\documentclass[12pt,leqno]{amsart}
\usepackage{amsmath,stmaryrd}
\usepackage{amssymb}
\usepackage{amsthm}
\usepackage{epsf}
\usepackage{graphicx}
\usepackage{esint} 
\usepackage{dsfont}
\usepackage[pagebackref]{hyperref} 
\hypersetup{pdfpagemode=FullScreen,  colorlinks=true} 
\usepackage{enumerate} 
\usepackage{xcolor,bbm,easybmat,bm}
\usepackage{tikz}
\usepackage{pgfplots}
\pgfplotsset{compat=1.18}
\usetikzlibrary{decorations.pathreplacing,arrows.meta}

\usepackage{amsfonts}
\usepackage{amsrefs}
\usepackage{verbatim}
\usepackage{booktabs}
\usepackage{color}
\usepackage{graphicx,float}

\numberwithin{equation}{section}

\newcommand\footnoteref[1]{\protected@xdef\@thefnmark{\ref{#1}}\@footnotemark}
\makeatother

\newcommand{\dr}{\partial}

\DeclareMathOperator{\divg}{div}

\DeclareMathOperator{\dist}{dist}
\DeclareMathOperator{\supp}{supp}
\DeclareMathOperator{\diam}{diam}

\DeclareMathOperator{\loc}{loc}

\newcommand{\1}{{\mathds 1}}

\newcommand{\ms}{\medskip}

\newcommand{\R}{\mathbb R}
\newcommand{\N}{\mathbb N}
\newcommand{\cF}{\mathcal F}
\renewcommand{\H}{\mathcal H}

\newcommand{\ep}{\hfill $\square$ \medskip}

\newcommand{\wt}{\widetilde}

\newcommand{\Base}{\mathcal{O}}

\newcommand{\LL}{\mathcal L}
\renewcommand{\L}{L}

\newcommand{\p}{\mathfrak p}

\newcommand{\C}{\mathcal C}

\newcommand{\OO}{\mathcal O}
\newcommand{\cS}{\mathcal S}
\newcommand{\cSS}{\mathcal S (q,\tau,r,r_0,\frac16\hbar)}
\newcommand{\cSSm}{\mathcal S (q,\tau,r,r_0,\frac1{12}\hbar)}
\newcommand{\W}{\mathcal W}

\newcommand{\Rn}{\mathbb R^n}
\newcommand{\norm}[1]{\left\Vert#1\right\Vert}
\newcommand{\abs}[1]{\left\vert#1\right\vert}
\newcommand{\br}[1]{\left(#1\right)}
\newcommand{\set}[1]{\left\{#1\right\}}

\renewcommand{\d}{\, \mathrm{d}} 

\newcommand{\om}{\Omega}
\newcommand{\pom}{\partial\Omega}

\newcommand{\E}{\mathsf{E}} 
\newcommand{\Lloc}{\L_{\operatorname{loc}}} 
\newcommand{\HT}{H_t} 
\newcommand{\dhalf}{D_t^{1/2}} 
\newcommand{\Hdot}{\dot{H}\protect{\vphantom{H}}} 
\newcommand{\mS}{{\mathcal S}} 
\newcommand{\IC}{\mathbb{C}}
\newcommand{\pd}{\partial}
\newcommand{\cl}[1]{\overline{#1}} 

\newcommand{\ree}{{\mathbb{R}^{n}}}

\newcommand{\dint}{\int\!\!\!\!\!\int}
\def\Yint#1{\mathchoice
	{\YYint\displaystyle\textstyle{#1}}%
	{\YYint\textstyle\scriptstyle{#1}}%
	{\YYint\scriptstyle\scriptscriptstyle{#1}}%
	{\YYint\scriptscriptstyle\scriptscriptstyle{#1}}%
	\!\dint}
\def\YYint#1#2#3{{\setbox0=\hbox{$#1{#2#3}{\iint}$}
		\vcenter{\hbox{$#2#3$}}\kern-.51\wd0}}
\def\longdash{\mkern-1.5mu{-}\mkern-7.5mu{-}} 

\def\fiint{\Yint\longdash}

\newcommand{\Z}{{\mathbb Z}}

\theoremstyle{plain}
\newtheorem{theorem}[equation]{Theorem}
\newtheorem{lemma}[equation]{Lemma}

\newtheorem{proposition}[equation]{Proposition}

\newtheorem{definition}[equation]{Definition}

\theoremstyle{definition}

\theoremstyle{remark}
\newtheorem{remark}[equation]{Remark}

\begin{document}

\title[The $L^p$ Neumann problem for parabolic operators]{The $L^p$ Neumann problem for parabolic operators with coefficients satisfying small Carleson condition}

\author[Dindo\v{s}]{Martin Dindo\v{s}}
\address{School of Mathematics, 
The University of Edinburgh and Maxwell Institute of Mathematical Sciences, Edinburgh, UK}
\email{M.Dindos@ed.ac.uk}

\author[Li]{Linhan Li}
\address{School of Mathematics, The University of Edinburgh and Maxwell Institute of Mathematical Sciences, Edinburgh, UK}
\email{linhan.li@ed.ac.uk}

\author[Pipher]{Jill Pipher}
\address{Department of Mathematics, 
 Brown University, RI, US}
\email{jill\_pipher@brown.edu}


\begin{abstract} 
In this paper, we resolve the question of whether the Neumann problem for the parabolic PDE
$-\partial_tu + \divg(A\nabla u)=0$ on a Lipschitz cylinder $\mathcal O\times\R$ is solvable for some $p\in (1,\infty)$ under the assumption that the matrix $A$ is elliptic with bounded and measurable coefficients that satisfy a natural Carleson condition (a parabolic analog of the so-called DKP-condition).

We prove that for any $1<p<\infty$ the Neumann problem is solvable under the assumption that both the Carleson norm of coefficients and the Lipschitz constant of the domain are sufficiently small (with dependence on $p$).  The question of what happens in the \lq\lq large Carleson norm/large Lipschitz constant" regime remains open, and even for elliptic PDEs this question has only been resolved in two dimensions.

This paper complements results from our recent manuscript in which the parabolic regularity problem has been fully resolved
in both the small and large Carleson norm regimes. Previously, the Dirichlet problem had been resolved under the same conditions by various authors. 
On Lipschitz graph domains we also estimate the boundary and nontangential
half-time derivatives for all $1<p<\infty$, allowing a Carleson drift term.
\end{abstract}

\maketitle


\ms\noindent 

\tableofcontents

\section{Introduction}

In this paper, we present the first results on the solvability of the
Neumann problem in all dimensions for a class of parabolic operators in divergence form with {\it time-varying coefficients} satisfying a minimal smoothness assumption. It is assumed that the coefficients of the matrix (not necessarily symmetric) defining the operators under consideration satisfy a well-studied Carleson measure condition. When the norm of this measure is sufficiently small, we obtain nontangential maximal function estimates on the gradient of the solution when the Neumann data belong to $L^p$, for all $1<p<\infty$.

Quantifying the regularity of these divergence form operators via a
Carleson measure condition on the coefficients
has a long history. This regularity condition on the coefficients arose from a 
particular change of 
variables that is relevant to the study of solutions to smooth elliptic or parabolic equations in Lipschitz domains (\cites{D,N}).  Dahlberg had conjectured that this condition gives rise to a ``good" class (in terms of solvability of boundary value problems) of
elliptic operators.  The conjecture, in the case of the Dirichlet problem, was resolved for elliptic equations in \cite{KP} and \cite{DPP1} and extended to parabolic equations in \cite{DDH} and \cite{DPP2}.

A related boundary value problem, the Regularity problem,  imposes one derivative of additional smoothness on the Dirichlet data, and asks for nontangential maximal function estimates on gradients of solutions. For elliptic operators, both the Regularity and Neumann problems were settled in \cite{DPR} in the small Carleson norm regime. When the Carleson condition is not small, matters are much less well understood. It took a long time to settle the elliptic
Regularity problem in this regime - this was only accomplished in 2023 independently and simultaneously in \cite{DHP} and (for very general domains) in \cite{MPT}. In dimension two, \cite{KR} had observed that there is a special dual relationship between Neumann problems and Regularity problems within a class of operators. When the regularity assumptions on the operator are preserved under this duality, Neumann and Regularity are equivalent. This does not persist in higher dimensions. 

Turning to the parabolic theory, the operators we consider here
have the form 
\begin{equation}\label{E:pde}
			\LL u:= -\dr_t u +\divg (A \nabla u)=0   \quad\text{in } \Omega, 
\end{equation}
and will be defined in a domain $\Omega = \mathcal O \times \R$ where $\mathcal O = \R^n_+$ or $\mathcal O$ is a bounded or unbounded Lipschitz domain.

The matrix $A= [a_{ij}(X, t)]$ is an $n\times n$ matrix satisfying the uniform ellipticity condition with $X \in \mathcal O$, $t\in \R$.
That is, there exist positive constants $\lambda$ and $\Lambda$ such that
\begin{equation}
	\label{E:elliptic}
	\lambda |\xi|^2 \leq \sum_{i,j} a_{ij}(X,t) \xi_i \xi_j \leq \Lambda |\xi|^2
\end{equation}
for almost every $(X,t) \in \Omega$ and all $\xi \in \R^n$. We do not assume that $A$ is symmetric, and for a nonsymmetric matrix the second inequality does not bound its entries. We therefore also require that $|A(X,t)|\le\Lambda$ for almost every $(X,t)\in\Omega$, as is implicit in calling $A$ bounded.

The importance of understanding the behavior of solutions to parabolic operators with rough coefficients or in domains with rough boundaries has been evident since the work of Nash in \cite{Nash}, which made it possible, and interesting, to study boundary value problems for equations with minimal smoothness. 
Many of the earlier developments for operators in such settings have focused on the time-independent case, $A(X,t) = A(X)$, or on
scalar equations in time-varying domains. New difficulties arise when the matrix depends on the time variable as well, for example in trying to use
Fourier transform methods. Despite the obstacles, the time-varying parabolic Dirichlet problems have now been settled for the operators under consideration in this paper in both the small and large Carleson norm regimes. 
The first results on solvability of the Dirichlet problem for the small Carleson norm regime were obtained in \cite{DH18}, for 
$p \geq 2$, and improved to $p > 1$ in \cite{DDH}.  The solvability of the Dirichlet problem for the large Carleson norm regime was 
established in \cite{DPP2}. For more history on solvability of parabolic boundary value problems with coefficients satisfying various mild regularity conditions, we refer the reader to  the introductions of \cite{DDH} and \cite{Din23}.
Recently, adapting new ideas and methods introduced in \cite{DHP}, the solvability of the Regularity problem for parabolic operators was obtained in \cite{DLP1}, with sharp results in both small and large norm regimes.

As in the elliptic setting, the parabolic Neumann problem is harder to understand and resolve. The only results we are aware of for non-smooth coefficients in the time-varying setting are in \cite{AEN}. There, the authors imposed structural conditions on their matrices different from those considered here, namely, requiring independence of the transversal spatial variable, and other constraints. This, in turn, built upon the prior results of \cite{Ny2016}, where the additional condition of time independence was assumed on the matrices. 
We also note that the parabolic Dirichlet and Regularity problems for transversally independent matrices without any additional assumptions on the coefficient matrix (beyond the natural ones: real, bounded measurable and elliptic) were resolved in \cite{AEN2} and \cite{DPU}, respectively.

In this paper, our matrices are allowed to be non-symmetric and time-varying, thus obtaining the full parabolic analog of the elliptic Neumann results of
\cite{DPR} as well as complementing our study of the parabolic Regularity problem in \cite{DLP1}.

\medskip

Specifically for parabolic equations, we assume that 
the measure defined by
\begin{equation}\label{E:1:carl}
d\mu = {\sup_{B_{\delta(X,t)/2}(X,t)}}\left( \delta(Y,s)|\nabla A|^2 + \delta(Y,s)^3|\partial_t A|^2  \right) dX\,dt
\end{equation}
is a Carleson measure on $\Omega$ with Carleson norm $\|\mu\|_C$,  or alternatively we ask for
\begin{equation}\label{E:1:carl2}
d\mu =  \delta^{-1}(X,t)\left({\mathrm{osc}_{B_{\delta(X,t)/2}(X,t)}} A\right)^2dX\,dt,
\end{equation}
to be a Carleson measure with $\|\mu\|_C<\infty$. This is an alternative condition introduced in \cite{DPR}.  Note that \eqref{E:1:carl} implies that
\begin{equation}\label{E:1:bound}
\delta(X,t)|\nabla A| + \delta(X,t)^2|\partial_t A| \leq K<\infty
\end{equation}
with $K\lesssim\norm{\mu}_C^{1/2}$.
Here and in the sequel, $\delta(X,t)$  denotes the parabolic distance to the boundary of $\om$, which is defined as: 
\[
\delta(X,t)=\inf_{(Y,\tau)\in\pom}\br{\abs{X-Y}^2+\abs{t-\tau}}^{1/2}\quad\mbox{and}\quad
\mathrm{osc}_B A=\sup_{(X,t),(Y,s)\in B}|A(X,t)-A(Y,s)|.
\]

(See \eqref{Cmeasure} in the next section for the definition of Carleson measures.)

\smallskip

Even in the elliptic setting, solvability of the Neumann problem without the small Carleson assumption remains open in dimensions larger than $2$.

Let us make some closer comparisons of our result with those of \cite{DLP1} and \cite{DPR}.
In \cite{DLP1}, we obtained solvability of the Regularity problem when the matrix $A(X,t)$ satisfied conditions \eqref{E:1:carl} and \eqref{E:1:bound}, for boundary data in $L^p$, in a range of $p$ dual to that of solvability of the Dirichlet problem for the adjoint operator. In particular, since the Dirichlet problem for such operators was solved for all $1<p<\infty$ under the
small Carleson condition in \cite{DDH}, this gives
solvability of the Regularity problem for all $1<p<\infty$ for operators with small Carleson norm.  
A key difference between these Regularity and Neumann boundary value problems 
is that the boundary data for the parabolic Regularity problem include a half derivative in time, $D^{1/2}_t f$, of the Dirichlet data, along with the spatial tangential derivatives. This is a further departure of the parabolic setting from the elliptic setting. We elaborate on this issue in the Appendix.

As alluded to earlier, in \cite{DPR},
the elliptic Regularity and Neumann problems were solved for boundary data in $L^p$ for all $1<p<\infty$ assuming that the Carleson norm, $\|\mu\|_C$, is sufficiently small.
However, to obtain this result in the parabolic setting, we cannot follow the elliptic roadmap precisely. In \cite{DPR}, the solvability of the elliptic Regularity problem was proven first, providing a bound on the nontangential maximal function in terms of the tangential derivatives on the boundary. Thus the solvability of the Neumann problem reduced to the $L^p$ equivalence of tangential and conormal derivatives on the boundary.  But in the parabolic setting, the Regularity data involve a complicating (non-local) half derivative in time. Hence a different approach is required to obtain the nontangential estimates for our Neumann problem. We will derive the appropriate nontangential maximal function estimates by 
making use of the so-called {\it $p$-adapted square function}, which first appeared in
\cite{DPP1}. Finally, a key component of the methodology for solving the Neumann problem in both the elliptic and parabolic settings is the discovery (first observed in \cite{DPR}) that there is a useful PDE satisfied by the conormal derivative. 

\medskip
We now state the main theorem of this paper.

\begin{theorem}\label{MainT} Let $\om=\OO\times\R$ with $\OO$ being a bounded or unbounded Lipschitz domain with Lipschitz character $(\ell,N,C_0)$ (see Definition~\ref{DefLipDomain}). Let $\LL= -\dr_t  +\divg (A \nabla \cdot)$ be a parabolic operator on $\om$ satisfying \eqref{E:elliptic} and either \eqref{E:1:carl} or \eqref{E:1:carl2}. Let $p\in(1,\infty)$. Then there exists $\delta=\delta(\lambda,\Lambda, n,p)$ such that if $\max\set{\ell,\norm{\mu}_{C}}<\delta$, then the $L^p$ Neumann problem $(N)_p^{\LL}$ for $\LL$ is solvable in $\om$.
In particular, for any $g\in\L^p(\pom)$ the solution $u$ solving $\LL u=0$ in $\Omega$ with Neumann datum $g$ satisfies the estimate
$$\|\tilde N(\nabla u)\|_{\L^p(\pom)}\le C(\lambda,\Lambda, n,p)\|g\|_{\L^p(\pom)}.$$
\end{theorem}

We also ask whether the Neumann data control the nontangential maximal
functions of $D_t^{1/2}u$ and $H_tD_t^{1/2}u$, where $H_t$ is the Hilbert
transform in time. Such estimates are known for the $L^2$ Neumann problem
in the half-space under the assumptions of \cite{AEN}, and for the
Regularity problem by \cite{Din23}. Appendix \ref{APA} shows that they
fail on bounded spatial domains and proves the following result on unbounded
Lipschitz graph domains, without any smallness of the Carleson norm and we also allow presence of a drift term.

\begin{theorem}\label{Thalf}
Let $\Omega=\{(x',x_n):x_n>\phi(x')\}\times\R$, where $\phi$ is
Lipschitz, and let $A$ satisfy \eqref{E:elliptic} and \eqref{E:1:carl}.
Consider
\[
 \mathcal L_{\mathbf b}=-\partial_t+\divg(A\nabla\cdot)
                                  +\mathbf b\cdot\nabla.
\]
Suppose that $\mathbf b$ has weak first spatial and time derivatives,
and, with $\delta$ denoting distance to the boundary,
\begin{equation}\label{eq:halfDriftHyp}
\begin{split}
 \delta|\mathbf b|+\delta^2|\nabla\mathbf b|
                  +\delta^3|\partial_t\mathbf b|&\le K_{\mathbf b}<\infty,\\
 d\mu_{\mathbf b}:=\big(\delta|\mathbf b|^2
             +\delta^3|\nabla\mathbf b|^2
             +\delta^5|\partial_t\mathbf b|^2\big)\,dX\,dt
 &\quad\text{is Carleson}.
\end{split}
\end{equation}
For every $1<p<\infty$, each energy solution of
$\mathcal L_{\mathbf b}u=0$ with
$\tilde N(\nabla u),S(\nabla u)\in L^p(\partial\Omega)$ satisfies
\begin{equation}\label{eq:halfDriftConclusion}
\begin{split}
 &\|D_t^{1/2}u\big|_{\pom}\|_{L^p(\partial\Omega)}
 +\|\tilde N(D_t^{1/2}u)\|_{L^p(\partial\Omega)}
 +\|\tilde N(H_tD_t^{1/2}u)\|_{L^p(\partial\Omega)}\\
 &\hspace{15mm}\le C\big(\|\tilde N(\nabla u)\|_{L^p(\partial\Omega)}
                         +\|S(\nabla u)\|_{L^p(\partial\Omega)}\big).
\end{split}
\end{equation}
The constant depends only on $n,p,\lambda,\Lambda$, the two Carleson
norms, $K_{\mathbf b}$ and $\|\nabla\phi\|_\infty$. No smallness or
boundary solvability assumption is imposed.

When $\mathbf b=0$, the assumption $S(\nabla u)\in L^p$ is unnecessary
and the right-hand side of \eqref{eq:halfDriftConclusion} can be
replaced by $C\|\tilde N(\nabla u)\|_{L^p(\partial\Omega)}$.

In particular, for the operator $\mathcal L=\mathcal L_0$, let $g\in L^p(\partial\Omega)$
and let $u$ be its Neumann solution with datum $\partial_\nu^Au=g$.
Then
\[
 \|\tilde N(D_t^{1/2}u)\|_{L^p(\partial\Omega)}
 +\|\tilde N(H_tD_t^{1/2}u)\|_{L^p(\partial\Omega)}
 \le C\|g\|_{L^p(\partial\Omega)},
\]
provided $(N)_p^{\mathcal L}$ is solvable. The constant in this
last estimate also depends on the $(N)_p^{\mathcal L}$ solvability constant.
\end{theorem}

\subsection{A guide to reading this paper}
Section \ref{Def} gives the definitions, constructs energy solutions,
and explains the change of variables for Lipschitz graph domains.
Section \ref{Sprf} assembles the proof of Theorem~\ref{MainT} in the
half-space. It also establishes the a priori finiteness of the gradient
maximal norm for smooth coefficients and data, which is needed to apply
the estimates proved later.

The proof uses the equations for the tangential derivatives and for
$H=\sum_j a_{nj}\partial_j u$.
Ellipticity then gives control of the remaining spatial derivative.
Section \ref{S.Caccio} proves the required Caccioppoli inequalities.
Section \ref{SS:43} bounds the gradient maximal function by the
$p$-adapted square function, and Section \ref{S.Sp_bd} estimates that
square function in terms of the Neumann data. Section \ref{S.S<N}
proves the usual square-function bound by the maximal function.

Section \ref{S.bddLip} treats bounded Lipschitz cylinders. The proof
combines localization, absorption at a small time scale, and exponential
decay away from the time support of the Neumann data.

Appendix \ref{APA} proves the half-time derivative estimates of
Theorem~\ref{Thalf} on Lipschitz graph domains and gives the obstruction
on bounded spatial domains. The solutionwise graph-domain proof uses
a bilinear height identity for $1<p\le2$ and a sharp-function estimate
for $p>2$, including the drift terms in \eqref{eq:halfDriftHyp}.
The Neumann-data assertion of Theorem~\ref{Thalf} follows from
$(N)_p^{\mathcal L}$ and approximation by energy data.

\section{Definitions}\label{Def}

\begin{definition}
$\Z \subset \R^n$ is an $\ell$-cylinder of diameter $d$ if there
exists an orthogonal coordinate system $(x,x_n)$  with $x\in\mathbb R^{n-1}$ and $x_n\in\mathbb R$ such that
\[
\Z = \{ (x,x_n)\; : \; |x|\leq d, \; -(\ell+1) d \leq x_n \leq (\ell+1) d \}
\]
and for $s>0$,
\[
s\Z:=\{(x,x_n)\;:\; |x|\le sd, -(\ell +1)s d \leq x_n \leq (\ell +1)s d \}.
\]
\end{definition}

\begin{definition}\label{DefLipDomain}
$\mathcal O\subset \R^n$ is a Lipschitz domain with Lipschitz
`character' $(\ell,N,C_0)$ if there exists a positive scale $r_0\in (0,
\infty]$ and
at most $N$ $\ell$-cylinders $\{{\Z}_j\}_{j=1}^N$ of diameter $d$, with
$\frac{r_0}{C_0}\leq d \leq C_0 r_0$ such that 
\vglue2mm

\noindent (i) $8{\Z}_j \cap {\partial\mathcal O}$ is the graph of a Lipschitz
function $\phi_j$, $\|\nabla\phi_j \|_\infty \leq \ell \, ;
\phi_j(0)=0$,\vglue2mm

\noindent (ii) $\displaystyle {\partial\mathcal O}=\bigcup_j ({\Z}_j \cap {\partial\mathcal O}
)$,

\noindent (iii) $\displaystyle{\Z}_j \cap \mathcal O \supset \left\{
(x,x_n)\in\mathcal O \; : \; |x|<d, \; \mathrm{dist}\left( (x,x_n),{\partial\mathcal O}
\right) \leq \frac{d}{2}\right\}$.

\noindent (iv) Each cylinder $\displaystyle{\Z}_j$ contains points from $\mathcal O^c={\mathbb R^n}\setminus\mathcal O$.

\noindent (v) If $r_0<\infty$ the domain $\mathcal O$ is a bounded set.
\vglue1mm
\end{definition}

\noindent{\it Remark.} If the scale $r_0$ is finite, then the domain $\mathcal O$ from the definition above is a bounded Lipschitz domain. However, we shall also allow both $r_0,\, d$ to be infinite, and in this case, since $\displaystyle{\Z}=\mathbb R^n$,
we have that  $\mathcal O$ can be written in some coordinate system as
\begin{equation}\label{eq.O=lipgph}
    \mathcal O = \{(x,x_n): x_n > \phi(x)\}\quad\mbox{ where $ \phi:\mathbb R^{n-1} \rightarrow \mathbb R$ is a Lipschitz function.}
\end{equation}

The set $\mathcal O\times\R$ will be called a parabolic Lipschitz cylinder with Lipschitz base $\mathcal O$.
We will need an extension of the definition of spaces $\dot L^p_{1,1/2}$ to Lipschitz cylinders.

 \begin{definition}
A parabolic cube on $\R^n\times\R$  centered at $(X,t)$ with sidelength $r$ is defined as
$$    Q_r(X,t):=\{ (Y, s) \in \R^{n}\times\R : |x_i - y_i| < r \ \text{ for } 1 \leq i \leq n, \ | t - s |^{1/2} < r \}.
$$
When writing a lowercase point $(x,t)$, we mean a boundary parabolic cube on $\R^{n-1}\times\R$,
which has an analogous definition but in one less spatial dimension:
\begin{equation}\label{eqdef.bdypcube}
    Q_r(x,t):=\{ (y, s) \in \R^{n-1}\times\R : |x_i - y_i| < r \ \text{ for } 1 \leq i \leq n-1, \ | t - s |^{1/2} < r \}.
\end{equation}
A parabolic ball on $\R^n\times\R$  centered at $(X,t)$ with radius $r$ is the ball
\begin{equation}
    B_r(X,t):=\{ (Y, s) \in \R^{n}\times\R : \|(X-Y,t-s)\|<r \},
\end{equation}
where $\|(\xi,\tau)\|$ on $\R^{n} \times \R$  is defined as
  the unique positive solution $\rho$ to the following equation
\begin{equation}
	\frac{|\xi|^2}{\rho^2} + \frac{\tau^2}{\rho^4} = 1.
\end{equation}
Recall that $\|\cdot\|$ scales as the parabolic distance function
\[
d_p((X,t),(Y,s)) := \br{\abs{X-Y}^2+\abs{t-s}}^{1/2}\sim \|(X-Y,t-s)\|.\]

For parabolic balls at the boundary we use notation $\Delta_r(X,t)=B_r(X,t)\cap \pom$. In the special case $\Omega=\R^n_+\times\R$ we drop the last coordinate and also write $\Delta_r(x,t)$ with the understanding that the ball is centered at $(X,t)=(x,0,t)$.

We sometimes use the notation $B_{X,t}:=B_{\delta(X,t)/2}(X,t)$ for $(X,t)\in\om$.  
\end{definition}

We define (parabolic) Carleson measures.

\begin{definition}
A measure $\mu$ is a Carleson measure on $\Omega = \mathcal O \times \R$  if there exists a constant $M$ such that for $(X,t) \in \partial \Omega$ and
$0<r<\diam(\Omega)$,
\begin{equation}\label{Cmeasure}
\mu(B_r(X,t) \cap \Omega) \leq  M r^{n+1}.
\end{equation}
The infimum over all $M$ such that \eqref{Cmeasure} holds is the Carleson norm, $\norm{\mu}_C$, of $\mu$.
\end{definition}

\begin{definition}
For $a>0$ and $(q,\tau)\in\pom$, unless otherwise defined, we denote the nontangential parabolic cones by
\begin{equation}
    \Gamma_a(q,\tau):=\set{(X,t)\in\om: d_p((X,t),(q,\tau))<(1+a)\delta(X,t)},
\end{equation}

and $\delta(\cdot)$ is the parabolic distance to the boundary: 
\[
\delta(X, t) = \inf_{(q,\tau)\in\pom}
d_p((X, t),(q, \tau)).\]
We also use the truncated parabolic cones: for $r>0$,
\[
\Gamma_a^r(q,\tau):=\set{(X,t)\in\om: d_p((X,t),(q,\tau))<(1+a)\delta(X,t),\, \delta(X,t)<r}
\]
is the parabolic cone with vertex $(q,\tau)$ truncated at height $r$.
\end{definition}

\begin{definition}
For $w\in L^\infty_{\loc}(\om)$, we define the nontangential maximal function of $w$ as
\[
 N_a(w)(q,\tau):=\sup_{(X,t)\in\Gamma_a(q,\tau)}\abs{w(X,t)} \quad\text{for }(q,\tau)\in\pom.
\]
If $w\in L^p_{\loc}(\om)$, $p\in(0,\infty)$, we need the modified nontangential maximal function
\begin{equation}\label{def.Nap}
    \wt N_{a,p}(w)(q,\tau):=\sup_{(X,t)\in\Gamma_a(q,\tau)}\br{\fiint_{B_{\delta(X,t)/2}(X,t)}\abs{w(Y,s)}^pdYds}^{1/p} \quad\text{for }(q,\tau)\in\pom,
\end{equation}
where $B_r(X,t)$ is a parabolic ball centered at $(X,t)$ with ``radius" $r$, that is, 
\[B_r(X,t):= \set{(Y,s)\in\Rn\times\R: d_p((Y,s),(X,t))<r}.\]
We simply denote 
\begin{equation}
    \wt N_{a}(w):=\wt N_{a,2}(w)
\end{equation}
when $p=2$ in \eqref{def.Nap}.
\end{definition}

\noindent We now introduce the parabolic square and area functions. 

\begin{definition}
For a function $w$ with $\nabla w\in L^2_{\loc}(\om)$, we define the square function of $w$ by 
\begin{equation}
S_a(w)(q,\tau):=\br{\iint_{\Gamma_a(q,\tau)}\abs{\nabla w(X,t)}^2\delta(X,t)^{-n}dX\,dt}^{1/2} \quad\text{for }(q,\tau)\in\pom,
\end{equation}
and for $p\in(1,\infty)$, define the $p$-adapted square function by
\begin{equation}
    S_{p,a}(w)(q,\tau):=\br{\iint_{\Gamma_a(q,\tau)}\abs{\nabla w(X,t)}^2|w|^{p-2}\delta(X,t)^{-n}dX\,dt}^{1/p} \quad\text{for }(q,\tau)\in\pom.
\end{equation}

We also need the following version of the area function defined as in \cite{DH18}. For a function $w: \om\to\R$, we define
\begin{equation}\label{DefArea}
     A_a(w)(q,\tau):=\left(\iint_{\Gamma_a(q,\tau)}|\partial_t w(X,t)|^2\delta(X,t)^{-n+2}\,dX\,dt\right)^{1/2}
\end{equation}
\end{definition}
If $w$ satisfies the parabolic PDE \eqref{E:pde}, then heuristically $\partial_t w\sim \nabla^2w$, and so the operators $A_a$ and $S_a$ can be related using Caccioppoli type estimates. In particular, as shown in \cite[Section 6]{DLP1}, for a solution $u$ to \eqref{E:pde}, one can prove the pointwise bound 
\begin{equation}
    A_a^2(\nabla u)(q,\tau)\lesssim S_a^2(\nabla u)(q,\tau)
+\iint_{\Gamma_{a'}(q,\tau)} (|\partial_t A|^{2}x_n^2+|\nabla A|^{2}) |\nabla u|^{2}x_n^{-n+2}\,dX\,dt,
\end{equation}
as well as the $L^p$ estimate for $p>1$
  \begin{equation}\label{A<S+cN.Lp}
    \|A_a(\nabla u)\|_{L^p(\R^{n-1}\times\R)}\le C \|S_a(\nabla u)\|_{L^p(\R^{n-1}\times\R)}+C\|\mu\|_C^{1/2}\|N(\nabla u)\|_{L^p(\R^{n-1}\times\R)}.
\end{equation}
We shall derive similar estimates for $H=a_{nj}\dr_ju$ in Section~\ref{ss1}. 
\medskip

It is well-known (using a level-sets argument) that for $p\in(0,\infty)$, the $L^p$ norms of $N_a(w)$, $\wt{N}_a(w)$, $S_a(w)$, $A_a(w)$ and $\tilde A_a(w)$ are invariant under changes of $a$ up to a constant multiple. For this reason, we  omit the dependence on the aperture $a$ of the cones when the aperture need not be specified.

\subsection{Reinforced weak and energy solutions}

We recall the reinforced weak solutions introduced in \cite{AEN}, using the weaker definition from \cite{Din23}, which requires only a local half derivative in time.
\vglue1mm

If $\mathcal O$ is an open subset of $\mathbb R^{n} $, we let $\H^1(\mathcal O)=\W^{1,2}(\mathcal O)$ be the standard Sobolev space of real-valued functions $v$ defined on $\mathcal O$, such that $v$ and $\nabla v$ are in $\L^{2}(\mathcal O;\R)$ and $\L^{2}(\mathcal O;\R^n)$, respectively. A subscripted `$\loc$' will indicate that these conditions hold locally.

We shall say that $u$ is a \emph{reinforced weak solution} of $-\partial_t u + \divg(A\nabla u)=0$ on $\Omega=\mathcal O\times \R$ if
\begin{align*}
 u\in \dot {\E}_{\loc}(\Omega):= \H_{\loc}^{1/2}(\R; \L^2_{\loc}(\mathcal O)) \cap \Lloc^2(\R; \W^{1,2}_{\loc}(\mathcal O))
\end{align*}
and if for all $\phi,\psi \in \C_0^\infty(\Omega)$,
\begin{equation}\label{2.10}
\iint_{\Omega}\left[
 A\nabla u\cdot{\nabla (\phi\psi)}+ \HT\dhalf (u\psi)\cdot {\dhalf \phi}+ \HT\dhalf (u\phi)\cdot {\dhalf \psi}\right]\, \d X \d t=0.
 \end{equation}
Here, $\dhalf$ is the half-order derivative and $\HT$ is the Hilbert transform with respect to the $t$ variable, normalized so that $\partial_{t}= \dhalf \HT \dhalf$. We have the following properties of $\dhalf$  and $\HT$:
\begin{equation}\label{eq:algebra}
D^{1/2}_tD^{1/2}_t=-H_t\partial_t,\quad
D^{1/2}_t\ \mbox{self-adjoint on }L^2(dt),\quad
H_t\ \mbox{skew-adjoint and isometric},
\end{equation}
$D^{1/2}_t$ commutes with $\nabla$ and annihilates constants, and
\begin{equation}
D^{1/2}_t f(t)=c_0\int_\R\frac{f(t)-f(s)}{|t-s|^{3/2}}\,ds,
\qquad
\int_{|\tau|>a^2}|\tau|^{-3/2}d\tau=\frac4a .
\end{equation}

The space $\Hdot^{1/2}(\R)$ is the homogeneous Sobolev space of order 1/2 - the completion of $\C_0^\infty(\R)$ in the norm $\|\dhalf (\cdot)\|_{2}$ -  and it embeds into the space $\mS'(\R)/\IC$ of tempered distributions, modulo constants.
The local space $\H_{\loc}^{1/2}(\R)$ consists of functions $u$ such that $u\phi\in \Hdot^{1/2}(\R)$ for all $\phi\in \C_0^\infty(\R)$. As shown in \cite{Din23} the space $\H_{\loc}^{1/2}(\R)$ is larger than $\Hdot^{1/2}(\R)$.
For  $u\in\Hdot^{1/2}(\R; \L^2_{\loc}(\mathcal O)), $ \eqref{2.10} simplifies to 
\begin{equation}\nonumber
\iint_{\Omega}\left[
 A\nabla u\cdot{\nabla \phi}+ \HT\dhalf u\cdot {\dhalf \phi}\right]\, \d X \d t=0.
 \end{equation}
  
 Our definition has the advantage that, in taking a cut-off of the function $u$,  we may improve the decay of $D^{1/2}_tu$ at infinity. 

At this point we remark that for any $u\in \Hdot^{1/2}(\R)$ and $\phi,\psi\in \C_0^\infty(\R)$ the formula
\begin{align*}
 \int_{\R} \left[\HT\dhalf (u\psi)\cdot {\dhalf\phi}+\HT\dhalf (u\phi)\cdot {\dhalf\psi}\right]\d t = - \int_{\R} u \cdot {\partial_{t}(\phi\psi)} \d t
\end{align*}
holds, where on the right-hand side we use the duality form extension of the complex inner product of $\L^2(\R)$, between $\Hdot^{1/2}(\R)$ and its dual $\Hdot^{-1/2}(\R)$. By taking $\psi=1$ on the set where $\phi$ is supported, it follows that a reinforced weak solution is a weak solution in the usual sense on $\Omega$ since it satisfies
$u\in \Lloc^2(\R; \W^{1,2}_{\loc}(\mathcal O))$ and for all $\phi\in \C_0^\infty(\Omega)$,
\begin{align*}
 \iint_{\Omega} A\nabla u\cdot{\nabla \phi} \d X \d t - \iint_{\Omega} u \cdot {\partial_{t}\phi} \d X  \d t=0.
\end{align*}
 This implies $\pd_{t}u\in \Lloc^2(\R; \W^{-1,2}_{\loc}(\mathcal O))$. Conversely, any  weak solution $u$ in    $ \H_{\loc}^{1/2}(\R; \L^2_{\loc}(\mathcal O))$ is a reinforced weak solution.\vglue2mm

Specializing to the case $\Omega=\mathcal O\times\mathbb R$, where $\mathcal O$ is either a bounded or an unbounded Lipschitz domain,
 we say that a 
reinforced weak solution $v\in \dot{\E}_{\loc}(\mathcal O\times\mathbb R)$ belongs to the \emph{energy class} $\dot \E(\mathcal O\times\mathbb R)$ if
\begin{align*}
 \|v\|_{\dot \E} := \bigg(\|\nabla v\|_{\L^2(\mathcal O\times\mathbb R)}^2 + \|\HT \dhalf v\|_{\L^2(\mathcal O\times\mathbb R)}^2 \bigg)^{1/2} < \infty.
\end{align*}
Consequently, these are called \emph{energy solutions}. When considered modulo constants, $\dot \E $ is a Hilbert space and it is in fact the closure of $\C_0^\infty\!\big(\,\cl{\mathcal O\times\mathbb R}\,\big)$ for the homogeneous norm $\|\cdot\|_{\dot \E}$
and it coincides with the space $\dot{L}^2_{1,1/2}(\Omega)$.

As shown in \cite{AEN} (with a small generalization), functions from $\dot \E $ have well-defined Dirichlet traces with values in
 the \emph{homogeneous parabolic Sobolev space} $\Hdot^{1/4}_{\pd_{t} - \Delta_x}(\partial\mathcal O\times\mathbb R)$. Here, $\Hdot^{s}_{\pm \pd_{t} - \Delta_x}(\R^n)$ is defined as the closure of Schwartz functions $v \in \mS(\ree)$ with Fourier support away from the origin in the norm $\|\cF^{-1}((|\xi|^2 \pm i \tau)^s \cF v)\|_2$. This yields a space of tempered  distributions modulo constants in $\Lloc^2(\ree)$ if $0 < s \leq 1/2$. 
 Conversely, any $g \in \Hdot^{1/4}_{\pd_{t} - \Delta_x}$ can be extended to a function $v \in  \dot \E$ with trace $v\big|_{\partial\mathbb R^{n+1}_+} = g$.  Using a partition of unity, we define 
 $\Hdot^{1/4}_{\pd_{t} - \Delta_x}(\partial\mathcal O\times\mathbb R)$ for $\mathcal O$ Lipschitz.
  
Hence, by the energy solution to $-\partial_tu + \divg(A\nabla u)=0$ with Dirichlet boundary datum $u\big|_{\partial\mathcal O\times\R} = f \in \Hdot^{1/4}_{\pd_{t} - \Delta_x}$ (understood in the trace sense) we mean $u \in \dot\E$ such that
\begin{align*}
a(u,v):= \iint_{\mathcal O\times\R} \left[A \nabla u \cdot{\nabla v} + \HT \dhalf u \cdot {\dhalf v}\right] \d X \d t = 0,
\end{align*}
holds for all $v \in \dot \E_0$, the subspace of $\dot \E$ with zero boundary trace.

Moving onto the Neumann problem, given any $u \in \dot \E(\Omega)$, the co-normal derivative $\partial^A_\nu u\Big|_{\partial\Omega}=:\langle A\nabla u,\nu\rangle \Big|_{\partial\Omega}=g$ is defined via the formula
\begin{align}\label{eq.NeumannBdy}
 \iint_{\mathcal O\times\R} \left[A \nabla u \cdot{\nabla v} + \HT \dhalf u \cdot {\dhalf v}\right] \d X \d t -\int_{\partial\mathcal O
 \times\R}gv \d x\d t= 0,
\end{align}
for all $v \in \dot \E$. Here, since the traces of $v$ belong to $\Hdot^{1/4}_{\pd_{t} - \Delta_x}(\partial\Omega)$ and all elements of the space $\Hdot^{1/4}_{\pd_{t} - \Delta_x}$ are realized by some $v \in \dot \E$,
 the Neumann boundary data must by duality  naturally belong to the space $\Hdot^{-1/4}_{\pd_{t} - \Delta_x}(\partial\Omega)$.

By \cite{AEN}, the key to solving these problems is the introduction of the modified sesquilinear form (introduced earlier in \cite{Ny2016}):
\begin{equation}
 a_\delta(u,v) := \iint_{\mathcal O\times\R} \left[A \nabla u \cdot {\nabla (1+\delta \HT) v} + \HT \dhalf u \cdot {\dhalf (1+\delta \HT) v}\right] \d X \d t,
\end{equation}
where $\delta$ is a  real number yet to be chosen. The Hilbert transform $\HT$ is a skew-symmetric isometric operator with inverse $-\HT$ on both $\dot \E$ and $\Hdot^{1/4}_{\pd_{t} - \Delta_x}$.  Hence, $1+\delta \HT$ is invertible on these spaces for any $\delta \in \R$. Hence for a fixed $\delta>0$  small enough, $a_\delta$ is coercive on $\dot \E$ since
\begin{equation}
 a_\delta(u,u) \ge (\lambda-\Lambda\delta )\|\nabla u\|_2^2 + \delta \|\HT \dhalf u \|_2^2.
\end{equation} 
In particular $\delta=\lambda/(\Lambda+1)$ would work.
To solve the Dirichlet problem we take an extension $w \in \dot \E$ of the data $f$ and apply the Lax-Milgram lemma to $a_\delta$ on $\dot \E_0$ to obtain some $u \in \dot \E_0$ such that 
\begin{align*}
 a_\delta(u,v) = - a_{\delta}(w,v) \qquad (v \in \dot \E_0).          \end{align*}
Hence, $u + w$ is an energy solution with data $f$. Should there exist another solution $v$, then $a_\delta(u+w-v,u+w-v) = 0$ and hence by coercivity $\|u +w - v \|_{\dot \E} = 0$. Thus the two solutions only differ by a constant. Thus the Dirichlet problem associated with our parabolic PDE is  \emph{well-posed} in the energy class. Similar arguments allow us to solve the Neumann problem by considering for the datum $g \in \Hdot^{-1/4}_{\pd_{t} - \Delta_x}(\partial\Omega)$ the solution $u
\in \dot \E$ such that
\begin{align*}
 a_\delta(u,v) = \langle g, \mbox{Tr } (1+\delta \HT)v\rangle\qquad (v \in \dot \E).          \end{align*}

\begin{definition}[$(D)_{p'}^{\LL^*}$]
    Let $p\in(1,\infty)$. We say that the $L^{p'}$ Dirichlet problem is solvable for $\LL^*$, denoted by $(D)_{p'}^{\LL^*}$, if there exists a constant $C>0$ such that for all $g \in \Hdot^{1/4}_{\pd_{t} - \Delta_x}\cap L^{p'}(\pom)$, the energy solution $u\in \dot \E(\om)$ to $\LL^* u=\partial_tu+\divg(A^T\nabla u) = 0$ in $\om$ with trace $u|_{\pom}=g$ satisfies the estimate
    \[
    \|N(u)\|_{L^{p'}(\pom)}\le C\|g\|_{L^{p'}(\pom)}.
    \]
\end{definition}
\begin{definition}[$(N)_p^{\LL}$]
     Let $p\in(1,\infty)$. We say that the $L^{p}$ Neumann problem is solvable for $\LL$, denoted by $(N)_p^{\LL}$, if there exists a constant $C>0$ such that for all $g \in \Hdot^{-1/4}_{\pd_{t} - \Delta_x}(\partial\Omega)\cap L^p(\pom)$, the energy solution $u\in \dot \E(\om)$ to $\LL u = 0$ in $\om$ with $\dr_\nu^Au|_{\pom}=g$ (defined as in \eqref{eq.NeumannBdy}) satisfies the estimate
    \[
    \|\wt N(\nabla u)\|_{L^p(\pom)}\le C\|g\|_{L^p(\pom)}.
    \]
\end{definition}

\subsection{Graph coordinates with constant Jacobian}

The change of variables used below must preserve the coefficient of the time
derivative as well as the spatial Carleson condition. We use the following
modification of the regularized graph map.
For a measure $\nu$ on $\R^n_+$, its spatial Carleson norm is
$\sup_Q\nu(Q\times(0,r))/|Q|$, where $Q\subset\R^{n-1}$ ranges over cubes
of side $r$. The spatial balls in the next lemma are Euclidean.

\begin{lemma}\label{lem:graphflat}
Let $\phi:\R^{n-1}\to\R$ be Lipschitz, $\ell=\|\nabla\phi\|_\infty$, and
$\OO=\{(x,x_n):x_n>\phi(x)\}$. There is a bi-Lipschitz map
$\Phi:\R^n_+\to\OO$, smooth in the interior, such that
\begin{equation}\label{eq:graphJac}
\Phi(x,0)=(x,\phi(x)),\qquad \det D\Phi\equiv c_0>0,
\qquad \delta(\Phi(x,h))\simeq_\ell h,
\end{equation}
where $c_0$ and the bi-Lipschitz constants depend only on $n$ and $\ell$.
Moreover $h|D^2\Phi(x,h)|\le C_\ell\ell$, and the measure
\begin{equation}\label{eq:graphGeomCM}
 d\nu_\Phi(x,h):=
 \sup_{(y,k)\in B_{h/2}(x,h)} k|D^2\Phi(y,k)|^2\,dx\,dh
\end{equation}
is a spatial Carleson measure on $\R^n_+$, with norm at most $C_\ell\ell^2$.
The constants are uniform for $\ell$ in a bounded interval.
\end{lemma}

\begin{proof}
Put $d=n-1$. Choose a radial $\theta\in C_0^\infty(\R^d)$ with integral one,
and write $\theta_h(x)=h^{-d}\theta(x/h)$,
$f(x,h)=\theta_h*\phi(x)$ and $q=\partial_h f$. Then
$\|\nabla_x f\|_\infty+\|q\|_\infty\le C\ell$. Choose
$c_0=C(1+\ell)$ large enough that $c_0+q\ge c_0/2$.
The map $P(x,h)=(x,c_0h+f(x,h))$ is bi-Lipschitz onto $\OO$, has the
required boundary values and satisfies $\det DP=c_0+q$.
We correct this determinant by a flow, following the volume-form method
of \cite{Moser65}; the estimates needed on this unbounded domain are given below.

The smooth radial function
$G=\left.\partial_h^2\theta_h\right|_{h=1}$ has compact support and integral
zero. The radial equation $\Delta_x\vartheta=G$ therefore has a smooth compactly
supported radial solution: its radial derivative is
$r^{1-d}\int_0^r s^{d-1}G(s)\,ds$, which vanishes outside the support of $G$,
and its additive constant is chosen to make $\vartheta$ vanish there. With
$\vartheta_h(x)=h^{-d}\vartheta(x/h)$ we obtain
$\Delta_x(\vartheta_h*\phi)=\partial_h^2 f$. Consequently the smooth vector field
\begin{equation}\label{eq:graphDiv}
 \begin{aligned}
 T(x,h)&:=\big(-h\nabla_x(\vartheta_h*\phi)(x),\ hq(x,h)\big),\\
 \divg T&=q,\qquad |T|\le C\ell h,\qquad |DT|\le C\ell.
 \end{aligned}
\end{equation}
For $0\le s\le1$ set
\[
 J_s=c_0+s q,\qquad V_s=-T/J_s,
 \qquad \partial_s J_s+\divg(J_sV_s)=0,
\]
and let $F_s$ solve
\begin{equation}\label{eq:graphFlow}
 \partial_s F_s(z)=V_s(F_s(z)),\qquad F_0(z)=z.
\end{equation}
The bounds $|V_s(x,h)|\le C_\ell h$ and $|DV_s|\le C_\ell$ are uniform in
$s$. Extend $V_s$ by zero on $h=0$. The flow and its inverse exist on the
whole half-space for $0\le s\le1$, are bi-Lipschitz, fix its boundary, and
satisfy $(F_s(x,h))_n\simeq_\ell h$. These statements follow from
Gronwall's inequality, first for the height and then for differences of
trajectories; $|V_s|\lesssim h$ also prevents finite-time escape to infinity.
Jacobi's formula and \eqref{eq:graphDiv} give
\[
 \frac{d}{ds}\big[J_s(F_s(z))\det DF_s(z)\big]=0,
 \qquad J_1(F_1(z))\det DF_1(z)=c_0.
\]
Thus $\Phi=P\circ F_1$ has constant determinant $c_0$ and the stated
boundary values and distance comparison.

We verify the derivative and Carleson bounds, including the local supremum
in \eqref{eq:graphGeomCM}. If $a\in C_0^\infty(\R^d)$ has integral zero
and $b\in L^\infty(\R^d)$, then
\begin{equation}\label{eq:graphKernelCM}
 \int_Q\int_0^r
 \sup_{\substack{|y-x|<h/2\\h/2<k<3h/2}}
 |a_k*b(y)|^2\,\frac{dh\,dx}{h}
 \le C_a|Q|\|b\|_\infty^2.
\end{equation}
Indeed, for $x\in Q$ and $h<r$ all convolutions use $b$ on a fixed dilate
of $Q$. For that restriction of $b$, Plancherel and
$\sup_{\xi\ne0}\int_0^\infty|\widehat a(h\xi)|^2\,dh/h<\infty$
give the estimate without the supremum. The local Sobolev inequality on
Whitney boxes bounds the supremum by averages of finitely many scaled
space and scale derivatives of $a_h*b$. Each is another smooth compactly
supported convolution kernel of integral zero, so the same argument applies.

By differentiation of the convolution formulas, every component of
\[
 hD^2P,\qquad hDq,\qquad h^2D^2q,\qquad hD^2T
\]
is a finite sum of kernels of this kind applied to components of
$\nabla\phi$. In particular they are bounded by $C\ell$ and satisfy
\eqref{eq:graphKernelCM} with right-hand side $C|Q|\ell^2$.
Differentiating $V_s=-T/J_s$ twice and using
$|T|\lesssim\ell h$, $|DT|\lesssim\ell$, $J_s\ge c_0/2$ and
$h|Dq|\lesssim\ell$ gives
\[
 h|D^2V_s|\le C_\ell
 \big(h|D^2T|+h|Dq|+h^2|D^2q|\big).
\]
Hence $hD^2V_s$ has the same pointwise and Carleson bounds, uniformly in $s$.
Differentiating \eqref{eq:graphFlow} twice and applying Gronwall yields
\[
 |D^2F_s(z)|\le C_\ell\int_0^s|D^2V_v(F_v(z))|\,dv.
\]
A bi-Lipschitz map fixing the flat boundary and preserving height up to a
constant preserves these Carleson bounds: images of boundary boxes lie
in fixed dilates of such boxes, and images of Whitney balls are covered
by boundedly many Whitney balls at comparable heights. Change of variables,
followed by Cauchy--Schwarz in $v$, therefore proves the pointwise bound
and \eqref{eq:graphGeomCM} for $F_s$ in place of $\Phi$.
The chain rule for $P\circ F_1$ proves both assertions for $\Phi$.
\end{proof}

\begin{lemma}[transfer of the equation]\label{lem:graphtransfer}
Let $\Phi$ be as in Lemma \ref{lem:graphflat}, put
$\rho(z,t)=(\Phi(z),t)$, and let $A$ be a bounded coefficient matrix on
$\OO\times\R$ satisfying \eqref{E:elliptic} and \eqref{E:1:carl}.
For $U=u\circ\rho$ define
\begin{equation}\label{eq:graphB}
 B(z,t)=D\Phi(z)^{-1}A(\Phi(z),t)D\Phi(z)^{-T}.
\end{equation}
Then $\LL u=0$ is equivalent to
$-\partial_t U+\divg(B\nabla U)=0$ on $\R^n_+\times\R$.
The matrix $B$ is bounded and elliptic and satisfies
\eqref{E:1:carl} and \eqref{E:1:bound}, with
\begin{equation}\label{eq:graphCM}
 \|\mu_B\|_C\le C_\ell\big(\|\mu_A\|_C+\|A\|_\infty^2\ell^2\big),
 \qquad K_B\le C_\ell\big(K_A+\|A\|_\infty\ell\big).
\end{equation}
The pullback preserves the energy class, commutes with $D_t^{1/2}$ and
$H_t$, and transfers Neumann and adjoint Dirichlet solvability at the same
exponents. In the Neumann case the transformed datum is
\begin{equation}\label{eq:graphg}
 g_0(x,t)=c_0^{-1}\sqrt{1+|\nabla\phi(x)|^2}\,
 g(x,\phi(x),t).
\end{equation}
\end{lemma}

\begin{proof}
Set $M=D\Phi$. The chain rule gives
$\nabla U=M^T(\nabla u)\circ\rho$, while
$\partial_tU=(\partial_tu)\circ\rho$ and $dX\,dt=c_0\,dz\,dt$.
For a test function $\zeta$ on the half-space, the original weak form
therefore becomes
\[
 c_0\iint_{\R^n_+\times\R}
 \big[B\nabla U\cdot\nabla\zeta-U\partial_t\zeta\big]\,dz\,dt=0.
\]
Division by the constant $c_0$ proves the asserted equation. The same
calculation with a boundary term and
$d\sigma=\sqrt{1+|\nabla\phi|^2}\,dx\,dt$ proves \eqref{eq:graphg}.
Transposition gives $B^T=M^{-1}(A^T\circ\rho)M^{-T}$, so the adjoint
transfers in the same way.

The derivative identities imply
\[
 |\nabla B|\le C_\ell\big(|\nabla A|\circ\rho+
 \|A\|_\infty|D^2\Phi|\big),\qquad
 |\partial_tB|\le C_\ell|\partial_tA|\circ\rho.
\]
Use \eqref{eq:graphJac}, \eqref{eq:graphGeomCM} and a finite Whitney
covering of the images of the averaging balls to obtain \eqref{eq:graphCM}.
Since the geometric measure is independent of time, its parabolic
Carleson bound follows from the spatial one by multiplication by the
length of the time interval, comparable to the square of the radius.
Ellipticity follows directly from \eqref{eq:graphB} and the bounds for
$M$ and $M^{-1}$.

Time is unchanged, so both time operators commute with pullback. The
constant volume Jacobian and the bounds for $M,M^{-1}$ give equivalent
energy norms. Taking traces and then duals gives the corresponding
isomorphisms of the Dirichlet and Neumann energy trace spaces.
The boundary Jacobian in \eqref{eq:graphg} is bounded above
and below. Cones map into cones with fixed enlarged apertures, and their
Whitney averaging regions have comparable measures and admit finite
Whitney coverings. Thus the corresponding maximal-function norms are
equivalent at every fixed finite exponent, with aperture changes made
globally. The weak formulations, energy uniqueness and these norm
comparisons give the asserted transfer of solvability. In particular the
transformed Neumann datum has zero integral if the original datum does.
\end{proof}

\section{Proof of Theorem~\ref{MainT}\label{Sprf} for $\om=\Rn_+\times\R$}
In this section, we prove Theorem~\ref{MainT} for $\om=\Rn_+\times\R$ using results obtained in Sections~\ref{SS:43}-\ref{S.S<N}. The case where $\om=\OO\times\R$ with $\OO$ being an unbounded Lipschitz graph domain (that is, of the form \eqref{eq.O=lipgph}) with sufficiently small Lipschitz constant follows from Lemmas \ref{lem:graphflat} and \ref{lem:graphtransfer}, and the case where $\OO$ is a bounded Lipschitz domain with sufficiently small Lipschitz constant is treated in Section~\ref{S.bddLip}.
\medskip

We first state a perturbation result that is proven in \cite{Ulm25}.
\begin{theorem}[{\cite[Theorem 1.6]{Ulm25}}]\label{thm.pert}
    Let $\LL_0= -\dr_t  +\divg (A_0 \nabla \cdot)$, $\LL_1= -\dr_t  +\divg (A_1 \nabla \cdot)$ be two parabolic operators on $\om=\OO\times\R$, with $\OO\subset\Rn$ Lipschitz. Assume that 
    \begin{equation}
        d\nu:=\sup_{B_{X,t}}|A_0-A_1|^2\frac{dX\,dt}{\delta(X,t)}\quad\text{ is a Carleson measure on }\om. 
    \end{equation}
    If $(N)_{q_0}^{\LL_0}$ and $(D)_{q_0'}^{\LL_1^*}$ are solvable for some $q_0>1$, then there exists a small $\varepsilon>0$ depending on $n$, the ellipticity constants of $A_0$ and $A_1$, $\OO$ and $q_0$, such that if $\norm{\nu}_{C}\le\varepsilon$, then $(N)_{q_0}^{\LL_1}$ is solvable.
\end{theorem}
We need the perturbation result above because we want the coefficients to be regular enough to apply the estimates in Sections~\ref{S.Caccio}-\ref{S.S<N}. This can be achieved by the following lemma.
\begin{lemma}\label{lem.CarImprov}
    Suppose that the matrix $A$ satisfies \eqref{E:elliptic} and either \eqref{E:1:carl} or \eqref{E:1:carl2} in $\om=\Rn_+\times\R$.
    Then there exist matrices $B$ and $D$ such that $A= B+D$, where $B$ satisfies \eqref{E:elliptic} with the same constants $\lambda$ and $\Lambda$. In addition, $B$ has entries in $C^\infty(\om)$, and the following estimates hold:
    \begin{equation}\label{B+CCM}
       d\mu':= \sup_{B_{X,t}}\br{\abs{\nabla B}^2\delta+\abs{\dr_t B}^2\delta^3+\abs{ D}^2\delta^{-1}}dX\,dt \quad\text{ is a Carleson measure on }\om,
    \end{equation}
    with $\norm{\mu'}_C\le C\norm{\mu}_C$ for some $C>0$, and 
     \begin{equation}\label{D2Bbdd}
        \abs{x_n\nabla B(X,t)}+\abs{x_n^2\dr_t B(X,t)}+\abs{x_n^{i+2j}\nabla^i\partial_t^j B(X,t)}\le C\norm{\mu}_C^{1/2} \quad\text{ for all }(X,t)\in\om,
    \end{equation}
    for $i,j\in \N\cup \{0\}$, $i+j\ge 1$ where the constant $C$ depends only on $n$, $\Lambda$ (and possibly also $i$ and $j$).
\end{lemma}
The construction in \cite[Lemma 3.32]{DLP1} gives this lemma with two
modifications: the Carleson estimates include Whitney-ball suprema, and
the bound \eqref{D2Bbdd} retains the factor $\norm{\mu}_C^{1/2}$.
The supremum condition on $A$ gives \eqref{B+CCM}. Differentiating the
same smoothing construction gives
 \begin{equation*}
        \sup_{B_{X,t}}\br{\abs{\nabla B}^2\delta+\abs{\dr_t B}^2\delta^3+\abs{\nabla^i\dr_t^j B}^2\delta^{2i+4j-1}}\,dX\,dt \quad\text{ is a Carleson measure on }\om,
    \end{equation*}
with Carleson norm at most $C\norm{\mu}_C$. The Whitney-ball supremum then gives \eqref{D2Bbdd}.
\medskip

Thanks to Lemma~\ref{lem.CarImprov} and Theorem~\ref{thm.pert}, it suffices to prove Theorem~\ref{MainT} for parabolic operators whose coefficient matrix satisfies the condition \eqref{D2Bbdd} instead of just \eqref{E:1:carl} or \eqref{E:1:carl2}. 

Let $p\in(1,\infty)$. Choose an integer $m\ge2$ such that $\p:=4m>p$; thus $\p\ge8$. We will show that 
\begin{equation}\label{pf0}
    \text{there exists }\delta>0 \text{ such that }(N)_{\p}^{\LL} \text{ is solvable if }\norm{\mu}_C<\delta \text{ in }\om=\Rn_+\times\R.
\end{equation} 
Once we obtain this, we choose $\delta$ small enough so that $(D)_{\p'}^{\LL^*}$ is solvable, which is possible due to \cite[Theorem 1.1]{DDH}. Then $(N)_p^{\LL}$ follows from interpolation of the Neumann problem (\cite[Theorem 1.6]{DLP2}), as desired.
Thus the direct estimates below are needed only for $\p=8,12,\dots$;
the interpolation step supplies every target exponent $p\in(1,\infty)$.

To prove \eqref{pf0}, let $g\in L^{\p}(\pom)\cap \Hdot^{-1/4}_{\pd_{t} - \Delta_x}(\partial\Omega)$ and let $u\in \dot \E$ be the energy solution to the Neumann problem $\LL u= 0$ in $\om$ with $\dr_{\nu}^A u|_{\pom}=g$. We assume a priori that $\norm{\wt N(\nabla u)}_{L^{\p}(\pom)}<\infty$.  We denote by $H:=a_{nj}\partial_ju$ (using summation convention) which at the boundary $\pom$ gives us the co-normal derivative of $u$. 
\medskip

We first bound the nontangential maximal function of $\nabla u$ by $p$-adapted square functions, which is done in Section~\ref{SS:43}. Precisely, for any $\varepsilon\in(0,1)$, Theorems~\ref{thm.NlessSdT} and \ref{thm.NleSpH} allow us to find $\delta>0$ such that if $\norm{\mu}_C<\delta$ (recall that we have additionally \eqref{D2Bbdd}), then there holds
\begin{equation}\label{pf.NdT<Sp}
    \|{\tilde N}(\nabla_Tu)\|_{L^{\p}(\pom)}^{\p}\le C\norm{S_{\p}(\nabla_T u)}_{L^{\p}(\pom)}^{\p}+ \varepsilon\br{\norm{\tilde N(\nabla u)}_{L^{\p}(\pom)}^{\p}+\norm{S_2(\nabla u)}_{L^{\p}(\pom)}^{\p}},
 \end{equation}  
 and 
\begin{equation}\label{pf.NH<Sp}
    \|{\tilde N}(H)\|_{L^{\p}(\pom)}^{\p}\le C\norm{S_{\p}(H)}_{L^{\p}(\pom)}^{\p}+ \varepsilon\br{\norm{\tilde N(\nabla u)}_{L^{\p}(\pom)}^{\p}+\norm{S_2(\nabla u)}_{L^{\p}(\pom)}^{\p}}.
 \end{equation} 
 Next, we control the $p$-adapted square functions $\norm{S_{\p}(\nabla_T u)}_{L^{\p}(\pom)}$ and $\norm{S_{\p}(H)}_{L^{\p}(\pom)}$ in Section~\ref{S.Sp_bd}. Specifically, the two key steps carried out in the proof of Lemma~\ref{NPl1} are 
 \begin{equation*}
     \norm{S_{\p}(\nabla_T u)}_{L^{\p}(\pom)}^{\p}\le C\norm{\nabla_Tu}_{L^{\p}(\pom)}^{\p}+ \varepsilon\br{\norm{\tilde N(\nabla u)}_{L^{\p}(\pom)}^{\p}+\norm{S_2(\nabla u)}_{L^{\p}(\pom)}^{\p}},
 \end{equation*}
 and 
 \begin{equation*}
     \norm{\nabla_Tu}_{L^{\p}(\pom)}^{\p}\le C \iint_{\Omega}|\nabla_T u|^{\p-2}|\nabla H|^2x_n\,dX\,dt +\varepsilon\br{\norm{\tilde N(\nabla u)}_{L^{\p}(\pom)}^{\p}+\norm{S_2(\nabla_T u)}_{L^{\p}(\pom)}^{\p}}.
 \end{equation*}
 Combining the two estimates with \eqref{pf.NdT<Sp}, one obtains the estimate summarized in Lemma~\ref{NPl1}:
 \begin{equation}\label{pf.NdT<H}
      \|{\tilde N}(\nabla_Tu)\|_{L^{\p}(\pom)}^{\p}\le C
\iint_{\Omega}|\nabla_T u|^{\p-2}|\nabla
H|^2x_n\,dX\,dt+\varepsilon\,C\br{\norm{\tilde N(\nabla u)}_{L^{\p}(\pom)}^{\p}+\norm{S_2(\nabla u)}_{L^{\p}(\pom)}^{\p}}.
 \end{equation}
We need to control the term $\iint_{\Omega}|\nabla_T u|^{\p-2}|\nabla
H|^2x_n\,dX\,dt$ and $\norm{S_{\p}(H)}_{L^{\p}(\pom)}$. Both terms are taken care of by Lemma~\ref{NPl2}, which shows that for any integer $0\le k\le \p-2$,
\begin{multline}\label{pf.iter}
\iint_{\Omega}|\nabla_T u|^{\p-k-2}|H|^{k}|\nabla
H|^2x_n\,dX\,dt\le C\int_{\partial\Omega}|H|^{\p}\,dx\,dt\\
+C(\p-k-2)\iint_{\Omega}|\nabla_T u|^{\p-k-3}|H|^{k+1}|\nabla
H|^2x_n\,dX\,dt+\varepsilon\,C\left[\|N(\nabla u)\|_{L^{\p}(\partial\Omega)}^{\p}+\|S(\nabla u)\|^{\p}_{L^{\p}(\partial\Omega)}\right].
\end{multline}
In fact, letting $k=\p-2$ in \eqref{pf.iter} gives that 
\[
\norm{S_{\p}(H)}_{L^{\p}(\pom)}^{\p}\le  C\norm{H}_{L^{\p}(\pom)}^{\p} + \varepsilon\,C\left[\|N(\nabla u)\|_{L^{\p}(\partial\Omega)}^{\p}+\|S(\nabla u)\|^{\p}_{L^{\p}(\partial\Omega)}\right],
\]
and applying \eqref{pf.iter} inductively for $k=0,1,\dots,\p-2$ gives that 
\[
\iint_{\Omega}|\nabla_T u|^{\p-2}|\nabla
H|^2x_n\,dX\,dt \le 
C\norm{H}_{L^{\p}(\pom)}^{\p} + \varepsilon\,C\left[\|N(\nabla u)\|_{L^{\p}(\partial\Omega)}^{\p}+\|S(\nabla u)\|^{\p}_{L^{\p}(\partial\Omega)}\right].
\]
Therefore, by \eqref{pf.NdT<H} and \eqref{pf.NH<Sp}, one obtains that 
\begin{equation*}
     \|{\tilde N}(\nabla_Tu)\|_{L^{\p}(\pom)}^{\p}\le C\norm{H}_{L^{\p}(\pom)}^{\p}+ \varepsilon\br{\norm{\tilde N(\nabla u)}_{L^{\p}(\pom)}^{\p}+\norm{S_2(\nabla u)}_{L^{\p}(\pom)}^{\p}},
\end{equation*}
and 
\begin{equation*}
      \|{\tilde N}(H)\|_{L^{\p}(\pom)}^{\p}\le C\norm{H}_{L^{\p}(\pom)}^{\p}+ \varepsilon\br{\norm{\tilde N(\nabla u)}_{L^{\p}(\pom)}^{\p}+\norm{S_2(\nabla u)}_{L^{\p}(\pom)}^{\p}}.
\end{equation*}
Since $\wt N(\partial_n u)$ is related to ${\wt N}(\nabla_Tu)$ and ${\wt N}(H)$ by using the definition of $H$ and the bounds for the matrix $A$:
$$|\wt N(\partial_n u)|\le C|\wt N(H)|+ C|\wt N(\nabla_Tu)|,$$ 
we have that
\[
  \|{\tilde N}(\nabla u)\|_{L^{\p}(\pom)}^{\p}\le C\norm{H}_{L^{\p}(\pom)}^{\p}+ \varepsilon\br{\norm{\tilde N(\nabla u)}_{L^{\p}(\pom)}^{\p}+\norm{S_2(\nabla u)}_{L^{\p}(\pom)}^{\p}}.
\]
This is almost what we want, except that we still have $\norm{S_2(\nabla u)}_{L^{\p}(\pom)}$ on the right-hand side. This is not an issue, as by Theorem~\ref{thm.S<Np}, $\norm{S(\nabla u)}_{L^{\p}}\le C\|N(\nabla u)\|_{L^{\p}}\le C\|\wt N(\nabla u)\|_{L^{\p}}$. Now we choose  
 $\varepsilon\in(0,1)$ sufficiently small, so that we get that
 \begin{equation}\label{NP-est}
  \|{\tilde N}(\nabla u)\|_{L^{\p}(\pom)}\le C\norm{H}_{L^{\p}(\pom)}=C\norm{g}_{L^{\p}(\pom)},
 \end{equation}
which proves \eqref{pf0} under the a priori assumption $\|{\tilde N}(\nabla u)\|_{L^{\p}(\pom)}<\infty$.

\bigskip
It remains to justify the finiteness of $\|{\tilde N}(\nabla u)\|_{L^{\p}(\pom)}$. We split the argument into multiple steps, some of which are similar to those of \cite{DLP1}, while others are new. 
The steps below involve the following key components: (1) reducing to operators with smooth coefficients whose gradient is finitely supported, and to smooth data; (2) using the inhomogeneous PDE satisfied by the tangential derivatives of the solution ($\partial_i u, \, i<n$), but complicated by the fact that the right-hand side involves all derivatives; (3) using the inhomogeneous PDE satisfied by the conormal derivative ($\sum_j a_{n,j}\partial_j$) to get the missing estimates on $\partial_n u$; (4) deriving a Moser iteration argument to improve regularity of $|\nabla u|$ that requires a careful intertwining of the estimates obtained on the tangential derivatives and the conormal derivative. The bootstrap estimate is provided in Lemma \ref{Lemma-moser}.
\medskip

{\bf Step 1. Approximation.} We approximate a matrix $B$ as in \eqref{D2Bbdd} by a sequence of matrices $B^i$, $i\ge2$, as follows.

Put $\rho(X,t)=(|X|^4+t^2)^{1/4}$ and choose
$\chi\in C^\infty(\R)$ with $0\le\chi\le1$, $\chi=1$ on
$(-\infty,0]$ and $\chi=0$ on $[1,\infty)$. Set
\[
 a_i=2^{1/4}i,\qquad b_i=2i^2,\qquad L_i=\log(b_i/a_i),\qquad
 \eta^i(X,t)=\chi\left(\frac{\log(\rho(X,t)/a_i)}{L_i}\right),
\]
with $\eta^i(0,0)=1$. This smooth cutoff equals one on $B(0,i)$
and vanishes outside $B(0,2i^2)$. Its derivatives satisfy
\[
 |\nabla\eta^i|\lesssim\frac{1}{L_i\rho},\qquad
 |\partial_t\eta^i|\lesssim\frac{1}{L_i\rho^2}.
\]
Consider
\begin{equation}
B^i(x,x_n,t)=\eta^i(x,x_n,t)B(x,x_n+1/i,t)+(1-\eta^{i})I,
\end{equation}
where $I$ is the $n\times n$ identity matrix. All matrices $B^i$ on the domain $\Omega=\R^n_+\times\R$ are uniformly elliptic, with constants independent of $i$, and each $B^i$ is the identity matrix outside a large ball centered at zero. Moreover,
 $\|\nabla^j \partial^k_t B^i\|_{L^\infty(\Omega)}<\infty$ for all nonnegative integers $j,k$ such that $j+k\ge 1$.
This is a consequence of the fact that $|\nabla B|\le Cx_n^{-1}$ and thus $|\nabla B^i|\le C(x_n+1/i)^{-1}\le Ci$, with similar arguments for  $\partial_t B^i$ and all higher derivatives.
Let $\mu_{B^i}$ and $\mu_B$ denote the derivative measures in
\eqref{E:1:carl} with $A$ replaced by $B^i$ and $B$, respectively.
Then
\begin{equation}\label{eq:approxCarleson}
 \|\mu_{B^i}\|_C
 \le C\|\mu_B\|_C+
       \frac{C\|B-I\|_{L^\infty(\Omega)}^2}{L_i^2},
 \qquad L_i\simeq\log i,
\end{equation}
with $C$ independent of $i$. Indeed, $x_n\le\rho$ gives
\[
 x_n|\nabla\eta^i|^2+x_n^3|\partial_t\eta^i|^2
 \lesssim L_i^{-2}\rho^{-1}.
\]
The measure $\rho^{-1}\,dX\,dt$ is Carleson: for a boundary cube
of radius $r$ far from the origin, $\rho\gtrsim r$ on its Carleson
box; otherwise the box lies in a parabolic ball of radius $Cr$ about
the origin, where the integral is at most $Cr^{n+1}$. Moreover,
$\rho$ is comparable throughout each Whitney ball, so this bound
also holds after taking the supremum in \eqref{E:1:carl}.

The shifted matrix $B(x,x_n+1/i,t)$ has derivative Carleson norm
at most $C\|\mu_B\|_C$. For boxes of radius $r\ge1/i$, this follows
by translation into a box of radius $2r$. For $r<1/i$, the scaled
derivative bounds give a Carleson quotient at most
$C\|\mu_B\|_C((ir)^2+(ir)^4)$.
The product rule now proves \eqref{eq:approxCarleson}.
By choosing the original Carleson norm sufficiently small and
discarding finitely many indices, the matrices $B^i$ satisfy the
smallness required for \eqref{NP-est}, uniformly in the remaining $i$.
Finally, $B^i\to B$ uniformly on compact subsets of $\Omega$ and
in $L^p(\Omega\cap B_R(0,0))$ for every $p<\infty$ and $R>0$.

Consider the approximate energy solutions $u^i$ to the PDE problems
$$\mathcal L^iu^i=\divg(B^i\nabla u^i)-\partial_tu^i=0\quad\mbox{in }\Omega\quad\mbox{with}\quad \partial^{B^i}_\nu u^i\Big|_{\partial\Omega}=g,$$
where $g$ is in the trace space 
 $\Hdot^{-1/4}_{\pd_{t} - \Delta_x}(\partial\Omega)$. Since, for all $i$, the ellipticity constants have the same bounds, the energy norms $\dot\E(\Omega)$ of the $u^i$ are uniformly bounded by a constant that only depends on the $\Hdot^{-1/4}_{\pd_{t} - \Delta_x}(\partial\Omega)$ norm of $g$.
 
The solutions $u^i$ are actually determined up to a constant, and enjoy interior H\"older continuity, so it may be assumed without loss of generality that 
$u^i(0,1,0)=0$ for all $i$. With this in hand, and the fact that $\sup_i \|u^i\|_{\dot\E(\Omega)}<\infty$
we conclude that there exists a subsequence $(u^{i_n})$ of the sequence $(u^i)$
 converging weakly in $\dot\E(\Omega)$ to a function $u\in\dot\E(\Omega)$ with
$u(0,1,0)=0$. From this point, the argument is analogous to that of Section 4 (Step 5) of \cite{DLP1} and hence we omit the details. In summary, there exists a subsequence $u^{i_n}$ converging strongly on compact subsets of $\Omega$. 

In step 5, we will prove that the approximate solutions $u^i$ have the a priori bound $\|\tilde{N}(\nabla u^i)\|_{L^\p(\pom)}<\infty$ (as required above for $g$ with finite ${L}^\p(\partial\Omega)$ norm). 
Assuming this bound for the moment, by \eqref{NP-est} we have that for each $i$, 
$$\|\tilde{N}(\nabla u^i)\|_{L^\p(\pom)}\le C\|g\|_{{L}^\p(\partial\Omega)},$$
with a constant $C$ which is independent of $i$, depending only on ellipticity of the original matrix $B$ and the Carleson norm of its coefficients. As in \cite{DLP1} we then conclude that
\begin{equation}\label{R_2solest}
\|\tilde{N}(\nabla u)\|_{L^\p(\pom)}\le C\|g\|_{{L}^\p(\partial\Omega)}.
\end{equation}
To see that $u$ attains the boundary value $g$, observe that for each $i$ and $v\in \dot\E(\Omega)$ by 
\eqref{eq.NeumannBdy}
$$
 \iint_{\mathcal O\times\R} \left[B^i \nabla u^i \cdot{\nabla v} + \HT \dhalf u^i \cdot {\dhalf v}\right] \d X \d t -\int_{\partial\mathcal O
 \times\R}gv \d x\d t= 0.
$$
Here it suffices to consider Lipschitz test functions $v$ with bounded support as they are dense in $\dot\E(\Omega)$. Recall that we have the weak convergence of $u^{i_n}\to u$ in $\dot\E(\Omega)$. Furthermore, clearly $(B^i)^T\nabla v\to B^T\nabla v$ in $L^2(\Omega)$ since $B^i\to B$ in $L^p(\Omega\cap B_R(0,0))$ for any $p<\infty$ and $R>0$ and $\nabla v\in L^\infty$ and has bounded support. Thus we
conclude that
$$
 \iint_{\mathcal O\times\R} \left[\nabla u \cdot{B^T\nabla v} + \HT \dhalf u \cdot {\dhalf v}\right] \d X \d t -\int_{\partial\mathcal O
 \times\R}gv \d x\d t= 0.
$$
Hence $u$ is the energy solution of the parabolic PDE with matrix $B$ and Neumann boundary datum $g$ as desired. 
\medskip

{\bf Step 2. Reduction to an inhomogeneous problem.} It follows that we only need to prove that  $\|\tilde{N}(\nabla u^i)\|_{L^\p(\pom)}<\infty$ 
for solutions $u^i$ as defined above, associated to matrices $B^i$
having the properties listed in step 1. We fix $i\in\N$ and consider a matrix $B^i$ as above. Dropping the index $i$ we may assume that $B$ is uniformly elliptic, satisfies the Carleson measure condition \eqref{B+CCM}, equals the identity matrix outside a large ball centered at zero, and also has the property
that $\|\nabla^\alpha \partial^\beta_t B\|_{L^\infty(\Omega)}<\infty$ for all nonnegative integers $\alpha,\beta$ such that $\alpha+\beta\ge 1$.
Let $u$ be an energy solution to $\LL u=0$ where $\LL=\divg(B\nabla\cdot)-\partial_t$. Without loss of generality assume also that $\partial_\nu^Bu\Big|_{\pom}=g$ for some smooth function $g$ (in all variables). 
It is sufficient to establish solvability for a class of functions that is dense in our underlying space since it then extends continuously to the whole space, thanks to estimate \eqref{R_2solest}.\medskip

We aim to find a differentiable function $v$ supported in a bounded set $\Omega\cap B_{2R}(0,0)$, for some large $R>0$, satisfying $\partial_\nu^B v=g$ on $\partial\Omega$. We start by choosing $R>0$ large enough such that outside the ball $B_{R}(0,0)$ we have $g=0$ on $\partial\Omega$
and $B=I$. We now find a smooth domain $\tilde{\Omega}$ such that
$$\Omega\cap B_{R}(0,0) \subset \tilde{\Omega} \subset \Omega\cap B_{2R}(0,0).$$
Writing $x_0=t$ and $\partial_0=\partial_t$, we regard 
$\tilde{\Omega}$ as a smooth bounded subset of $\R^{n+1}$ with variables $x_0,x_1,\dots,x_n$.
We introduce a new $(n+1)\times(n+1)$ matrix $\tilde B$ defined by
$$
\tilde B=\left[ \begin{array}{c|c}
   1 & 0 \\
   \midrule
   0 & B \\
\end{array}\right]. $$
Let $\tilde{g}$ be a smooth function on $\partial\tilde\Omega$ such that $\tilde g=g$ on $\partial\tilde\Omega\cap \partial\Omega$, but also satisfying the condition $\int_{\partial\tilde\Omega} \tilde g=0$.
For this datum $\tilde{g}$, one can solve the elliptic Neumann problem:
$$\divg_{x_0,x}(\tilde B\nabla_{x_0,x}\tilde v)=0\quad\mbox{in }\tilde\Omega,\qquad \partial_\nu^{\tilde B}\tilde v=\tilde g\quad \mbox{on }\partial\tilde\Omega,$$
by the method of layer potentials.
(For example, see \cite{DMem} Theorem 4.4 for solvability via the method of layer potentials.)
In short, we have \lq\lq ellipticized"  the problem.
The new matrix $\tilde B$ has the required smoothness - it is twice differentiable in $x_0$ and $x$. It follows that
$\nabla_{x,x_0}\tilde v\in C^\alpha(\tilde\Omega)$ for some $\alpha>0$; and in fact, as the matrix $\tilde B$
is smooth, $\tilde{v}$ has as high differentiability as one might require.  Observe that on
$\partial\tilde\Omega\cap \partial\Omega$, it is the case that $\partial_\nu^{\tilde B}\tilde v=\partial_\nu^{B}\tilde v$ and hence
$\tilde v$ satisfies our original Neumann data. The final step is to multiply $\tilde v$ by a smooth cutoff function $\varphi$
supported on $B_{2R}(0,0)$, equal to $1$ near the set $\partial\Omega\cap B_R(0,0)$ which is constant in the $x_n$ variable near the set $\partial\Omega$. All of this is achievable and the resulting function $v=\tilde v\varphi$
is supported in a bounded set $\Omega\cap B_{2R}(0,0)$ for some large $R>0$ and $\partial_\nu^B v=g$ on $\partial\Omega$. It now follows that $u=v+w$ where $w\in\dot\E(\Omega)$ solves the parabolic PDE:
$$-\partial_t w+\divg(B\nabla w):=h=\partial_t v-\divg(B\nabla v)\quad\mbox{in }\Omega,\qquad \partial_\nu^B w=0\quad\mbox{on }\partial\Omega.$$
Here $h$ on the right-hand side is an $L^\infty(\Omega)$ function with bounded support.
If we prove that $\|\tilde{N}(\nabla w)\|_{L^\p(\pom)}<\infty$ we are done, since $\nabla v\in L^\infty(\Omega)$
has bounded support, and hence $\|\tilde{N}(\nabla v)\|_{L^p(\pom)}<\infty$ for any $1\le p\le\infty$. 

Let us now use an even reflection and extend the function $w$ across the boundary $x_n=0$. The extended solution $\tilde w$ is defined as
$$\tilde w(X,t)=\begin{cases}w(x,x_n,t),&\quad\mbox{when }x_n>0,\\
w(x,-x_n,t),&\quad\mbox{when }x_n<0,\end{cases}$$
and solves the parabolic PDE: 
\begin{equation}-\partial_t \tilde w+\divg(\tilde B\nabla \tilde w):=\tilde h
\end{equation} in the whole domain 
$\R^n\times\R$. Here $\tilde h$ is the even reflection of $h$, and the coefficients of the matrix $\tilde B$
are bounded functions of all variables and in fact smooth in $x$ and $x_n$, except across the hyperplane $x_n=0$. Obviously, $\tilde B$ is uniformly elliptic. By interior H\"older regularity of solutions it follows that $\tilde w\in L^\infty(\R^n\times\R)$.
\medskip

{\bf Step 3. Regularity of $\nabla \tilde w$.} Let us first consider $\partial_i\tilde w$, for $i=1,2,\dots, n-1$.
For simplicity of notation set $w_i=\partial_i\tilde w$. Clearly, $\tilde B$ is differentiable with respect to $x_i$
since its possible jump is across $x_n=0$, whereas $x_i$ is a tangential variable.  Hence $w_i$ satisfies the PDE
\begin{equation}\label{eq-tildewi}-\partial_t w_i+\divg(\tilde B\nabla w_i):=\partial_i \tilde h-\divg((\partial_i\tilde B)\nabla\tilde w)
\end{equation} in the whole domain $\R^n\times\R$. A priori we know that $\partial_i\tilde B\in L^\infty(\R^n\times\R)$ and $\nabla\tilde w\in L^2(\R^n\times\R)$ because $w\in\dot \E(\Omega)$. 
Since it is also the case that $\tilde h\in L^2(\R^n\times\R)$, this implies that $w_i\in\dot \E(\R^n\times\R)$, $i=1,2,\dots,n-1$.

The following lemma gives $w_i\in L^\infty(\R,L^2(\R^n))$ together
with a gain in spatial integrability. In dimension two that gain has
a different time exponent.

\begin{lemma}\label{Lemma-moser} Let $n\ge2$ and $2\le p<\infty$. Suppose $u\in\dot\E(\R^n\times\R)$ is an energy solution of the PDE
$$-\partial_t u+\divg(A\nabla u)=-h+\divg F,\qquad\quad\mbox{in } \R^n\times\R,$$
where $A$ is elliptic, measurable and bounded and 
\begin{equation}
S:=\int_\R\left(\int_{\R^n}|h|^pdx\right)^{1/p}dt+\left(\int_\R\left(\int_{\R^n}|F|^pdx\right)^{2/p}dt\right)^{1/2}<\infty.
\end{equation}
Assume zero past data, understood by taking limits of solutions with
zero initial data at times tending to $-\infty$.
Define
\[
 \mathcal X_p^{(n)}=
 \begin{cases}
 L^\infty(\R;L^p(\R^n))\cap L^p(\R;L^{np/(n-2)}(\R^n)),& n>2,\\
 L^\infty(\R;L^p(\R^2))\cap L^{2p}(\R;L^{2p}(\R^2)),& n=2,
 \end{cases}
\]
with the sum of the two norms. Then
\begin{equation}\label{eq:bootstrapMixedNorm}
 \|u\|_{\mathcal X_p^{(n)}}\le C(n,p,\lambda)S.
\end{equation}
\end{lemma}
\begin{proof} We split the solution into two parts, $u = u_1+u_2$,
where $u_1$ has right-hand side $\divg F$ and $u_2$ has right-hand side $-h$.
First, we consider $u_1$. It follows that after we multiply the PDE $u_1$ satisfies by $|u_1|^{p-2}u_1$ and integrate over the spatial variables we get that
\begin{multline}\label{eq-316}
p^{-1}\frac{d}{dt}\int_{\R^n}|u_1|^p dx =-(p-1)\int_{\R^n} A\nabla u_1\cdot\nabla u_1|u_1|^{p-2}\,dx\\
+(p-1)\int_{\R^n}F\cdot \nabla u_1|u_1|^{p-2}dx.
\end{multline}
For the first term on the right-hand side we have (by ellipticity)
$$-(p-1)\int_{\R^n} A\nabla u_1\cdot\nabla u_1|u_1|^{p-2}\,dx\le -\frac{4(p-1)}{p^2}\lambda \int_{\R^n} |\nabla(u_1|u_1|^{p/2-1})|^2dx,$$
which is strictly negative. Using Cauchy-Schwarz for the last term of \eqref{eq-316}, and then H\"older's inequality, we get
\begin{multline}
\left|(p-1)\int_{\R^n}F\cdot \nabla u_1|u_1|^{p-2}dx\right|\le \frac{2(p-1)}{p^2}\lambda \int_{\R^n} |\nabla(u_1|u_1|^{p/2-1})|^2dx \\+
\frac{p-1}2\lambda^{-1}\int_{\R^n}|F|^2|u_1|^{p-2}dx
\le \frac{2(p-1)}{p^2}\lambda \int_{\R^n} |\nabla(u_1|u_1|^{p/2-1})|^2dx \\+\frac{p-1}2\lambda^{-1}\left(\int_{\R^n}|u_1|^{p}dx\right)^{(p-2)/p} \left(\int_{\R^n}|F|^pdx\right)^{2/p}.
\end{multline}
From \eqref{eq-316}, moving the negative term to the left-hand side, we see
that
\begin{multline}\label{eq-317a}
\frac{d}{dt}\int_{\R^n}|u_1|^p dx +\frac{2(p-1)}{p}\lambda \int_{\R^n} |\nabla(u_1|u_1|^{p/2-1})|^2dx\\\le \frac{(p-1)p}2\lambda^{-1}\left(\int_{\R^n}|u_1|^{p}dx\right)^{(p-2)/p} \left(\int_{\R^n}|F|^pdx\right)^{2/p}.
\end{multline}
Dropping the (positive) second term in the inequality and rearranging the remaining terms gives
$$\frac{d}{dt}\left(\int_{\R^n}|u_1|^p dx\right)^{2/p}\le (p-1)\lambda^{-1}\left(\int_{\R^n}|F|^pdx\right)^{2/p}.
$$
Integrating over the interval $(-\infty,t]$, and recalling that
$u_1\to 0$ as $t \to-\infty$, 
yields
$$
\left(\int_{\R^n}|u_1(t,x)|^p dx\right)^{2/p}\le (p-1)\lambda^{-1}\int_{-\infty}^t\left(\int_{\R^n}|F(\tau,x)|^pdx\right)^{2/p}d\tau.$$
It follows that 
\begin{equation}
\sup_{t\in\R}\|u_1(\cdot,t)\|_{L^p(\R^n)}\le (p-1)^{1/2}\lambda^{-1/2}
\left(\int_{\R}\left(\int_{\R^n}|F|^pdx\right)^{2/p}dt\right)^{1/2}.
\end{equation}
Inserting this estimate for the ``$u_1$" integrand on the right-hand side of \eqref{eq-317a} gives
\begin{multline}\label{eq-317aa}
\frac{d}{dt}\int_{\R^n}|u_1|^p dx +\frac{2(p-1)}{p}\lambda \int_{\R^n} |\nabla(u_1|u_1|^{p/2-1})|^2dx\\\le \frac{[(p-1)\lambda^{-1}]^{p/2}p}2\left(\int_{\R}\left(\int_{\R^n}|F(\tau,x)|^pdx\right)^{2/p}d\tau\right)^{(p-2)/2} \left(\int_{\R^n}|F|^pdx\right)^{2/p}.
\end{multline}
After integrating \eqref{eq-317aa} over $\R$, we conclude that
\begin{multline}\label{eq-317aas}
\int_\R \int_{\R^n} |\nabla(u_1|u_1|^{p/2-1})|^2dxd\tau\\\le \frac{(p-1)^{p/2-1}p^2\lambda^{-p/2-1}}4\left(\int_{\R}\left(\int_{\R^n}|F(\tau,x)|^pdx\right)^{2/p}d\tau\right)^{(p-2)/2} \int_\R\left(\int_{\R^n}|F|^pdx\right)^{2/p}d\tau\\
=\frac{(p-1)^{p/2-1}p^2\lambda^{-p/2-1}}4\left(\int_{\R}\left(\int_{\R^n}|F(\tau,x)|^pdx\right)^{2/p}d\tau\right)^{p/2}.
\end{multline}
When $n>2$, we see that
$u_1(t,\cdot)|u_1(t,\cdot)|^{p/2-1}\in W^{1,2}(\R^n)$ and therefore $u_1(t,\cdot)|u_1(t,\cdot)|^{p/2-1}\in L^q(\R^n)$ for $q=2n/(n-2)$.
By Sobolev embedding in the spatial variables and \eqref{eq-317aas} we get that
\begin{equation*}
\left(\int_\R\left(\int_{\R^n}|u_1|^{np/(n-2)}dx\right)^{(n-2)/n}dt\right)^{1/p}\le C(p,\lambda) \left(\int_{\R}\left(\int_{\R^n}|F|^pdx\right)^{2/p}dt\right)^{1/2},
\end{equation*}
or in other words, 
$$\|u_1\|_{L^p(\R,L^{np/(n-2)}(\R^n))}\le C(p,\lambda)\|F\|_{L^2(\R,L^{p}(\R^n))}.
$$
\medskip

The calculation for $u_2$ is analogous. We start by observing that for the term containing $h$, by H\"older's  inequality,
$$\left|\int_{\R^n}h u_2|u_2|^{p-2}dx\right|\le \left(\int_{\R^n}|u_2|^pdx\right)^{(p-1)/p}\left(\int_{\R^n}|h|^pdx\right)^{1/p},
$$
which, as above, then yields the estimate
\begin{equation}
\sup_{t\in\R}\|u_2(\cdot,t)\|_{L^p(\R^n)}\le \int_{\R}\left(\int_{\R^n}|h|^pdx\right)^{1/p}dt.
\end{equation}
Integrating the energy inequality for $u_2$ and using this supremum
bound also gives
\[
 \iint_{\R^n\times\R}
 |\nabla(u_2|u_2|^{p/2-1})|^2\,dx\,dt
 \le C(p,\lambda)\|h\|_{L^1(\R;L^p(\R^n))}^p.
\]
For $n>2$, Sobolev embedding therefore gives
\begin{equation*}
\left(\int_\R\left(\int_{\R^n}|u_2|^{np/(n-2)}dx\right)^{(n-2)/n}dt\right)^{1/p}\le C(p,\lambda) \int_{\R}\left(\int_{\R^n}|h|^pdx\right)^{1/p}dt,
\end{equation*}
or in other words, 
$$\|u_2\|_{L^p(\R,L^{np/(n-2)}(\R^n))}\le C(p,\lambda)\|h\|_{L^1(\R,L^{p}(\R^n))}.
$$
For $n=2$, put $z_j=u_j|u_j|^{p/2-1}$ for $j=1,2$.
The two-dimensional Gagliardo--Nirenberg inequality is
\[
 \|z_j(t)\|_{L^4(\R^2)}^4
 \le C\|z_j(t)\|_{L^2(\R^2)}^2
          \|\nabla z_j(t)\|_{L^2(\R^2)}^2.
\]
Integrating in time, using the supremum bounds and the preceding
energy estimates, yields
\[
 \int_\R\|u_j(t)\|_{L^{2p}(\R^2)}^{2p}\,dt
 \le C\sup_t\|u_j(t)\|_{L^p(\R^2)}^p
              \iint|\nabla z_j|^2
 \le C(p,\lambda)S^{2p}.
\]
Thus $u_j\in L^{2p}(\R;L^{2p}(\R^2))$. The triangle inequality
for $u=u_1+u_2$ proves \eqref{eq:bootstrapMixedNorm} in both cases.
The energy tests can first be made with truncated powers and finite
initial time; their uniform bounds justify the limits defining the
zero-past solution.
\end{proof}

The forcing terms in \eqref{eq-tildewi} have bounded time support.
We use their causal solutions; uniqueness in the energy class identifies
them with the derivatives already constructed. Applying the lemma with
$p=2$ gives $w_i\in\mathcal X_2^{(n)}$. In particular, when $n=2$,
\[
 w_i\in L^\infty(\R;L^2(\R^2))\cap L^4(\R;L^4(\R^2)).
\]
\medskip

{\bf Step 4. Regularity of $H$.}  In step 3 we obtained an initial estimate for $\nabla_{x}\tilde{w}$, and hence for $\nabla_x u$. It remains to obtain a good estimate for $\partial_{n}u$. To this end, we use the 
results of subsection \ref{ss1}, namely that
$$\partial_{n}u=a_{nn}^{-1}(H-\sum_{j<n}a_{nj}\partial_j u).$$
The tangential derivatives are controlled by step 3. We now obtain
the corresponding mixed-norm bounds for the inhomogeneous part of $H$;
its bounded homogeneous part will be handled on the support of the forcing.
Recall that $H$ solves the PDE given by \eqref{eq.Hpde1} in $\Omega$ with right-hand side bounded by \eqref{boundsF}.
At the boundary, $H$ is equal to $g$, and so $H$ solves a Dirichlet problem. We decompose $H$ as $H_0+H_1$, where $-\partial_tH_0+\tilde LH_0=0$ with Dirichlet datum $g$, while
$H_1$ solves $-\partial_tH_1+\tilde LH_1=F$ with vanishing Dirichlet boundary condition.
The maximum principle gives $\|H_0\|_{L^\infty(\Omega)}\le\|g\|_{L^\infty}$.
For the fixed exponent $\p$, we may choose the coefficient smallness
threshold also to ensure $\|\tilde N(H_0)\|_{L^\p(\partial\Omega)}<\infty$
by \cite[Theorem 1.1]{DDH}. It remains to consider $H_1$. We use an odd reflection to extend this function
to the whole space $\R^n\times\R$. The right-hand side $F$ will be also extended by an odd reflection to a function $\tilde F$. Since $H_1\in\dot \E(\Omega)$, the extended function $\tilde H_1$ satisfies 
\begin{equation}-\partial_t \tilde H_1+\divg((\tilde B)^T\nabla \tilde H_1):=\tilde F
\end{equation} in the whole domain 
$\R^n\times\R$. We now examine $F$, as defined by \eqref{eq.Hpde1}. Replacing $u_t$ by $Lu$ (justified by using the PDE for $u$), we can write $F$ in the form
$$F=\divg(\varphi(B,\nabla B)\nabla u)+\psi(\partial_tB, \nabla^2 B,\nabla B)\nabla u.$$
where $\varphi,\psi$ are smooth functions enjoying the bounds
\begin{equation}\label{eq-322}
|\varphi(B,\nabla B)|\le C|\nabla B|\,\,\,\text{and}\,\,\,\, |\psi(\partial_tB, \nabla^2 B,\nabla B)|\le C(|\partial_t B|+|\nabla^2 B|+|\nabla B|^2).
\end{equation}
In particular, after reflection, $\tilde F$ has the same form as in the right-hand side of Lemma \ref{Lemma-moser},
with \eqref{eq-322} guaranteeing that $\tilde{F}$ is compactly supported ($\tilde B=I$ outside the ball $B_R(0,0)$). Because $\nabla u\in L^2(\Omega)$ and the time support is bounded, the divergence forcing is in $L^2_tL^2_X$ and the scalar forcing is in $L^1_tL^2_X$.
Lemma \ref{Lemma-moser}, applied to the causal solution, gives
$\tilde H_1\in\mathcal X_2^{(n)}$.
\medskip

{\bf Step 5. The bootstrap argument.} Fix a bounded spatial ball $D$
and a bounded time interval $I$ containing the supports of $\tilde h$,
the derivatives of $\tilde B$, and the extension $v$ from step 2.
The forcing terms in the equations for $w_i$ and $\tilde H_1$ are
supported in $D\times I$. Set
\[
 Z=(w_1,\ldots,w_{n-1},\tilde H_1).
\]
Steps 3 and 4 give $Z\in\mathcal X_2^{(n)}$, componentwise.
On $D\times I$, the identity for $\partial_nu$ in step 4, together
with $u=v+w$ and the boundedness of $H_0$ and $\nabla v$, gives
\[
 |\nabla u|+|\nabla\tilde w|\le C(1+|Z|),
\]
with the original functions reflected when necessary. Thus the bounded
homogeneous part $H_0$ only enters norms on this bounded set.

Suppose inductively that $Z\in\mathcal X_p^{(n)}$, with $p\ge2$.
Put
\[
 (p_+,a)=
 \begin{cases}
 (np/(n-2),p),&n>2,\\
 (2p,2p),&n=2.
 \end{cases}
\]
Then the preceding bound implies that both $\nabla u$ and
$\nabla\tilde w$ belong to $L^a(I;L^{p_+}(D))$ on the relevant
half-spaces. All coefficient factors in \eqref{eq-tildewi} and
\eqref{eq-322} are bounded and supported in $D\times I$.
Moreover, $\partial_i\tilde h=\divg(e_i\tilde h)$, where $e_i$ is
the $i$-th coordinate vector. Since $a\ge2$, H\"older's inequality
in time gives
\[
 \|f\|_{L^2(I;L^{p_+}(D))}
 \le |I|^{1/2-1/a}\|f\|_{L^a(I;L^{p_+}(D))},\qquad
 \|f\|_{L^1(I;L^{p_+}(D))}
 \le |I|^{1-1/a}\|f\|_{L^a(I;L^{p_+}(D))}.
\]
Consequently the divergence and scalar forcing terms satisfy the
hypotheses of Lemma \ref{Lemma-moser} at exponent $p_+$.
Applying it to all components of $Z$ gives
$Z\in\mathcal X_{p_+}^{(n)}$.

For $n>2$, the exponents are $p_j=2(n/(n-2))^j$, so after finitely
many steps $p_j>n$. For $n=2$, the first two applications give
\[
 Z\in L^\infty_tL^2_X\cap L^4_tL^4_X
 \quad\Longrightarrow\quad
 Z\in L^\infty_tL^4_X\cap L^8_tL^8_X.
\]
In particular, in dimension two we can take $q=4>n$.
In either case we have obtained $Z\in L^\infty(\R;L^q(\R^n))$
for some finite $q>n$.
The support and coefficient bounds now put the divergence forcing
of every component in $L^\infty_tL^q_X$ and the scalar forcing in
$L^\infty_tL^{q/2}_X$; the latter follows also from the bounded
spatial support.

The local boundedness estimate for scalar parabolic equations with
these inhomogeneous terms applies; see \cite[Theorem B, p.~616]{Aro68}.
Its strict integrability conditions are
$n/(2q)<1/2$ for the divergence term and $n/q<1$ for the scalar
term, both consequences of $q>n$.
On cylinders of a fixed radius, its right-hand side is bounded
uniformly in the centre by the $L^\infty_tL^q_X$ norm of each
component and the forcing norms just established.
It follows that $w_i,\tilde H_1\in L^\infty(\R^n\times\R)$.
The identity for $\partial_nu$, the boundedness of $H_0$, and
$\nabla v\in L^\infty$ therefore give $\nabla u\in L^\infty(\Omega)$.
In particular, $\tilde N(\nabla u)$ belongs to $L^p$ on bounded
subsets of $\partial\Omega$ for every $1\le p\le\infty$.

For the estimate outside a bounded set, recall that there is an
$R$ large enough such that the matrix $B$ equals $I$ in the complement of the ball $B(0,R)$. It follows that, in $\Omega\setminus B(0,R)$, the PDE for $w_i$ is exactly the heat equation
$\Delta w_i-\partial_tw_i=0$.
Hence on this set $w_i$ is comparable with the fundamental solution for the heat equation.  Choosing a pole  $(Y,s)$ with time $s$ 
in the past of the ball $B(0,R)$, there exists a large constant $M>0$ for which $0\le |w_i|\le MG$
on $\partial B(0,R)$ (since $w_i$ is bounded on 
$\partial B(0,R)$).  Hence by the comparison principle, applied to both the positive and negative parts of $w_i$,  

$$0<|w_i|(x,x_n,t)\le M G((x,x_n,t),(Y,s)),\qquad\mbox{for all } (x,x_n,t)\notin B(0,R),$$
where $G$ is the Green's function of the whole space satisfying the bounds
 \begin{equation}
     G(X,t,Y,s)\le C(t-s)^{-n/2}\exp\br{-\frac{\abs{X-Y}^2}{C(t-s)}} \quad\text{ for all } t>s,
 \end{equation}
as follows from \cite{Aro68}.
The Gaussian bound and the boundedness on the remaining compact region
give, outside a fixed boundary ball,
\[
 \tilde N(\partial_i u)(x,t)
 \lesssim (1+|x|+|t|^{1/2})^{-n},\qquad i<n.
\]
This majorant belongs to $L^p$ at infinity when $np>n+1$, which
suffices for the exponent $\p\ge8$ used here. The same argument
applies to $H_1$ and hence to the remaining component $\partial_nu$.
This establishes the required a priori bound $\|\tilde N(\nabla u)\|_{L^\p}<\infty$. 

\section{Caccioppoli-type inequalities }\label{S.Caccio}
\subsection{Caccioppoli with cut-off for $\nabla \partial_tu$}
\begin{lemma}[Caccioppoli with cut-off]\label{lem.CacciCO}
    Let $u$ be a weak solution to $-\dr_t u+\divg (A\nabla u)=0$ in $\om$. Let $\Psi$ be a cutoff function that satisfies $0\le\Psi\le 1$, $|\nabla\Psi(X,t)|+x_n|\dr_t\Psi(X,t)|\lesssim x_n^{-1}$ for  $(X,t)\in\om$. Let $Q$ be a parabolic Whitney ball in $\om$ with size $r$, $2Q\subset\om$. Then for any $m\ge 2$,
    \[
    \iint_Q|\nabla \dr_tu|^2\Psi^m dX\,dt\le \frac{C}{r^2}\iint_{2Q}|\dr_t u|^2\Psi^{m-2}dX\,dt+ C\iint_{2Q}|\dr_t A|^2|\nabla u|^2\Psi^mdX\,dt.
    \]
\end{lemma}
\begin{proof}
    The proof is a simple modification of the argument in the beginning of \cite[Section 6]{DLP1}. Let $\eta$ be a smooth cutoff function with $\supp\eta\subset 2Q$, $\eta=1$ on $Q$,  $|\nabla\eta|\lesssim r^{-1}$, and $|\dr_t\eta|\lesssim r^{-2}$. Since  $w=\dr_tu$ satisfies 
\[
-\dr_t w + \divg (A\nabla w)=-\divg((\dr_t A) \nabla u),
\]
we have 
\begin{multline*}
    \iint A\nabla w\cdot \nabla w\Psi^m\eta^2dX\,dt
    =\underbrace{-\iint(\dr_t w)w\Psi^m\eta^2dX\,dt}_{=:I_1}+\underbrace{\iint \divg ((\dr_tA)\nabla u)w\Psi^m\eta^2dX\,dt}_{:= I_2}\\
    -m\iint A\nabla w\nabla\Psi\br{w\Psi^{m-1}\eta^2}dX\,dt
    -2\iint A\nabla w\nabla\eta \br{w\Psi^m\eta}dX\,dt.
\end{multline*}
For the two unlabeled terms on the right-hand side, we use the bounds $|\nabla\eta|\lesssim r^{-1}$ and $|\nabla\Psi|\lesssim x_n^{-1}\approx r^{-1}$ on $2Q$ and the Cauchy-Schwarz inequality to get that they are bounded by 
\[
\varepsilon \iint |\nabla w|^2\Psi^m\eta^2dX\,dt+\frac{C_\varepsilon}{r^2}\iint_{2Q}w^2\Psi^{m-2}dX\,dt.
\]
We integrate by parts in $t$ in $I_1$ to get that 
\[
|I_1|=\abs{m\iint w^2\Psi^{m-1}\eta^2\dr_t\Psi dX\,dt+2\iint w^2\Psi^m\eta\dr_t\eta dX\,dt}\le \frac{C}{r^2}\iint_{2Q}w^2\Psi^{m-1}dX\,dt
\]
by the bounds on $|\dr_t\Psi|$ and $|\dr_t\eta|$. For $I_2$, the divergence theorem gives that 
\begin{multline*}
    I_2=-\iint (\dr_t A)\nabla u\nabla w\br{\Psi^m\eta^2}dX\,dt
    -m\iint (\dr_t A)\nabla u\nabla\Psi \br{u\Psi^{m-1}\eta^2}dX\,dt\\
    -2\iint (\dr_t A)\nabla u\nabla\eta \br{w\Psi^m\eta}dX\,dt.
\end{multline*}
Using the bounds on $|\nabla\Psi|$ and $|\nabla\eta|$ in the last two terms and Cauchy-Schwarz, one sees that 
\[
|I_2|\le \varepsilon\iint|\nabla w|^2\Psi^m\eta^2+(C+C_\varepsilon)\iint_{2Q}|\dr_t A|^2|\nabla u|^2\Psi^m+\frac{C}{r^2}\iint_{2Q}w^2\Psi^{m-2}.
\]
The lemma follows from taking $\varepsilon$ small in the above estimates. 
\end{proof}

\subsection{The PDEs for $H$ and  the Caccioppoli inequalities for $\partial_t H$ and $\nabla^2 H$}\label{ss1}

Let $u$ be a weak solution to $-\dr_t u+\divg(A\nabla u)=0$ in $\om$. Recall that $H=a_{nj}\partial_ju$. We start by deriving PDEs that $H$ satisfies.
\medskip

For an arbitrary $b_{ij}$ satisfying the ellipticity condition and sufficient smoothness let us denote by $\tilde{L}$ the operator $\partial_i(b_{ij}\partial_j\cdot)$. Then clearly

$$\tilde{L}H=\partial_i(b_{ij}\partial_j (a_{nk}\partial_k u))=\sum_{j<n}\partial_i(b_{ij}\partial_j (a_{nk}\partial_k u))+\partial_i(b_{in}\partial_n (a_{nk}\partial_k u)).$$
Using the PDE that $u$ satisfies, we replace $\partial_n (a_{nk}\partial_k u)$ by the remaining terms of the PDE for $u$
to obtain
$$
\tilde{L}H=\sum_{j<n}\left[\partial_i(b_{ij}\partial_j (a_{nk}\partial_k u))-\partial_i(b_{in}\partial_j (a_{jk}\partial_k u))\right]+\partial_i(b_{in}\partial_t u).
$$
Observe that the choice of coefficients $b_{ij}=a_{ji}$ implies cancellation of all terms containing three spatial derivatives $\partial_i\partial_j\partial_k u$. With these coefficients fixed we further observe that
$$
\partial_i(b_{in}\partial_t u)=\partial_t(a_{ni}\partial_i u)+(\partial_i a_{ni})\partial_t u-(\partial_ta_{ni})\partial_i u=\partial_tH+(\partial_i a_{ni})\partial_t u-(\partial_ta_{ni})\partial_i u.
$$
Hence it follows that
\begin{equation}\label{eq.Hpde1}
\tilde{\mathcal L}H=-\partial_tH+\tilde{L}H
=F:=(\partial_i a_{ni})\partial_t u-(\partial_ta_{ni})\partial_i u+
\sum_{j<n}\left[\partial_i(b_{ij}\partial_j (a_{nk}\partial_k u))-\partial_i(b_{in}\partial_j (a_{jk}\partial_k u))\right].
\end{equation}
Exploring the nature of the function on the right-hand side of \eqref{eq.Hpde1} we see that this function enjoys a bound by
\begin{equation}\label{boundsF}
|F|\lesssim |\nabla A||\partial_t u|+|\partial_tA||\nabla u|+|\nabla A|^2|\nabla u|+|\nabla A||\nabla ^2 u|+|\nabla ^2 A||\nabla u|,
\end{equation}
since the terms containing $\nabla^3 u$ cancel out.
\medskip

Alternatively, for certain arguments we will need $b_{nn}=1$. Hence we choose $b_{ij}=a_{ji}/a_{nn}$. This again guarantees that terms where three derivatives
fall on $u$ vanish as these are the terms:
\begin{eqnarray} &&\sum_{j<n} [b_{ij}a_{nk} (\partial_i\partial_j\partial_k u)-
b_{kn}a_{ji} (\partial_i\partial_j\partial_k
u)]=\sum_{j<n}a_{nn}^{-1}(a_{ji}a_{nk}-a_{nk}a_{ji})\partial_i\partial_j\partial_k
u=0.
\end{eqnarray}

We also rearrange the term containing $\partial_t$ derivative and write it as
$$\partial_i(b_{in}\partial_tu)=\partial_t\left(\frac{a_{ni}}{a_{nn}}\partial_iu\right)+(\partial_i b_{in})\partial_tu-(\partial_tb_{in})\partial_iu=\partial_t\left(\frac{H}{a_{nn}}\right)+(\partial_i b_{in})\partial_tu-(\partial_tb_{in})\partial_iu.$$
It follows that
\begin{multline} \label{eq.Hpde2}-a_{nn}^{-1}(\partial_t H)+\widetilde{L}H=T:=(\partial_i b_{in})\partial_tu+(\partial_ta_{nn}^{-1})H-(\partial_tb_{in})\partial_iu\\+\sum_{j<n} [\partial_i(b_{ij}\partial_j(a_{nk} \partial_k u))-
\partial_k(b_{kn} \partial_j(a_{ji} \partial_i u))].
\end{multline}

We will use this PDE for ${H}$ in several places.
When our method relies on the assumption that $k<n-1$ in the proofs of Subsection~\ref{S.NleSp2}, we need to use the equation \eqref{eq.Hpde2} to get the results for $H$. Otherwise, it is slightly easier to use the equation \eqref{eq.Hpde1} for $H$. 
\medskip

Our next aim is to establish a Caccioppoli inequality for $\nabla^2 H$ and we use \eqref{eq.Hpde1} to achieve it. In order to achieve this, let $H_k=\partial_k H$ for $k=1,2,\dots, n$. Differentiating through \eqref{eq.Hpde1} we obtain
\begin{equation}
-\partial_tH_k+\divg(A^T (\nabla H_k))=\divg_x((\partial_kA^T)\nabla H)+\partial_k F.
\end{equation}
Given an interior parabolic Whitney cube $Q$ of side-length $r$ with $3Q\subset\Omega$
we multiply both sides by $H_k\zeta^2$
where $\zeta$ is a cutoff function equal to one on $Q$ for an interior Whitney cube and zero outside $2Q$. After integration by parts we obtain:

\[\begin{split}
&\frac{1}{2}\iint_{2Q} \left[(H_k\zeta)^2\right]_{t} \,dX\,dt
+ \iint_{2Q} A^T \nabla(H_k\zeta)\cdot \nabla(H_k\zeta)\,dX\,dt 
\\
&=  \iint_{2Q} H_k^{2}\zeta\zeta_{t} \,dX\,dt  + \iint_{2Q} H_k^2 A^T\nabla\zeta\cdot\nabla \zeta \,dX\,dt +\iint_{2Q} A \nabla (H_k\zeta)\cdot H_k\nabla\zeta\,dX\,dt\\
&\quad- \iint_{2Q} A^T \nabla (H_k\zeta)\cdot H_k\nabla\zeta\,dX\,dt
 - \iint_{2Q}  (\partial_{k} A^T)\nabla H\zeta\cdot (\nabla H_k\zeta)\,dX\,dt \\ &\quad- \iint_{2Q} (\partial_{k} A^T)\nabla H( H_k\cdot( \zeta\nabla\zeta)) \,dX\,dt- \iint_{2Q} F\partial_{k}(H_k\zeta)\zeta \,dX\,dt-
  \iint_{2Q} FH_k\zeta(\partial_k\zeta) \,dX\,dt
\end{split}\]

Using the ellipticity and boundedness of the coefficients and the Cauchy-Schwarz inequality and the fact that $r|\nabla A|\le C$, it follows that
\[\begin{split}
& \lambda \iint_{Q} \left|\nabla H_k \right|^{2}\,dX\,dt \leq \frac{C}{r^2}  \iint_{2Q} |\nabla H|^{2} \,dX\,dt
 + C \iint_{2Q} |F|^{2} \,dX\,dt.
\end{split}\]
for some constant $C= C(\Lambda)$. As this holds for every $k=1,2,\dots,n$ it follows that
\[\begin{split}
&  \iint_{Q} \left|\nabla^2 H \right|^{2}\,dX\,dt \leq \frac{C}{r^2}  \iint_{2Q} |\nabla H|^{2} \,dX\,dt
 + C \iint_{2Q} |F|^{2} \,dX\,dt.
\end{split}\]
Observe that given the PDE $\partial_tH$ satisfies we have that $|\partial_tH|\lesssim |\nabla ^2H|+|\nabla A||\nabla H|+|F|$ and hence also
\begin{equation}
\iint_{Q} \left|\partial_t H \right|^{2}\,dX\,dt+\iint_{Q} \left|\nabla^2 H \right|^{2}\,dX\,dt \lesssim \frac{1}{r^2}  \iint_{2Q} |\nabla H|^{2} \,dX\,dt
 + \iint_{2Q} |F|^{2} \,dX\,dt.
\end{equation}
This is the final version of the Caccioppoli inequality for $H$. We use it to obtain bounds for the Area function of $H$ defined earlier. For a fixed $P\in\pom$
we cover $\Gamma_a(P)$ by a collection of Whitney cubes $(Q_i)_{i}$ of side length $r_i$ such that
the collection $(2Q_i)_{i\in\mathbb Z}$ has finite overlap. We then 
multiply the above inequality by $r_i^{-n+2}$, and sum over all indices $i$ for $Q=Q_i$ to obtain:

\begin{equation}\label{sqrA}
{A}^2(H)(P)=\iint_{\Gamma_{a}(P)}|\partial_t H|^2x_n^{-n+2}\,dX\,dt\lesssim \iint_{\Gamma_{a'}(P)}|\nabla H|^2x_n^{-n}\,dX\,dt
\end{equation}
$$+\iint_{\Gamma_{a'}(P)} |F|^{2} x_n^{-n+2}\,dX\,dt.
$$
Here $\Gamma_{a'}(P)$ is a cone of wider aperture (smaller slope) than the original cone so that $\bigcup_i 2Q_i\subset\Gamma_{a'}(P)$. We note that \eqref{sqrA} also holds if the cones are truncated.
Clearly, the last term of \eqref{sqrA} requires further work using \eqref{boundsF}. Given that $|u_t|\lesssim |\nabla ^2 u|+|\nabla A||\nabla u|$ we obtain using \eqref{boundsF} that
\begin{equation}
{A}^2(H)(P)\lesssim S^2(H)(P)
+\iint_{\Gamma_{a'}(P)} (|\partial_t A|^{2}+|\nabla^2 A|^{2}+|\nabla A|^{4}) |\nabla u|^{2}x_n^{-n+2}\,dX\,dt+
\end{equation}
$$+\iint_{\Gamma_{a'}(P)} |\nabla A|^2|\nabla^2 u|^2x_n^{-n+2}\,dX\,dt.$$
Again, a truncated version of this inequality also holds. Using truncated cones at the height $r$ and integrating over a boundary ball $\Delta_r$ we obtain that (after using $|\nabla A|x_n\le C$ in one of the terms):
$$\|A(H)\|^2_{L^2(\Delta_r)}\lesssim \|S(H)\|^2_{L^2(2\Delta_r)}
+\iint_{T(2\Delta_r)} (|\partial_t A|^{2}x_n^2+|\nabla^2 A|^{2}x_n^2+|\nabla  A|^{2}) |\nabla u|^{2}x_n\,dX\,dt
$$
$$+\iint_{T(2\Delta_r)} |\nabla A|^{2}|x_n\nabla^2 u|^{2}x_n\,dX\,dt  \le \|S(H)\|^2_{L^2(2\Delta_r)}
+\|\mu\|_C\|N(\nabla u)\|^2_{L^2(2\Delta_r)},$$
using the Carleson condition for $\mu$ for the operator $A$. We are also using the fact that
$N(x_n\nabla^2 u)\lesssim N(\nabla u)$ using the Caccioppoli inequality for the second gradient of $u$.
In particular after taking $r\to\infty$ we obtain
\begin{equation}
    \|A(H)\|_{L^2(\R^{n-1}\times\R)}\le C \|S(H)\|_{L^2(\R^{n-1}\times\R)}+C\|\mu\|_C^{1/2}\|N(\nabla u)\|_{L^2(\R^{n-1}\times\R)}.
\end{equation}

We remark that the $L^p$ version of this result also holds; we refer the reader to \cite[Section 10]{DLP1} and in particular the equation (6.5) where an analogous $L^p$ result is proven from the $L^2$
version using the good-$\lambda$ technique. We record it here:
\begin{equation}\label{eq.sqrAAL2}
    \|A(H)\|_{L^p(\R^{n-1}\times\R)}\le C \|S(H)\|_{L^p(\R^{n-1}\times\R)}+C\|\mu\|_C^{1/2}\|N(\nabla u)\|_{L^p(\R^{n-1}\times\R)}.
\end{equation}

\section{$N<S_p$ estimates}\label{SS:43}

We follow the approach of \cite{DLP1} to prove nontangential maximal function estimates, with some important modifications. The approach originated in \cite{KKPT}, with critical new  developments in \cite{DHM} and \cite{DFM}. In
\cite{DLP1}, we obtained bounds for the nontangential maximal function by the area integral and square function ( `$N<S+A$' estimates) but only for parabolic operators in block form.
Key innovations here include: (1) removing the block-form assumption on the parabolic operators, (2) using the $p$-adapted square function $S_p$ (rather than the square function) to derive `$N<S_p$' estimates for both the tangential gradient and the `conormal derivative' of solutions.    

\medskip
Let us begin by introducing some notions that will be used throughout the three subsections.
As we now work on the upper half-space, we define, as in \cite[Section 10]{DLP1}, for a fixed background parameter
$a>0$, the nontangential cone $\gamma_a(q,\tau)$ for a boundary point $(q,\tau)\in\partial(\R^n_+\times \R)$ as follows:
\begin{equation}\label{gamma2.11}
\gamma_a(q,\tau)=\{(X,t)=(x,x_n,t):\, x_n>a^{-1}\|(x-q,t-\tau)\|\}.
\end{equation}
The cones \eqref{gamma2.11} are better for our purposes here, as the special direction $x_n$ is separated from all other variables. We will refer to the parameter $a$, as before, as the aperture, and
$a^{-1}$ we shall call the slope of the nontangential cone $\gamma_a$. For the remainder of this section we assume that $N$ and $S$
are defined with respect to the cones $\gamma_a$.

These cones, as defined currently, have vertices at the boundary, but in order to define the stopping time construction we also need to define interior cones (with vertices inside $\Omega$). In the elliptic setting these interior cones
are defined by simply shifting the boundary cone, that is for $(Q,\tau)=(q,q_n,\tau)\in\Omega$
\begin{equation}\label{oldG}
\gamma_a(q,q_n,\tau)=(0,q_n,0)+\{(X,t)=(x,x_n,t):\, x_n>a^{-1}\|(x-q,t-\tau)\|\}.
\end{equation}
This is the approach also taken in \cite{RN}. However, we prefer to modify this construction and consider \lq\lq enlarged" cones that have an additional feature which we explain below. These novel cones were introduced in \cite{DLP1}.

Firstly, let us consider the set $S(X,t)=\{(q,\tau)\in \partial(\R^n_+\times \R):\, (X,t)\in\gamma_a(q,\tau)\}$. Given the way the cones are defined, it follows that
\begin{equation}\label{def.SXt}
    S(X,t)=\{(q,\tau):\, \|(x-q,t-\tau)\|<a\,x_n\},
\end{equation}
which is an open parabolic ball centered at $(x,t)$ of \lq\lq radius" $a\,x_n$. We define the cone $\gamma_a$ at $(X,t)$ as the interior of the intersection of all boundary cones that originate in $S(X,t)$. Hence
\begin{equation}\label{newG}
\gamma_a(X,t)=\gamma_a(x,x_n,t)=\mbox{int}\bigcap \{\gamma_a(q,\tau):\, \|(x-q,t-\tau)\|<a\,x_n\}.
\end{equation}

To explain this definition we compare \eqref{oldG} and \eqref{newG}. When the time variable is not present there is no difference between these definitions, and thus there is no reason to consider \eqref{newG}. This is due to the fact that in this case the boundary of the cone \eqref{oldG} is a union of straight lines of slope $a^{-1}$ originating from its vertex and intersections of such cones result in a cone of the same type. 

This is not the case with the time variable. In this variable the cone defined by \eqref{oldG} has a cusp of shape
given by the curve $a^{-1}\sqrt{|t|}$ at zero.
The definition  \eqref{newG} removes this cusp for interior cones at $x_n>0$; it is easy to see that the boundary cones 
$\gamma_a(q,\tau)$ that define $\gamma_a(X,t)$ have bounded slope of their boundary in $t$-variable at height $x_n>0$. This slope is bounded by $\frac{a^{-2}}{2x_n}$ (this is derived from the slope of the function $a^{-1}\sqrt{|t|}$). 

Thus it follows that the boundary of the cone $\gamma_a(X,t)$ is given by the graph of a function that is Lipschitz 
with Lipschitz constant $a^{-1}$ in spatial variables and Lipschitz constant $\frac{a^{-2}}{2x_n}$ in the time variable.

Moreover, our new cones \eqref{newG} maintain a key property of the cones from \eqref{oldG}, namely, that
if $(Q,\tau)$ is any point (including a boundary point) and $(X,t)\in \gamma_a(Q,\tau)$ then
$$\overline{\gamma_a(X,t)}\subset \gamma_a(Q,\tau).$$

As in \cite{DHM} we consider the map $\hbar:\partial\Omega\to\R$ given at each $(x,t)\in\partial\Omega$ defined by 
\begin{equation}\label{h}
\hbar_{\nu,a}(w)(x,t):=\inf\left\{x_n>0:\,\sup_{(Z,s)\in\gamma_{a}(x,x_n,t)}w(Z,s)<\nu\right\}.
\end{equation}

The function $w$ is defined below in \eqref{def.w} and varies
depending on the exact form of the `$N<S_p$' estimate. The properties of the cones are mainly used to derive some key properties of the function $\hbar$ in Lemma~\ref{S3:L5}. 

\medskip
\subsection{$N<S_p$ for $u$}\label{S.NleSp1}
The goal of this subsection is to prove Theorem \ref{thm.NlessS}. We start by only considering $\mathcal Lu=0$. In the next two subsections, we will adjust the proof to more general settings so that this can be applied to $\nabla_T u$ and $H$. In the theorem below it might look strange that we consider energy solutions with given boundary data $f$ (as the objective of this paper is to solve the Neumann problem), but for the purposes of this section our objective is to prove $N<S_p$ for all solutions $u\in\dot\E(\Omega)$ and all such solutions are generated by solving the Dirichlet problem with $f\in \Hdot^{1/4}_{\partial_t-\Delta_x}(\partial\Omega)$.

\begin{theorem}\label{thm.NlessS}
   Let $A$ be a matrix with bounded measurable coefficients that satisfies the ellipticity condition, the Carleson condition for coefficients $(a_{nj})_{j=1,2,\dots,n}$.   Let $\LL=-\dr_t+\divg(A\nabla\cdot)$.
   
   Then there exists $\delta>0$ such that if $$\||\nabla \mathbf{a_{n}}|^2x_n\|_{C}^{1/2}+\||\partial_{t}a_{nn}|^2x_n^3\|_{C}^{1/2}+\||\partial_ta_{nn}|x_n^2\|_{L^\infty}<\delta,$$
   where $ \mathbf{a_{n}}$ is the $n$-th row of the matrix $A$, then for any $p\ge8$ such that $p/2$ is an even integer, for any $a>0$, $f\in  L^{q}(\pom)\cap  \Hdot^{1/4}_{\partial_t-\Delta_x}(\partial\Omega)$,   there exists a 
constant $C=C(n,p,a,\delta, \lambda,\Lambda)>0$, such that for the energy solution $u$ to $\LL u=0$, $u\big|_{\pom}=f$, there holds for all $q>p/2$
\begin{equation}
    \|{\tilde N}(u)\|_{L^q(\pom)}\le C\norm{S_p(u)}_{L^q(\pom)} ,
 \end{equation}   
provided we know a priori that $\|\tilde N( u)\|_{L^q(\partial\Omega)}<\infty$. The same claim also holds for solutions $u_k$ on the domains $\Omega^k=\R^{n-1}\times(0,k)\times\R$ with boundary data $u_k(\cdot,0,\cdot)=f$ and  $u_k(\cdot,k,\cdot)=0$.
\end{theorem}

We introduce solutions $u_k$, 
in order to gain a necessary decay as $x_n\to\infty$. Let 
\begin{equation}
\LL u_k=0\mbox{ on }\R^{n-1}\times(0,k)\times\R,\quad u_k(x,0,t)=f(x,t),\quad u_k(x,k,t)=0,\quad\mbox{for all }(x,t)\in\pom.
\end{equation}
For simplicity we extend $u_k$ to $\Omega$ by defining  $u_k(x,x_n,t)=0$ when $x_n>k$, which extends each $u_k$ continuously.
We now consider the limit of $u_k$, as $k \to \infty$. For each $k\ge 1$, 
the Lax-Milgram lemma used on the domain $\Omega^k=\R^{n-1}\times(0,k)\times\R$  gives us
\begin{equation}
 \|u_k\|_{\dot \E} = \br{\|\nabla u_k\|^2_{\L^2(\Omega)} + \|\HT \dhalf u_k\|^2_{\L^2(\Omega)}}^{1/2} \le C \|\mbox{Tr }u_k\|_{\Hdot^{1/4}_{\pd_{t} - \Delta_x}(\partial\Omega^k)}= C\|f\|_{\Hdot^{1/4}_{\pd_{t} - \Delta_x}(\partial\Omega)},
\end{equation}
where the constant $C\in(0,\infty)$ depends on the coercivity constant of a certain 
sesquilinear form (see \cite{DLP1}, Section 2)
 in $\Omega^k$ (which is uniform in $k$), and $\|A\|_{L^\infty}=\Lambda$. Therefore, this 
bound is uniform in $k$. This uniformity and the fact that $\text{Tr}(u_k)=f$ imply that a weak convergence argument 
yields a sub-sequence convergent to some $v$ with $ \|v\|_{\dot \E}\leq C\|f\|_{\Hdot^{1/4}_{\pd_{t} - \Delta_x}(\partial\Omega)}$ 
and $\text{Tr}(v)=f$. This sub-sequence (which we somewhat imprecisely also call $(u_k)$) is therefore strongly convergent to $v$ in $L^2_{\rm loc}({\R}^n_+\times\R)$
by a standard functional analysis argument.
It follows that the $L^2$ averages of $u_k$ converge locally and uniformly to 
the $L^2$ averages of $v$ in $C_{\rm loc}({\R}^n_{+}\times\R)$. Additionally, because each $u_k$ has uniform interior H\"older continuity estimates, the convergence of $u_k\to v$ also holds in $C^\alpha_{\rm loc}({\R}^n_{+}\times\R)$ for some $\alpha>0$. 

Taking the limit in $k$, it now follows that $v$ must be a weak solution to $\mathcal Lv=0$ in $\Omega$. Since $\text{Tr}(v)=f$ and $ \|v\|_{\dot \E}<\infty$, it must hold that $u=v$ since energy solutions are unique. Consider now $u_k-u$, a solution in $\Omega^k$ enjoying the bound
\begin{equation}
 \|u_k-u\|_{\dot \E}  \le C \|\mbox{Tr }(u_k-u)\|_{\Hdot^{1/4}_{\pd_{t} - \Delta_x}(\partial\Omega^k)}= C\|\mbox{Tr }u\|_{\Hdot^{1/4}_{\pd_{t} - \Delta_x}(\R^{n-1}\times\{k\}\times\R)}.
\end{equation}
The trace $\|\mbox{Tr }u\|_{\Hdot^{1/4}_{\pd_{t} - \Delta_x}(\R^{n-1}\times\{k\}\times\R)}\to 0$ as $k\to\infty$, otherwise $ \|u\|_{\dot \E}=\infty$, which is false. Thus the weak convergence $\nabla u_k\to\nabla u$ in $L^2$ upgrades to strong convergence. \vglue1mm

We define $w$ as the $L^p$ average of $|u_k|$ for some fixed $k\in\mathbb N$, that is, for $p>1$,
\begin{equation}\label{def.w}
w(X,t)=w(x,x_n,t)=\left(\fiint_{Q_{x_n/2}(X,t)}|u_k(Y,s)|^p\,dY\,ds\right)^{1/p}.
\end{equation}
At this point we observe that $\hbar_{\nu,a}w(x,t)<\infty$ for all points $(x,t)\in\partial\Omega$. 
This is due to the fact that $u_k=0$ for $x_n>k$ and therefore $w=0$ for $x_n>2k$.

For a constant $\nu>0$, define the set
\begin{equation}\label{E}
E_{\nu,a}:=\big\{(x,t)\in\partial\Omega:\,N_{a}(w)(x,t)>\nu\big\}.
\end{equation}

The following lemma is proved in \cite[Lemma 10.10]{DLP1}.
\begin{lemma}\label{S3:L5}
Let $w$ be as in \eqref{def.w}. Then the following properties hold.
\vglue2mm

\noindent (i)
The function $\hbar_{\nu, a}(w)$ is Lipschitz in spatial variables with constant $a^{-1}$. It is also 
Lipschitz in the time variable above a positive height and the following holds:
\begin{equation}\label{Eqqq-5}
\left|\hbar_{\nu,a}(w)(x,t)-\hbar_{\nu,a}(w)(y,s)\right|\leq a^{-1}|x-y|+\frac{a^{-2}}{2\min\{\hbar_{\nu,a}(w)(x,t),\hbar_{\nu,a}(w)(y,s)\}}|t-s|.
\end{equation}
for all $(x,t),(y,s)\in\partial\Omega$. In particular,
\begin{equation}\label{derh}
|\nabla \hbar_{\nu,a} (x,t)|\le a^{-1},\qquad |\partial_t \hbar_{\nu,a} (x,t)|\le \frac{a^{-2}}{2\hbar_{\nu,a} (x,t)}.
\end{equation}
\vglue2mm

\noindent (ii)
The function $\hbar_{\nu,a}(w)$ satisfies the following: there exists $C_n>0$ that depends only on $n$ such that
\begin{equation}\label{derh'}
    \abs{\hbar_{\nu,a}(w)(x,t)-\hbar_{\nu,a}(w)(y,s)}\le C_na^{-1}\br{\abs{x-y}+\abs{t-s}^{1/2}} 
\end{equation}
for all $(x,t),(y,s)\in\partial\Omega$.
\vglue2mm

\noindent (iii)
Given an arbitrary $(x,t)\in E_{\nu,a}$, where $E_{\nu,a}$ is defined in \eqref{E}, let $x_n:=\hbar_{\nu,a}(w)(x,t)$. Then there exists a 
point $(y,y_n,s)\in\partial\gamma_{a}(x,x_n,t)$ such that $w(y,y_n,s)=\nu$ and $\hbar_{\nu,a}(w)(y,s)=y_n$. 		
\end{lemma}

The following statement (for the square function $S_{p,b}$ with $p=2$ only) appears in \cite{DLP1} in Appendix A, where the matrices were also assumed to be in block form.

\begin{lemma}\label{l6} 
Assume that $w$ is defined as in \eqref{def.w} and $\Omega=\R^n_+\times\R$. 
Then for any $a>0$ there exists $0<b_0=b_0(a)$ and $\gamma_0=\gamma_0(a)>0$ such that the following holds. 
Having fixed an arbitrary $\nu>0$, $b\ge b_0$, and $\gamma\le\gamma_0$ for each point $(x,t)\in\partial\Omega$ from the set 
\begin{equation}\label{Eqqq-17}
\big\{(x,t):\, N_{a}(w)(x,t)>\nu\mbox{ and }S_{p,b}(u_k)(x,t)\leq\gamma\nu\big\}
\end{equation}
there exists a boundary ball $R$ with $(x,t)\in K_0 R$,
where $K_0 R$ is an enlargement of $R$ by a factor $K_0 >1$ that depends only on $p$, dimension, and the aperture $a$, and such that
\begin{equation}\label{Eqqq-18}
\big|w\big(z,\hbar_{\nu,a}(w)(z,\tau),\tau\big)\big|>\nu/{8}\,\,\text{ for all }\,\,(z,\tau)\in R.
\end{equation}
\end{lemma}

\begin{proof}[Proof of Lemma~\ref{l6}] The statement follows by a modified argument from Appendix A of \cite{DLP1} adapted to our case.
For a point $(x,t)$ in \eqref{Eqqq-17}, part (iii) of Lemma \ref{S3:L5} gives the existence of a point $(y,y_n,s)\in\partial\gamma_a(x,x_n,t)$ on the graph of $\hbar$ with $w(y,y_n,s)=\nu$.
We consider an interior parabolic cylinder $Q_{y_n/2}(y,y_n,s)$. There exists $b=b(a)>0$ such that 
$Q_{4y_n/5}(y,y_n,s)\subset\gamma_b(x,t)$ and hence \eqref{Eqqq-18} will follow if we prove that for all points inside  $Q_{\delta y_n}(y,y_n,s)$ the function $w$ is at least $\nu/8$, for some $\delta >0$ depending only on $p$.

Firstly,  $w(y,y_n,s)=\nu$ implies that the $L^p$ average of $|u_k|$ over the cube $Q_{y_n/2}(y,y_n,s)$ is $\nu$. Denote by $\tilde B$ the orthogonal projection of $Q_{cy_n/2}(y,y_n,s)$, onto the spatial domain $\Rn_+$, where $c=c(p)$ is a constant (larger than $1$ and less than $8/5$) to be determined later.  Denote by $B$ the orthogonal projection of $Q_{y_n/2}(y,y_n,s)$ onto $\Rn_+$. Let $\eta$ be a smooth nonnegative cutoff
function in spatial variables, obtained by taking a smooth function with
values in $[0,1]$, equal to one on $B$ and supported in $\tilde B$, and
normalizing its $L^2$ norm to one.
Thus
\[
 \int_{\tilde B}\eta^2\,dX=1,\qquad
 \|\eta\|_{L^\infty}\lesssim y_n^{-n/2},\qquad
 \|\nabla\eta\|_{L^\infty}\lesssim y_n^{-n/2-1},\qquad
 \inf_B\eta^2\ge |\tilde B|^{-1}.
\]
The implicit constants may depend on the fixed ratio $c(p)$.
Consider the function
\[
v_{av}(\tau):=\int_{\Rn}|u_k|^{p}(X,\tau)\eta^2(X)\,dX.
\] With the properties of $\eta$ and the fact that 
\(
\fiint_{Q_{y_n/2}(y,y_n,s)}|u_k|^pdX\,dt = \nu^p,
\) 
we must have 
\begin{equation}\label{claim.v_av1}
    v_{av}(\tau)>(\nu/2)^p \text{ for some }\tau\in  I_{y_n}(s):=[s-(y_n)^2/2,s+(y_n)^2/2],
\end{equation}
because otherwise one would have
\[
    y_n^2\br{\nu/2}^p\ge\int_{\tau\in I_{y_n}(s)}v_{av}(\tau)d\tau\ge \br{\inf_B \eta^2}\iint_{Q_{y_n/2}(Y,s)}\abs{u_k}^pdXd\tau\ge \frac{y_n^2|B|}{2|\tilde B|}\nu^p = c^{-n}\frac{y_n^2}{2} \nu^p,
\]
which is a contradiction if $p>1$ and $c = c(p)$ is chosen sufficiently close to $1$.
We claim that \eqref{claim.v_av1} implies that 
\begin{equation}\label{claim.v_av2}
    v_{av}(\tau')>(\nu/4)^p \text{ for all }\tau'\in I_{y_n}(s).
\end{equation}
Indeed,  calculating as in \cite{Din23}, we multiply the equation $\LL u=0$ by $\eta^2$ and integrate it over the region
$\mathbb R^n\times [t_0,t_0']$ for any $t_0,t_0'\in I_{y_n}(s)$, extending the function $u_k\eta$ by zero outside the support of $\eta$.
Because $\eta$ is a  
function of the spatial variables only, we have
\begin{multline}\label{eq.v'-v}
    v_{av}(t_0')-v_{av}(t_0)=\iint_{\mathbb R^n\times\{t_0'\}}|u_k|^{p}\eta^2\,dX-\iint_{\mathbb R^n\times\{t_0\}}|u_k|^{p}\eta^2\,dX\\
    =
p\iint_{\mathbb R^n\times[t_0,t_0']}(\partial_t u_k)|u_k|^{p-2}u_k\,\eta^2\,dX\,dt
=p\iint_{\mathbb R^n\times[t_0,t_0']}\mbox{\rm div}(A\nabla u_k)|u_k|^{p-2}u_k\,\eta^2\,dX\,dt\\
=\iint_{\mathbb R^n\times[t_0,t_0']}\mbox{\rm div}(A\nabla (|u_k|^{p}))\eta^2\,dX\,dt
-p(p-1)\iint_{\mathbb R^n\times[t_0,t_0']}(A\nabla u_k\cdot\nabla u_k)|u_k|^{p-2}\eta^2\,dX\,dt\\
=-2p\iint_{\mathbb R^n\times[t_0,t_0']}(A\nabla u_k\cdot\nabla\eta)|u_k|^{p-2}u_k\eta\,dX\,dt\\
-p(p-1)\iint_{\mathbb R^n\times[t_0,t_0']}(A\nabla u_k\cdot\nabla u_k)|u_k|^{p-2}\eta^2\,dX\,dt.
\end{multline}
The interval $[t_0,t_0']$ has length at most $y_n^2$.
We use $|\nabla\eta|\lesssim y_n^{-n/2-1}$ and place the remaining
factor $\eta$ entirely in the $|u_k|^p$ factor of Cauchy--Schwarz.
The energy factor is unweighted, so
\begin{multline}\label{eq.v'-v2}
    \abs{\iint_{\mathbb R^n\times[t_0,t_0']}(A\nabla u_k\cdot\nabla\eta)|u_k|^{p-2}u_k\eta\,dX\,dt}\\
    \lesssim \br{y_n^{-n}\iint_{\tilde B\times I_{y_n}(s)}\abs{\nabla u_k}^2\abs{u_k}^{p-2}dX\,dt}^{1/2}\br{y_n^{-2}\iint_{\Rn\times I_{y_n}(s)}\abs{u_k}^p\eta^2 dX\,dt}^{1/2}\\
 \lesssim \br{\iint_{Q_{\frac{4y_n}{5}}(y,y_n,s)}\abs{\nabla u_k}^2\abs{u_k}^{p-2}y_n^{-n}dX\,dt}^{1/2}\left(\sup_{\tau\in I_{y_n}(s)}\int_{\mathbb R^n\times\{\tau\}}|u_k|^p\eta^2\,dX\right)^{1/2}\\
 \le S_{p,b}(u_k)^{p/2}(x,t)\sup_{\tau\in I_{y_n}(s)}v_{av}(\tau)^{1/2}
 \le \frac1{4p} \sup_{\tau\in I_{y_n}(s)}v_{av}(\tau) + C S_{p,b}(u_k)^{p}(x,t).
\end{multline}
The energy term in \eqref{eq.v'-v} is bounded by $CS_{p,b}(u_k)^p(x,t)$,
using $\eta^2\lesssim y_n^{-n}$ and the same cylinder inclusion.
Therefore, we have that for any $t_0',t_0\in I_{y_n}(s)$,
$$\left| v_{av}(t_0')-v_{av}(t_0)\right|\le \frac12\sup_{\tau\in I_{y_n}(s)}v_{av}(\tau)+ CS_{p,b}( u_k)^p(x,t).$$
Since $S_{p,b}(u_k)(x,t)\le \gamma\nu$ and $\sup_{\tau\in I_{y_n}(s)}v_{av}(\tau)\ge (\nu/2)^p$, the inequality above implies that 
$$
\inf_{\tau\in I_{y_n}(s)}v_{av}(\tau)\ge \br{2^{-p-1}-C\gamma^{p}}\nu^p\ge \left(\frac{\nu}4\right)^p
$$
must hold for sufficiently small  $\gamma>0$. Hence we have shown the claim \eqref{claim.v_av2}.

Pick now any point $(X_0,t_0)\in Q_{\delta y_n}(y,y_n,s)$ where
$\delta = \frac{c(p) - 1}{3} > 0$. Denote by $B_{X_0}$ the orthogonal projection of $Q_{x_n/2}(X_0,t_0)$ onto the spatial domain $\Rn_+$. We write  $w^p(X_0,t_0)$  as
\begin{equation}
w^p(X_0,t_0)=\fint_{s'=t-(x_n/2)^2}^{t+(x_n/2)^2}\fint_{B_{X_0}}|u_k(X,s')|^p\,dX\,ds'=:\fint_{t-(x_n/2)^2}^{t+(x_n/2)^2}w_{av}(s')ds'.
\end{equation}
We aim to obtain a lower bound for $w(X_0,t_0)^p$ by comparing $w_{av}(s')$ with $v_{av}(s')$. In fact, it is easier to compare the averages without the cutoff function $\eta$ and we will control
\begin{equation*}
  \fint_{\tilde B} |u_k(X,s')|^pdX - w_{av}(s')= \fint_{\tilde B} |u_k(X,s')|^pdX - \fint_{B_{X_0}}|u_k(X,s')|^p\,dX.
\end{equation*}
By the fundamental theorem of calculus in the spatial variables (c.f. \cite[Lemma 5.2]{DHM}), we obtain that 
\[
     \fint_{\tilde B} |u_k(X,s')|^pdX - w_{av}(s')
   \le C\int_{X\in H}|u_k(X,s')|^{p-1}\abs{\nabla u_k(X,s')}y_n^{1-n}dX, 
\]
where $H$ is the convex hull of the set $\tilde B\cup B_{X_0}$ in the spatial domain. Note that our choice of $X_0$, $\delta$ and $\tilde B$ guarantees that $B_{X_0}\subset \tilde B$, and thus $H=\tilde B$. It follows from Cauchy-Schwarz and Young's inequality that
\begin{multline*}
     \fint_{\tilde B\times\set{s'}} |u_k|^pdX - w_{av}(s')
     \le C_\varepsilon \,y_n^2\int_{\tilde B \times\{s'\}}\abs{\nabla u_k}^2\abs{u_k}^{p-2}y_{n}^{-n}dX+\varepsilon \int_{\tilde B\times\set{s'}}|u_k|^py_n^{-n}dX\\
     \le  C\,y_n^2\int_{\tilde B \times\{s'\}}\abs{\nabla u_k}^2\abs{u_k}^{p-2}y_{n}^{-n}dX+ \frac12\fint_{\tilde B\times\set{s'}}|u_k|^pdX
\end{multline*}
by choosing $\varepsilon$ sufficiently small. Reordering the terms gives that
\[
w_{av}(s')\ge \frac{1}{2}\fint_{\tilde B\times\set{s'}} |u_k|^pdX - C\,y_n^2\int_{\tilde B \times\{s'\}}\abs{\nabla u_k}^2\abs{u_k}^{p-2}y_{n}^{-n}dX.
\]
Integrating both sides in $s'\in[t_0-(x_n/2)^2,t_0+(x_n/2)^2]$ and taking averages, one gets that     
\begin{equation}\label{eq.wav}
    \fint_{t-(x_n/2)^2}^{t+(x_n/2)^2}w_{av}(s')ds'\ge \frac{1}{2}\fint_{t-(x_n/2)^2}^{t+(x_n/2)^2}\fint_{\tilde B\times\set{s'}} |u_k|^pdXds'-CS_{p,b}(u)(x,t)^p,
\end{equation}
where we have used $x_n\approx y_n$. The property of $\eta$ ensures that $v_{av}(s')\le \fint_{\tilde B\times\set{s'}} |u_k|^pdX$ for all $s'$. Thus, \eqref{eq.wav}, the claim \eqref{claim.v_av2} (notice that the interval $[t_0-(x_n/2)^2,t_0+(x_n/2)^2]$ is contained in $I_{y_n}(s)$), and the bound on $S_{p,b}(u)(x,t)$ imply that 
\[
w^p(X_0,t_0)\ge \frac12 (\nu/4)^p- C\gamma^p\nu^p\ge (\nu/8)^p
\]
for $\gamma>0$ sufficiently small. 
Hence we have proved that on $Q_{\delta y_n}(y,y_n,s)$ the function $w$ is at least $\nu/8$.
\end{proof}

\begin{lemma}[maximal function on a stopping graph]\label{S3:L6}
Let $w\ge0$, let $h=\hbar_{\nu,a}(w)$ be finite and satisfy
\eqref{derh'}, and let $P\in\{N_a(w)>\nu\}$.
Suppose a boundary ball $R$ satisfies $P\in K_0R$ and
$w(z,h(z,\tau),\tau)\ge\nu/8$ on $R$. Then
\[
 (M_hw)(P,h(P))\ge c\nu,
\]
where $c>0$ depends only on $n$, $a$ and $K_0$. The same conclusion holds
for the maximal function restricted to radii at most $Lr$ if the fixed
enlargement of $R$ containing $P$ has radius at most $Lr$.
\end{lemma}
\begin{proof}
The parabolic Lipschitz bound \eqref{derh'} identifies measure on the graph
with boundary measure, up to a fixed factor. A fixed enlargement of $R$
contains $P$, so its graph average is at least $c\nu$. That ball is also
admissible for the restricted maximal function under the stated radius
condition. In particular, Lemma \ref{l6} gives this conclusion whenever
$S_{p,b}(u_k)(P)\le\gamma\nu$.
\end{proof}

Finally, we are ready to state the key lemma that allows us to formulate a good-$\lambda$ inequality later.
The calculation here is based on \cite{DLP1} with the necessary modifications taking into account the use of the $p$-adapted square function instead of the more usual $L^2$ based square function.

\begin{lemma}\label{S3:L8-alt1}  Fix $p\ge8$ such that $p/2$ is an even integer.
Let $\Omega={\mathbb R}^n_+\times\R$ and let ${\mathcal L}=-\partial_t+\divg(A\nabla\cdot)$ be a parabolic operator.
Suppose $u_k$ is a weak solution of $\mathcal Lu_k=0$ in $\Omega^k$ satisfying $u_k(x,k,t)=0$, extended to $\Omega$ by setting $u_k=0$ for $x_n>k$.  For a fixed (sufficiently large) $a>0$ determined below, consider an arbitrary function $\hbar:{\mathbb R}^{n-1}\times\R\to \mathbb R$ such that it satisfies
the estimates \eqref{Eqqq-5}, \eqref{derh} and $\hbar\ge 0$. 
Then 
we have the following:\vglue1mm

For any $\varepsilon>0$ there exists $\delta>0$
such that if
$$\||\nabla \mathbf{a_{n}}|^2x_n\|_{C}+\||\partial_{t}a_{nn}|^2x_n^3\|_{C}+\||\partial_ta_{nn}|x_n^2\|_{L^\infty}<\delta,$$
where $ \mathbf{a_{n}}$ is the $n$-th row of the matrix $A$,
and $|\nabla A|x_n$ is bounded, then
for all parabolic surface balls $\Delta_r\subset{\mathbb R}^{n-1}\times\R$ of radius $r$ such that at least one point of $\Delta_r$
satisfies $\hbar(x,t)\le 2r$, we have the following estimate for an arbitrary $c\ge 0$:

\begin{multline}\label{TTBBMM}
\int_{1/6}^6\int_{\Delta_r}\big|u^{p/2}_k \big(x,\theta\hbar(x,t),t\big)-c\big|^2\,dx\,dt\,d\theta
\leq C(1+\varepsilon^{-1})\iint_{\cS(\Delta_r,\hbar)}\abs{\nabla u_k}^2|u_k|^{p-2}x_n\, dx_ndtdx\\
+\varepsilon\br{\|\tilde{N}_{2,a}\br{(u_k^{p/2}-c)\1_{\cS(\Delta_r,\hbar)}}\|_{L^2(\Delta_{2r})}^2
+\|{N}_{a}(u_k\,\1_{\cS(\Delta_r,\hbar)})\|_{L^p(\Delta_{2r})}^p}\\+
\varepsilon  \int_{2\Delta_r}\br{\iint_{\gamma_a(y,s)\cap\cS(\Delta_r,\hbar)}\abs{\nabla u_k}^2x_n^{-n}}^{p/2}\,dx_ndx\,dt
+\frac{C}{r}\iint_{\mathcal{K}_\varepsilon}|u_k^{p/2}-c|^{2}\,dX\,dt,
\end{multline}
for some $C\in(0,\infty)$ that only depends on $a,\Lambda,n$ but not on $k$, $u_k$, $c$, $\varepsilon$ or $\Delta_r$. 
Here, \[\cS(\Delta_r,\hbar):=\set{(x,x_n,t):\, (x,t)\in\Delta_{2r} \text{ and } \frac{\hbar(x,t)}{12}<x_n<18r},\]
where $\mathcal K_{\varepsilon}:=\cS(\Delta_r,\hbar)\cap\set{(x,x_n,t): x_n>\varepsilon r}$.
The cones used to define the nontangential 
maximal functions in this lemma have vertices on $\partial\Omega$.
\end{lemma}

\begin{proof} Let $\Delta_r$ be as in the statement of our Lemma, and assume that $(q,\tau)$ is the center of $\Delta_r$. Let $\zeta$ be a smooth cutoff function of the form $\zeta(x,x_n,t)=\zeta_{0}(x_n)\zeta_{1}(x,t)$ where
\begin{equation}
\zeta_{0}= 
\begin{cases}
1 & \text{ in } (-\infty, r_0+r], 
\\
0 & \text{ in } [r_0+2r, \infty),
\end{cases}
\qquad
\zeta_{1}= 
\begin{cases}
1 & \text{ in } \Delta_{r}(q,\tau), 
\\
0 & \text{ in } \mathbb{R}^{n}\setminus \Delta_{2r}(q,\tau)
\end{cases}
\end{equation}
and
\begin{equation}
r|\partial_{x_n}\zeta_{0}|+r|\nabla_{x}\zeta_{1}|+r^2|\partial_t\zeta_{1}|\leq c
\end{equation}
for some constant $c\in(0,\infty)$ independent of $r$. Here 
$r_0=6\sup_{(x,t)\in \Delta_r}\hbar(x,t)$. Observe that $ r_0\le (12+C_na) r$ with $C_n$ as in \eqref{derh'}, and so for $a$ sufficiently large, one has $r_0\le 13 r$ and 
$$\theta\hbar(x,t)\le  r_0+r \qquad \mbox{for all }(x,t)\in \Delta_{2r}, \text{ for }\theta\in (1/6,6).$$

Our goal is to control the $L^2$ norm of $u_k^{p/2}\big(\cdot,\theta\hbar(\cdot)\big)-c$ on $\Delta_r$.  We proceed to estimate
\begin{align}
&\hskip -0.20in
\int_{\Delta_{r}(q,\tau)}(u_k^{p/2}(x,\theta\hbar(x,t),t)-c)^2\,dx\,dt \le \mathcal I:=\int_{\Delta_{2r}(q,\tau)}(u_k^{p/2}(x,\theta\hbar(x,t),t)-c)^2\zeta(x,\theta\hbar(x,t),t)\,dx\,dt
\nonumber\\[4pt]
&\hskip 0.70in
=-\iint_{\mathcal S(q,\tau,r,r_0,\theta\hbar)}\partial_{x_n}\left[(u_k^{p/2}(x,x_n,t)-c)^2\zeta(x,x_n,t)\right]\,dx_n\,dx\,dt,
\nonumber
\end{align}
where $\mathcal S(q,\tau,r,r_0,\theta\hbar)=\{(x,x_n,t):(x,t)\in \Delta_{2r}(q,\tau)\mbox{ and }\theta\hbar(x,t)<x_n<r_0+2r\}$. Hence:

\begin{align}\nonumber
&\hskip 0.10in
\mathcal I \le-2\iint_{\mathcal S(q,\tau,r,r_0,\theta\hbar)}(u_k^{p/2}-c)\partial_{x_n}(u_k^{p/2}-c)\zeta\, dx_n\,\,dt\,dx  
\\[4pt]
&\hskip 0.70in
\quad-\iint_{\mathcal S(q,\tau,r,r_0,\theta\hbar)}(u_k^{p/2}-c)^2(x,x_n,t)\partial_{x_n}\zeta\,dx_n\,dx\,dt
=:\mathcal{A}+V.\label{u6tg}
\end{align}
We further expand the term $\mathcal A$ as a sum of three terms obtained 
via integration by parts with respect to $x_n$ as follows:
\begin{align}\label{utAA}
\mathcal A &=-2\iint_{\mathcal S(q,\tau,r,r_0,\theta\hbar)}(u_k^{p/2}-c)\partial_{x_n} 
(u_k^{p/2}-c)(\partial_{x_n}x_n)\zeta\,dx_n\,dt\,dx
\nonumber\\[4pt]
&=2\iint_{\mathcal S(q,\tau,r,r_0,\theta\hbar)}\left|\partial_{x_n}(u_k^{p/2})\right|^{2}x_n\zeta\,dx_n\,dt\,dx 
\nonumber\\[4pt]
&\quad +2\iint_{\mathcal S(q,\tau,r,r_0,\theta\hbar)}(u_k^{p/2}-c)\partial^2_{x_nx_n}(u_k^{p/2}-c)x_n\zeta\,dx_n\,dt\,dx 
\nonumber\\[4pt]
&\quad +2\iint_{\mathcal S(q,\tau,r,r_0,\theta\hbar)}(u_k^{p/2}-c)(\partial_{x_n}u_k^{p/2})\,x_n\partial_{x_n}\zeta\,dx_n\,dt\,dx
\nonumber\\[4pt]
&\quad -2\int_{\mathcal \partial S(q,\tau,r,r_0,\theta\hbar)}(u_k^{p/2}-c)\partial_{x_n}(u_k^{p/2})\,x_n\zeta\,\nu_n dS
\nonumber\\[4pt]
&=:I+II+III+IV.
\end{align}

Here the boundary term $IV$ is nonvanishing only on the graph of the function $\theta\hbar$.

We start by analyzing the term $II$. As each $u_k$ solves the PDE in $\Omega^k$ and equals zero at $x_n=k$,
we can if necessary reflect $u_k$ across the $x_n=k$ boundary so that it solves the parabolic PDE even above $x_n=k$. Then we have for $u_k^{p/2}-c$:
$${\mathcal L}(u_k^{p/2}-c)=\frac{p}2u_k^{p/2-1}\left[-\partial_t u_k+\divg(A\nabla u_k)\right]+\frac{p}2\left(\frac{p}2-1\right)u_k^{p/2-2}A\nabla u_k\cdot\nabla u_k.$$
As the first term equals zero (since $u_k$ is a solution to $\mathcal L$) we are left with
$${\mathcal L}(u_k^{p/2}-c)=\frac{p}2\left(\frac{p}2-1\right)u_k^{p/2-2}A\nabla u_k\cdot\nabla u_k>0.$$
It follows that we can write
\begin{multline}\label{S3:T8:E01-x}
\partial^2_{x_nx_n}(u_k^{p/2}-c)=\frac1{a_{nn}}\partial_{x_n}(a_{nn}\partial_{x_n}(u_k^{p/2}))-\frac{\partial_{x_n}a_{nn}}{a_{nn}}\partial_{x_n}(u_k^{p/2})=\\
\frac1{a_{nn}}\left[\partial_t(u_k^{p/2})-\sum_{(i,j)\ne(n,n)}\partial_{x_i}(a_{ij}\partial_{x_j}(u_k^{p/2}))+\frac{p}2\left(\frac{p}2-1\right)u_k^{p/2-2}A\nabla u_k\cdot\nabla u_k-\partial_{x_n}a_{nn}\partial_{x_n}(u_k^{p/2})\right].
\end{multline}
In turn, this permits us to write the term $II$ as
\begin{align}\label{eq.NleS1-II}
II &=-2\sum_{(i,j)\ne(n,n)}\iint_{\mathcal S(q,\tau,r,r_0,\theta\hbar)}\frac1{a_{nn}}(u_k^{p/2}-c)\partial_{{x_i}}\left({a}_{ij}\partial_{{x_j}}(u_k^{p/2})\right)x_n\zeta\,dx_n\,dt\,dx
\nonumber\\[4pt]
&\quad-2\iint_{\mathcal S(q,\tau,r,r_0,\theta\hbar)}(u_k^{p/2}-c)\frac{\partial_{x_n}a_{nn}}{a_{nn}}\partial_{x_n}(u_k^{p/2})
\zeta\,dx_n\,dt\,dx \nonumber\\
&\quad+p\left(\frac{p}2-1\right)\iint_{\mathcal S(q,\tau,r,r_0,\theta\hbar)}(u_k^{p/2}-c)u_k^{p/2-2}A\nabla u_k\cdot\nabla u_k\,x_n\zeta\,dx_n\,dt\,dx \nonumber\\
&\quad+2\iint_{\mathcal S(q,\tau,r,r_0,\theta\hbar)}\frac1{a_{nn}}(u_k^{p/2}-c)\partial_{t}(u_k^{p/2}-c)x_n\zeta\,dx_n\,dt\,dx=\mathcal B+\mathcal C+\mathcal D+II_7.
\end{align}
We leave the analysis of the last term $II_7$ for later. Observe that, under the assumption of the lemma $c\ge 0$, and using non-negativity of $u_k^{p/2-2}$ we get for the term $\mathcal D$:
\begin{align}
\mathcal D&\le p\left(\frac{p}2-1\right)\iint_{\mathcal S(q,\tau,r,r_0,\theta\hbar)}u_k^{p-2}A\nabla u_k\cdot\nabla u_k\,x_n\zeta\,dx_n\,dt\,dx\\
&\le C\iint_{\cSS}|u_k|^{p-2}\abs{\nabla u_k}^2x_n\, dx_ndtdx.\nonumber
\end{align}
For the term $\mathcal B$ we integrate by parts w.r.t. $\partial_{x_i}$, provided $i<n$ calling the resulting terms $\mathcal B_i$. When $i=n$ we must have $j<n$ and we then swap the two derivatives since 
$$\partial_{x_n}\left({a}_{nj}\partial_{x_j}(u_k^{p/2})\right)=\partial_{{x_j}}\left({a}_{nj}\partial_{{x_n}}(u_k^{p/2})\right)
+\partial_{x_n}(a_{nj})\partial_{x_j}(u_k^{p/2})-\partial_{x_j}(a_{nj})\partial_{{x_n}}(u_k^{p/2}),$$
and subsequently we can integrate by parts w.r.t. $\partial_{x_j}$ instead. (We denote the resulting terms $\mathcal B_j$). Notice that the last two terms in the formula above are of a type similar to $\mathcal C$, and we denote those by $\mathcal B_{res1}$. 

This then yields
\begin{align}
\mathcal B_i+\mathcal B_j&=2\sum_{i<n}\iint_{\mathcal S(q,\tau,r,r_0,\theta\hbar)}\frac{a_{ij}}{a_{nn}}\partial_{x_i}(u_k^{p/2})\partial_{x_j}(u_k^{p/2})\,x_n\zeta\,dx_n\,dt\,dx
\nonumber\\[4pt]
&+2\sum_{j<n}\iint_{\mathcal S(q,\tau,r,r_0,\theta\hbar)}\frac{a_{nj}}{a_{nn}}\partial_{x_n}(u_k^{p/2})\partial_{x_j}(u_k^{p/2})\,x_n\zeta\,dx_n\,dt\,dx
\nonumber\\[4pt]
&-2\sum_{i<n}\iint_{\mathcal S(q,\tau,r,r_0,\theta\hbar)}\frac{a_{ij}\partial_{x_i}a_{nn}}{a_{nn}^2}(u_k^{p/2}-c)\partial_{x_j}(u_k^{p/2})\,x_n\zeta\,dx_n\,dt\,dx
\nonumber\\[4pt]
&-2\sum_{j<n}\iint_{\mathcal S(q,\tau,r,r_0,\theta\hbar)}\frac{a_{nj}\partial_{x_j}a_{nn}}{a_{nn}^2}(u_k^{p/2}-c)\partial_{x_n}(u_k^{p/2})\,x_n\zeta\,dx_n\,dt\,dx
\end{align}
\begin{align}
\nonumber\\[4pt]
\hskip1cm&+2\sum_{i<n}\iint_{\mathcal S(q,\tau,r,r_0,\theta\hbar)}\frac{a_{ij}}{a_{nn}}(u_k^{p/2}-c)\partial_{x_j}(u_k^{p/2})\,x_n\partial_{x_i}\zeta\,dx_n\,dt\,dx
\nonumber\\[4pt]
&+2\sum_{j<n}\iint_{\mathcal S(q,\tau,r,r_0,\theta\hbar)}\frac{a_{nj}}{a_{nn}}(u_k^{p/2}-c)\partial_{x_n}(u_k^{p/2})\,x_n\partial_{x_j}\zeta\,dx_n\,dt\,dx
\nonumber\\[4pt]
&\quad-2\sum_{i<n}\int_{\partial\mathcal S(q,\tau,r,r_0,\theta\hbar)}(\mbox{boundary terms})x_n\zeta\nu_i\,dS
\nonumber\\[4pt]
&=:II_{1}+II_{2}+\mathcal B_{res2}+\mathcal B_{res3}+II_{3}+II_{4}+II_{5}.\nonumber
\end{align}
The boundary integral (term $II_5$) vanishes everywhere except on the graph of the function $\theta\hbar$. Note that on the graph of $\theta\hbar$,
\begin{equation}\label{eq.nui}
    \nu_i=\frac{-\theta \dr_i \hbar}{J_h}dx\,dt, \quad dS=J_h\, dx\,dt \quad\text{for }i=1,\dots,n-1,
\end{equation}
where $J_h:=\br{1+(\theta\nabla_x\hbar)^2+(\theta\dr_t\hbar)^2}^{1/2}$, and so
$\abs{\nu_i}dS\le\theta\abs{\dr_i\hbar}dx\,dt\le \theta\, a^{-1}dx\,dt$. Hence,
\begin{align*}
&|II_5|\nonumber\\
&\le C\sum_{i<n}\int_{\Delta_{2r}(q,\tau)}|(u_k^{p/2}-c)(x,\theta\hbar(x,t),t)\nabla_x(u_k^{p/2})(x,\theta\hbar(x,t),t)\hbar(x,t)\zeta(x,\theta\hbar(x,t),t) |dt\,dx\nonumber\\
&\le \frac14\int_{\Delta_{2r}(q,\tau)}(u_k^{p/2}(x,\theta\hbar(x,t),t)-c)^2\zeta(x,\theta\hbar(x,t),t)\,dt\,dx\nonumber\\&\quad+
C\int_{\Delta_{2r}(q,\tau)}|\nabla_x (u_k^{p/2})(x,\theta\hbar(x,t),t)|^2|\hbar(x,t)|^2\,dt\,dx =\frac14\mathcal I+II_6.
\end{align*}

In the last step we used Cauchy-Schwarz. We note that an analogous estimate also applies to the term $IV$.
 We will hide the term $\frac14\mathcal I$ on the left-hand side of \eqref{u6tg}, while the second term becomes, after integrating $II_6$ in $\theta$:
\begin{align}
\int_{1/6}^6|II_6|\,d\theta&\le C\int_{1/6}^6\int_{\Delta_{2r}(q,\tau)}|\nabla (u_k^{p/2})(x,\theta\hbar(x,t),t)|^2 |\hbar(x,t)|^2dt\,dxd\theta.\nonumber\\
&\le C\iint_{\cS(q,\tau,r,r_0,\frac16\hbar)}|\nabla u_k|^2|u_k|^{p-2} x_n\,dt\,dxdx_n,
\end{align}
which is an estimate similar to that obtained for $\mathcal D$. Observe also that the terms $I$, $II_1$ and $II_2$ also enjoy a similar estimate which we record as:
\begin{equation}
|I+II_1+II_2|\le C\iint_{\cSS}|\nabla u_k|^2|u_k|^{p-2}x_n\, dx_ndtdx.
\end{equation}
Next, we handle the terms $\mathcal B_{res1}$-$\mathcal B_{res3}$ and $\mathcal C$. The Carleson condition for $|\nabla \mathbf{a_n}|^2x_n$ and the Cauchy-Schwarz inequality imply
\begin{align}
&|\mathcal B_{res1}+\mathcal B_{res2}+\mathcal B_{res3}+\mathcal C|\lesssim
\iint_{\mathcal S(q,\tau,r,r_0,\theta\hbar)}|\nabla \mathbf{a_n}||u_k^{p/2}-c||\nabla_{x}(u_k^{p/2})|\,x_n\zeta\,dx_n\,dt\,dx
\nonumber\\
&\le \br{\iint_{\cSS}\abs{\nabla \mathbf{a_n}}^2\abs{u_k^{p/2}-c}^2x_n\,dx_ndx\,dt}^{1/2}\br{\iint_{\cSS}\abs{\nabla(u_k^{p/2})}^2x_n\,dx_ndx\,dt}^{1/2}\nonumber\\
   &\le C \|\abs{\nabla \mathbf{a_n}}^2x_n\|_C^{1/2}\|\tilde {N}_a((u_k^{p/2}-c)\,\1_{\cSS})\|_{L^2(\Delta_{2r})}
   \cdot\br{\iint_{\cSS}\abs{\nabla(u_k^{p/2})}^2x_n\,dx_ndx\,dt}^{1/2}\nonumber\\
&\le \varepsilon\|\tilde {N}_a((u_k^{p/2}-c)\,\1_{\cSS})\|_{L^2(\Delta_{2r})}^2 + C\iint_{\cSS}|\nabla u_k|^2|u_k|^{p-2}x_n\, dx_ndtdx.
\end{align}

Next, since $r|\nabla\zeta|\le c$, if the derivative falls on the cutoff function $\zeta$ we have
\begin{align*}
&|II_3+II_4+III| \lesssim \iint_{\cSS}\left|\nabla _x(u_k^{p/2})\right||u_k^{p/2}-c|\frac{x_n}{r}\,dx_n\,dt\,dx
\nonumber\\[4pt]
&\lesssim \left(\iint_{\cSS}|u_k^{p/2}-c|^{2}\frac{x_n}{r^{2}}\,dx_n\,dt\,dx\right)^{1/2} 
\br{\iint_{\cSS}\abs{\nabla u_k}^2|u_k^{p-2}|x_n\,dX\,dt}^{1/2}.
\end{align*}
By Fubini's theorem, and the fact that 
$\iint_{\gamma_a(x,t)\cap \cSS}y_n^{-n-1}dydsdy_n\le Cr$, the first integral on the right-hand side can be controlled by the nontangential maximal function:
\begin{multline}\label{eq.Neta-c}
    \iint_{\cSS}|u_k^{p/2}-c|^{2}\frac{x_n}{r^{2}}\,dx_n\,dt\,dx
\le \frac Cr\int_{\Delta_{2r}}\iint_{\gamma_a(x,t)\cap\cSS}\abs{u_k^{p/2}-c}^2y_n^{-n-1}dydsdy_n\\
\le C\int_{\Delta_{2r}}\tilde N_a\br{(u_k^{p/2}-c)\1_{\cSS}}^2dx\,dt.
\end{multline}
Hence, once again
\begin{multline}\nonumber
\abs{II_3+II_4+III}\le\\ \varepsilon\|\tilde {N}_a((u_k^{p/2}-c)\,\1_{\cSS})\|_{L^2(\Delta_{2r})}^2 + C\iint_{\cSS}|\nabla u_k|^2|u_k|^{p-2}x_n\, dx_ndtdx.
\end{multline}

In a similar spirit we analyze the term $II_7$. We write it as
\begin{eqnarray}\label{eq.DtNeta-c}
II_7&=&
\iint_{\mathcal S(q,\tau,r,r_0,\theta\hbar)}\partial_{t}[a_{nn}^{-1}(u_k^{p/2}-c)^2]x_n\zeta\,dx_n\,dt\,dx\nonumber
\\&-&\iint_{\mathcal S(q,\tau,r,r_0,\theta\hbar)}\partial_{t}(a_{nn}^{-1})(u_k^{p/2}-c)^2x_n\zeta\,dx_n\,dt\,dx\\\nonumber
&=&-\iint_{\mathcal S(q,\tau,r,r_0,\theta\hbar)}\partial_{t}(a_{nn}^{-1})(u_k^{p/2}-c)^2x_n\zeta\,dx_n\,dt\,dx\\\nonumber
&-&\iint_{\mathcal S(q,\tau,r,r_0,\theta\hbar)}a_{nn}^{-1}(u_k^{p/2}-c)^2x_n(\partial_t\zeta)\,dx_n\,dt\,dx\\\nonumber
&+&\int_{\partial\mathcal S(q,\tau,r,r_0,\theta\hbar)}(\mbox{boundary terms})x_n\zeta\nu_t\,dS=:II_8+II_9+II_{10},
\end{eqnarray}
where the boundary terms only arise on the portion of $\partial\mathcal S(q,\tau,r,r_0,\theta\hbar)$ where
$\zeta$ is supported (which means on the graph of the function $\theta\hbar$). $\nu_t$ denotes the co-normal 
in the $t$-direction. 
Note that on the graph of $\theta\hbar$, $\nu_t=-\frac{\dr_t\hbar}{J_h}$ and $dS=J_hdx\,dt$, where $J_h$ is defined above the computation for $II_{5}$. By \eqref{derh}, we have that $\abs{\nu_t}dS\le \frac{a^{-2}}{2x_n}dx\,dt$.
It follows that 
$$|II_{10}|\lesssim \frac{1}{2a^2}\int_{\Delta(q,\tau)}|(u_k^{p/2}-c)^2(x,\theta\hbar(x,t),t)|\zeta\,dt\,dx\le \frac14\mathcal I,
$$
where the last estimate holds if $a$  is chosen sufficiently large. We do so now. The solid integral term $II_9$ can be bounded as follows. Since $|\partial_t\zeta|\lesssim r^{-2}$ we split the area of integration into  two parts, the first one for $x_n/r\le\varepsilon$ and its complement.
On the set where $x_n/r\le\varepsilon$  we get that the integral is bounded by
\begin{equation}\label{smallf}
 \varepsilon\iint_{\cSS}|u_k^{p/2}-c|^{2}\frac{1}{r}\,dx_n\,dt\,dx\le\varepsilon \|\tilde{N}_a\br{(u_k^{p/2}-c)\1_{\cSS}}\|_{L^2(\Delta_{2r})}^2
\end{equation}
using \eqref{eq.Neta-c}.
The remaining part of the integral is bounded by $\frac{C}{r}\displaystyle\iint_{\mathcal{K}_\varepsilon}|u_k^{p/2}-c|^2\,dX\,dt$.

Another term with a similar estimate is $V$; the fact that $\partial_{x_n}\zeta$ vanishes on the set
$(-\infty,r_0+r)\cup(r_0+2r,\infty)$ means that it is supported away from the boundary of $\partial\Omega$ and therefore,
\begin{equation}
|V|\le \frac{C}{r}\iint_{\cSS\cap\set{x_n\ge r_0+r}}|u_k^{p/2}-c|^{2}\,dx_n\,dt\,dx
\le \frac{C}{r}\iint_{\mathcal K_{\varepsilon}}|u_k^{p/2}-c|^{2}\,dx_n\,dt\,dx.
\end{equation}

This leaves the term $II_8$ where we introduce $x_n=\frac12\partial^2_{x_nx_n}(x_n^2)$ in order to integrate by parts.
\begin{eqnarray}
2II_8&=&\iint_{\mathcal S(q,\tau,r,r_0,\theta\hbar)}\partial^2_{tx_n}(a_{nn}^{-1})(u_k^{p/2}-c)^2x_n^2\zeta\,dx_n\,dt\,dx\\\nonumber
&+&\iint_{\mathcal S(q,\tau,r,r_0,\theta\hbar)}\partial_{t}(a_{nn}^{-1})(u_k^{p/2}-c)\partial_{x_n}(u_k^{p/2})x_n^2\zeta\,dx_n\,dt\,dx\\\nonumber
&+&\iint_{\mathcal S(q,\tau,r,r_0,\theta\hbar)}\partial_t(a_{nn}^{-1})(u_k^{p/2}-c)^2x_n^2(\partial_{x_n}\zeta)\,dx_n\,dt\,dx\\\nonumber
&-&\int_{\partial\mathcal S(q,\tau,r,r_0,\theta\hbar)}\partial_t(a_{nn}^{-1})(u_k^{p/2}-c)^2x_n^2\zeta\,dS=:II_{11}+II_{12}+II_{13}+II_{14}.
\end{eqnarray}
Here, using the Carleson condition for $|\partial_t a_{nn}|^2x_n^3$ and the Cauchy-Schwarz inequality, term $II_{12}$  can be bounded by
\begin{align}
&| II_{12}|\lesssim
\iint_{\mathcal S(q,\tau,r,r_0,\theta\hbar)}|\partial_t a_{nn}||u_k^{p/2}-c||\nabla_{x}(u_k^{p/2})|\,x_n^2\zeta\,dx_n\,dt\,dx
\nonumber\\
&\le \br{\iint_{\cSS}\abs{\partial_t a_{nn}}^2\abs{u_k^{p/2}-c}^2x_n^3\,dx_ndx\,dt}^{1/2}\br{\iint_{\cSS}\abs{\nabla(u_k^{p/2})}^2x_n\,dX\,dt}^{1/2}\nonumber\\
   &\le C \|\abs{\partial_t a_{nn}}^2x_n^3\|_C^{1/2}\|\tilde {N}_a((u_k^{p/2}-c)\,\1_{\cSS})\|_{L^2(\Delta_{2r})}
   \cdot\br{\iint_{\cSS}\abs{\nabla(u_k^{p/2})}^2x_n\,dx_ndx\,dt}^{1/2}\nonumber\\
&\le \varepsilon\|\tilde {N}_a((u_k^{p/2}-c)\,\1_{\cSS})\|_{L^2(\Delta_{2r})}^2 + C\iint_{\cSS}|\nabla u_k|^2|u_k|^{p-2}x_n\, dx_ndtdx.
\end{align}
Since we assume $|\partial_t a_{nn}|x_n^2\le \delta$ for sufficiently small $\delta>0$, term $II_{14}$ is
bounded by $\frac14\mathcal I$. For term $II_{13}$, by the same estimate and the fact that $|\nabla_x\zeta|\lesssim r^{-1}$, we have
$$|II_{13}|\lesssim \delta \iint_{\mathcal S(q,\tau,r,r_0,\theta\hbar)} |u_k^{p/2}-c|^2\frac1r \,dx_n\,dt\,dx,$$
which therefore has the same bound as \eqref{smallf}, again for small $\delta>0$.
This leaves the term $II_{11}$, where we integrate by parts in $t$. When $\partial_t$ falls on $\zeta$  we obtain a term similar to $II_9$ since $|\partial_t a_{nn}|x_n^2$ is bounded. The same is true for the resulting boundary term which has bounds identical to the term $II_{10}$. It remains to bound
\begin{align}
&\left|-\iint_{\mathcal S(q,\tau,r,r_0,\theta\hbar)}\partial_{x_n}(a_{nn}^{-1})\partial_t[(u_k^{p/2}-c)^2]x_n^2\zeta\,dx_n\,dt\,dx\right|\nonumber\\&\lesssim
\iint_{\mathcal S(q,\tau,r,r_0,\theta\hbar)}|\partial_{x_n} a_{nn}||u_k^{p/2}-c||\partial_{t}(u_k^{p/2})|\,x_n^2\zeta\,dx_n\,dt\,dx
\nonumber\\
&\le \br{\iint_{\cSS}\abs{\partial_{x_n} a_{nn}}\abs{u_k^{p/2}-c}^2x_n\,dX\,dt}^{1/2}\br{\iint_{\cSS}\abs{\partial_t(u_k^{p/2})}^2x_n^3\,dx_ndx\,dt}^{1/2}\nonumber\\
   &\le C \||\partial_{x_n}a_{nn}|^2x_n\|_C^{1/2}\|\tilde {N}_a((u_k^{p/2}-c)\,\1_{\cSS})\|_{L^2(\Delta_{2r})}
   \cdot\br{\iint_{\cSS}\abs{\partial_t(u_k^{p/2})}^2x_n^3\,dx_ndx\,dt}^{1/2}\nonumber.
\end{align}
The constant in front of this term is small because the Carleson norm of the coefficient $a_{nn}$ is small. The last term in this product can be further estimated as a product of the Area function and the nontangential maximal functions: 

\begin{align}\label{eq.dtukp/2}
&\hskip-0.3cm\iint_{\cSS}\abs{\partial_t(u_k^{p/2})}^2x_n^3\,dx_ndx\,dt\\
&\lesssim \int_{2\Delta_r}{N}^{p-2}_a(u_k\,\1_{\cSS})\left(\iint_{\gamma_a(x,t)\cap \cSS}|\partial_t u_k|^2y_n^{-n+2}dYds\right) dx\,dt\nonumber\\
&\lesssim \int_{2\Delta_r}{N}^{p-2}_a(u_k\,\1_{\cSS})\left(\iint_{\gamma_a(x,t)\cap \cSS}[|\nabla^2 u_k|^2+y_n^{-2}|\nabla u_k|^2]y_n^{-n+2}\right) dx\,dt,\nonumber
\end{align}
since $|\partial_tu_k|=|\mathcal Lu_k|\sim|\nabla_x^2u_k|+|\nabla A||\nabla u_k|$.
Caccioppoli inequality for the second gradient implies that for any $j\in\set{1,\dots,n}$, and for any parabolic cube $Q_\rho$ with $Q_{2\rho}\subset\om$, there holds
\[
\iint_{Q_\rho}|\nabla \dr_j u|^2dX\,dt\lesssim\frac{1}{\rho^2}\iint_{Q_{2\rho}}|\dr_ju|^2dX\,dt+\iint_{Q_{2\rho}}|\dr_jA|^2|\nabla u|^2dX\,dt.
\]
For a parabolic Whitney cube $Q_{\rho}$, we further have $|\nabla A|\lesssim \rho^{-1}$ in $Q_{2\rho}$, and so the above gives that $\iint_{Q_\rho}|\nabla^2 u|^2dX\,dt\lesssim\frac{1}{\rho^2}\iint_{Q_{2\rho}}|\nabla u|^2dX\,dt$. 
Since the set $\gamma_a(y,s)\cap \cSS$ can be decomposed as a union of Whitney cubes whose enlargements have finite overlap, we apply the Caccioppoli inequality on each Whitney cube and H\"older's inequality with exponent $p/(p-2)$ to get that
\begin{multline*}
    \iint_{\cSS}\abs{\partial_t(u_k^{p/2})}^2x_n^3\,dX\,dt\\
\lesssim \| {N}_a(u_k\,\1_{\cSS})\|_{L^{p}(\Delta_{2r})}^{p}+
 \int_{2\Delta_r}\br{\iint_{\gamma_a(x,t)\cap\cSS}\abs{\nabla u_k}^2y_n^{-n}dYds}^{p/2}\,dX\,dt
\end{multline*}
Hence,  after combining all estimates we have that
\begin{align}
&\left|-\iint_{\mathcal S(q,\tau,r,r_0,\theta\hbar)}\partial_{x_n}(a_{nn}^{-1})\partial_t[(u_k^{p/2}-c)^2]x_n^2\zeta\,dx_n\,dt\,dx\right|\nonumber\lesssim\varepsilon \|\tilde {N}_a((u_k^{p/2}-c)\,\1_{\cSS})\|^2_{L^2(\Delta_{2r})}\\\nonumber&+\varepsilon
\| {N}_a(u_k\,\1_{\cSS})\|_{L^p(\Delta_{2r})}^{p}+\varepsilon \int_{2\Delta_r}\br{\iint_{\gamma_a(y,s)\cap\cSS}\abs{\nabla u_k}^2x_n^{-n}}^{p/2}\,dX\,dt.
\end{align}

Finally, we put together all terms and integrate in $\theta$. Observing that $\cSSm\subset \cS(\Delta_r,\hbar)$, the above analysis ultimately yields \eqref{TTBBMM}.
Thus the claim follows.
\end{proof}

The following two deductions will also be used for tangential derivatives
and $H$. Their hypotheses separate the graph geometry from the PDE estimates.

\begin{lemma}[local graph reduction]\label{lem:graphReduction}
Let $F\subset\Delta$, where $\Delta$ is a boundary ball of radius $r$,
and let $h\ge0$ satisfy \eqref{Eqqq-5}, \eqref{derh} and \eqref{derh'}.
For a function $z$ write
$z_{L^2}(X,t)=(\fint_{Q_{x_n/2}(X,t)}|z|^2)^{1/2}$.
Fix $\beta,c_*,L>0$ and $c\in\R$. Assume
\[
 (M_{h,Lr}z_{L^2})(P,h(P))\ge c_*\beta\quad(P\in F),
 \qquad |c|\le c_*\beta/2.
\]
Suppose also that the localized graph maximal operator is uniformly
bounded on $L^2$ and that the Whitney averages obey
\[
 \int_{C_1\Delta}(z-c)_{L^2}(x,h(x,t),t)^2\,dx\,dt
 \lesssim\int_{1/6}^6\int_{C_2\Delta}
 |z(x,\theta h(x,t),t)-c|^2\,dx\,dt\,d\theta,
\]
where $C_1\Delta$ contains the graph balls used by $M_{h,Lr}$ above $F$.
Then
\begin{equation}\label{eq:graphReduction}
 |F|\lesssim\beta^{-2}\int_{1/6}^6\int_{C_2\Delta}
 |z(x,\theta h(x,t),t)-c|^2\,dx\,dt\,d\theta.
\end{equation}
\end{lemma}
\begin{proof}
Minkowski's inequality gives $z_{L^2}\le(z-c)_{L^2}+|c|$.
Thus the graph maximal function of $(z-c)_{L^2}$ is at least
$c_*\beta/2$ on $F$. Square and integrate, use its $L^2$ bound, and then
apply the assumed Whitney-average comparison.
\end{proof}

\begin{lemma}[integration of a good-$\lambda$ inequality]\label{lem:goodLambdaIntegration}
Let $F,G\ge0$ on $\pom$ and suppose that, for every $\beta>0$ and
$0<\gamma<\gamma_0$,
\[
 |\{F>\beta,\ G\le\gamma\beta\}|
 \le C_0\gamma^2|\{MF>\kappa\beta\}|,
 \qquad 0<\kappa<1.
\]
If $1<q<\infty$ and $F\in L^q$, then
$\|F\|_{L^q}\le C\|G\|_{L^q}$, where $C$ depends only on
$n,q,C_0,\kappa$ and $\gamma_0$.
\end{lemma}
\begin{proof}
Integration against $q\beta^{q-1}\,d\beta$ gives
\[
 \|F\|_q^q\le C_0\gamma^2\kappa^{-q}\|MF\|_q^q
 +\gamma^{-q}\|G\|_q^q.
\]
The $L^q$ bound for $M$ allows us to choose $\gamma$ so that the first
term is at most $\|F\|_q^q/2$, and absorb it.
\end{proof}

\begin{lemma}\label{LGL} Fix $p\ge8$ such that $p/2$ is an even integer. Let $\mathcal L$ be an operator as in Lemma \ref{S3:L8-alt1} and fix $a>0$ sufficiently large as in Lemma \ref{S3:L8-alt1}.  
There exist $\kappa\in(0,1)$, $b>a$, $\gamma_0>0$ and $C>0$, depending only on $n,p,a$ and the ellipticity constants, such that the following holds.

For any $\varepsilon\in(0,1)$, if $\||\nabla \mathbf{a_{n}}|^2x_n\|_{C}^{1/2}+\||\partial_{t}a_{nn}|^2x_n^3\|_{C}^{1/2}+\||\partial_ta_{nn}|x_n^2\|_{L^\infty}$ is sufficiently small (depending on $\varepsilon$), where $ \mathbf{a_{n}}$ is the $n$-th row of the matrix $A$, then for any $u_k$ as in Lemma \ref{S3:L8-alt1}, for any $\gamma\in(0,\gamma_0)$, $\beta>0$, there holds
\begin{equation}    
  \abs{E_{1,\beta}}\le C\gamma^2\abs{\set{(x,t)\in\R^{n-1}\times\R: M(\tilde{N}_{2,a}(u_k^{p/2}))(x,t)>\kappa\beta}}.\label{eq.gdLmd}
\end{equation}
where
\begin{multline}\label{def.E1}
    E_{1,\beta}:=
\Big\{(x,t)\in {\mathbb R}^{n-1}\times\R:\, \tilde{N}_{2,a}(u_k^{p/2})>\beta,\\ \varepsilon^{-1/p}S_{p,b}(u_k)+{\varepsilon^{1/p}S_{2,b}(u_k)}\le(\gamma\beta)^{2/p}, \varepsilon^{1/p}\tilde N_{2,b}(u_k)<(\gamma\beta)^{2/p}\Big\}.
\end{multline}
\end{lemma}

\begin{proof}
Let $b_0$ and $\gamma_0$ be as in Lemma \ref{l6}.
Let $\varepsilon>0$, $\gamma\in(0,\gamma_0)$, $\beta>0$ be fixed. 
Observe that $E_{2,\beta}:=\set{(x,t)\in\R^{n-1}\times\R: M(\tilde{N}_{2,a}(u_k^{p/2}))(x,t)>\kappa\beta}$ is an open subset of $\R^{n-1}\times\R$. We can assume that $E_{2,\beta}$ is a nonempty proper subset of $\R^{n-1}\times\R$, as otherwise the estimate \eqref{eq.gdLmd} is trivial. 
By Vitali's covering lemma, we can find a non-overlapping collection of parabolic balls $\set{\Delta(x_i,t_i)}_{i\in \N}$ 
such that $2\Delta(x_i,t_i)\subset E_{2,\beta}$, $10\Delta(x_i,t_i)\cap E_{2,\beta}^c\neq\emptyset$ for each $i$, and  that $E_{2,\beta}\subset\bigcup_{i\in \N}5\Delta(x_i,t_i)$. Denote $\Delta_i:=5\Delta(x_i,t_i)$, and denote by $r_i$ the radius of $\Delta_i$.
Set $F_i:=E_{1,\beta}\cap\Delta_i$. Note that $E_{1,\beta}\subset\bigcup_{i\in\N}F_i$ as $E_{1,\beta}\subset E_{2,\beta}$. 
The claim \eqref{eq.gdLmd} will follow if we prove that $|F_i|\le C\gamma^2|\Delta_i|$.
Let us denote 
\[
 \eta_{L^p}(x,x_n,t):=  \left(\fiint_{Q_{x_n/2}(x,x_n,t)}\abs{u_k(Y,s)}^pdYds\right)^{1/2},
\]
where $Q_{x_n/2}(x,x_n,t)$ is a parabolic Whitney cube centered at $(x,x_n,t)$ in $\Rn_+\times\R$. Take $b\ge b_0$ and let $\hbar:=\hbar_{\beta,a}(\eta_{L^p})$ be as in \eqref{h}. By Lemmas \ref{l6} and \ref{S3:L6}, $(M_{\hbar}(\eta_{L^p}))\big(x,\hbar(x,t),t\big)\geq\,C\beta$ for $(x,t)\in F_i$. We claim that this is also true when $M_{\hbar}(\eta_{L^p})$ is replaced  with the localized maximal function $M_{\hbar, 6ar_i}(\eta_{L^p})$ at scale $\sim r_i$. In fact, since $2\Delta_i\cap E_{2,\beta}^c\neq\emptyset$, we have
\begin{equation}\label{eq.MNsmall}
    M(\tilde N_{2,a}(u_k^{p/2}))(y,s)\le\kappa\beta \quad\text{for some }(y,s)\in 2\Delta_i\cap E_{2,\beta}^c.
\end{equation}
From this one deduces that 
\begin{equation}\label{eq.etasmall}
  \eta_{L^p}(x,x_n,t)
  \le 2^{n+1}\kappa\beta \quad\text{for  }(x,t)\in\Delta_i \text{ and }x_n>2r_i.
\end{equation}
In fact, suppose that \eqref{eq.etasmall} is false for some $(Z,\tau)=(z,z_n,\tau)$ with $(z,\tau)\in\Delta_i$ and $z_n>2r_i$. Then $\tilde{N}_{2,a}(u_k^{p/2})(x',t')>2^{n+1}\kappa\beta$ for all $(x',t')\in S(Z,\tau)$, where $S(Z,\tau)$ is defined as in \eqref{def.SXt}, which is a parabolic ball centered at $(z,\tau)$ of radius at least $2ar_i$. Therefore, since $a>1$,
\[
\kappa\beta<\br{\frac{2a}{3+2a}}^{n+1}2^{n+1}\kappa\beta<\fint_{\Delta_{(3+2a)r_i}(y,s)}\tilde{N}_{2,a}(u_k^{p/2})\,dx'dt'\le M(\tilde N_{2,a}(u_k^{p/2}))(y,s),
\]
a contradiction to \eqref{eq.MNsmall}.

We choose $\kappa\le 2^{-n-1}$. In view of \eqref{eq.etasmall}, for $(x,t)\in F_i$, $\tilde N_{2,a}(u_k^{p/2})(x,t)=\tilde N_{2,a}^{2r_i}(u_k^{p/2})(x,t)$. It also guarantees that $\hbar\le 2r_i$ on $\Delta_i$. Moreover, it implies that the ``radius'' of the boundary ball $R$ constructed in Lemma \ref{l6} is bounded by $6a r_i$. Therefore, we have justified that 
\begin{equation}\label{eq.Mhrbig}
    \text{ for } (x,t)\in F_i,\ (M_{\hbar,6ar_i}(\eta_{L^p}))\big(x,\hbar(x,t),t\big)\geq\,C\beta.
\end{equation}
We define 
\begin{equation}
    \hbar_i(x,t):=\max\set{\hbar(x,t),c_{a,p}\dist_p\br{(x,t),F_i}},
\end{equation}
where $c_{a,p}>0$ is fixed sufficiently small, and set
$\cS_i:=\cS(\Delta_{20ar_i},\hbar_i)$ as in Lemma \ref{S3:L8-alt1}.
The squared parabolic distance to $F_i$ is $1$-Lipschitz in time,
so $c_{a,p}\le a^{-1}$ ensures that $\hbar_i$ satisfies
\eqref{Eqqq-5}, \eqref{derh} and \eqref{derh'}.
For the lower-bound ball $R$ produced from a stopping point of height
$y_n$ in Lemma \ref{l6}, its proof gives
$\inf_R\hbar\gtrsim y_n$ and $\sup_{Q\in R}d_p(Q,P)\lesssim_{a,p}y_n$,
where $P\in F_i$ is the original boundary point. Choosing $c_{a,p}$
still smaller gives $\hbar_i=\hbar$ on that ball, as well as at $P$.
Thus Lemma \ref{S3:L6} gives \eqref{eq.Mhrbig} on the modified graph.
Increasing $b$ to account for $c_{a,p}^{-1}$ also gives 
\begin{equation}\label{SiProp}
    \text{for any }(x,t)\in40a\,\Delta_i\setminus F_i,\ \exists\, (y,s)\in F_i\text{ such that }
    \quad\cS_i\cap\gamma_{2a}(x,t)\subset \gamma_b(y,s)
\end{equation}
for $b\ge b_0$ sufficiently large.
Property \eqref{SiProp} holds because we ensure that the graph of $\hbar_i$ is sufficiently far away from the boundary even when it is not above $F_i$. 
 We choose the constant  
\[c:=\fint_{t\in I_{i}^0}\int_{X\in\Rn} u_k^{p/2}(X,t)\eta^2(X)\,dX\,dt\ge 0.\]
Here $\eta$ is a smooth nonnegative cutoff
function in spatial variables, supported in $B_i^0$ and satisfying
\[
 \int_{B_i^0}\eta^2\,dX=1,\qquad
 \|\eta\|_{L^\infty}\lesssim r_i^{-n/2},\qquad
 \|\nabla\eta\|_{L^\infty}\lesssim r_i^{-n/2-1}.
\]
Such a function is obtained by rescaling a fixed smooth bump and
normalizing its $L^2$ norm. $I_{i}^0$ is an interval in $\R$ centered at $t_i$ with length $\sim r_i^2$ and $B_i^0\subset\Rn$ is a ball with radius $\sim r_i$ located so that $B_i^0\times I_i^0 \subset \cS_i\cap\set{(x,x_n,t): (x,t)\in\Delta_i, x_n>2r_i}$.
By \eqref{eq.etasmall} and $\eta^2\lesssim r_i^{-n}$, $0\le c\le C\kappa\beta$.
Choose $\kappa$ small relative to the constant in \eqref{eq.Mhrbig}.
Apply Lemma \ref{lem:graphReduction} with $z=u_k^{p/2}$ and $h=\hbar_i$:
the lower bound on the modified graph was established above.
The graph maximal operator is $L^2$ bounded, and the required comparison
with $C_1\Delta_i=10a\Delta_i$ and $C_2\Delta_i=20a\Delta_i$ is the
Whitney-averaging argument of \cite{DLP1}, adapted from
\cite[Lemma 5.6]{DHM}. Thus \eqref{eq:graphReduction} and
Lemma \ref{S3:L8-alt1} give
\begin{multline}\label{eq.Fiest}
    C'\abs{F_i}\le \frac{C}{\beta^2}(1+\varepsilon^{-1})\iint_{\cS_i}\abs{\nabla u_k}^2|u_k|^{p-2}x_n\, dx_ndtdx\\
+\frac{\varepsilon}{\beta^2}\br{\|\tilde{N}_{2,a}\br{(u_k^{p/2}-c)\1_{\cS_i}}\|_{L^2(40a\Delta_{i})}^2
+\| N_{a}(u_k\,\1_{\cS_i})\|_{L^p(40a\Delta_{i})}^p}\\+\frac{\varepsilon}{\beta^2}\left(\iint_{\cS_i}\abs{\nabla u_k}^2x_n\, dx_ndtdx\right)^{p/2}    +\frac{C}{\beta^2r_i}\iint_{\mathcal{K}_\varepsilon}|u_k^{p/2}-c|^{2}\,dX\,dt,
\end{multline}
where $\mathcal{K}_\varepsilon=\mathcal S_i\cap \set{x_n>\varepsilon r_i}$.
Our goal is to bound the right-hand side of \eqref{eq.Fiest} by $C\gamma^2\abs{\Delta_i}$, which will imply the desired estimate \eqref{eq.gdLmd} by summing over $i$.
To this end, we write 
\[
\iint_{\cS_i}\abs{\nabla u_k}^2|u_k|^{p-2}x_n\, dx_ndtdx
\le C\int_{40a\Delta_i}\iint_{\gamma_a(x,t)}\abs{\nabla u_k}^2|u_k|^{p-2}\1_{\cS_i}(y,y_n,s)\frac{dydy_nds}{y_n^n}dx\,dt.
\]    
When $(x,t)\in F_i$, it holds trivially that $\iint_{\gamma_a(x,t)}\abs{\nabla u_k}^2|u_k|^{p-2}\1_{\cS_i}(y,y_n,s)\frac{dydy_nds}{y_n^n}\le S_{p,b}(u_k)(x,t)^{p}$. 
When $(x,t)\in 40a\Delta_i\setminus F_i$, by \eqref{SiProp}, we can find some $(x_0,t_0)\in F_i$ so that \newline
$\iint_{\gamma_a(x,t)}\abs{\nabla u_k}^2|u_k|^{p-2}\1_{\cS_i}(y,y_n,s)\frac{dydy_nds}{y_n^n}\le S_{p,b}(u_k)(x_0,t_0)^{p}$. Hence, by the definition of $E_{1,\beta}$, we conclude that 
\begin{equation}
\iint_{\gamma_a(x,t)}\abs{\nabla u_k}^2|u_k|^{p-2}\1_{\cS_i}(y,y_n,s)\frac{dydy_nds}{y_n^n}\le \varepsilon\gamma^2\beta^2 \quad\text{for all }(x,t)\in 40a\Delta_i,
\end{equation}
and so 
\begin{equation}\label{eq.sqrSi}
    \iint_{\cS_i}\abs{\nabla u_k}^2|u_k|^{p-2}x_n\,dx_ndtdx\le C\varepsilon\gamma^2\beta^2\abs{\Delta_i}.
\end{equation}
A similar argument applies for the term
$\left(\iint_{\gamma_a(x,t)}\abs{\nabla u_k}^2\1_{\cS_i}(y,y_n,s)\frac{dydy_nds}{y_n^n}\right)^{p/2}$
which is $\le  S_{2,b}(u_k)(x,t)^p\le \varepsilon^{-1}\gamma^2\beta^2$ when  $(x,t)\in F_i$, with an analogous estimate outside $F_i$.

The terms  $\|\tilde{N}_{2,a}((u_k^{p/2}-c)\1_{\cS_i})\|_{L^2(40a\Delta_{i})}^2$ and $\| N_a(u_k\1_{\cS_i})\|_{L^p(40a\Delta_{i})}^p$ are controlled similarly. Recall that  since $\abs{c}\le|\eta_{L^p}|$ has the bound given in \eqref{eq.etasmall} the contribution of this term is at most $C\gamma^2\beta^2|\Delta_i|$.
We note that for any $(x,t)\in 40a\Delta_i$, there exists $(x_0,t_0)\in F_i$ such that $N_a(u_k\1_{\cS_i})^{p/2}(x,t)\lesssim \tilde N_{2,2a}(u_k^{p/2}\1_{\cS_i})(x,t)\le \tilde N_{2,b}(u_k)(x_0,t_0)^{p/2}$ due to \eqref{SiProp} and the fact that $L^\infty$ version of the nontangential maximal function for solutions $u_k$ is controlled by $L^p$ version of this function for cones of slightly enlarged aperture. Therefore, by the definition of $F_i$ we get that
\begin{equation}
\varepsilon\|\tilde N_a(u_k^{p/2}\,\1_{\cS_i})\|_{L^2(40a\Delta_{i})}^2+\varepsilon\| N_a(u_k\1_{\cS_i})\|_{L^p(40a\Delta_{i})}^p\le C\gamma^2\beta^2\abs{\Delta_i}.
\end{equation}
 For the term $\frac{1}{r_i}\iint_{\mathcal{K}_\varepsilon}|u_k^{p/2}-c|^{2}\,dX\,dt$, we first use the triangle inequality to write 
\begin{multline*}
   \abs{u_k^{p/2}(X,t)-c}^2 \le 3\abs{u_k^{p/2}(X,t)-\fint_{B_i^0}u_k^{p/2}(Y,t)dY}^{2}\\
   + 3\abs{\fint_{B_i^0}u_k^{p/2}(Y,t)dY-\int_{\Rn}u_k^{p/2}(Y,t)\eta^2(Y)dY}^2
   + 3\abs{\int_{\Rn}u_k^{p/2}(Y,t)\eta^2(Y)dY-c}^2,
\end{multline*}
and we control the term $\frac{1}{r_i}\iint_{\mathcal{K}_\varepsilon}|u_k^{p/2}-c|^{2}\,dX\,dt$ by the sum of 3 integrals corresponding to the 3 terms on the right-hand side. We call the 3 integrals $J_i$, $i=1,2,3$. For $J_1$, we use the Poincar\'e inequality in the spatial variables to get that 
\[
J_1\le Cr_i\iint_{\mathcal{K}_\varepsilon}\abs{\nabla(u_k^{p/2})}^2dX\,dt\le C\varepsilon^{-1}\iint_{\cS_i}\abs{\nabla u_k}^2|u_k|^{p-2}x_n\,dx_ndtdx,
\]
where in the last step we recalled that $x_n>\varepsilon r_i$ in $\mathcal{K}_{\varepsilon}$. So this term is again bounded by an estimate similar to \eqref{eq.sqrSi}. 
To estimate $J_2$ and $J_3$, we denote by $c(t):=\int_{\Rn}u_k^{p/2}(Y,t)\eta^2(Y)dY$, and denote by $I_i$ the time interval of $\mathcal{K}_\varepsilon$. We write 
\begin{multline*}
  J_2\le Cr_i^{n-1}\int_{I_i} \abs{\fint_{B_i^0}u_k^{p/2}(Y,t)dY-c(t)}^2dt\\
  =Cr_i^{n-1}\int_{I_i} \abs{\int_{Z\in\Rn}\br{u^{p/2}(Z,t)-\fint_{B_i^0}u_k^{p/2}(Y,t)dY}\eta^2(Z)dZ}^2dt\\
  \le \frac{C}{r_i}\int_{I_i} \int_{B_i^0}\abs{u^{p/2}(Z,t)-\fint_{B_i^0}u_k^{p/2}(Y,t)dY}^2dZ\,dt.
\end{multline*}
By the Poincar\'e inequality in the spatial variables again and the observation that $I_i\times B_i^0\subset \mathcal{S}_i\cap\set{x_n>2r_i}$,
$J_2$ can be bounded as $J_1$. To estimate $J_3$, set
$m=p/2$, $v=u_k$, $r=r_i$, and $Q=B_i^0\times I_i$, and write
\[
 \mathcal E_q(v;Q):=r^{-n}\iint_Q |v|^{q-2}|\nabla v|^2\,dX\,dt
 \qquad(q\ge2).
\]
Choose $I_i^0\subset I_i$, as we may in the construction above, so
$|c(t)-c|\le\operatorname*{osc}_{I_i}c$.
We calculate the evolution of $c(t)=\int v^m\eta^2$ directly.
Testing the equation with $mv^{m-1}\eta^2$ gives, for $t<t'$ in $I_i$,
\begin{multline}\label{eq:halfMomentEvolution}
 c(t')-c(t)
 =-2m\iint_{B_i^0\times[t,t']}v^{m-1}\eta\,A\nabla v\cdot\nabla\eta\,dX\,ds\\
 -m(m-1)\iint_{B_i^0\times[t,t']}v^{m-2}\eta^2
                         A\nabla v\cdot\nabla v\,dX\,ds.
\end{multline}
The identity is justified by Steklov time averages and truncations of
the polynomial test function, using local boundedness and local energy
integrability.
For the slab solutions, the test vanishes on $x_n=k$, so their zero
extension contributes no additional boundary term.
The estimates hold for almost every pair of endpoints and hence for
the absolutely continuous representative of $c$.
Since $2m-2=p-2$, Cauchy--Schwarz bounds the first integral by
\[
 C\left(\iint_Q |v|^{p-2}|\nabla v|^2\right)^{1/2}
   \left(\iint_Q\eta^2|\nabla\eta|^2\right)^{1/2}
 \le C\mathcal E_p(v;Q)^{1/2}.
\]
Here $\iint_Q\eta^2|\nabla\eta|^2\lesssim r^{-n}$, using
$\int\eta^2=1$ and $|I_i|\lesssim r^2$. The second integral is
bounded by $C\mathcal E_m(v;Q)$. Thus, suppressing $(v;Q)$,
\[
 \operatorname*{osc}_{I_i}c\le C_p(\mathcal E_p^{1/2}+\mathcal E_{p/2}),
 \qquad
 \mathcal E_{p/2}\le
 \mathcal E_p^{\frac{p-4}{2(p-2)}}
 \mathcal E_2^{\frac{p}{2(p-2)}},
\]
where the second inequality is H\"older's inequality. Squaring and
applying Young's inequality gives
\begin{equation}\label{eq.ct-c}
 \begin{split}
 |c(t)-c|^2
 &\le C_p\big(1+\varepsilon^{-2/(p-4)}\big)\mathcal E_p
                       +C_p\varepsilon\mathcal E_2^{p/2}\\
 &\le C_p\varepsilon^{-1}\mathcal E_p
                       +C_p\varepsilon\mathcal E_2^{p/2}.
 \end{split}
\end{equation}
The last step uses $p\ge8$ and $0<\varepsilon<1$.
Since $Q\subset\mathcal S_i\cap\{x_n>2r_i\}$, the square-function
bounds already established using \eqref{SiProp} imply
\[
 \mathcal E_p(v;Q)\le C\varepsilon\gamma^2\beta^2,
 \qquad \mathcal E_2(v;Q)^{p/2}\le C\varepsilon^{-1}\gamma^2\beta^2.
\]
Consequently,
\(
J_3\le Cr_i^{n+1}\sup_{t\in I_i}|c(t)-c|^2\le C\gamma^2\beta^2 r_i^{n+1}
\), and after adding up all estimates we obtain the desired bound: $|F_i|\le C\gamma^2|\Delta_i|$.

\end{proof}

To prove Theorem \ref{thm.NlessS}, put $m=p/2$, $q>1$ and
$r=mq>p/2$. Apply Lemma \ref{lem:goodLambdaIntegration} with
\[
 F=\tilde N_{2,a}(u_k^m),\qquad
 G=C_p\big[\varepsilon^{-1/2}S_{p,b}(u_k)^m
 +\varepsilon^{1/2}\{S_{2,b}(u_k)^m+\tilde N_{2,b}(u_k)^m\}\big].
\]
For $C_p$ sufficiently large, $\{F>\beta,G\le\gamma\beta\}$ is contained
in \eqref{def.E1}, so Lemma \ref{LGL} supplies the required good-$\lambda$
inequality. Under the a priori finiteness assumption, this gives
\begin{equation}\label{eq:commonNpBound}
 \|\tilde N_{p,a}(u_k)\|_{L^r}
 \le C\varepsilon^{-1/p}\|S_{p,b}(u_k)\|_{L^r}
 +C\varepsilon^{1/p}\big(\|S_{2,b}(u_k)\|_{L^r}
 +\|\tilde N_{2,b}(u_k)\|_{L^r}\big).
\end{equation}
Here $F=\tilde N_{p,a}(u_k)^m$, and $\gamma$ was chosen independently
of $\varepsilon$. Since $S_2\lesssim\tilde N_2$ for solutions,
$\tilde N_2\le\tilde N_p$, and changes of aperture preserve the $L^r$
norms up to a constant, the last two terms are absorbed by choosing
$\varepsilon$ small. We then choose the Carleson norm small enough for
Lemma \ref{LGL}. This gives the required $\tilde N_2$ bound for $u_k$;
letting $k\to\infty$ proves Theorem \ref{thm.NlessS}.\qed

\subsection{$N<S_p$ for $\nabla_T u$}\label{S.NleSp2}

\begin{theorem}\label{thm.NlessSdT}
   Let $A$ be a matrix with bounded measurable coefficients that satisfies the ellipticity condition, and that
    $\norm{|\nabla A|x_n}_{L^\infty}+\||\dr_t A|^2x_n^3\|_{\mathcal{C}}<\infty$.
   Let $\LL=-\dr_t+\divg(A\nabla\cdot)$.
   For any $\varepsilon\in(0,1)$, there exists $\delta>0$ such that if $$\||\nabla A|^2x_n\|_{C}+\norm{|\nabla_T A|x_n}_{L^\infty}+\||\partial_{t}a_{nn}|^2x_n^3\|_{C}+\||\partial_ta_{nn}|x_n^2\|_{L^\infty}<\delta,$$  then for any $p\ge8$ such that $p/2$ is an even integer, for any $a>0$, 
   $f\in  L^{q}(\pom)\cap  \Hdot^{-1/4}_{\partial_t-\Delta_x}(\partial\Omega)$,
   there exists a 
constant $C=C(n,p,a,\varepsilon, \lambda,\Lambda)>0$, such that for the energy solution\footnote{As can be seen from the proof, the boundary condition does not matter. This remark applies also to Theorem~\ref{thm.NleSpH}, Lemmas~\ref{NPl1} and \ref{NPl2}.} $u$ to $\LL u=0$, 
$\partial^A_\nu u\Big|_{\partial\Omega}=f$,
there holds for all $q>p/2$
\begin{equation}\label{eq.N<SLp2}
    \|{\tilde N}(\nabla_Tu)\|_{L^q(\pom)}\le C\norm{S_p(\nabla_T u)}_{L^q(\pom)}+ \varepsilon\br{\norm{\tilde N(\nabla u)}_{L^q(\pom)}+\norm{S_2(\nabla u)}_{L^q(\pom)}},
 \end{equation}   
provided we know a priori that $\|\tilde N(\nabla_T u)\|_{L^q(\partial\Omega)}<\infty$. 
\end{theorem}

We shall only highlight the modifications needed from Subsection~\ref{S.NleSp1}. First, we change the definition of $w$: it is now the $L^p$ average of $|\nabla_Tu|$, that is, for $(X,t)=(x,x_n,t)\in\om$,
\begin{equation}\label{def.w2}
w(X,t)=\left(\fiint_{Q_{x_n/2}(X,t)}|\nabla_T u(Y,s)|^p\,dY\,ds\right)^{1/p}
=\left(\fiint_{Q_{x_n/2}(X,t)}\br{\sum_{i=1}^{n-1}|\dr_{y_i} u(Y,s)|^2}^{p/2}\,dY\,ds\right)^{1/p}.
\end{equation}

We also introduce for $k=1,\dots,n-1$,
\begin{equation}\label{def.wk}
w_k(X,t):=\left(\fiint_{Q_{x_n/2}(X,t)}|\dr_{y_k} u(Y,s)|^p\,dY\,ds\right)^{1/p}, \quad \text{and }u_k:=\dr_k u.
\end{equation}

\begin{lemma}\label{lem.NleSp2-1} 
Let $k\in\set{1,2,\dots,n-1}$. Let $w_k$ be defined as in \eqref{def.wk} and let $\hbar_{\nu,a}$ be defined as in \eqref{h}.
Then for any $a>0$ there exists $b_0=b_0(a)>a$ and $\gamma_0=\gamma_0(a)>0$ such that the following holds. 
Having fixed an arbitrary $\nu>0$, $b\ge b_0$, and $\gamma\le\gamma_0$ for each point $(x,t)\in\partial\Omega$ from the set 
\begin{equation}
\big\{(x,t):\, N_{a}(w_k)(x,t)>\nu\mbox{ and }S_{p,b}(u_k)(x,t)+\norm{|\dr_{k}A|x_n}_{L^\infty}^{1/p}N_b(\nabla u)\leq\gamma\nu\big\}
\end{equation}
there exists a boundary ball $R$ with $(x,t)\in K_0 R$ for some $K_0=K_0(n,p,a)>1$ and such that
\begin{equation}
\big|w_k\big(z,\hbar_{\nu,a}(w_k)(z,\tau),\tau\big)\big|>\nu/{8}\,\,\text{ for all }\,\,(z,\tau)\in R.
\end{equation}
\end{lemma}

\begin{proof}
    We proceed as in the proof of Lemma \ref{l6}. The first modification is needed in \eqref{eq.v'-v}, where the PDE of $u_k$ is used. Since $u_k$ satisfies 
    \begin{equation}\label{eq.ukPDE}
        -\dr_t u_k + \divg(A\nabla u_k)=-\divg((\dr_{x_k}A)\nabla u), 
    \end{equation}
    there is one more term in \eqref{eq.v'-v}. That is, we have 
    \begin{multline}\label{eq.v'-v2a}
    v_{av}(t_0')-v_{av}(t_0)=\iint_{\mathbb R^n\times\{t_0'\}}|u_k|^{p}\eta^2\,dX-\iint_{\mathbb R^n\times\{t_0\}}|u_k|^{p}\eta^2\,dX\\
   =-2p\iint_{\mathbb R^n\times[t_0,t_0']}(A\nabla u_k\cdot\nabla\eta)|u_k|^{p-2}u_k\eta\,dX\,dt\\
-p(p-1)\iint_{\mathbb R^n\times[t_0,t_0']}(A\nabla u_k\cdot\nabla u_k)|u_k|^{p-2}\eta^2\,dX\,dt
+p\,I,
\end{multline}
where 
\[
I :=  \iint_{\Rn\times [t_0,t_0']}\divg ((\dr_{k}A)\nabla u)|u_k|^{p-2}u_k\eta^2dX\,dt
\]
is the new term coming from the right-hand side of \eqref{eq.ukPDE}. Integrating by parts, we get 
\begin{multline*}
    I = - \iint_{\Rn\times [t_0,t_0']}(\dr_{k}A)\nabla u \cdot\nabla\br{|u_k|^{p-2}u_k\eta^2}dX\,dt\\
    =-(p-1) \iint_{\Rn\times [t_0,t_0']}(\dr_{k}A)\nabla u \cdot\nabla u_k |u_k|^{p-3}u_k\eta^2 dX\,dt\\
    -2 \iint_{\Rn\times [t_0,t_0']}(\dr_{k}A)\nabla u \cdot\nabla \eta |u_k|^{p-2}u_k\eta dX\,dt=: I_1+I_2.
\end{multline*}
By Cauchy-Schwarz, 
\begin{multline*}
    |I_1|\lesssim \br{\iint_{\Rn\times [t_0,t_0']}\abs{\nabla u_k}^2|u_k|^{p-2}\eta^2dX\,dt}^{1/2} \br{\iint_{\Rn\times [t_0,t_0']}|\dr_{k}A|^2|\nabla u|^2 |u_k|^{p-2}\eta^2 dX\,dt}^{1/2}\\
    \lesssim \norm{|\dr_k A|x_n}_{L^\infty} S_{p,b}(u_k)^{p/2}(x,t)N_b(\nabla u)^{p/2}(x,t)\\
    \le C S_{p,b}(u_k)^{p}(x,t)+\norm{|\dr_k A|x_n}_{L^\infty} N_b(\nabla u)^p(x,t)
\end{multline*}
using the properties of $\eta$ and the fact that $I_{y_n}(s)\times \wt B\subset Q_{4y_n/5}\subset \gamma_b(x,t)$. For $I_2$, we use
$|\eta\nabla\eta|\lesssim y_n^{-n-1}$, the bound $x_n\approx y_n$
on the supporting cylinder, and its volume $\lesssim y_n^{n+2}$ to obtain
\[
    |I_2|\lesssim y_n^{-n-1}\iint_{\tilde B\times [t_0,t_0']}|\dr_{k}A||\nabla u|^p\,dX\,dt\lesssim \norm{|\dr_k A|x_n}_{L^\infty} N_b(\nabla u)^p(x,t).
\]
Altogether, we have that 
for any $t_0',t_0\in I_{y_n}(s)$, 
$$
\left| v_{av}(t_0')-v_{av}(t_0)\right|\le \frac12\sup_{\tau\in I_{y_n}(s)}v_{av}(\tau)+ CS_{p,b}( u_k)^p(x,t) + C\norm{|\dr_k A|x_n}_{L^\infty} N_b(\nabla u)^p(x,t).
$$
The condition $S_{p,b}(u_k)(x,t)+\norm{|\dr_k A|x_n}_{L^\infty}^{1/p} N_b(\nabla u)(x,t) \le \gamma\nu$ ensures that the rest of the argument goes exactly as in the proof of Lemma \ref{l6}.
\end{proof}

For $w$ in \eqref{def.w2}, Lemma \ref{S3:L6} gives
$(M_{\hbar_{\nu,a}}w)(x,\hbar_{\nu,a}(w)(x,t),t)\gtrsim\nu$ whenever
$N_a(w)(x,t)>\nu$ and
\[
 S_{p,b}(\nabla_Tu)(x,t)
 +\norm{|\nabla_TA|x_n}_{L^\infty}^{1/p}N_b(\nabla u)(x,t)
 \le\gamma\nu.
\]
At the stopping point for $w$, some component $w_k$ is at least
$\nu/(n-1)$. The oscillation calculation in the proof of
Lemma \ref{lem.NleSp2-1}, applied at that point, gives the lower-bound
ball for $w$ on its own stopping graph, with adjusted dimensional constants.

The following lemma is the analog of Lemma~\ref{S3:L8-alt1}. 
\begin{lemma}\label{lem.NleSp2-2} 
Fix $p\ge8$ such that $p/2$ is an even integer.
Let $\Omega={\mathbb R}^n_+\times\R$ and let ${\mathcal L}=-\partial_t+\divg(A\nabla\cdot)$ be a parabolic operator. 
Let $u$ be an energy solution to $\mathcal Lu=0$ in $\Omega$ and let $u_k=\dr_{x_k}u$ with $k\in\set{1,\dots,n-1}$.  For a fixed (sufficiently large) $a>0$ determined below, consider an arbitrary function $\hbar:{\mathbb R}^{n-1}\times\R\to \mathbb R$ such that it satisfies
the estimates \eqref{Eqqq-5}, \eqref{derh} and $\hbar\ge 0$. 
Then 
we have the following:\vglue1mm

For any $\varepsilon>0$ there exists $\delta>0$
such that if
$\norm{|\nabla A|x_n}_{L^\infty}+\||\partial_{t}A|^2x_n^3\|_{C}<\infty$ and $$
\||\nabla A|^2x_n\|_{C}+\||\partial_{t}a_{nn}|^2x_n^3\|_{C}+\||\partial_ta_{nn}|x_n^2\|_{L^\infty}<\delta,
$$ then
for all parabolic surface balls $\Delta_r\subset{\mathbb R}^{n-1}\times\R$ of radius $r$ such that at least one point of $\Delta_r$
satisfies $\hbar(x,t)\le 2r$, we have the following estimate for an arbitrary $c\ge 0$:
\begin{multline}\label{eq.NleSp2-2}
\int_{1/6}^6\int_{\Delta_r}\big|u^{p/2}_k \big(x,\theta\hbar(x,t),t\big)-c\big|^2\,dx\,dt\,d\theta
\leq C(1+\varepsilon^{-1})\iint_{\cS(\Delta_r,\hbar)}\abs{\nabla u_k}^2|u_k|^{p-2}x_n\, dx_ndtdx\\
+\varepsilon\br{\|\tilde{N}_{2,a}\br{(u_k^{p/2}-c)\1_{\cS(\Delta_r,\hbar)}}\|_{L^2}^2
+\|{N}_{a}(\nabla u\,\1_{\cS(\Delta_r,\hbar)})\|_{L^p}^p}\\+
\varepsilon  \int_{\Delta_{2r}}\br{\iint_{\gamma_a(x,t)\cap\cS(\Delta_r,\hbar)}\abs{\nabla^2 u}^2y_n^{-n}dYds}^{p/2}dX\,dt
+\frac{C}{r}\iint_{\mathcal{K}_\varepsilon}|u_k^{p/2}-c|^{2}\,dX\,dt,
\end{multline}
for some $C\in(0,\infty)$ that only depends on $a,\Lambda,n$ but not on $k$, $u_k$, $c$, $\varepsilon$ or $\Delta_r$. 
Here, \[\cS(\Delta_r,\hbar):=\set{(x,x_n,t):\, (x,t)\in\Delta_{2r} \text{ and } \frac{\hbar(x,t)}{12}<x_n<18r},\]
where 
$\mathcal K_{\varepsilon}:=\cS(\Delta_r,\hbar)\cap\set{(x,x_n,t): x_n>\varepsilon r}$.
The cones used to define the nontangential 
maximal functions in this lemma have vertices on $\partial\Omega$.
\end{lemma}

\begin{proof}
    The first modification in the proof of Lemma~\ref{S3:L8-alt1} happens when dealing with the term $II$ from \eqref{utAA}, where the PDE of $u_k^{p/2}-c$ is used. In the current setting, there is an extra term from $\mathcal L (u_k^{p/2}-c)$, namely, we have 
\[
{\mathcal L}(u_k^{p/2}-c)=\frac{p}2\left(\frac{p}2-1\right)u_k^{p/2-2}A\nabla u_k\cdot\nabla u_k
+ \frac{p}2u_k^{p/2-1} \divg((\dr_k A)\nabla u).
\]
As a result, we get an extra term $J$ in \eqref{eq.NleS1-II} and we write
\begin{equation}
    II =\mathcal B+\mathcal C+\mathcal D+II_7 +J,
\end{equation}
where 
\[
J:= p \iint_{\mathcal S(q,\tau,r,r_0,\theta\hbar)}(u_k^{p/2}-c)u_k^{p/2-1}\divg((\dr_kA)\nabla u)x_n\zeta dx_n\,dt\,dx,
\]
and the terms $\mathcal B,\mathcal C,\mathcal D$, and $II_7$ are the same as in \eqref{eq.NleS1-II}. By the product rule, $J$ is further split into two terms: $J=J_1+J_2$, where
\[
J_1:= p \iint_{\mathcal S(q,\tau,r,r_0,\theta\hbar)}(u_k^{p/2}-c)u_k^{p/2-1}(\dr_kA)\Delta u\,x_n\zeta dx_n\,dt\,dx,
\]
\[
J_2:=p \iint_{\mathcal S(q,\tau,r,r_0,\theta\hbar)}(u_k^{p/2}-c)u_k^{p/2-1}\nabla(\dr_kA)\nabla u\,x_n\zeta dx_n\,dt\,dx.
\]
By Cauchy-Schwarz and the Carleson condition on $\dr_k A$, we get that 
\begin{multline}
    |J_1|\lesssim \br{\iint_{\cS(q,\tau,r,r_0,\theta\hbar)}\abs{\nabla^2 u}^2|u_k|^{p-2}x_n}^{1/2}\br{\iint_{\cS(q,\tau,r,r_0,\theta\hbar)}\br{u_k^{p/2}-c}^2\abs{\dr_k A}^2x_n}^{1/2}\\
    \lesssim \br{\int_{2\Delta_r}N_a^{p-2}(u_k\1_{\cSS})(x,t)\iint_{\gamma_{a}(x,t)\cap\cSS}|\nabla^2u(Y,s)|^2y_n^{-n}dYds\,dx\,dt}^{1/2}\\
    \cdot \norm{|\dr_k A|^2x_n}_C^{1/2}\br{\int_{\Delta_{2r}}N_a\br{(u_k^{p/2}-c)\1_{\cSS}}^2dx\,dt}^{1/2}.
\end{multline}
Applying H\"older's inequality (with exponent $p/(p-2)$) to the first term and then using Young's inequality, one gets that for any $\varepsilon>0$,
\begin{multline*}
    |J_1|\le C_\varepsilon\norm{|\dr_k A|^2x_n}_C \norm{N_a\br{(u_k^{p/2}-c)\1_{\cSS}}}_{L^2(2\Delta_r)}^2\\
    +\varepsilon \norm{N_a(u_k\1_{\cSS})}_{L^p(2\Delta_{r})}^p +\varepsilon  \int_{2\Delta_r}\br{\iint_{\gamma_a(x,t)\cap\cS(\Delta_r,\hbar)}\abs{\nabla^2 u}^2y_n^{-n}dYds}^{p/2}\,dX\,dt.
\end{multline*}
For $J_2$, we integrate by parts again by pulling off $\dr_k A$. Since $k<n$, it gives 5 terms as follows. $J_2=\sum_{j=1}^4 J_{2j}+J_{2b}$, where
\[
J_{21}=-p(p/2-1)\iint_{\mathcal S(q,\tau,r,r_0,\theta\hbar)}(\dr_ku_k)u_k^{p/2-2}(u_k^{p/2}-c)\nabla A\cdot \nabla u\,x_n\zeta dx_n\,dt\,dx,
\]
\[
J_{22}=-p^2/2\iint_{\mathcal S(q,\tau,r,r_0,\theta\hbar)}(\dr_ku_k)u_k^{p-2}\nabla A\cdot \nabla u\,x_n\zeta dx_n\,dt\,dx,
\]
\[
J_{23}=-p\iint_{\mathcal S(q,\tau,r,r_0,\theta\hbar)}u_k^{p/2-1}(u_k^{p/2}-c)\nabla A\cdot \nabla \dr_ku\,x_n\zeta dx_n\,dt\,dx,
\]
\[
J_{24}=-p\iint_{\mathcal S(q,\tau,r,r_0,\theta\hbar)}u_k^{p/2-1}(u_k^{p/2}-c)\nabla A\cdot \nabla u\,x_n\,\dr_k\zeta dx_n\,dt\,dx,
\]
and the boundary term 
\begin{equation}\label{eq.J2b}
    J_{2b}=p\int_{\partial\mathcal S(q,\tau,r,r_0,\theta\hbar)}(u_k^{p/2}-c)u_k^{p/2-1}\nabla A\cdot \nabla u\,x_n\zeta\nu_k\,dS.
\end{equation}
The terms $J_{21}$ and $J_{23}$ are similar to $J_1$ and satisfy
\begin{multline*}
    |J_{21}|+|J_{23}|\le C_\varepsilon\norm{|\nabla A|^2x_n}_C \norm{N_a\br{(u_k^{p/2}-c)\1_{\cSS}}}_{L^2(2\Delta_r)}^2 \\+\varepsilon \norm{N_a(\nabla u\1_{\cSS})}_{L^p(2\Delta_{r})}^p+\varepsilon  \int_{2\Delta_r}\br{\iint_{\gamma_a(x,t)\cap\cS(\Delta_r,\hbar)}\abs{\nabla u_k}^2y_n^{-n}dYds}^{p/2}\,dX\,dt.
\end{multline*}
For the term $J_{22}$, we have 
\begin{multline*}
    |J_{22}|\lesssim \br{\iint_{\cSS}\abs{\dr_k u_k}^2 u_k^{p-2}x_n\,dX\,dt}^{1/2}\br{\iint_{\cSS}\abs{\nabla A}^2\abs{\nabla u}^px_n\,dX\,dt}^{1/2}\\
    \le C\iint_{\cSS}|\dr_ku_k|^2|u_k|^{p-2}x_n\,dX\,dt
    + \norm{|\nabla A|^2x_n}_C\norm{N_a(\nabla u \1_{\cSS})}_{L^p(\Delta_{2r})}^p.
\end{multline*}
For $J_{24}$, we use Cauchy-Schwarz, the bound $|\dr_k\zeta|\le C/r$, and the estimate \eqref{eq.Neta-c} to obtain that 
\begin{multline*}
    |J_{24}|\le C\norm{|\nabla A|^2x_n}_C^{1/2}\norm{N_a(u_k^{p/2}-c)}_{L^2(\Delta_{2r})}\norm{N(\nabla u\1_{\cSS})}_{L^p(\Delta_{2r})}^{p/2}\\
    \le \varepsilon\norm{N_a(u_k^{p/2}-c)}_{L^2(\Delta_{2r})}^2+C_{\varepsilon}\norm{|\nabla A|^2x_n}_C\norm{N(\nabla u\1_{\cSS})}_{L^p(\Delta_{2r})}^{p}.
\end{multline*}
We are left with the boundary term $J_{2b}$. By the computation of $\nu_k$ on the graph of $\theta\hbar$ (see \eqref{eq.nui}) and the Cauchy-Schwarz inequality, one has
\begin{multline*}
    |J_{2b}|\le \frac18\int_{\Delta_{2r}}\abs{u_k^{p/2}(x,\theta\hbar(x,t),t)-c}^2\zeta\, dx\,dt\\
    +C\int_{\Delta_{2r}}\abs{\nabla A}^2(x,\theta\hbar(x,t),t)\hbar(x,t)^2|u_k|^{p-2}\abs{\nabla u}^2(x,\theta\hbar(x,t),t)\zeta\,dx\,dt=\frac18\mathcal I+J_{2b'}.
\end{multline*}
The term $\frac18\mathcal I$ can be absorbed into the left-hand side of \eqref{eq.NleSp2-2}. For the term $J_{2b'}$, we integrate in $\theta$, noting the change of variable $dx_n=\hbar(x,t)d\theta$, and get that
\begin{multline*}
    \int_{1/6}^6J_{2b'}\,d\theta\le C\iint_{\cSS}\abs{\nabla A}^2|\nabla u|^p x_n\,dX\,dt\\
    \le C\norm{|\nabla A|^2x_n}_{C}\norm{N(\nabla u \1_{\cSS})}_{L^p(\Delta_{2r})}^p.
\end{multline*}

The last modification is required for estimating $\iint_{\cSS}\abs{\partial_t(u_k^{p/2})}^2x_n^3\,dX\,dt$ as in \eqref{eq.dtukp/2}, where the PDE for $u_k$ is used. We start with Fubini's theorem as before, and write 
\begin{multline*}
    \iint_{\cSS}\abs{\partial_t(u_k^{p/2})}^2x_n^3\,dX\,dt =\iint_{\cSS}\abs{\partial_t u_k}^2|u_k|^{p-2}x_n^3\,dX\,dt\\
\lesssim \int_{2\Delta_r}{N}^{p-2}_a(u_k\,\1_{\cSS})\left(\iint_{\gamma_{a/2}(x,t)\cap \cSS}|\partial_t u_k|^2y_n^{-n+2}dYds\right) dx\,dt.
\end{multline*}
 Following the notation \eqref{DefArea}, we write for any set $E\subset\Rn$,
\begin{equation}\label{defeq.A}
    A_{a/2}\br{u_k\,\1_{E}}=\left(\iint_{\gamma_{a/2}(x,t)\cap E}|\partial_t u_k|^2y_n^{-n+2}dYds\right)^{1/2}.
\end{equation}
Using the arguments in \cite[Section 6]{DLP1} (or \eqref{A<S+cN.Lp}), we can get that for any $q>1$, there holds
\begin{multline}\label{eq.A_S(uk)}
    \norm{A_{a/2}(u_k\,\1_{\cSS})}_{L^q(\Delta_{2r})}
    \lesssim \br{\int_{\Delta_{2r}}\br{\iint_{\gamma_a(x,t)\cap\cS(\Delta_r,\hbar)}\abs{\nabla^2 u}^2y_n^{-n}dYds}^{q/2}\,dX\,dt}^{1/q}\\
    +\norm{\mu}_{\mathcal{C}}^{1/2}\norm{N_a(\nabla u \1_{\cS(\Delta_r,\hbar)})}_{L^q}.
\end{multline}
 Then using \eqref{eq.A_S(uk)} with $q=p$, and Young's inequality with exponent $p/(p-2)$, one sees that 
\begin{multline*}
    \iint_{\cSS}\abs{\partial_t(u_k^{p/2})}^2x_n^3\,dX\,dt\le C\norm{N_a(\nabla u\1_{\cSS})}_{L^p}^p\\ + C\int_{\Delta_{20ar}}\br{\iint_{\gamma_a(x,t)\cap\cS(\Delta_r,\hbar)}\abs{\nabla^2 u}^2y_n^{-n}dYds}^{p/2}\,dX\,dt.
\end{multline*}
Here, we don't need the constants to be small as the term $ \iint_{\cSS}\abs{\partial_t(u_k^{p/2})}^2x_n^3\,dX\,dt$ is already multiplied by $\varepsilon$. 
This completes the modifications.
\end{proof}

\begin{lemma}\label{lem.NleS2-GL} Fix $p\ge8$ such that $p/2$ is an even integer. Let $k\in\set{1,\dots,n-1}$, and let $\mathcal L$ be an operator as in Lemma \ref{lem.NleSp2-2} that satisfies $\norm{|\nabla A|x_n}_{L^\infty}+\||\dr_t A|^2x_n^3\|_{\mathcal{C}}<\infty$. Fix $a>0$ sufficiently large as in Lemma \ref{lem.NleSp2-2}.  
For any $\varepsilon\in(0,1)$, if $\||\nabla A|^2x_n\|_{C}+\norm{|\dr_k A|x_n}_{L^\infty}+\||\partial_{t}a_{nn}|^2x_n^3\|_{C}+\||\partial_ta_{nn}|x_n^2\|_{L^\infty}$ is sufficiently small (depending on $\varepsilon$), then \eqref{eq.gdLmd} holds with the constants of Lemma \ref{LGL},
$u_k=\partial_k u$, and the following good set:
\begin{multline*}
    E_{1,\beta}:=
\Big\{(x,t)\in {\mathbb R}^{n-1}\times\R:\, \tilde{N}_{2,a}(u_k^{p/2})>\beta,\\ \varepsilon^{-1/p}S_{p,b}(u_k)+{\varepsilon^{1/p}S_{2,b}(\nabla u)}\le(\gamma\beta)^{2/p}, \varepsilon^{1/p}\tilde N_{2,b}(\nabla u)<(\gamma\beta)^{2/p}\Big\}.
\end{multline*}
\end{lemma}

\begin{proof}
Use the stopping graphs and covering from Lemma \ref{LGL}.
Lemma \ref{lem.NleSp2-1} supplies the graph lower bound, and
Lemma \ref{lem:graphReduction} now applies with the graph integral
bounded by Lemma \ref{lem.NleSp2-2}. Property \eqref{SiProp} controls
the resulting square and maximal terms by the defining good-set bounds.
For the interior oscillation term, the spatial Poincar\'e estimates
are unchanged. For its time part, set $v=\partial_k u$, $m=p/2$,
$Q=B_i^0\times I_i$, and $r=r_i$, and use the notation
$\mathcal E_q(v;Q)$ from the proof of Lemma \ref{LGL}. Write
\[
 U=\|\nabla u\|_{L^\infty(Q)},\qquad
 \mathcal T=r^{-n}\iint_Q|\nabla^2u|^2,\qquad
 K=\|x_n\partial_kA\|_{L^\infty}.
\]
The equation for $v$ is $v_t=\divg(A\nabla v)+\divg G$, where
$G=(\partial_kA)\nabla u$. Thus the identity
\eqref{eq:halfMomentEvolution} for $c(t)=\int v^m\eta^2$ has the
additional term
\[
 \mathcal R_k=-m\iint_{B_i^0\times[t,t']}
                    G\cdot\nabla(v^{m-1}\eta^2)\,dX\,ds.
\]
Since $x_n>2r$ on $Q$, we have $|G|\le CKU/r$.
The two terms obtained by differentiating the product satisfy
\begin{align*}
 |\mathcal R_k|
 &\le C_p\iint_Q |G|\big(|v|^{m-2}|\nabla v|\eta^2
                       +|v|^{m-1}|\eta\nabla\eta|\big)\\
 &\le C_pK\big(U^{m-1}\mathcal T^{1/2}+U^m\big).
\end{align*}
Here $|v|\le U$, $|\nabla v|\le|\nabla^2u|$, and the normalized
cutoff bounds give $\iint_Q\eta^2\lesssim r^2$ and
$\iint_Q|\eta\nabla\eta|\lesssim r$. Hence Young's inequality gives
\[
 |\mathcal R_k|^2\le C_pK^2\big(\mathcal T^{p/2}+U^p\big).
\]
The homogeneous part is estimated as in \eqref{eq.ct-c}, with
$\mathcal E_2(v;Q)\le\mathcal T$. Choosing $K^2\le\varepsilon$,
we obtain
\begin{equation}\label{eq:tangentialHalfOsc}
 \big(\operatorname*{osc}_{I_i}c\big)^2
 \le C_p\varepsilon^{-1}\mathcal E_p(v;Q)
             +C_p\varepsilon\big(\mathcal T^{p/2}+U^p\big).
\end{equation}
Property \eqref{SiProp} and the local interior estimates used above
bound these quantities on the good set by
\[
 \mathcal E_p(v;Q)\le C\varepsilon\gamma^2\beta^2,
 \qquad \mathcal T^{p/2}+U^p\le C\varepsilon^{-1}\gamma^2\beta^2.
\]
This proves the required $C\gamma^2\beta^2|\Delta_i|$ bound for
the interior term. Summing over $i$ proves \eqref{eq.gdLmd}.
\end{proof}

Apply Lemma \ref{lem:goodLambdaIntegration} as in
\eqref{eq:commonNpBound}, with $u_k=\partial_k u$ in the leading term
and $\nabla u$ in the two error terms. Lemma \ref{lem.NleS2-GL} gives,
after reparametrizing $\varepsilon$, the component estimate 
\begin{equation}
    \|{\tilde N}(\dr_ku)\|_{L^q(\pom)}\le C_\varepsilon\norm{S_p(\dr_k u)}_{L^q(\pom)}+ \varepsilon\br{\norm{\tilde N(\nabla u)}_{L^q(\pom)}+\norm{S_2(\nabla u)}_{L^q(\pom)}},
 \end{equation}   
provided we know a priori that $\|\tilde N(\dr_k u)\|_{L^q(\partial\Omega)}<\infty$.

\subsection{$N<S_p$ for $H$}
In this subsection we derive the following theorem for $H=\sum_{j=1}^na_{nj}\dr_ju$. 

\begin{theorem}\label{thm.NleSpH}
   Let $A$ be a matrix with bounded measurable coefficients that satisfies the ellipticity condition.
   Let $\LL=-\dr_t+\divg(A\nabla\cdot)$.
   For any $\varepsilon\in(0,1)$, there exists $\delta>0$ such that if $$\||\nabla A|^2x_n\|_{C}+\norm{|\nabla A|x_n}_{L^\infty}
   +\norm{|\nabla^2 A|x_n^2}_{L^\infty}+\||\partial_{t}A|^2x_n^3\|_{C}+\||\partial_t A|x_n^2\|_{L^\infty}<\delta,$$  then for any $p\ge8$ such that $p/2$ is an even integer, for any $a>0$, 
   $f\in  L^{q}(\pom)\cap  \Hdot^{-1/4}_{\partial_t-\Delta_x}(\partial\Omega)$,
   there exists a 
constant $C=C(n,p,a,\varepsilon, \lambda,\Lambda)>0$, such that for the energy solution $u$ to $\LL u=0$, 
$\partial^A_\nu u\Big|_{\partial\Omega}=f$,
there holds for all $q>p/2$
\begin{equation}
    \|{\tilde N}(H)\|_{L^q(\pom)}\le C\norm{S_p(H)}_{L^q(\pom)}+ \varepsilon\br{\norm{\tilde N(\nabla u)}_{L^q(\pom)}+\norm{S_2(\nabla u)}_{L^q(\pom)}},
 \end{equation}   
provided we know a priori that $\|\tilde N(H)\|_{L^q(\partial\Omega)}<\infty$. 
\end{theorem}

As before, we only highlight the modifications needed from Subsections~\ref{S.NleSp1} and \ref{S.NleSp2}. 
In this subsection, $w$ is  the $L^p$ average of $H$, that is, for $(X,t)=(x,x_n,t)\in\om$,
\begin{equation}\label{def.w3}
w(X,t)=\left(\fiint_{Q_{x_n/2}(X,t)}|H|^p\,dY\,ds\right)^{1/p}.
\end{equation}

\begin{lemma}\label{lem.NleSp3-1} 
For any $a>0$ there exists $b_0=b_0(a)>a$ and $\gamma_0=\gamma_0(a)>0$ such that the following holds. 
Having fixed an arbitrary $\nu>0$, $b\ge b_0$, $\gamma\le\gamma_0$, and $\varepsilon\in(0,1)$, for each point $(x,t)\in\partial\Omega$ from the set 
\begin{equation}
\big\{(x,t):\, N_{a}(H)(x,t)>\nu, \,S_{p,b}(H)+\varepsilon^{1/p} S_{2,b}(\nabla u)(x,t)+\delta(\varepsilon,A) N_b(\nabla u)\leq\gamma\nu\big\},
\end{equation}
where $\delta(\varepsilon,A):=\varepsilon^{-1/(p(p-1))}\norm{|\nabla A|x_n}_{L^\infty}^{1/(p-1)}+ \norm{|\nabla A|x_n}_{L^\infty}^{2/p}+\norm{|\dr_t A|x_n^2}_{L^\infty}^{1/p}+\norm{|\nabla^2A|x_n^2}_{L^\infty}^{1/p}$,
there exists a boundary ball $R$ with $(x,t)\in K_0 R$ with $K_0=K_0(n,p,a)>1$ and such that
\begin{equation}
\big|w\big(z,\hbar_{\nu,a}(w)(z,\tau),\tau\big)\big|>\nu/{8}\,\,\text{ for all }\,\,(z,\tau)\in R.
\end{equation}
\end{lemma}

\begin{proof}
    We adopt the notation in Lemma~\ref{l6} (and Lemma~\ref{lem.NleSp2-1}), keeping in mind that $u_k$ is replaced by $H$. We use the PDE \eqref{eq.Hpde1} for $H$, and so we need to modify our estimates on the function \(
v_{av}(\tau):=\int_{\Rn}|H|^{p}(X,\tau)\eta^2(X)\,dX\).
We have 
\begin{multline*}
    v_{av}(t_0')-v_{av}(t_0)
   =-2p\iint_{\mathbb R^n\times[t_0,t_0']}(A^T\nabla H\cdot\nabla\eta)|H|^{p-2}H\eta\,dX\,dt\\
-p(p-1)\iint_{\mathbb R^n\times[t_0,t_0']}(A^T\nabla H\cdot\nabla H)|H|^{p-2}\eta^2\,dX\,dt
+p\,I,
\end{multline*}
where $I:=-\iint_{\mathbb R^n\times[t_0,t_0']}F|H|^{p-2}H\eta^2dX\,dt$ and $F$ is as defined in \eqref{eq.Hpde1}. We use the bound \eqref{boundsF} on $F$ and get that 
\[
    |I|\lesssim \iint_{\mathbb R^n\times[t_0,t_0']}\br{ |\nabla A||\partial_t u|+|\partial_tA||\nabla u|+|\nabla A|^2|\nabla u|+|\nabla A||\nabla ^2 u|+|\nabla ^2 A||\nabla u|}|H|^{p-1}\eta^2.
\]
Since $|\dr_t u| \sim |\nabla^2 u|+\abs{\nabla A}\abs{\nabla u}$, we can simplify the above as 
\begin{multline*}
    |I|\lesssim  \iint_{\mathbb R^n\times[t_0,t_0']}|\nabla A||\nabla^2u||H|^{p-1}\eta^2+ \iint_{\mathbb R^n\times[t_0,t_0']}\br{ |\nabla A|^2+|\partial_tA|+|\nabla^2A|}|\nabla u||H|^{p-1}\eta^2\\
    =:I_1+I_2.
\end{multline*}
By Cauchy-Schwarz, the property of $\eta$, and the fact that $Q_{4y_n/5}\subset\gamma_b(x,t)$,
\begin{multline*}
    I_1\lesssim \br{ \iint_{\mathbb R^n\times[t_0,t_0']}|\nabla^2u|^2|H|^{p-2}\eta^2}^{1/2} \br{\iint_{\mathbb R^n\times[t_0,t_0']}|\nabla A|^2|H|^{p}\eta^2}^{1/2}\\
    \le C\norm{|\nabla A|x_n}_{L^\infty}S_2(\nabla u)(x,t)N(H)^{p-1}(x,t)\\
    \le \varepsilon S_2(\nabla u)^p(x,t)+ C\varepsilon^{-1/(p-1)}\norm{|\nabla A|x_n}_{L^\infty}^{p/(p-1)}N_b(\nabla u)^{p}(x,t),
\end{multline*}
using the observation that $|H|\lesssim|\nabla u|$ in the last inequality. 
For $I_2$, we use again that $|H|\lesssim|\nabla u|$ to get 
\[
I_2\le C\br{\norm{|\nabla A|x_n}_{L^\infty}^2+\norm{|\dr_t A|x_n^2}_{L^\infty}+\norm{|\nabla^2A|x_n^2}_{L^\infty}}N_b(\nabla u)^p(x,t).
\]
Since $\varepsilon^{1/p} S_{2,b}(\nabla u)(x,t)+\delta(\varepsilon,A) N_b(\nabla u)\leq\gamma\nu$, we can proceed as in the proof of Lemma~\ref{l6} and get \[
\left| v_{av}(t_0')-v_{av}(t_0)\right|\le \frac12\sup_{\tau\in I_{y_n}(s)}v_{av}(\tau)+ C(\gamma\nu)^p.
\]
The rest of the proof proceeds as in the proof of Lemma~\ref{l6}. 
\end{proof}

For $w$ in \eqref{def.w3}, the same graph-maximal conclusion follows
from Lemmas \ref{lem.NleSp3-1} and \ref{S3:L6} under the condition
\[
 N_a(w)>\nu,\qquad
 S_{p,b}(H)+\varepsilon^{1/p}S_{2,b}(\nabla u)
 +\delta(\varepsilon,A)N_b(\nabla u)\le\gamma\nu,
\]
with $\delta(\varepsilon,A)$ as in Lemma \ref{lem.NleSp3-1}.

The following lemma is the analog of Lemma~\ref{S3:L8-alt1} and Lemma~\ref{lem.NleSp2-2}. 
\begin{lemma}\label{lem.NleSp3-2} 
Fix $p\ge8$ such that $p/2$ is an even integer.
Let $\Omega={\mathbb R}^n_+\times\R$ and let ${\mathcal L}=-\partial_t+\divg(A\nabla\cdot)$ be a parabolic operator. 
Let $u$ be an energy solution to $\mathcal Lu=0$ in $\Omega$ and let $H=\sum_{j=1}^na_{nj}\dr_{j}u$. For a fixed (sufficiently large) $a>0$ determined below, consider an arbitrary function $\hbar:{\mathbb R}^{n-1}\times\R\to \mathbb R$ such that it satisfies
the estimates \eqref{Eqqq-5}, \eqref{derh} and $\hbar\ge 0$. 
Then 
we have the following:\vglue1mm

For any $\varepsilon>0$ there exists $\delta>0$
such that if
$\norm{|\nabla A|x_n}_{L^\infty}<\infty$ and $$
\||\nabla A|^2x_n\|_{C}+\||\partial_{t}A|^2x_n^3\|_{C}+\||\partial_ta_{nn}|x_n^2\|_{L^\infty}<\delta,
$$ then
for all parabolic surface balls $\Delta_r\subset{\mathbb R}^{n-1}\times\R$ of radius $r$ such that at least one point of $\Delta_r$
satisfies $\hbar(x,t)\le 2r$, we have the following estimate for an arbitrary $c\ge 0$:
\begin{multline}\label{eq.NleSp3-2}
\int_{1/6}^6\int_{\Delta_r}\big|H^{p/2} \big(x,\theta\hbar(x,t),t\big)-c\big|^2\,dx\,dt\,d\theta
\leq C(1+\varepsilon^{-1})\iint_{\cS(\Delta_r,\hbar)}\abs{\nabla H}^2|H|^{p-2}x_n\, dx_ndtdx\\
+\varepsilon\br{\|\tilde{N}_{2,a}\br{(H^{p/2}-c)\1_{\cS(\Delta_r,\hbar)}}\|_{L^2}^2
+\|{N}_{a}(\nabla u\,\1_{\cS(\Delta_r,\hbar)})\|_{L^p}^p}\\+
\varepsilon  \int_{\Delta_{2r}}\br{\iint_{\gamma_a(x,t)\cap\cS(\Delta_r,\hbar)}\abs{\nabla^2 u}^2y_n^{-n}dYds}^{p/2}dX\,dt
+\frac{C}{r}\iint_{\mathcal{K}_\varepsilon}|H^{p/2}-c|^{2}\,dX\,dt,
\end{multline}
for some $C\in(0,\infty)$ that only depends on $a,\Lambda,n$ but not on $u$, $H$, $c$, $\varepsilon$ or $\Delta_r$. 
Here, \[\cS(\Delta_r,\hbar):=\set{(x,x_n,t):\, (x,t)\in\Delta_{2r} \text{ and } \frac{\hbar(x,t)}{12}<x_n<18r},\]
where 
$\mathcal K_{\varepsilon}:=\cS(\Delta_r,\hbar)\cap\set{(x,x_n,t): x_n>\varepsilon r}$.
The cones used to define the nontangential 
maximal functions in this lemma have vertices on $\partial\Omega$.
\end{lemma}

\begin{proof}
    We need to control the term 
    \[
    II=2 \iint_{\mathcal S(q,\tau,r,r_0,\theta\hbar)}(H^{p/2}-c)\dr_{nn}^2(H^{p/2}-c)x_n\zeta\, dX\,dt.
    \]
    We write 
    \begin{align*}
        \dr_{nn}^2(H^{p/2})&= \divg\br{B\nabla (H^{p/2})} - \sum_{(i,j)\neq (n,n)}\dr_i(b_{ij}\dr_j(H^{p/2}))\\
        &=\tilde{\mathcal{L}}(H^{p/2})+\frac{1}{a_{nn}}\dr_t(H^{p/2})- \sum_{(i,j)\neq (n,n)}\dr_i(b_{ij}\dr_j(H^{p/2})),
    \end{align*}
    where $\tilde{\mathcal{L}}= -a_{nn}^{-1}\dr_t+\divg(B\nabla\cdot)$. We have 
    \begin{align*}
         \tilde{\mathcal{L}}(H^{p/2})
         &=\frac{p}2H^{p/2-1}\left[-a_{nn}^{-1}\partial_t H+\divg(B\nabla H)\right]+\frac{p}2\left(\frac{p}2-1\right)H^{p/2-2}B\nabla H\cdot\nabla H\\
         &=\frac{p}2H^{p/2-1}T+\frac{p}2\left(\frac{p}2-1\right)H^{p/2-2}B\nabla H\cdot\nabla H,
    \end{align*}
    by \eqref{eq.Hpde2}. So the new term that we need to consider is 
    \[
    II_{new}=p\iint_{\mathcal S(q,\tau,r,r_0,\theta\hbar)}(H^{p/2}-c)H^{p/2-1}Tx_n\zeta\,dX\,dt,
    \]
    which is the sum of the following 4 terms:
    \[
    J_1:= p\iint_{\mathcal S(q,\tau,r,r_0,\theta\hbar)}(H^{p/2}-c)H^{p/2-1}(\dr_i b_{in})\dr_tu\,x_n\zeta\,dX\,dt,
    \]
     \[
    J_2:= -p\iint_{\mathcal S(q,\tau,r,r_0,\theta\hbar)}(H^{p/2}-c)H^{p/2-1}(\dr_t b_{in})\dr_iu\,x_n\zeta\,dX\,dt,
    \]
     \[
    J_3:= -p\iint_{\mathcal S(q,\tau,r,r_0,\theta\hbar)}(H^{p/2}-c)H^{p/2-1}\dr_t(a_{nn}^{-1})\,H\,x_n\zeta\,dX\,dt,
    \]
    \[
    J_4:= -p\iint_{\mathcal S(q,\tau,r,r_0,\theta\hbar)}(H^{p/2}-c)H^{p/2-1}\sum_{j<n} [\partial_i(b_{ij}\partial_j(a_{nk} \partial_k u))-
\partial_k(b_{kn} \partial_j(a_{ji} \partial_i u))]\,x_n\zeta\,dX\,dt.
    \]
    
 For the term $J_1$, we use the PDE for $u$ to get that $|\dr_t u|\sim |\nabla A||\nabla u| + |\nabla^2 u|$. So by Cauchy-Schwarz,
 \begin{multline*}
     |J_1|\lesssim
     \br{\iint_{\cSS}\abs{\nabla B}^2|H^{p/2}-c|^2x_n}^{1/2}\\
     \cdot\left[\br{\iint_{\cSS}|\nabla A|^2|H|^{p-2}|\nabla u|^2x_n}^{1/2}+\br{\iint_{\cSS}|H|^{p-2}|\nabla^2 u|^2x_n}^{1/2}\right]\\
     \le C\norm{|\nabla A|^2x_n}_{\mathcal{C}}\norm{N((H^{p/2}-c)\1_{\cSS})}_{L^2(\Delta_{2r})}
     \norm{N(\nabla u \1_{\cSS})}_{L^p(\Delta_{2r})}^{p/2}\\
     + C\norm{|\nabla A|^2x_n}_{\mathcal{C}}^{1/2}\norm{N((H^{p/2}-c)\1_{\cSS})}_{L^2(\Delta_{2r})}\norm{N(H\1_{\cSS})}_{L^p(\Delta_{2r})}^{(p-2)/2}\\
\cdot\br{\int_{2\Delta_r}\br{\iint_{\gamma_a(x,t)\cap\cS(\Delta_r,\hbar)}\abs{\nabla^2 u}^2y_n^{-n}dYds}^{p/2}dx_ndx\,dt}^{1/p}.
 \end{multline*}
 By Young's inequality and the observation that $|H|\lesssim|\nabla u|$, we get that 
 \begin{multline*}
    |J_1|\le  C_\varepsilon\norm{|\nabla A|^2x_n}_{\mathcal{C}}\|N\br{(H^{p/2}-c)\1_{\cS(\Delta_r,\hbar)}}\|_{L^2(\Delta_{2r})}^2
+\varepsilon\|{N}_{a}(\nabla u\,\1_{\cS(\Delta_r,\hbar)})\|_{L^p(\Delta_{2r})}^p\\+
\varepsilon  \int_{2\Delta_r}\br{\iint_{\gamma_a(x,t)\cap\cS(\Delta_r,\hbar)}\abs{\nabla^2 u}^2y_n^{-n}dYds}^{p/2}dX\,dt.
 \end{multline*}
 Next we estimate the term $J_4$. Using the product rule, we get terms from the bracket that can be bounded by $|\nabla A|^2|\nabla u|$ and $|\nabla A||\nabla^2 u|$, and thus can be treated as $J_1$. When 2 derivatives fall on the coefficients $A$, one of them is always in $x_j$ for $j<n$. For these terms, we integrate by parts in $x_j$ to move $\dr_j$ away from the coefficients $a_{ji}$ and $a_{nk}$. This process generates terms that can be bounded by (up to a multiplicative constant) the sum of the following terms:
\begin{align*}
      J_{41}&:= \iint_{\cSS}|\dr_j(H^{p/2})||H|^{p/2-1}|\nabla A||\nabla u|x_n\zeta\,dX\,dt\\
      &\le\frac{p}2\iint_{\cSS}|\dr_j H||H|^{p-2}|\nabla A||\nabla u|x_n\zeta\,dX\,dt,
\end{align*}  
\begin{align*}
    J_{42}&:= \iint_{\cSS}|H^{p/2}-c||\dr_j(H^{p/2-1})||\nabla A||\nabla u|x_n\zeta\,dX\,dt\\
    &\le \br{\frac{p}2-1}\iint_{\cSS}|H^{p/2}-c||\dr_j H||H|^{p/2-2}|\nabla A||\nabla u|x_n\zeta\,dX\,dt,
\end{align*}
\[
J_{43}:= \iint_{\cSS}|H^{p/2}-c||H|^{p/2-1}|\nabla A|^2|\nabla u|x_n\zeta\,dX\,dt,
\]
\[J_{44}:= \iint_{\cSS}|H^{p/2}-c||H|^{p/2-1}|\nabla A||\nabla \dr_ju|x_n\zeta\,dX\,dt,
\]
\[J_{45}:= \iint_{\cSS}|H^{p/2}-c||H|^{p/2-1}|\nabla A||\nabla u|x_n\,\dr_j\zeta\,dX\,dt,
\]
and the boundary term 
\[
J_{4b}:=\int_{\partial\mathcal S(q,\tau,r,r_0,\theta\hbar)}|H^{p/2}-c||H|^{p/2-1}|\nabla A||\nabla u|\,x_n\zeta\,|\nu_j|\,dS.
\]
As before, we use Cauchy-Schwarz for all these terms. For $J_{41}$, we let  $|\dr_j H|^2|H|^{p-2}x_n\zeta$ be one integrand and $|\nabla A|^2|H|^{p-2}|\nabla u|^2x_n\zeta$ be the other. The latter is further bounded by  $|\nabla A|^2|\nabla u|^px_n\zeta$. So 
\[
J_{41}\le \iint_{\cSS}\abs{\nabla H}^2|H|^{p-2}x_n\, dx_ndtdx + C\norm{|\nabla A|^2x_n}_{\mathcal{C}}\norm{N(\nabla u \1_{\cSS})}_{L^p(\Delta_{2r})}^p.
\]
For $J_{42}$, we get $|\dr_j H|^2|H|^{p-2}x_n\zeta$ as one integrand and $|\nabla A|^2|H^{p/2}-c|^2|H|^2|\nabla u|^2x_n\zeta$ as the other. The latter  is further bounded by $|\nabla A|^2|H^{p/2}-c|^2x_n\zeta$. This gives 
\[
J_{42}\le \iint_{\cSS}\abs{\nabla H}^2|H|^{p-2}x_n\, dx_ndtdx + C\norm{|\nabla A|^2x_n}_{\mathcal{C}}\norm{N((H^{p/2}-c)\1_{\cSS})}_{L^2(\Delta_{2r})}^2.
\]
For $J_{43}$, we put $|\nabla A|^2|H^{p/2}-c|^2x_n\zeta$ as one integrand and $|\nabla A|^2|H|^{p-2}|\nabla u|^2x_n\zeta$ as the other. The latter is  bounded by  $|\nabla A|^2|\nabla u|^px_n\zeta$. So 
\[
J_{43}\le C\norm{|\nabla A|^2x_n}_{\mathcal{C}}\br{\norm{N((H^{p/2}-c)\1_{\cSS})}_{L^2(\Delta_{2r})}^2 + \norm{N(\nabla u \1_{\cSS})}_{L^p(\Delta_{2r})}^p}.
\]
$J_{44}$ can be estimated in the same way as the term involving $\nabla^2$ from  $J_1$.  For $J_{45}$, we put $|H^{p/2}-c|^2x_n|\dr_j\zeta|^2$ in the first integral and  $|\nabla A|^2|H|^{p-2}|\nabla u|^2x_n\zeta$ in the second. The second one is  further bounded by  $|\nabla A|^2|\nabla u|^px_n\zeta$, while the first one can be estimated as in \eqref{eq.Neta-c}, using the bound $r|\dr_j\zeta|\le C$. By Young's inequality, we get 
\[
J_{45}\le \varepsilon \norm{N((H^{p/2}-c)\1_{\cSS})}_{L^2(\Delta_{2r})}^2 +C_\varepsilon \norm{|\nabla A|^2x_n}_{\mathcal{C}}\norm{N(\nabla u \1_{\cSS})}_{L^p(\Delta_{2r})}^p.
\]
The term $J_{4b}$ is the same as the boundary term $J_{2b}$ in \eqref{eq.J2b}, with $u_k$ replaced by $H$. Note that we have taken care of all the terms coming from $J_4$, as there are no terms involving 3 derivatives on $u$ from the bracket in $J_4$.

The terms $J_2$ and $J_3$ need more work, as we get a $t$ derivative in the coefficients, so we are lacking $x_n$. Therefore, we write  $x_n=\frac12\dr_n(x_n^2)$ in $J_2$ and $J_3$, and then integrate by parts in $x_n$. For the term $J_2$, this process
gives terms that can be controlled by the sum of 6 terms: $J_{21}$-$J_{25}$ and $J_{2b}$. The terms $J_{21}$, $J_{22}$, $J_{24}$, $J_{25}$ and $J_{2b}$ correspond to the terms $J_{41}$, $J_{42}$, $J_{44}$, $J_{45}$ and $J_{4b}$, with $\dr_j$ replaced by $\dr_n$, $|\nabla A|$ replaced by $|\dr_t A|x_n$, and $\nabla u$ replaced with $\nabla_{T} u$, and additionally for the boundary term, $\nu_j$ replaced by $\nu_n$. Since $|\nabla A|$ and $|\dr_t A|x_n$ satisfy the same Carleson condition, these terms can be bounded the same as the corresponding terms from $J_4$, replacing $\norm{|\nabla A|^2x_n}_{C}$ with  $\norm{|\dr_t A|^2x_n^3}_{C}$. We are left with 
\[
J_{23}:= \abs{\iint_{\mathcal S(q,\tau,r,r_0,\theta\hbar)}(H^{p/2}-c)H^{p/2-1}(\dr_t \dr_nb_{in})\dr_iu\,x_n^2\zeta\,dX\,dt}.
\]
To estimate $J_{23}$, we integrate by parts in $t$ to move $\dr_t$ away from $\dr_n b_{in}$. This gives $J_{23}\lesssim\sum_{k=1}^4J_{23k}+J_{23b}$, where 
\[
J_{231}:= \iint_{\cSS}|\dr_t H||H|^{p-2}|\dr_nA||\nabla_T u|x_n^2\zeta\,dX\,dt,
\]
\[
J_{232}:= \iint_{\cSS}|\dr_t H||H|^{p/2-2}|H^{p/2}-c||\dr_nA||\nabla_T u|x_n^2\zeta\,dX\,dt,
\]
\[
J_{233}:= \iint_{\cSS}|H^{p/2}-c||H|^{p/2-1}|\dr_nA||\dr_t\nabla_Tu|x_n^2\zeta\,dX\,dt,
\]
\[
J_{234}:= \iint_{\cSS}|H^{p/2}-c||H|^{p/2-1}|\dr_nA||\nabla_Tu|x_n^2\,|\dr_t\zeta|\,dX\,dt,
\]
and 
\[
J_{23b}:=\int_{\partial\mathcal S(q,\tau,r,r_0,\theta\hbar)}|H^{p/2}-c||H|^{p/2-1}|\dr_nA||\nabla_T u|\,x_n^2\zeta\,|\nu_t|\,dS.
\]
For $J_{231}$, we get by Cauchy-Schwarz and $|H|\lesssim|\nabla u|$ that 
\[
J_{231}\lesssim \br{ \iint_{\cSS}|\dr_t H|^2|H|^{p-2}x_n^3 dX\,dt}^{1/2}\br{\iint_{\cSS}|\dr_nA|^2|\nabla u|^px_n\,dX\,dt}^{1/2}.
\] We need to control
    \(
    \iint_{\cSS}|\dr_t H|^2|H|^{p-2}x_n^3 dX\,dt.
    \)
    Observe that by the product rule,  
\[
|\dr_t H|^2\lesssim |\dr_t A|^2|\nabla u|^2+|\dr_t(\nabla u)|^2,
\]
which gives that 
\begin{multline}\label{eq.SpHt}
    \iint_{\cSS}|\dr_t H|^2|H|^{p-2}x_n^3 dX\,dt \lesssim\int_{\Delta_{2r}}N_a(H)^{p-2}(x,t)\\
  \cdot \left[\iint_{\gamma_{a/2}(x,t)\cap\cSS}|\dr_tA|^2|\nabla u|^2y_n^{-n+2}dYds+A_{a/2}\br{\nabla u\,\1_{\cSS}}^2(x,t)\right]dx\,dt,
\end{multline}
where $A_{a/2}\br{\nabla u\,\1_E}$ is given as in \eqref{defeq.A}.
By the Caccioppoli inequality, or the argument in \cite[Section 6]{DLP1}, the term in the bracket is bounded by
\[
S^2_{2,a}(\nabla u\,\1_{\cS(\Delta_r,\hbar)})(x,t)+S^2_A(u\,\1_{\cS(\Delta_r,\hbar)})(x,t),
\]
where $S_{2,a}(u\,\1_{E})$ is the square function with aperture $a$ of $u$ restricting to the set $E$, and 
\[
S_A(u\,\1_E)(x,t):=\br{\iint_{\gamma_{a}(x,t)\cap E}\br{|\dr_tA|^2y_n^2+|\nabla A|^2}|\nabla u|^2y_n^{-n}dYds}^{1/2}.
\]
By a good-$\lambda$ argument as in \cite[Section 6]{DLP1}, one has that 
\[
\norm{S_A(u\,\1_{\cS(\Delta_r,\hbar)})}_{L^q}\le C\norm{\mu}_{C}^{1/2}\norm{N(\nabla u\1_{\cS(\Delta_r,\hbar)})}_{L^q} \quad \text{for all }q>1.
\]
Therefore, from applying H\"older's inequality to \eqref{eq.SpHt}, it follows  that
\begin{multline}\label{eq.SpHt'}
    \iint_{\cSS}|\dr_t H|^2|H|^{p-2}x_n^3 \,dX\,dt\\
    \le C\norm{N(H\1_{\cSS})}_{L^p}^{p-2}\left(\norm{S_{2,a}(\nabla u\,\1_{\cS(\Delta_r,\hbar)})}_{L^p(\Delta_{2r})}^{2}
   +\norm{\mu}_{C}\norm{N(\nabla u)\1_{\cS(\Delta_r,\hbar)}}_{L^p}^2\right)\\
   \le  C\norm{N((\nabla u)\1_{\cS(\Delta_r,\hbar)})}_{L^p}^{p}+C\norm{S_{2,a}(\nabla u\,\1_{\cS(\Delta_r,\hbar)})}_{L^p(\Delta_{2r})}^{p}.
\end{multline}

This implies that 
\begin{multline*}
    J_{231}\le \varepsilon  \int_{\Delta_{2r}}\br{\iint_{\gamma_a(x,t)\cap\cS(\Delta_r,\hbar)}\abs{\nabla^2 u}^2y_n^{-n}dYds}^{p/2}dx_n\,dx\,dt\\
    +C_\varepsilon\norm{|\dr_nA|^2x_n}_{C}\norm{N((\nabla u)\1_{\cS(\Delta_r,\hbar)})}_{L^p}^{p}.
\end{multline*}
The term $J_{232}$ is similar. Using \eqref{eq.SpHt'} and Cauchy-Schwarz, one has 
\begin{multline*}
    J_{232}\le \varepsilon  \int_{\Delta_{2r}}\br{\iint_{\gamma_a(x,t)\cap\cS(\Delta_r,\hbar)}\abs{\nabla^2 u}^2y_n^{-n}dYds}^{p/2}dx_n\,dx\,dt\\
    +\varepsilon\norm{N((\nabla u)\1_{\cS(\Delta_r,\hbar)})}_{L^p}^{p}
    +C_\varepsilon\norm{|\dr_nA|^2x_n}_{C}\norm{N((H^{p/2}-c)\1_{\cSS})}_{L^2}^{2}.
\end{multline*}
For $J_{233}$, we write
\[
J_{233}\lesssim
\br{\iint_{\cSS}|\dr_nA|^2|H^{p/2}-c|^2x_n\zeta\,dX\,dt}^{1/2}
\br{\iint_{\cSS}|H|^{p-2}|\nabla_T\dr_t u|^2x_n^3\zeta\,dX\,dt}^{1/2}.
\]
{Note that the second integral on the right-hand side can be estimated as \eqref{eq.SpHt}, using $A_{a/2}\br{\nabla_T u\,\1_{\cSS}}$.} We thus get $J_{233}$ can be bounded by the same quantity as $J_{232}$. 
For $J_{234}$, the Cauchy-Schwarz inequality gives
\[
J_{234}\lesssim \br{\iint_{\cSS}|H^{p/2}-c|^2|\dr_t\zeta|^2x_n^3}^{1/2}\br{\iint_{\cSS}|\dr_nA|^2|H|^{p-2}|\nabla u|^2x_n}^{1/2}.
\]
The second term on the right-hand side is bounded by $\norm{|\dr_n A|^2x_n}_{\mathcal{C}}^{1/2}\norm{N((\nabla u)\1_{\cSS})}_{L^p}^{p/2}$. 
The first integral could be estimated in the same way as the term $II_9$ in \eqref{eq.DtNeta-c}, but since it can be multiplied by $\varepsilon$ after applying the Cauchy-Schwarz inequality, we can simply use the bound $|\dr_t\zeta|^2x_n^{3}\lesssim r^{-1}$ and control $|H^{p/2}-c|$ pointwise by its nontangential maximal function. Therefore, we have 
\[
J_{234}\le \varepsilon\norm{N((H^{p/2}-c)\1_{\cSS})}_{L^2}^{2} + C_\varepsilon\norm{|\dr_nA|^2x_n}_{C}\norm{N((\nabla u)\1_{\cSS})}_{L^p}^{p}.
\]

For the boundary term $J_{23b}$, we use Cauchy-Schwarz to get 
\[
    J_{23b}\le \frac1{10}\int_{\partial\mathcal S(q,\tau,r,r_0,\theta\hbar)}|H^{p/2}-c|^2x_n|\nu_t|dS +C\int_{\dr\mathcal S(q,\tau,r,r_0,\theta\hbar)}|\dr_nA|^2|\nabla u|^px_n^3|\nu_t|dS.
\]
Using the bound $|\nu_t|dS\le \frac{a^{-2}}{2x_n}\,dx\,dt$, we see that the first term on the right-hand side is bounded by $\frac1{10}\mathcal I$, and thus can be hidden on the left-hand side of \eqref{eq.NleSp3-2}. The second term can be estimated by integrating in $\theta$ over $[1/6,6]$, and the resulting integral is bounded by 
\[
C\norm{|\dr_nA|^2x_n}_{C}\norm{N((\nabla u)\1_{\cSS})}_{L^p}^p.
\]
We have completed the estimate for $II_{new}$.
Notice that while in the proof of Lemma \ref{lem.NleSp2-2} we have to modify the way we estimate $\iint_{\cSS}\abs{\partial_t(u_k^{p/2})}^2x_n^3\,dx_n\,dx\,dt$ from  \eqref{eq.dtukp/2}, in the current setting this term (replacing $u_k$ by $H$) is exactly \eqref{eq.SpHt'}. This completes the modifications. 
\end{proof}

\begin{lemma}\label{lem.NleS3-GL} 
Fix $p\ge8$ such that $p/2$ is an even integer. Let $\mathcal L$ be an operator as in Lemma \ref{lem.NleSp3-2}. Fix $a>0$ sufficiently large as in Lemma \ref{lem.NleSp3-2}.  
For any $\varepsilon\in(0,1)$, if \(\||\nabla A|^2x_n\|_{C}+\norm{|\nabla A|x_n}_{L^\infty}
   +\norm{|\nabla^2 A|x_n^2}_{L^\infty}+\||\partial_{t}A|^2x_n^3\|_{C}+\||\partial_t A|x_n^2\|_{L^\infty}\) is sufficiently small (depending on $\varepsilon$), then the conclusion of Lemma \ref{LGL} holds for energy solutions
$u$ of $\mathcal Lu=0$, with the following substitutions in
\eqref{eq.gdLmd} and \eqref{def.E1}:
\[
 u_k^{p/2}\mapsto H^{p/2},\quad
 S_{p,b}(u_k)\mapsto S_{p,b}(H),\quad
 S_{2,b}(u_k)\mapsto S_{2,b}(\nabla u),\quad
 \tilde N_{2,b}(u_k)\mapsto\tilde N_{2,b}(\nabla u).
\]
\end{lemma}
\begin{proof}
Apply the same stopping construction and Lemma \ref{lem:graphReduction}
with $z=H^{p/2}$, using Lemmas \ref{lem.NleSp3-1} and
\ref{lem.NleSp3-2} for the graph lower bound and graph integral.
The square and maximal terms are controlled by \eqref{SiProp}.
For the interior oscillation term, the spatial Poincar\'e argument
from Lemma \ref{LGL} is unchanged. Set $m=p/2$, $Q=B_i^0\times I_i$,
$r=r_i$, and $c(t)=\int H^m\eta^2$. Use $U$ and $\mathcal T$ from
the proof of Lemma \ref{lem.NleS2-GL}, and put
\[
 \alpha_1=\|x_n\nabla A\|_{L^\infty},\qquad
 \alpha_0=\alpha_1^2+\|x_n^2\partial_tA\|_{L^\infty}
                         +\|x_n^2\nabla^2A\|_{L^\infty}.
\]
By \eqref{eq.Hpde1}, $H_t=\divg(A^T\nabla H)-F$. Thus
\eqref{eq:halfMomentEvolution}, with $v=H$ and $A$ replaced by $A^T$,
has the additional term
\[
 \mathcal R_H=-m\iint_{B_i^0\times[t,t']}FH^{m-1}\eta^2\,dX\,ds.
\]
The original equation gives
$|u_t|\le C(|\nabla^2u|+|\nabla A||\nabla u|)$.
Consequently \eqref{boundsF} and $x_n>2r$ on $Q$ imply
\[
 |F|\le C\big(\alpha_1r^{-1}|\nabla^2u|
                              +\alpha_0r^{-2}|\nabla u|\big).
\]
Since $|H|\le CU$, Cauchy--Schwarz with the normalized cutoff yields
\[
 |\mathcal R_H|\le C_p\big(\alpha_1U^{m-1}\mathcal T^{1/2}
                                      +\alpha_0U^m\big),
 \qquad
 |\mathcal R_H|^2\le C_p(\alpha_1^2+\alpha_0^2)
                              (\mathcal T^{p/2}+U^p).
\]
Also $\nabla H=A_n\nabla^2u+(\nabla A_n)\nabla u$, so
\[
 \mathcal E_2(H;Q)\le C(\mathcal T+\alpha_1^2U^2).
\]
Apply the homogeneous estimate \eqref{eq.ct-c} to the first two terms
of the moment identity, and add the estimate for $\mathcal R_H$.
Taking $\alpha_1^2+\alpha_0^2\le\varepsilon$ gives
\begin{equation}\label{eq:conormalHalfOsc}
 \big(\operatorname*{osc}_{I_i}c\big)^2
 \le C_p\varepsilon^{-1}\mathcal E_p(H;Q)
             +C_p\varepsilon(\mathcal T^{p/2}+U^p).
\end{equation}
The good-set bounds control $\mathcal E_p(H;Q)$ by
$C\varepsilon\gamma^2\beta^2$ and $\mathcal T^{p/2}+U^p$ by
$C\varepsilon^{-1}\gamma^2\beta^2$, exactly as in the tangential case.
Thus the interior term is at most $C\gamma^2\beta^2|\Delta_i|$,
as required.
\end{proof}

Lemma \ref{lem:goodLambdaIntegration}, applied with the same substitutions
as in Lemma \ref{lem.NleS3-GL}, gives Theorem \ref{thm.NleSpH} after
reparametrizing $\varepsilon$.

\section{Bounds of the $p$-adapted square functions}\label{S.Sp_bd}
In this section, we establish key estimates for several variants of the 
$p$-adapted square functions, which in turn provide bounds for the nontangential maximal function of the gradient of the solution. It provides a parabolic analog of the results in \cite[Section 6]{DPR}, but the arguments differ significantly from the elliptic setting in many places.
\medskip

We start with the following lemma, which is a parabolic analog of \cite[Lemma 6.1]{DPR}. In the elliptic setting, one can start directly with bounding $\norm{N(\nabla u)}_{L^p(\pom)}$ by $\|\nabla_Tu\|_{L^p(\pom)}$ using the solvability of the $L^p$ Regularity problem. In the parabolic setting, however, this strategy would produce an extra term $\|D_t^{1/2}u\|_{L^p(\pom)}$, which cannot be controlled (at least easily) by a $p$-adapted square function or its variants. Therefore, we start instead from the `$N<S_p$' estimate (Theorem~\ref{thm.NlessSdT}), then derive an appropriate estimate on the $p$-adapted function.

\begin{lemma}\label{NPl1} Let $p=4m$, for some integer $m\ge2$ and $\Omega=\mathbb R^n_+\times\R$.
Let $A$ be as in Theorem~\ref{thm.NlessSdT}, and 
let $\LL=-\dr_t+\divg(A\nabla\cdot)$. Then there exists $\delta>0$
such that if $$\||\nabla A|^2x_n\|_{C}+\norm{|\nabla_T A|x_n}_{L^\infty}+\||\partial_{t}a_{nn}|^2x_n^3\|_{C}+\||\partial_ta_{nn}|x_n^2\|_{L^\infty}<\delta,$$   then for some
$K=K(\lambda,\Lambda,n,p)>0$ we have for an energy solution $\mathcal Lu=0$ in $\Omega$ the estimate:
\begin{eqnarray} \label{NPe1}
&&\int_{\partial\Omega} N^p(\nabla_T u)\,dx\,dt\le K
\iint_{\Omega}|\nabla_T u|^{p-2}|\nabla
H|^2x_n\,dX\,dt+K\delta\int_{\partial\Omega}\left[N^p(\nabla u)+S^p(\nabla u)\right]\,dx\,dt.
\end{eqnarray}
Here $H=\sum_ja_{nj}\partial_ju$, and in particular $H\big|_{\partial\Omega}=A\nabla u\cdot\nu\big|_{\partial\Omega}$ is the Neumann data on $\partial\Omega$.
\end{lemma}

\begin{proof} We apply first the results of Section \ref{SS:43} to the components of the tangential gradients
$\nabla_Tu$. It follows that for all sufficiently small $\delta=\delta(p,k,\lambda,\Lambda)>0$ we have by Theorem \ref{thm.NlessSdT} for $p=4m$, $m\ge2$
\begin{equation}\label{NS-T}
\|N(\nabla_Tu)\|_{L^p(\partial\Omega)}^p\le C\sum_{k=1}^{n-1}\iint_{\Omega}|\partial_k u|^{p-2}|\nabla \partial_k u
|^2x_n\,dX\,dt + C\delta \int_{\partial\Omega}\left[N^p(\nabla u)+S^p(\nabla u)\right]\,dx\,dt.
\end{equation}

The second term on the right-hand side of \eqref{NS-T} already has the desired form. We now estimate the first term. From here up to \eqref{NPe2xx}, the calculation proceeds as in \cite[Lemma 8.5]{DLP1} but with some further modifications as in our case $p\ne 2$ and also the matrix $A$ is not in block form. \vskip2mm

We start with a local estimate on $(0,r)\times Q_r$, where $Q_r=Q_r(y,s)$ is a parabolic cube on the boundary as defined earlier. Let $\zeta=\zeta(x,t)$ be a smooth cutoff function that satisfies
\[
\zeta=\begin{cases}
    1\quad \text{in }Q_r,\\
    0\quad \text{outside }Q_{2r},
\end{cases}
\]
and that for some constant $0<c<\infty$
\[r\abs{\dr_{x_i}\zeta}+r^2\abs{\dr_t\zeta}\le c\quad \text{where }1\le i\le n-1.\]

To lighten notation, we denote $w_k:=\dr_k u$ for $1\le k\le n$. Fix $k\in\set{1,2,\dots, n-1}$. Then
we can write using the ellipticity condition and the fact that $|w_k|^{p-2}=w_k^{p-2}$ as $p-2$ is even
\begin{multline}\label{eq.pfSdk1}\lambda\Lambda^{-1}(p-1)\iint_{Q_r\times (0,r)}|w_k|^{p-2}|\nabla w_k
|^2x_n\,dX\,dt \le\\
    \int_0^r\int_{Q_{2r}}a_{nn}^{-1}A\nabla w_k\cdot \nabla (w_k^{p-1})(\zeta^2x_n)dX\,dt
    =\int_0^r\int_{Q_{2r}}\divg\br{w_k^{p-1}\zeta^2x_n a_{nn}^{-1}A\nabla w_k}\\
    -\int_0^r\int_{Q_{2r}}w_k^{p-1}A\nabla w_k\cdot\nabla\br{a_{nn}^{-1}\zeta^2x_n}
    -\int_0^r\int_{Q_{2r}}\divg(A\nabla w_k)(w_k^{p-1}a_{nn}^{-1}\zeta^2x_n).
\end{multline}
We then use the divergence theorem for the first term on the right-hand side, which gives us a boundary term
\begin{equation}\label{eq.Sdkbdy}
    \frac{r}{p}\int_{Q_{2r}}\dr_n\br{w_k^p}(x,r,t)\zeta(x,t)^2\,dx\,dt
+ \frac{r}{p}\sum_{j<n}\int_{Q_{2r}}a_{nn}^{-1}a_{nj}\dr_j\br{w_k^p}(x,r,t)\zeta(x,t)^2\,dx\,dt.
\end{equation} 
For the second term, we write 
$\nabla\br{a_{nn}^{-1}\zeta^2x_n}=\nabla(a_{nn}^{-1})\zeta^2x_n+\nabla(\zeta^2)a_{nn}^{-1}x_n+\zeta^2a_{nn}^{-1}\nabla x_n$, which gives three terms correspondingly. Notice that $a_{nn}^{-1}A\nabla w_k\cdot\nabla x_n= \dr_n w_k+\sum_{j<n}a_{nn}^{-1}a_{nj}\partial_jw_k$, and so the corresponding term in the integral 
for the leading term
\begin{multline*}
    -\int_0^r\int_{Q_{2r}}w_k^{p-1}\dr_nw_k\zeta^2dX\,dt
    =-\frac1p\int_0^r\int_{Q_{2r}}\dr_n\br{w_k^p\zeta^2}dX\,dt\\
    =\frac1p\int_{Q_{2r}}w_k(x,0,t)^p\zeta(x,t)^pdx\,dt-\frac1p\int_{Q_{2r}}w_k(x,r,t)^p\zeta(x,t)^2dx\,dt,
\end{multline*}
while the terms $j<n$ can be further written as (after two integrations by parts):
\begin{multline}\label{eqwq}
    -\int_0^r\int_{Q_{2r}}w_k^{p-1}a_{nn}^{-1}a_{nj}\partial_jw_k\zeta^2dX\,dt
    =-\frac1p\int_0^r\int_{Q_{2r}}a_{nn}^{-1}a_{nj}\dr_j\br{w_k^p}\zeta^2(\partial_n x_n)dX\,dt\\
    =\frac1p\int_0^r\int_{Q_{2r}}\partial_n(a_{nn}^{-1}a_{nj})\dr_j\br{w_k^p}\zeta^2 x_n\,dX\,dt
    -\frac1p\int_0^r\int_{Q_{2r}}\partial_j(a_{nn}^{-1}a_{nj})\dr_n\br{w_k^p}\zeta^2 x_n\,dX\,dt\\
    -\frac1p\int_{Q_{2r}}a_{nn}^{-1}a_{nj}\partial_j(w_k^p)\zeta^2 r\,dx\,dt-\frac1p\int_0^r\int_{Q_{2r}}a_{nn}^{-1}a_{nj}\dr_n\br{w_k^p}\partial_j(\zeta^2) x_n\,dX\,dt.
\end{multline}
The first term of the last line above cancels part of the terms from the boundary term \eqref{eq.Sdkbdy}. To summarize, the  right-hand side of \eqref{eq.pfSdk1} is equal to
\begin{multline}\label{eq.pfSdk2}
    \frac{r}{p}\int_{Q_{2r}}\dr_n\br{w_k^p}(x,r,t)\zeta(x,t)^2dx\,dt
    +\frac1p\int_{Q_{2r}}w_k(x,0,t)^p\zeta(x,t)^2dx\,dt\\
    -\frac1p\int_{Q_{2r}}w_k(x,r,t)^p\zeta(x,t)^2dx\,dt
    -\int_0^r\int_{Q_{2r}}w_k^{p-1}a_{nn}^{-1}A\nabla w_k\cdot\nabla_x(\zeta^2)x_n\,dX\,dt\\
    -\int_0^r\int_{Q_{2r}}w_k^{p-1}a_{nn}^{-1}a_{nj}(\dr_n w_k)\partial_j(\zeta^2) x_n\,dX\,dt+C
    -\int_0^r\int_{Q_{2r}}a_{nn}^{-1}\divg(A\nabla w_k)(w_k^{p-1}\zeta^2x_n)\,dX\,dt.
\end{multline}
Here by $C$ we have denoted terms that enjoy a Carleson bound (the two terms in the second line of \eqref{eqwq} and the term arising from $\nabla(a_{nn}^{-1})\zeta^2x_n$) which are of the form
$$|C|\lesssim \int_0^r\int_{Q_{2r}}w_k^{p-1}|\nabla w_k||\nabla A|\zeta^2x_n dX\,dt.$$
Using Cauchy-Schwarz, Young's inequality and the Carleson condition for $\nabla A$ this is further bounded by
$$|C|\le \varepsilon \int_0^r\int_{Q_{2r}} |w_k|^{p-2}|\nabla w_k|^2(\zeta^2x_n)dX\,dt+C_\varepsilon \|\mu\|_{C}\int_{Q_{2r}}N(\nabla u)^p dx\,dt.$$
Here we choose $\varepsilon>0$ such that we can hide the first term on the right-hand side into the first term of
the second line of  \eqref{eq.pfSdk1}. The last term above is as required in Theorem~\ref{thm.NlessS}, provided
$\|\mu\|_{C}<\delta$ is sufficiently small. 

Now we use the PDE of $w_k$ to treat the last term of \eqref{eq.pfSdk2}. Since $\divg A\nabla u=\dr_tu$, a direct computation shows that 
\[-\divg A\nabla w_k=-\dr_tw_k+\divg\br{(\dr_kA)\vec{w}},\]
where we have used the notation $\vec{w}:=(w_1,w_2,\cdots,w_n)^T=\nabla u$. Therefore,
\begin{multline*}
     -\int_0^r\int_{Q_{2r}}a_{nn}^{-1}\divg(A\nabla w_k)(w_k^{p-1}\zeta^2x_n)dX\,dt
     =-\int_0^r\int_{Q_{2r}}a_{nn}^{-1}\dr_tw_k(w_k^{p-1}\zeta^2x_n)dX\,dt\\
     +\int_0^r\int_{Q_{2r}}a_{nn}^{-1}\divg\br{(\dr_kA)\vec{w}}w_k^{p-1}\zeta^2x_ndX\,dt
     =:I_1+I_2.
\end{multline*}
We continue to compute
\begin{multline*}
    I_1=-\frac1p\int_0^r\int_{Q_{2r}}\dr_t\br{a_{nn}^{-1}w_k^p\zeta^2x_n}dX\,dt\\+\frac1p\int_0^r\int_{Q_{2r}}a_{nn}^{-1}w_k^p\dr_t(\zeta^2)x_ndX\,dt+\frac1p\int_0^r\int_{Q_{2r}}\dr_t(a_{nn}^{-1})w_k^p\zeta^2x_ndX\,dt\\
=\frac1p\int_0^r\int_{Q_{2r}}a_{nn}^{-1}w_k^p\dr_t(\zeta^2)x_ndX\,dt+\frac1p\int_0^r\int_{Q_{2r}}\dr_t(a_{nn}^{-1})w_k^p\zeta^2x_ndX\,dt,
\end{multline*}    
since the first term equals zero. For term $I_2$ we have
\begin{multline*}
I_2=\int_0^r\int_{Q_{2r}}a_{nn}^{-1}(\dr_kA)(\divg\vec{w})w_k^{p-1}\zeta^2x_ndX\,dt-
\int_0^r\int_{Q_{2r}}(\partial_k a_{nn}^{-1})\nabla A\cdot \vec{w}w_{k}^{p-1}\zeta^{2}x_ndX\,dt\\
-\int_0^r\int_{Q_{2r}} a_{nn}^{-1}\nabla A\cdot \partial_k(\vec{w}w_{k}^{p-1})\zeta^{2}x_ndX\,dt
-\int_0^r\int_{Q_{2r}} a_{nn}^{-1}\nabla A\cdot \vec{w}w_{k}^{p-1}\partial_k(\zeta^{2})x_ndX\,dt.
\end{multline*}
Notice, with the exception of the last term we may again use the Cauchy-Schwarz estimate, Carleson condition and Young's inequality to write
$$I_2=D-\int_0^r\int_{Q_{2r}} a_{nn}^{-1}\nabla A\cdot \vec{w}w_{k}^{p-1}\partial_k(\zeta^{2})x_ndX\,dt,$$
where 
$$|D|\le \varepsilon \int_0^r\int_{Q_{2r}} |w_k|^{p-2}|\nabla w_k|^2(\zeta^2x_n)dX\,dt+C_\varepsilon \|\mu\|_{C}\int_{Q_{2r}}[N(\nabla u)^p+S(\nabla u)^p] dx\,dt.$$
Indeed, we have
\begin{multline}\label{NPe10}
\iint_{2Q\times[0,2r]}|\nabla u|^{p-1}|\nabla^2 u||\nabla
A|x_n\,dX\,dt\\
\le \left(\iint_{2Q\times[0,2r]}|\nabla u|^{p-2}|\nabla^2
u|^2x_n\,dX\,dt\right)^{1/2}\left(\iint_{2Q\times[0,2r]}|\nabla
u|^{p}|\nabla
A|^2x_n\,dX\,dt\right)^{1/2}\\
\le \left(\int_{2Q}N(\nabla u)^{p-2}\iint_{\Gamma(x,t)}|\nabla^2
u(X)|^2{x_n}^{-n}\,dX\,dt\right)^{1/2}\left(\iint_{2Q\times[0,2r]}|\nabla
u|^{p}|\nabla A|^2x_n\,dX\,dt\right)^{1/2}\\
\le \left(\int_{2Q}N(\nabla u)^{p-2}S^2(\nabla
u)\,dx\right)^{1/2}\|\mu\|_{C}^{1/2} \|N(\nabla
u)\|^{p/2}_{L^p(2Q)}\\ \le C\|\mu\|_{C}\left[\|S(\nabla
u)\|_{L^p(2Q)}^p+\|N(\nabla u)\|_{L^p(2Q)}^{p}\right],
\end{multline}
with an analogous estimates for other terms.

We now use \eqref{eq.pfSdk2} and our computation of $I_1$ and $I_2$ to obtain a global estimate on $(0,r)\times\pom$.

Let $\set{Q_{2r}^\ell}_{\ell=1}^\infty$ be a collection of disjoint parabolic cubes of sidelength $2r$ that covers $\pom$. Let $\set{\zeta_\ell^2}_\ell$ be a partition of unity subordinate to this collection, that is, $\supp\zeta_\ell\subset Q_{2r}^\ell$ and $\sum_{\ell=1}^\infty\zeta_\ell^2\equiv 1$ on $\pom$. Moreover, for any $\ell$,
$\zeta_\ell=1$ in $Q^\ell_r$,
and for some constant $0<c<\infty$
\[r\abs{\dr_{x_i}\zeta_\ell}+r^2\abs{\dr_t\zeta_\ell}\le c\quad \text{where }1\le i\le n-1.\]
Note that 
\begin{equation}\label{eq.pou0}
    \sum_\ell\dr_{x_i}\br{\zeta_\ell^2}=0 \quad\text{for }1\le i\le n, \quad \text{and }\sum_\ell\dr_{t}\br{\zeta_\ell^2}=0.
\end{equation}
By \eqref{eq.pfSdk2}, the computation of $I_1$ and $I_2$, and \eqref{eq.pou0}, we can write for each $k<n$
\begin{multline}\label{eq.pfSdk3}
C(\lambda,\Lambda,p)\int_0^r\int_{\pom}|w_k|^{p-2}|\nabla w_k|^2x_ndX\,dt\\
   \le \frac{r}{p}\int_{\pom}\dr_n\br{w_k^p}(x,r,t)dx\,dt
    +\frac1p\int_{\pom}w_k(x,0,t)^pdx\,dt
    -\frac1p\int_{\pom}w_k(x,r,t)^pdx\,dt\\
+C\delta\int_{\partial\Omega}[N(\nabla u)^p+S(\nabla u)^p] dx\,dt
    \underbrace{+\frac1p\int_0^r\int_{\partial\Omega}\dr_t(a_{nn}^{-1})w_k^px_n\,dX\,dt}_J.
\end{multline}
The last term $J$ requires more work. We introduce $x_n=\frac12\partial_n(x_n^2)$ and further integrate by parts. It follows that
\begin{multline*}
    \abs{J}=\underbrace{\frac1{2p}\abs{\int_{\partial\Omega}\partial_t(a_{nn}^{-1})w_k^p(x,r,t)r^2\,dx\,dt}}_{=:J_1}+\underbrace{\frac12
    \int_0^r\int_{\pom}|\dr_tA|w_k^{p-1}|\nabla w_k|\,x_n^2\,dX\,dt}_{:= J_2}\\
    +\abs{\frac1{2p}\int_0^r\int_{\partial\Omega}\dr_t\partial_n(a_{nn}^{-1})w_k^px_n^2\,dX\,dt}
    \le J_1+J_2+\frac12{\int_0^r\int_{\partial\Omega}|\partial_n(a_{nn}^{-1})||w_k|^{p-1}|\partial_t w_k|x_n^2dX\,dt}.
\end{multline*}
We denote the last term $J_3$. 
By Cauchy-Schwarz, the Carleson conditions for $\partial_tA$ and $\nabla A$
and Young's inequality we obtain the following bounds for $J_2$ and $J_3$:
$$|J_2|+|J_3|\le \delta\int_{\partial\Omega}[S(\nabla u)^p+A(\nabla u)^p]dx\,dt+C_\delta\|\mu\|_{C}\int_{\partial\Omega}N(\nabla u)^pdx\,dt.$$
By \eqref{A<S+cN.Lp}, these terms are like on the right-hand side of \eqref{NPe1}. For the term $J_1$ we observe that since $|\partial_t A|r^2\le \delta^{1/2}$, this term is small and can be bounded by $C\delta^{1/2}\int_{\partial\Omega}w_k^p(x,r,t)dx\,dt$.

Fix any $r_0>0$. We integrate \eqref{eq.pfSdk3} in the $r$ variable over $[0,r_0]$ and divide both sides by $r_0$. Since $(\dr_n w^p)x_n=\dr_n(w^px_n)-w^p$, 
\[\int_0^{r_0}\int_{\pom}r\dr_n(w_k^p)(x,r,t)dx\,dt\,dr=r_0\int_{\pom}w_k(x,r_0,t)^pdx\,dt-\int_0^{r_0}\int_{\pom}w_k^p(x,x_n,t)dX\,dt,\]
and hence
\begin{multline*}
    C(\lambda,\Lambda,p)\int_0^{r_0}\int_{\pom}\br{x_n-\frac{x_n^2}{r_0}}|w_k|^{p-2}\abs{\nabla w_k}^2dX\,dt\\
    \le \int_{\pom}w_k(x,r_0,t)^pdx\,dt
    +\int_{\pom}w_k(x,0,t)^pdx\,dt
    -\frac{2-C\delta^{1/2}}{r_0}\int_0^{r_0}\int_{\pom}w_k(x,x_n,t)^pdX\,dt\\+C\delta\int_{\partial\Omega}[N(\nabla u)^p+S(\nabla u)^p] dx\,dt.
\end{multline*}
Truncating the integral on the left-hand side to $[0,\frac{r_0}{2}]$ and summing over all $k<n$ we obtain for all small $\delta>0$ after we let $r_0\to\infty$:
\begin{multline*}
    \sum_{k<n}\int_0^\infty \int_{\pom}|w_k|^{p-2}\abs{\nabla w_k}^2x_n\,dX\,dt
    \le  C\int_{\pom}w_k(x,0,t)^pdx\,dt+C\delta\int_{\partial\Omega}[N(\nabla u)^p+S(\nabla u)^p] dx\,dt,
\end{multline*}
since by the assumed decay of $\nabla u$ we have that $\int_{\pom}w_k(x,r_0,t)^pdx\,dt\to 0$ as $r_0\to\infty$.
Hence by  \eqref{NS-T} it follows that
\begin{equation}          \label{NPe2xx}
\int_{\partial\Omega}N(\nabla_T u)^pdx\,dt\le  C\sum_{k<n}\int_{\pom}w_k(x,0,t)^pdx\,dt+C\delta\int_{\partial\Omega}[N(\nabla u)^p+S(\nabla u)^p] dx\,dt.
\end{equation}

It remains to establish estimates for $\int_{\pom}w_k(x,0,t)^pdx\,dt$. We again localize and work 
on one of the sets $\set{Q_{2r}^\ell}_{\ell=1}^\infty$, from the collection of disjoint parabolic cubes of sidelength $2r$ that covers $\pom$. With cutoffs $\zeta_\ell$ as above such that $(\zeta^2_\ell)_{\ell=1}^\infty$ is a partition of unity we consider another cutoff function 
 $\eta:[0,\infty)\to \mathbb R$ such that
$\eta(x_n)=1$ on $[0,r]$, $|\eta'|\le 5/r$ and its support is contained
in $[0,2r]$. Now let
\begin{equation} \label{cutoff}
\xi_\ell(X,t)=\xi_\ell(x,x_n,t)=\zeta^2_\ell(x,t)\eta(x_n).
\end{equation}

We work on estimating $\int_{\pom}w_k(x,0,t)^p\xi_\ell(x)dx\,dt$ in local
coordinates on $2Q_\ell\times (0,5r)$. We now drop the dependence on $\ell$ (working on a single parabolic cube). For each $k\le n-1$ we have
\begin{multline}
 \int_{{\mathbb
R}^{n}}|w_k(x,0,t)|^p\xi(x)dx\,dt=-\iint_{{\mathbb
R}^{n+1}_+}\partial_n(|w_k|^p\xi)(X)dX\,dt\\
 =-\iint_{{\mathbb
R}^{n+1}_+}|w_k|^{p-2}w_k(\partial_nw_k)\xi\,dX\,dt-\iint_{{\mathbb
R}^{n+1}_+}|w_k|^p\zeta^2\eta'\,dX\,dt=I+II.
\end{multline}

The term $II$ is
controlled by $\frac1r\iint_{K_Q}|\nabla u|^p$, where $K_Q=\{X=(x,x_n,t);(x,t)\in
2Q\text{ and }r\le x_n\le 2r\}$. We deal with the term $I$.
Since $\partial_nw_k=\partial_kw_n$, by writing \[w_n=\frac{a_{ni}}{a_{nn}}w_i-\sum_{i<n}\frac{a_{ni}}{a_{nn}}w_i=\frac{H}{a_{nn}}-\sum_{i<n}\frac{a_{ni}}{a_{nn}}w_i\] we have
\begin{multline}
I=-p\iint_{{\mathbb R}^{n+1}_+}|w_k|^{p-2}w_k(\partial_kw_n)\xi\,dX\,dt\\
=-p\iint_{{\mathbb
R}^{n+1}_+}|w_k|^{p-2}w_k\partial_k\left(\frac{H}{a_{nn}}\right)\xi\,dX\,dt+ p\sum_{i<n}\iint_{{\mathbb
R}^{n+1}_+}|w_k|^{p-2}w_k\partial_k\left(\frac{a_{ni}}{a_{nn}}w_i\right)\xi\,dX\,dt.\label{NPe3}
\end{multline}

The second term of (\ref{NPe3}) can be further written as
\begin{multline}
p\sum_{i<n}\iint_{{\mathbb
R}^{n+1}_+}|w_k|^{p-2}w_k\partial_k\left(\frac{a_{ni}}{a_{nn}}w_i\right)\xi\,dX\,dt\\
=p\sum_{i<n}\iint_{{\mathbb
R}^{n+1}_+}|w_k|^{p-2}w_kw_i\partial_k\left(\frac{a_{ni}}{a_{nn}}\right)\xi\,dX\,dt+\sum_{i<n}\iint_{{\mathbb
R}^{n+1}_+}\partial_i(|w_k|^{p})\frac{a_{ni}}{a_{nn}}\xi\,dX\,dt.\label{NPe4}
\end{multline}
We introduce $\partial_n x_n=1$ into both the terms of (\ref{NPe4})
and integrate by parts. This gives
\begin{multline}\label{NPe5}
-\sum_{i<n}\left[p\iint_{{\mathbb
R}^{n+1}_+}\partial_n\left(|w_k|^{p-2}w_kw_i\partial_k\left(\frac{a_{ni}}{a_{nn}}\right)\xi\right)
x_n\,dX\,dt+\iint_{{\mathbb
R}^{n+1}_+}\partial_n\left(\partial_i(|w_k|^{p})\frac{a_{ni}}{a_{nn}}\xi\right)
x_n\,dX\,dt\right]\\
=-\sum_{i<n}\Bigg[p\iint_{{\mathbb
R}^{n+1}_+}\partial_n(|w_k|^{p-2}w_k)w_i\partial_k\left(\frac{a_{ni}}{a_{nn}}\right)\xi
x_n\,dX\,dt+    \iint_{{\mathbb
R}^{n+1}_+}\partial_i(|w_k|^{p})\partial_n\left(\frac{a_{ni}}{a_{nn}}\right)\xi
x_n\,dX\,dt      \\
+p\iint_{{\mathbb
R}^{n}_+}|w_k|^{p-2}w_k(\partial_nw_i)\partial_k\left(\frac{a_{ni}}{a_{nn}}\right)\xi
x_n\,dX\,dt  +\\
p\iint_{{\mathbb
R}^{n+1}_+}|w_k|^{p-2}w_kw_i\partial_k\left(\frac{a_{ni}}{a_{nn}}\right)\zeta^2\eta'
x_n\,dX\,dt+\iint_{{\mathbb R}^{n+1}_+}\partial_i(|w_k|^{p})\frac{a_{ni}}{a_{nn}}\zeta^2\eta' x_n\,dX\,dt\\
-p\iint_{{\mathbb
R}^{n}_+}|w_k|^{p-2}w_kw_i\partial_n\partial_k\left(\frac{a_{ni}}{a_{nn}}\right)\xi
x_n\,dX\,dt+\iint_{{\mathbb
R}^{n}_+}\partial_n\partial_i(|w_k|^{p})\frac{a_{ni}}{a_{nn}}\xi
x_n\,dX\,dt\Bigg].
\end{multline}

In the last two terms we integrate by parts one more time, moving the $\partial_k$ or $\partial_i$ derivatives.
This gives
\begin{multline}\label{NPe6}
\sum_{i<n}\Bigg[p\iint_{{\mathbb
R}^{n}_+}\partial_k(|w_k|^{p-2}w_k)w_i\partial_n\left(\frac{a_{ni}}{a_{nn}}\right)\xi
x_n\,dX\,dt+p\iint_{{\mathbb R}^{n+1}_+}|w_k|^{p-2}w_k(\partial_k
w_i)\partial_n\left(\frac{a_{ni}}{a_{nn}}\right)\xi
x_n\,dX\,dt                  \\
+ \iint_{{\mathbb R}^{n+1}_+}\partial_n
(|w_k|^p)\partial_i\left(\frac{a_{ni}}{a_{nn}}\right)\xi
x_n\,dX\,dt+
p\iint_{{\mathbb
R}^{n}_+}|w_k|^{p-2}w_kw_i\partial_n\left(\frac{a_{ni}}{a_{nn}}\right)(\partial_k\zeta^2)\eta
x_n\,dX\,dt\\+\iint_{{\mathbb R}^{n+1}_+}\partial_n
(|w_k|^p)\frac{a_{ni}}{a_{nn}}(\partial_i\zeta^2)\eta
x_n\,dX\,dt\Bigg]
\end{multline}

The first three terms on the right-hand side of both (\ref{NPe5}) and
(\ref{NPe6}) enjoy the same bound as \eqref{NPe10} we have already done above.

The fourth term on right-hand side of (\ref{NPe5}) can be estimated
by
\begin{multline}\label{NPe11}
C\iint_{2Q\times[r,2r]}|\nabla u|^{p}|\nabla
A|\textstyle{\frac{x_n}{r}}\,dX\,dt\\
\le \left(\iint_{2Q\times[r,2r]}|\nabla
u|^{p}{\textstyle\frac{x_n}{r^2}}\displaystyle\,dX\,dt\right)^{1/2}\left(\iint_{2Q\times[0,2r]}|\nabla
u|^{p}|\nabla
A|^2x_n\,dX\,dt\right)^{1/2}\\\le\left(\int_{2Q}
N(\nabla u)^p(x)\,dx\,dt \right)^{1/2}  \|\mu\|_{C}^{1/2} \|N(\nabla
u)\|^{p/2}_{L^p(2Q)}=\|\mu\|_{C}^{1/2} \|N(\nabla u)\|^{p}_{L^p(2Q)}.
\end{multline}

The fifth term on right-hand side of (\ref{NPe5}) can be estimated
by
\begin{multline}\label{NPe12}
C\iint_{2Q\times[r,2r]}|\nabla u|^{p-1}|\nabla^2
u|\textstyle{\frac{x_n}{r}}\,dX\,dt\lesssim \displaystyle r^{n+2}
 \left(\fiint_{2Q\times[r,2r]}|\nabla
u|^{p}\displaystyle\,dX\,dt\right)^{1/2}\times\\\left(\fiint_{2Q\times[0,2r]}|\nabla u|^{p-2}|\nabla^2
u|^{2}\,dX\,dt\right)^{1/2}\le
\displaystyle r^{n+2}
 \left(\fiint_{2Q\times[r,2r]}|\nabla
u|^{p}\displaystyle\,dX\,dt\right)^{1/2}\times\\\left(\fint_{2Q}N(\nabla u)dx\,dt\right)^{(p-2)/2}\left(\fiint_{2Q\times [r,2r]}|\nabla ^2 u|^2dX\,dt\right)^{1/2}\\
 \le \frac{C_\delta}r\iint_{\widetilde{K_Q}}|\nabla u|^p\,dX\,dt+\delta\int_{2Q}N(\nabla u)^{p}dx\,dt .
\end{multline}
To get the last line we used some standard parabolic estimates away
from the boundary (such as Caccioppoli inequality for the second gradient which holds when $|\nabla
A|r\le C$). Here and in the sequel, $\widetilde{K_Q}$ is a small enlargement of $K_Q$.

We return now to the first term on the right-hand side of (\ref{NPe3}). 
By introducing $\dr_nx_n=1$ and integrating by parts, we have
\begin{multline}\label{NPe7} -p\iint_{\mathbb{R}_+^{n+1}}|w_k|^{p-2}w_k\partial_k\left(\frac{H}{a_{nn}}\right)\xi\,dX\,dt=-p\iint_{\mathbb{R}_+^{n+1}}|w_k|^{p-2}w_k\partial_k\left(\frac{H}{a_{nn}}\right)\xi(\partial_n
x_n)\,dX\,dt\\
= p\iint_{\mathbb{R}_+^{n+1}}\partial_n(|w_k|^{p-2}w_k)\partial_k\left(\frac{H}{a_{nn}}\right)\xi
x_n\,dX\,dt+ p\iint_{\mathbb{R}_+^{n+1}}|w_k|^{p-2}w_k\partial_k\left(\frac{H}{a_{nn}}\right)\zeta^2\eta'
x_n\,dX\,dt\\
+ p\iint_{\mathbb{R}_+^{n+1}}|w_k|^{p-2}w_k\partial_n\partial_k\left(\frac{H}{a_{nn}}\right)\xi
x_n\,dX\,dt,
\end{multline}
where the last term further yields after moving the  $\partial_k$ derivative:
\begin{equation}
\label{NPe8} -p\iint_{\mathbb{R}_+^{n+1}}\partial_k(|w_k|^{p-2}w_k)\partial_n\left(\frac{H}{a_{nn}}\right)\xi
x_n\,dX\,dt- p\iint_{\mathbb{R}_+^{n+1}}|w_k|^{p-2}w_k\partial_n\left(\frac{H}{a_{nn}}\right)(\partial_k\zeta^2)\eta
x_n\,dX\,dt.
\end{equation}
If the derivative in the first two terms on the right-hand side of
(\ref{NPe7}) and (\ref{NPe8}) falls on $a_{nn}^{-1}$, we obtain terms we have already bounded above (see
(\ref{NPe10}) and (\ref{NPe11})). If the derivative falls on $H$
the first term on the right-hand side of both (\ref{NPe7}) and
(\ref{NPe8}) is bounded by
\begin{multline*}
 C\iint_{\mathbb{R}_+^{n+1}}|w_k|^{p-2}|\nabla
w_k||\nabla H|\xi x_n\,dX\,dt\\
\le C\left(\iint_{\mathbb{R}_+^{n+1}}|w_k|^{p-2}|\nabla
w_k|^2\xi x_n\,dX\,dt\right)^{1/2} \left(\iint_{\mathbb{R}_+^{n+1}}|w_k|^{p-2}|\nabla H|^2\xi x_n\,dX\,dt\right)^{1/2}\\
\lesssim \left(\int_{2Q}N^{p-2}(w_k)(x,t)\int_{\Gamma(x,t)}|\nabla
w_k(Y,s)|^2y_n^{-n} dYds\,dx\,dt\right)^{1/2}\left(\iint_{\mathbb{R}_+^{n+1}}|\nabla_T
u|^{p-2}|\nabla
H|^2 \xi x_n\,dX\,dt\right)^{1/2}\\
=\left(\int_{2Q}N^{p-2}(w_k)(x,t)
S^2(w_k)(x,t)\,dx\,dt\right)^{1/2}\left(\iint_{\mathbb{R}_+^{n+1}}|\nabla_T
u|^{p-2}|\nabla
H|^2\xi x_n\,dX\,dt\right)^{1/2}\\
\le
C_\delta\iint_{\Omega}|\nabla_T
u|^{p-2}|\nabla H|^2 \xi x_n\,dX\,dt+\delta\left[\|N(w_k)\|_{L^p(2Q)}^{p}+\|S(w_k)\|^p_{L^p(2Q)}\right].
\end{multline*}

If the derivative falls on $H$ in the second term of (\ref{NPe7}),
we get terms of the same form as (\ref{NPe11}) and (\ref{NPe12}).

It follows that for all $k\le n-1$ we have
\begin{multline}
\label{NPe13} \int_{{\mathbb R}^{n}}w_k(x,0,t)^p\zeta^2(x,t)dx\,dt
\le C_\delta\iint_{\Omega}|\nabla_T
u|^{p-2}|\nabla H|^2 \xi x_n\,dX\,dt\\+\delta\left[\|N(\nabla u)\|_{L^p(2Q)}^{p}+\|S(w_k)\|^p_{L^p(2Q)}\right]
+
\frac{C}r\iint_{\widetilde{K_Q}}|\nabla u|^p\,dX\,dt+E.
\end{multline}
Here $E$ denotes remainder terms; these are the last two terms of
(\ref{NPe6}) and the last term of (\ref{NPe8}) when the derivative
falls on $H$. We now sum (\ref{NPe13}) over all $k\le n-1$ and
also sum over all coordinate patches $Q_\ell$. We notice that the
error terms $E$ will completely cancel out as $\sum_\ell (\partial_k
\zeta^2_\ell)=0$. This yields a global estimate for all $r>0:$
\begin{multline}\label{NPe13a}
\int_{\partial\Omega}|\nabla_T u|^pdx\,dt\le
C\iint_{\Omega}|\nabla_T
u|^{p-2}|\nabla H|^2 x_n\,dX\,dt\\+C\delta\left[\|N(\nabla u)\|_{L^p(\partial\Omega)}^{p}+\|S(\nabla_Tu)\|^p_{L^p(\partial\Omega)}\right]+\frac{C}r\iint_{\partial\Omega\times (r/2,3r/2)}|\nabla u|^p\,dX\,dt.
\end{multline}
As before since for the energy solutions
$\int_{\pom}\nabla u(x,r,t)^pdx\,dt\to 0$ as $r\to\infty$, it follows that the last term converges to zero and hence
by \eqref{NPe2xx} the desired estimate
(\ref{NPe1}) holds.
\end{proof}

In the following lemma we again assume that $\Omega=\mathbb R^n_+\times\R$.
\begin{lemma}\label{NPl2} Let $p\ge 2$ be an integer, $k$ be an integer such that $0\le k\le p-2$.
Let $A$ be a matrix with bounded measurable coefficients that satisfies the ellipticity condition, the Carleson condition for coefficient matrix $A$.   Let $\LL=-\dr_t+\divg(A\nabla\cdot)$. Then there exists $\delta>0$
such that if $\|\mu\|_{C}<\delta$ then for some
$K=K(\lambda,\Lambda,n,p)>0$ we have for an energy solution $\mathcal Lu=0$ in $\Omega$ the estimate:

\begin{multline} \label{NPe14}
\iint_{\Omega}|\nabla_T u|^{p-k-2}|H|^{k}|\nabla
H|^2x_n\,dX\,dt\le K\int_{\partial\Omega}|H|^{p}\,dx\,dt\\
+K(p-k-2)\iint_{\Omega}|\nabla_T u|^{p-k-3}|H|^{k+1}|\nabla
H|^2x_n\,dX\,dt+K\delta\left[\|N(\nabla u)\|_{L^p(\partial\Omega)}^{p}+\|S(\nabla u)\|^p_{L^p(\partial\Omega)}\right].
\end{multline}
\end{lemma}

\noindent {\it Proof of Lemma \ref{NPl2}.} Fix any $k\in\{0,1,2,\dots,p-2\}$.
Let us choose a cutoff function $\xi$ as in (\ref{cutoff}). We aim to bound
\begin{eqnarray} \nonumber
\iint_{\R^{n-1}\times (0,r)\times \R}|\nabla_T u|^{p-k-2}|H|^{k}|\nabla
H|^2 x_n\,dX\,dt,
\end{eqnarray}
and then let $r\to\infty$.
Clearly, it suffices to get bounds for
\begin{eqnarray} \nonumber
I:=\iint_{\Omega}|\nabla_T
u|^{p-k-2}|H|^{k}b_{ij}(\partial_iH)(\partial_jH)\xi x_n\,dX\,dt
\end{eqnarray}
for some matrix $B$ (to be specified later) satisfying a uniform ellipticity condition and then sum over the decomposition of $\R^n=\bigcup_\ell Q_\ell$.

We integrate $I$ by parts. Since  $|H|^k\partial_i H=1/(k+1)\partial_i(|H|^kH)$, this gives
\begin{multline} \label{NPe17} 
I=-\frac1{k+1} \iint_{\Omega}|\nabla_T
u|^{p-k-2}|H|^{k}H\partial_i (b_{ij}\partial_jH)\xi x_n\,dX\,dt\\
 -\frac1{k+1} \iint_{\Omega}|\nabla_T
u|^{p-k-2}|H|^{k}Hb_{nj}(\partial_jH)\xi\,dX\,dt\\
\quad\qquad-\frac1{k+1} \iint_{\Omega}|\nabla_T
u|^{p-k-2}|H|^{k}Hb_{nj}(\partial_jH)(\partial_i\xi) x_n\,dX\,dt\\
 -\frac{p-k-2}{k+1} \iint_{\Omega}|\nabla_T
u|^{p-k-4}(\nabla_Tu\cdot
\partial_i(\nabla_T u))|H|^{k}Hb_{ij}(\partial_jH)\xi x_n\,dX\,dt .
\end{multline}
The second term only appears in (\ref{NPe17}) if $i=n$ as $\partial_nx_n=1$. 

Consider first the third term of (\ref{NPe17}) for $i=n$. As $|\eta'|\le
2/r$ and $\eta'=0$ on $[0,r]$ we have that this term is bounded by
\begin{multline} \label{NPe17a} 
\iint_{Q \times [r,2r]}|\nabla_Tu|^{p-k-2}|H|^{k+1}|\nabla
H|\textstyle\frac{x_n}{r}\displaystyle\,dX\,dt \\\le \delta \int_{2Q}
N^{p}(\nabla u)\,dx\,dt +
\frac{C_\delta}{r}\iint_{\widetilde{K_Q}}|\nabla u|^p\,dX\,dt,
\end{multline}
since this term is of the same type as  (\ref{NPe12}) we have estimated previously. When $i < n$, the third terms of  \eqref{NPe17}
will cancel each other out after we sum over the indices $\ell$ in the partition of unity $(Q_\ell)_\ell$.

In order  to handle  the second term of
(\ref{NPe17}) we make a choice for the matrix $B$ such that $b_{nn}=1$.
 Then the second term of (\ref{NPe17}) for $j=n$ looks as follows:
\begin{multline}\label{NPe18}
  -\frac{1}{(k+1)(k+2)} \iint_{\Omega}|\nabla_T
u|^{p-k-2}(\partial_n|H|^{k+2})\xi\,dX\,dt\\
= -\frac{1}{(k+1)(k+2)}
\iint_{\Omega}\partial_n(|\nabla_T
u|^{p-k-2}|H|^{k+2}\xi)\,dX\,dt\\
\qquad+\frac{1}{(k+1)(k+2)}
\iint_{\Omega}\partial_n(|\nabla_T
u|^{p-k-2})|H|^{k+2}\xi\,dX\,dt\\
+\frac{1}{(k+1)(k+2)} \iint_{\Omega}|\nabla_T
u|^{p-k-2}|H|^{k+2}(\partial_n\xi)\,dX\,dt.
\end{multline}

Here the last term again can be estimated by a solid integral
$\frac{C}r\iint_{\widetilde{K_Q}}|\nabla u|^p\,dX\,dt$. The first term is equal to a boundary
integral
\begin{multline} \nonumber \frac{1}{(k+1)(k+2)}
\int_{\partial\Omega}|\nabla_T u|^{p-k-2}|H|^{k+2}\xi\,dX\,dt\le C_\delta \int_{2Q}|H|^pdx\,dt\\+
\delta \int_{2Q}|\nabla_Tu|^pdx\,dt\le C_\delta \int_{2Q}|H|^pdx\,dt+
\delta \int_{2Q}N(\nabla u)^pdx\,dt.
\end{multline}

It remains to deal with the second term of (\ref{NPe18}). We
differentiate and change the order of derivatives $\partial_n$ and
$\nabla_T$:
\begin{eqnarray} &&\label{NPe19}\frac{p-k-2}{(k+1)(k+2)}
\iint_{\Omega}|\nabla_T u|^{p-k-4}(\nabla_Tu\cdot
\nabla_T\partial_n u)|H|^{k+2}\xi \,dX\,dt.
\end{eqnarray}
We reintroduce the co-normal derivative $H$ as $\partial_n
u=\frac{H}{a_{nn}}-\sum_{j<n}\frac{a_{nj}}{a_{nn}}w_j$. We also
insert the term $\partial_n x_n=1$ into both integrals. Then we
integrate by parts again in the $\partial_n$ derivative. When
exactly one derivative falls on the coefficients (either $a_{nn}$
or $\frac{a_{nj}}{a_{nn}}$) we obtain terms bounded by
\begin{equation}\label{NPe20}
\iint_{2Q\times(0,2r)}|\nabla A||\nabla u|^{p-1}|\nabla^2 u|x_n\,dX\,dt
\end{equation}
which is the term of type (\ref{NPe10}) and is bounded by  $\|\mu\|_{C}[\|S(\nabla u)\|_{L^p(2Q)}^p+\|N(\nabla
u)\|_{L^p(2Q)}^{p}]$.

If both the $\partial_n$ and $\nabla_T$ derivatives fall on
coefficients, there are two possibilities. The first possibility
is that they fall on the same coefficient and so then we do a
further integration by parts in $\nabla_T$ moving this derivative
on other terms. This again will yield a term of the form (\ref{NPe20}).
The second possibility is that they fall on separate coefficients 
and the term has a simple bound using the Carleson condition by 
$\|\mu\|_{C}\|N(\nabla u)\|_{L^p(2Q)}^p$.

We obtain another error term when
$\partial_n$ falls on $\xi$, however in that case we get  terms of
type (\ref{NPe17a})  we have handled before.

Let us deal with the term
when both derivatives fall on $H$. In that case we have
\begin{eqnarray} &&-\frac{p-k-2}{(k+1)(k+2)}
\iint_{2Q\times(0,2r)}\frac1{a_{nn}}|\nabla_T
u|^{p-k-4}(\nabla_Tu\cdot \nabla_T\partial_n H)|H|^{k+2}\xi x_n\,dX\,dt.
\end{eqnarray}
We move the $\nabla_T$ derivative off $\partial_n H$. We can get a
term of type (\ref{NPe20}) and two terms that can be dominated by
\begin{eqnarray} &&\label{NPe22}C(p-k-2)
\iint_{2Q\times(0,2r)}|\nabla_T u|^{p-k-4}|\nabla(\nabla_Tu)||\nabla
H||H|^{k+2} x_n\,dX\,dt\\
&+&C(p-k-2)\nonumber\iint_{2Q\times(0,2r)}|\nabla_T u|^{p-k-3}|\nabla
H|^2|H|^{k+1} x_n\,dX\,dt.
\end{eqnarray}
Also, when $\nabla_T$ lands on $\xi$ we get error terms which will cancel when we sum over all $Q_\ell$. Observe also that the last term of (\ref{NPe17}) can be controlled by
\begin{equation} \label{nov1}
C(p-k-2) \iint_{\Omega_{2r}}|\nabla_T u|^{p-k-3}|\nabla(\nabla_Tu)||\nabla
H||H|^{k+1} x_n\,dX\,dt
\end{equation}
We now deal with the terms arising from
$-\sum_{j<n}\frac{a_{nj}}{a_{nn}}w_j$. Here we write
\[
\nabla_T\left(\sum_{j<n}\frac{a_{nj}}{a_{nn}}w_j\right) = \sum_{j<n}\nabla_T\left(\frac{a_{nj}}{a_{nn}}\right)\partial_ju + \sum_{j<n}\frac{a_{nj}}{a_{nn}}\partial_j(\nabla_Tu).
\]
The contribution of the first term here, when substituted in \eqref{NPe19}, can be dealt with by again introducing the factor $\partial_n x_n$ and integrating by parts. When $\partial_n$ lands on $\nabla_T(a_{nj}/a_{nn})$, we can move the tangential derivatives off by again integrating by parts. All this yields terms of the form we have handled previously (such as \eqref{NPe20}). Substituting the second term in \eqref{NPe19} yields
\begin{eqnarray} &&\frac{1}{(k+1)(k+2)}\sum_{j<n}
\iint_{\Omega}\frac{a_{nj}}{a_{nn}}\partial_j(|\nabla_T
u|^{p-k-2})|H|^{k+2}\xi (\partial_n x_n)\,dX\,dt.
\end{eqnarray}
Moving $\partial_n$ across using integration by parts and if necessary moving $\partial_j$
we obtain terms bounded by \eqref{NPe20}, \eqref{NPe22} or \eqref{nov1}. Thus the analysis of the
second term of (\ref{NPe17}) for $j=n$ reduces to controlling \eqref{NPe22} and \eqref{nov1}, a task which we will postpone for now. When $j<n$ in the second
term of (\ref{NPe17}) we again introduce $(\partial_n x_n)$. This
gives
\begin{eqnarray}
&&\nonumber -\iint_{\Omega}|\nabla_T
u|^{p-k-2}b_{nj}\partial_j(|H|^{k+2})\xi (\partial_n x_n)\,dX\,dt
\end{eqnarray}
We integrate by parts. When $\partial_n$ falls on $|\nabla_T
u|^{p-k-2}$ we can dominate such a term by \eqref{nov1},
when $\partial_n$ falls on $b_{nj}$ we use Cauchy-Schwarz, Carleson condition and
\eqref{NPe20} (i.e., we must choose matrix $B$ whose coefficients also satisfy the Carleson condition). If $\partial_n$ hits
$\xi$ we get terms which can be bounded by \eqref{NPe11} and \eqref{NPe12}. Finally the remaining term
is
\begin{eqnarray}
&&\nonumber \iint_{\Omega}|\nabla_T
u|^{p-k-2}b_{nj}\partial_j\partial_n(|H|^{k+2})\xi x_n\,dX\,dt.
\end{eqnarray}
We integrate by parts again in $\partial_j$ giving us terms
of type \eqref{NPe20}, \eqref{NPe22} and \eqref{nov1}. The only remaining terms
we have not yet bounded are the first term of \eqref{NPe17},
\eqref{NPe22} and \eqref{nov1}. The second term of (\ref{NPe22}) is already of
desired form (see right-hand side of (\ref{NPe14})). By the
the Cauchy-Schwarz inequality, the first term of (\ref{NPe22}) can be bounded by
\begin{multline*}  C(p-k-2)\left(
\iint_{\Omega}|\nabla_T u|^{p-k-3}|\nabla
H|^2|H|^{k+1} \xi x_n\,dX\,dt\right)^{1/2}\times\\
\left( \iint_{\Omega}|\nabla_T
u|^{p-k-5}|\nabla(\nabla_Tu)|^2|H|^{k+3}
\xi x_n\,dX\,dt\right)^{1/2}\\
\le C(p-k-2)\left( \iint_{\Omega}|\nabla_T
u|^{p-k-3}|\nabla H|^2|H|^{k+1}\xi x_n\,dX\,dt\right)^{1/2}\times
\\\|N(\nabla
u)\|^{p/2-1}_{L^p(2Q)}\|S(\nabla
u)\|_{L^p(2Q)}.
\end{multline*}
The last line can be further bounded by
$$\delta[\|N(\nabla
u)\|^{p}_{L^p(2Q)}+\|S(\nabla
u)\|^{p}_{L^p(2Q)}] +
C_\delta(p-k-2)^2\iint_{\Omega}|\nabla_T u|^{p-k-3}|\nabla
H|^2|H|^{k+1} \xi x_n\,dX\,dt.$$ 
Term \eqref{nov1} can be dealt with in a very similar fashion. We summarize what we have
so far. After summing these inequalities over all cubes $Q_\ell$ we have that

\begin{multline} \label{NPe24}
\iint_{\Omega}|\nabla_T u|^{p-k-2}|H|^{k}(B\nabla H\cdot\nabla H)\,\eta x_n\,dX\,dt\\ \le K(p-k-2)
\iint_{\Omega}|\nabla_T u|^{p-k-3}|H|^{k+1}|\nabla
H|^2\eta x_n\,dX\,dt\\+K\int_{\partial\Omega}|H|^{p}\,dx\,dt+\frac{C}{r}\iint_{\partial\Omega\times (r/2,3r/2)}|\nabla
u|^p\,dX\,dt\\
-\frac{1}{k+1} \iint_{\Omega}|\nabla_T
u|^{p-k-2}|H|^{k}H(\widetilde{L}H)\eta x_n\,dX\,dt.
\end{multline}
Here $\widetilde{L}H=\mbox{ div}(B\nabla H)$.  Clearly, (\ref{NPe24}) is the desired estimate
(\ref{NPe14}) modulo the last two terms. The penultimate term is harmless since $\int_{\partial\Omega}|\nabla u|^p(\cdot,r,\cdot)\to 0$ as $r\to\infty$ and hence this term gets eliminated in the limit.
\vglue2mm

Let us consider now the last term of (\ref{NPe24}). By the calculation in subsection \ref{ss1} we see that by inserting (\ref{eq.Hpde2}) into the last term of (\ref{NPe24}) we obtain several terms. Consider first the terms that do not contain $\partial_t$, i.e., those that arise from the second line of \eqref{eq.Hpde2}. When two derivatives hit one of the coefficients we get terms of the form
\begin{eqnarray} &&\label{NPe27}
\iint_{\Omega}\sum_{j<n}[b_{ij}(\partial_i\partial_j
a_{ij})(\partial_k u)-b_{kn}(\partial_k\partial_j
a_{ji})(\partial_i u)]|\nabla_T u|^{p-k-2}|H|^{k}H\eta x_n\,dX\,dt
\end{eqnarray}
and the remaining ones can be bounded by
\begin{eqnarray} &&\label{NPe28}
\iint_{\Omega}|\nabla u|^{p-1}[|\nabla u||\nabla A||\nabla
B|+|\nabla^2 u||\nabla A||B|+|\nabla^2 u||\nabla B||A|]\eta x_n\,dX\,dt.
\end{eqnarray}
The terms in (\ref{NPe27}) have two derivatives on coefficients $a_{ij}$
however one is $\partial_j$ and $j<n$. We therefore integrate by
parts in $\partial_j$. This yields additional terms, but all are
of the form (\ref{NPe28}). However, by an estimate similar to
(\ref{NPe20}) we get that all the terms of (\ref{NPe28}) are
smaller than $C\|\mu\|_{C}\int_{\partial\Omega} [N^p(\nabla
u)+S^p(\nabla
u)]\,dx\,dt$. Hence for $\|\mu\|_{C}<\delta$ these terms are as expected.

It remains to consider terms arising from the first line of (\ref{eq.Hpde2}). We omit the first term containing $\partial_t H$, the remaining terms give rise to the following terms
\begin{multline}\label{eq3.44}
-\frac{1}{2(k+1)}\iint_\Omega |\nabla_Tu|^{p-k-2}|H|^kH\left[(\partial_i b_{in})\partial_tu+(\partial_ta_{nn}^{-1})H-(\partial_tb_{in})\partial_iu\right]\eta (\partial_n x_n^2)dX\,dt\\
=\frac{1}{2(k+1)}\Bigg[\iint_\Omega |\nabla_Tu|^{p-k-2}|H|^kH\left[(\partial_i b_{in})\partial_tu+(\partial_ta_{nn}^{-1})H-(\partial_tb_{in})\partial_iu\right]\eta'  x_n^2 dX\,dt\\
+\iint_\Omega \partial_n(|\nabla_Tu|^{p-k-2}|H|^kH)(\partial_i b_{in})\partial_tu\,\eta x_n^2dX\,dt
+\iint_\Omega \partial_n(|\nabla_Tu|^{p-k-2}|H|^{k+2})(\partial_t a_{nn}^{-1})\,\eta x_n^2dX\,dt\\
-\iint_\Omega \partial_n(|\nabla_Tu|^{p-k-2}|H|^{k}H\partial_i u)(\partial_t b_{in})\,\eta x_n^2dX\,dt+\iint_\Omega |\nabla_Tu|^{p-k-2}|H|^kH(\partial_i b_{in})\partial^2_{tn}u\\
+\iint_\Omega |\nabla_Tu|^{p-k-2}|H|^kH\left[(\partial^2_{in} b_{in})\partial_tu+(\partial^2_{tn}a_{nn}^{-1})H-(\partial^2_{tn}b_{in})\partial_iu\right]\eta x_n^2 dX\,dt\Bigg].
\end{multline}
We estimate these terms line by line. Starting with the first line on the right-hand side of \eqref{eq3.44}
we see that these terms can be estimated using the Carleson condition. For the term $\partial_tu$
we use $|\partial_tu|\le |\nabla ^2u|+|\nabla A||\nabla u|$. For  $\nabla ^2u$ we then again use Caccioppoli inequality for the second gradient. This yields estimates similar to \eqref{NPe11}-\eqref{NPe12}:
$$|\mbox{1st line of \eqref{eq3.44}}|\le \delta\int_{\partial\Omega}N(\nabla u)^pdx\,dt+\frac{C_\delta}{r}\iint_{\partial\Omega\times(r/2,3r/2)}|\nabla u|^pdX\,dt.$$
The terms in the second and third lines have the following forms:
\begin{multline}\label{eq3.45}
|\mbox{2nd/3rd line of \eqref{eq3.44}}|\le \iint_{\Omega}\Big[|\nabla^2 u|^2|\nabla u|^{p-2}|\nabla A|
+|\nabla^2 u||\nabla A|^2|\nabla u|^{p-1}\\+|\nabla A|^3|\nabla u|^p+|\nabla^2 u||\nabla u|^{p-1}|\partial_t A|+|\nabla u|^{p}|\nabla A||\partial_t A|+|\nabla\partial_t u||\nabla u|^{p-1}|\nabla A| \Big]\eta x_n^2\, dX\,dt.
\end{multline}
For the first five terms of \eqref{eq3.45}, we use:
\begin{itemize}
\item the pointwise bound $|\nabla A|x_n\le\delta^{1/2}$ for terms one
and three, followed respectively by the square-function bound and
spatial Carleson embedding;
\item Cauchy-Schwarz and the Carleson measure $|\nabla A|^2x_n\,dX\,dt$
for term two;
\item Cauchy-Schwarz and $|\partial_tA|^2x_n^3\,dX\,dt$ for term four,
with the spatial Carleson measure as well for term five.
\end{itemize}
For the last term, Cauchy-Schwarz and Young's inequality give
\begin{equation}\label{eq3.55z}
C_\delta\|\mu\|_{C}\int_{\partial\Omega}N(\nabla u)^pdx\,dt+\delta \iint_\Omega|\nabla \partial_t u|^2|\nabla u|^{p-2}x_n^3\eta\,dX\,dt.
\end{equation}
 The second term is bounded by 
 $\delta\int_{\pom}N(\nabla u)^{p-2}A(\nabla u)^2dx\,dt$, which can be bounded  using $S(\nabla u)$ and $N(\nabla u)$ by H\"older's inequality and \eqref{A<S+cN.Lp}.
Hence, each term of \eqref{eq3.45} contains a small term due to the presence of derivatives of $A$ and thus 
$$
|\mbox{2nd/3rd line of \eqref{eq3.44}}|\le C\delta\int_{\partial\Omega}N(\nabla u)^pdx\,dt+C\delta\int_{\partial\Omega}S(\nabla u)^pdx\,dt.
$$

It remains to consider terms in the last line of \eqref{eq3.44}. Each of the terms has
too many derivatives on one of the coefficients of $A$ or $B$. Starting with the first term where we have 
$\partial^2_{in}b_{in}$ we observe that we only need to consider $i<n$ since $b_{nn}=1$ and thus
$\partial^2_{nn}b_{nn}=0$. But for $i<n$ we may integrate by parts in $\partial_i$ and obtain terms we have already considered in  \eqref{eq3.44}. 

Similarly for the remaining terms on the last line of \eqref{eq3.44} we integrate by parts in $\partial_t$ to remove $\partial^2_{tn}$ from the coefficients. This means $\partial_t$ will either fall on $H$
(and then we use \eqref{eq.sqrAAL2}) or $\partial_t$ falls on $\nabla_T u$ and hence it is as in \eqref{eq3.55z}.
It follows that all terms arising from \eqref{eq.Hpde2} have been handled with the exception of the first one. To do that, we need to consider the term 
\begin{multline}
-\frac{1}{k+1}\iint_{\Omega}a_{nn}^{-1}|\nabla_Tu|^{p-k-2}|H|^kH(\partial_t H)\eta x_n\, dX\,dt\\
=-\frac{1}{(k+1)(k+2)}\iint_{\Omega}a_{nn}^{-1}|\nabla_Tu|^{p-k-2}\partial_t (|H|^{k+2})\eta x_n\, dX\,dt\\
=\frac{1}{(k+1)(k+2)}\iint_{\Omega}\left[(\partial_t a_{nn}^{-1})|\nabla_T u|^{p-k-2}|H|^{k+2}+a_{nn}^{-1}
\partial_t(|\nabla_T u|^{p-k-2})|H|^{k+2}\right]\eta x_n\, dX\,dt,
\end{multline}
after integration by parts. The first term on the last line is harmless (we introduce $x_n=\frac12\partial_n(x_n^2)$, integrate by parts in $\partial_n$ and obtain terms analogous to that of 
\eqref{eq3.44}). We omit the details. 

The second term must vanish when $p-k-2=0$ and hence we only consider the case when $q=p-k-2\ge 1$. In a calculation that mimics \eqref{S3:T8:E01-x}


\begin{multline}
\partial_t(|w_k|^{q})=\divg(A\nabla(|w_k|^{q}))+
q|w_k|^{q-2}w_k\left[\partial_t w_k-\divg(A\nabla w_k)\right]\\-q(q-1)|w_k|^{q-2}A\nabla w_k\cdot\nabla w_k
\\=\divg(A\nabla(|w_k|^{q}))-q|w_k|^{q-2}w_k\divg((\partial_kA)\nabla u)-q(q-1)|w_k|^{q-2}A\nabla w_k\cdot\nabla w_k,
\end{multline}
where as before $w_k=\partial_k u$. We have also used the PDE for $w_k$. Replacing $|\nabla_T u|$   by sum  of the terms $k<n$ containing $|w_k|$ we now have:
\begin{multline}
A_k+B_k+C_k:=\frac{1}{(k+1)(k+2)}\Bigg[\iint_{\Omega}a_{nn}^{-1}\divg(A\nabla(|w_k|^q))|H|^{k+2}\eta x_n\, dX\,dt\\
-q\iint_{\Omega}a_{nn}^{-1}|w_k|^{q-2}w_k\divg((\partial_kA)\nabla u)|H|^{k+2}\eta x_n\, dX\,dt\\
-q(q-1)\iint_{\Omega}a_{nn}^{-1}|w_k|^{q-2}A\nabla w_k\cdot\nabla w_k|H|^{k+2}\eta x_n\, dX\,dt
\Bigg].
\end{multline}

Observe that $C_k\le 0$ by ellipticity and it can be dropped, as we are seeking a bound from above. The  term $B_k$ (which we expect to be small) satisfies
\begin{multline*}
B_k=C(p,k)\iint_{\Omega}a_{nn}^{-1}|w_k|^{p-k-4}w_k\partial_k(\divg(A)\nabla u)|H|^{k+2}\eta x_n\, dX\,dt\\
+C(p,k)\iint_{\Omega}a_{nn}^{-1}|w_k|^{p-k-4}w_k(\partial_kA)\nabla^2 u|H|^{k+2}\eta x_n\, dX\,dt.
\end{multline*}
By Cauchy-Schwarz, the second term is bounded by 
\begin{equation}\label{eqzz}
C\iint_{\Omega}|\nabla ^2u||\nabla u|^{p-2}|\nabla A|\eta x_n\le C\|\mu\|^{1/2}_{C}\|S(\nabla u)\|^2\|N(\nabla u)\|^{p-2}.
\end{equation}
For the first term we integrate by parts in $\partial_k$ and obtain terms bounded by
\begin{equation}\label{eqzzz}
C\iint_{\Omega}|\nabla A|^2|\nabla u|^p\eta x_n+C\iint_{\Omega}|\nabla ^2u||\nabla u|^{p-2}|\nabla A|\eta x_n.
\end{equation}
The first term is bounded by $\|\mu\|_{C}\|N(\nabla u)\|^p$, while the second term was handled earlier.

Returning to $A_k$ we consider the divergence in variables $i<n$ first. Hence
$$A_k=A_{k,i<n}+A_{k,i=n}.$$ For $A_{k,i<n}$ we integrate by parts in $\partial_i$ to obtain 
\begin{multline*}
A_{k,i<n}=-\frac{1}{(k+1)(k+2)}\sum_{i<n}\iint_{\Omega}\Bigg[\partial_i(a_{nn}^{-1})a_{ij}(\partial_j |w_k|^{p-k-2})|H|^{k+2}\\
+(k+2)a_{nn}^{-1}a_{ij}(\partial_j|w_k|^{p-k-2})|H|^{k}H(\partial_i H)\Bigg]\eta x_n\,dX\,dt.
\end{multline*}
The first term is again a small term as it is of the type \eqref{eqzz} considered earlier.  For the last term we use Cauchy-Schwarz to get that this term is bounded from above by
\begin{multline}\label{eq3.51}
C\left(\iint_\Omega |\nabla H|^2|H|^{k+1}|\nabla_T u|^{p-k-3}\eta x_n\,dX\,dt\right)^{1/2}
\left(\iint_\Omega |\nabla \nabla_T u|^2|\nabla u|^{p-2}\eta x_n\,dX\,dt\right)^{1/2}\\
\le K_\delta\iint_{\Omega}|\nabla_T u|^{p-k-3}|H|^{k+1}|\nabla
H|^2x_n\,dX\,dt+K\delta\left[\|N(\nabla u)\|_{L^p(\partial\Omega)}^{p}+\|S(\nabla u)\|^p_{L^p(\partial\Omega)}\right],
\end{multline}
by the Young's inequality. These terms correspond to the right-hand side of \eqref{NPe14} as required.

When $i=n$ we notice that
\begin{multline*}
a_{nj}(\partial_j |w_k|^{p-k-2})=(p-k-2)|w_k|^{p-k-4}w_ka_{nj}\left(\partial^2_{jk}u\right)=
(p-k-2)|w_k|^{p-k-4}w_k\big[\partial_k(a_{nj}\partial_j u)\\-(\partial_ka_{nj})\partial_ju\big]
=(p-k-2)|w_k|^{p-k-4}w_k\left[\partial_kH-(\partial_k a_{nj})\partial_j u\right],
\end{multline*}
which implies that 
\begin{multline}\label{eq3.52}
A_{k,i=n}=\frac{p-k-2}{(k+1)(k+2)}\Bigg[\iint_{\Omega}a_{nn}^{-1}\partial_{n}(|w_k|^{p-k-4}w_k\partial_kH)|H|^{k+2} \eta x_n\,dX\,dt\\
-\iint_{\Omega}a_{nn}^{-1}\partial_n[(\partial_ka_{nj})|w_k|^{p-k-4}w_kw_j]|H|^{k+2}\eta x_n\, dX\,dt\Bigg].
\end{multline}
Here in the first term if the derivative $\partial_n$ hits $|w_k|^{p-k-4}w_k$ we get a term that can be bounded by \eqref{eq3.51}. When it falls on $H$ we peel off the $\partial_k$ derivative from $H$ to again get a term like \eqref{eq3.51},
or (if the $\partial_k$ falls on the coefficient) a term like \eqref{eqzz}.

Similarly, for the second term of \eqref{eq3.52}, if $\partial_n$ falls on $|w_k|^{p-k-4}w_kw_j$ we have a term like 
\eqref{eqzz}. Otherwise we have two derivatives on $\partial^2_{nk}(a_{nj})$ and, after peeling off $\partial_k$,
we obtain terms that look like \eqref{eqzz} or like the first term of \eqref{eqzzz}. For both types we have already established the
desired bounds. Thus we have shown the claim.\ep

\ms

\section{$S< N$ for $\nabla u$.}\label{S.S<N}
The goal of this section is to prove the following theorem.
\begin{theorem}\label{thm.S<Np}
Let $\LL=-\dr_t+\divg(A\nabla\cdot)$ be a parabolic operator that satisfies \eqref{E:elliptic}, \eqref{E:1:carl}, and \eqref{E:1:bound}. Let $u$ be an energy solution to $\LL u=0$ in $\om=\Rn_+\times\R$.
For any $p\in(0,\infty)$, there is $C>0$ such that
    \[
    \norm{S(\nabla u)}_{L^p}\le C\norm{N(\nabla u)}_{L^p}. 
    \]
\end{theorem}
To prove Theorem \ref{thm.S<Np}, we first derive an $S<N$ estimate on sawtooth domains using cutoff functions as in \cite[Section 9]{DLP1}.  
Let $\Psi$ be a cutoff function on $\om$ that satisfies the following properties.
\begin{enumerate}
    \item $0\le \Psi\le1$,
    \item for any $p\in(0,\infty)$, the measure given by the density
 \begin{equation}
        \br{\abs{\nabla\Psi(X,t)}^px_n^{p-1}+\abs{\dr_t\Psi(X,t)}^px_n^{2p-1}}dxdx_ndt
    \end{equation}
    is a Carleson measure on $\om$.
\end{enumerate}
\smallskip

The key difference from \cite[Theorem 9.1]{DLP1} is that we do not assume that the coefficient matrix of the operator is in block form. To achieve that, we first treat $S(\nabla_T u)$, and then $S(H)$, where $H=\sum_{1\le j\le n}a_{nj}\dr_j u$. Observe that 
\[
|\nabla^2u|^2\lesssim |\nabla(\nabla_Tu)|^2+|\dr_{nn}u|^2\lesssim  |\nabla(\nabla_Tu)|^2+|\dr_nH|^2+\abs{\dr_nA}^2|\nabla u|^2
\]
and so 
\begin{multline}\label{eq.grad2u_split}
      \iint |\nabla^2u|^2\Psi^mx_n\,dX\,dt\lesssim   \iint |\nabla(\nabla_Tu)|^2\Psi^mx_n\,dX\,dt+  \iint |\dr_nH|^2\Psi^mx_n\,dX\,dt\\
      + \norm{|\dr_nA|^2x_n}_{C} \int_{\pom} N\br{\Psi^{m/2}\nabla u }  ^2dx\,dt.
\end{multline}

\begin{lemma}\label{lem.Sdk}
Let $\LL$ and $u$ be as in Theorem \ref{thm.S<Np}. Let $0\le\Psi\le1$ be a cutoff function that satisfies the following conditions:  
\begin{equation}\label{eq.Psi_bd}
    |\nabla \Psi(X,t)|+x_n|\dr_t \Psi(X,t)|\lesssim x_n^{-1}\quad\text{ for }(X,t)\in\om,
\end{equation}  and the measure $\nu$ given by the density 
\begin{equation}
d\nu=\br{\abs{\nabla\Psi}^2x_n+\abs{\nabla\Psi}+\abs{\dr_t\Psi}x_n}dX\,dt \quad\text{ is a Carleson measure on }\om.
\end{equation}
Let $m\ge 3$ be an integer, and let $k\in\set{1,\dots,n-1}$. Then for any $\varepsilon>0$,  there holds
\begin{multline}
     \iint \abs{\nabla(\dr_ku)}^2\Psi^mx_n\, dX\,dt
    \le \varepsilon \iint|\nabla^2u|^2\Psi^mx_n dX\,dt\\
    + C_\varepsilon\br{1+\norm{\mu}_{C}+\norm{\nu}_{C}}\int_{\pom}\abs{N(\Psi^{\frac{m}{2}-1}\nabla u)}^2dx\,dt.
\end{multline}
The constants may depend on $n$, $m$, the ellipticity and coefficient bounds,
$K$ in \eqref{E:1:bound}, and the implicit constant in \eqref{eq.Psi_bd}.
They are independent of the support of $\Psi$.
\end{lemma}
\medskip
\begin{lemma}\label{lem.SH}
Let $\LL$ and $u$ be as in Theorem \ref{thm.S<Np}. Let $\Psi$ be a cutoff function as in Lemma \ref{lem.Sdk}. 
Let $m\ge 3$ be an integer. Then for any $\varepsilon>0$,
\begin{multline*}
    \iint \abs{\nabla H}^2\Psi^mx_n\, dX\,dt
    \le \varepsilon \iint|\nabla^2u|^2\Psi^mx_n dX\,dt\\
    +C_\varepsilon\br{1+\norm{\mu}_{C}+\norm{\nu}_{C}}\int_{\pom}\abs{N(\Psi^{\frac{m}{2}-1}\nabla u)}^2dx\,dt.
\end{multline*}
\end{lemma}

Combining Lemma \ref{lem.Sdk} and \ref{lem.SH} with the estimate \eqref{eq.grad2u_split}, one gets that 
\begin{lemma}\label{lem.S2cutoff}
    Let $\LL$ and $u$ be as in Theorem \ref{thm.S<Np}, and let $\Psi$ be a cutoff function as in Lemma \ref{lem.Sdk}. 
Let $m\ge 3$ be an integer. Then 
\[
    \iint \abs{\nabla^2u}^2\Psi^mx_n\, dX\,dt
    \le C\br{1+\norm{\mu}_{C}+\norm{\nu}_{C}}\int_{\pom}\abs{N(\Psi^{\frac{m}{2}-1}\nabla u)}^2dx\,dt.
\]
\end{lemma}

Once we have Lemma~\ref{lem.S2cutoff}, Theorem~\ref{thm.S<Np} follows from a good-$\lambda$ inequality (Lemma \ref{lem.gdlambda} below) using the exact same argument as in \cite[Section 9]{DLP1}. We note that, in the analog of the lemma above, namely \cite[Lemma 9.3]{DLP1}, the condition \eqref{eq.Psi_bd} 
has not been required; however, it follows from the construction of the cutoff function adapted to a sawtooth domain in \cite[Definition 9.7]{DLP1}.
For these cutoffs $\|\nu\|_C$ and the constant in \eqref{eq.Psi_bd} are
uniform. At a level $\lambda_0>0$ and for $0<\gamma<1$, each boundary cube $\Delta$ in
the good-$\lambda$ argument has the following bound: the squared maximal-function norm in
Lemma \ref{lem.S2cutoff}, taken over its full boundary support, is at
most $C\gamma^2\lambda_0^2|\Delta|$. Thus the factor
$1+\|\mu\|_C+\|\nu\|_C$ changes only the constant in the good-$\lambda$
inequality and imposes no smallness condition on the coefficients.
To remove any a priori finiteness assumption, first restrict the square
functions to $\varepsilon<h<R$. Spatial difference-quotient Caccioppoli
on Whitney boxes, with $h|\nabla A|\le K$, gives
\[
 S_a^{\varepsilon,R}(\nabla u)
 \le C\big(1+\log(R/\varepsilon)\big)^{1/2}N_{a'}(\nabla u).
\]
Indeed, each dyadic height layer in a cone is covered by boundedly many
Whitney boxes, and its Hessian integral is bounded by the squared
maximal function at the vertex. Thus the truncated square function is
in $L^p$ whenever the right-hand side in Theorem~\ref{thm.S<Np} is
finite. The same good-$\lambda$ construction applies to these truncated
cones: multiply its sawtooth cutoffs by height cutoffs at
$\varepsilon$ and $R$. Their additional Carleson densities are bounded
by $C/h$ on fixed-ratio height layers, uniformly in both truncations.
Lemma~\ref{lem.S2cutoff} therefore gives the good-$\lambda$ inequality
below with uniform constants for the truncated square functions.
Integrating it against $p\nu^{p-1}d\nu$, using change of aperture and
choosing $\gamma$ small, yields
\[
 \|S_a^{\varepsilon,R}(\nabla u)\|_p
 \lesssim_p\|N_{2a}(\nabla u)\|_p.
\]
Monotone convergence as $\varepsilon\downarrow0$ and $R\uparrow\infty$
proves Theorem~\ref{thm.S<Np} without assuming its left-hand side finite.
\medskip

\begin{lemma}[Good-$\lambda$ inequality]\label{lem.gdlambda}
    For any $\nu>0$ and any $\gamma\in(0,1)$, 
 \begin{multline}
        \left|\set{(x,t)\in\R^{n-1}\times\R:\, S_a(\nabla u)(x,t)>\nu, \, N_{2a}(\nabla u)(x,t)\le\gamma\nu}\right|\\
        \le C\gamma^2\abs{\set{(x,t)\in\R^{n-1}\times\R:\, S_{2a}(\nabla u)(x,t)>\nu/2}}.
    \end{multline}
\end{lemma}

We now prove Lemmas~\ref{lem.Sdk} and \ref{lem.SH}.
We first take $\Psi$ compactly supported in $\om$, so all integrations
stay a positive distance from the boundary. Interior difference quotients
justify the differentiated equations. Mixed derivatives of coefficients
are interpreted distributionally and transferred by integration by parts;
only their first derivatives occur in the estimates. The removal of this
support restriction is given after the proof of Lemma \ref{lem.SH}.

\begin{proof}[Proof of Lemma \ref{lem.Sdk}]
    The proof has a similar idea as that of Lemma \ref{NPl1}, starting from \eqref{eq.pfSdk1}, and in particular, it might be viewed as the special case when $p=2$. But since in the current setting there is a cutoff function $\Psi$, the proof goes differently in multiple places. 
  
    For simplicity, we write $w_k=\dr_k u$. We start by writing
    \begin{multline*}
         \iint \abs{\nabla w_k}^2\Psi^mx_n\, dX\,dt\lesssim 
         \iint \frac{A}{a_{nn}}\nabla w_k\cdot\nabla w_k\Psi^mx_n\,dX\,dt\\
         =\underbrace{-\iint\frac{a_{nj}}{a_{nn}}\dr_jw_k w_k\Psi^m dX\,dt}_{:=I}
         -m\iint \frac{A}{a_{nn}}\nabla w_k\cdot\nabla \Psi\br{\Psi^{m-1}w_kx_n}dX\,dt\\
         \underbrace{-\iint \divg\br{\frac{A}{a_{nn}}\nabla w_k}w_k\Psi^mx_ndX\,dt}_{=:J}
    \end{multline*}
by ellipticity and the divergence theorem. The unlabeled term on the right-hand side can be bounded by the desired quantities using the Cauchy-Schwarz inequality and the assumption that $|\nabla\Psi|^2x_ndX\,dt$ is a Carleson measure: 
\begin{multline*}
    \abs{m\iint \frac{A}{a_{nn}}\nabla w_k\cdot\nabla \Psi\br{\Psi^{m-1}w_kx_n}dX\,dt}\\
    \le \varepsilon \iint  \abs{\nabla w_k}^2\Psi^mx_n\, dX\,dt 
    + C_\varepsilon\norm{\nu}_{C}\int\abs{N(\Psi^{\frac{m}{2}-1}w_k)}^2dx\,dt.
\end{multline*}
For the term $I$, we integrate by parts in $x_j$ to get
\begin{multline*}
\frac12\iint \dr_j\br{\frac{a_{nj}}{a_{nn}}}w_k^2\Psi^mdX\,dt+\frac{m}2\iint\br{\frac{a_{nj}}{a_{nn}}}w_k^2\dr_j\Psi\,\Psi^{m-1}dX\,dt 
+ \frac12\int_{\pom} w_k^2 \Psi^m dx\,dt\\ =:I_1+I_2 +I_3.
\end{multline*}
The term $I_2$ is controlled by $C\norm{\nu}_{C}\int\abs{N(\Psi^{\frac{m-1}{2}}w_k)}^2dx\,dt$ since $|\nabla\Psi|dX\,dt$ is a Carleson measure.
Here $I_3=0$ for the cutoffs currently under consideration. If $u$ and
$\Psi$ extend smoothly to the boundary, the displayed $I_3$ is the
boundary term from $j=n$. Since $0\le\Psi\le1$, it is bounded by
\[
 I_3\le\frac12\int_{\pom}N(\Psi^{m/2-1}\nabla u)^2\,dx\,dt.
\]
This bound has coefficient one, independently of the two Carleson norms.
The cutoff limit below handles general energy solutions without assuming
a boundary trace of $w_k$. For the term $I_1$, observe that $j<n$, and introducing the term $\dr_nx_n$, we write:
\[
I_1=\frac12\sum_{j<n}\iint\dr_j\br{\frac{a_{nj}}{a_{nn}}}w_k^2\Psi^m\dr_nx_n\,dX\,dt.
\]
We then integrate by parts in $\dr_n$ and get three terms: (1) when $\dr_n$ hits $w_k^2$, we use the Cauchy-Schwarz inequality and the assumption that $|\dr_j A|^2x_ndX\,dt$ is a Carleson measure to bound the term by allowable quantities
\(
\varepsilon\iint  \abs{\nabla w_k}^2\Psi^mx_n\, dX\,dt 
    + C_\varepsilon\norm{\mu}_{C}\int\abs{N(\Psi^{\frac{m}{2}}w_k)}^2dx\,dt\);
(2) when $\dr_n$ hits $\Psi^m$, we use the bound $ |\nabla A(X,t)|x_n\lesssim 1$ and the assumption that $|\nabla\Psi|dX\,dt$ is a Carleson measure to bound the term by $C\norm{\nu}_{C}\int\abs{N(\Psi^{\frac{m-1}{2}}w_k)}^2dx\,dt$; and (3) when $\dr_n$ hits $ \dr_j\br{\frac{a_{nj}}{a_{nn}}}$, we integrate by parts again in $x_j$. Since $j<n$, this gives
\[
\sum_{j<n}\iint\dr_n\br{\frac{a_{nj}}{a_{nn}}}\dr_j w_k\,w_k\Psi^mx_n\,dX\,dt
+\frac{m}2\sum_{j<n}\iint\dr_n\br{\frac{a_{nj}}{a_{nn}}}w_k^2\dr_j\Psi\,\Psi^{m-1}x_n\,dX\,dt.\]
Observe that both terms can be estimated in the same way as the previous two terms and are controlled by allowable quantities. 

It remains to estimate the term $J$. Since 
\(
 \divg\br{\frac{A}{a_{nn}}\nabla w_k}=\frac{1}{a_{nn}}\divg(A\nabla w_k)-\frac{\nabla a_{nn}}{a_{nn}^2}\cdot (A\nabla w_k)\),
we can use the PDE for $w_k$ and get
\begin{multline*}
    J=-\iint \frac{1}{a_{nn}}(\dr_tw_k)w_k\Psi^mx_n\,dX\,dt
    +\iint\frac{1}{a_{nn}}\divg\br{(\dr_k A)\nabla u}w_k\Psi^mx_n\,dX\,dt\\
    +\iint \frac{\nabla a_{nn}}{a_{nn}^2} (A\nabla w_k)w_k\Psi^mx_n\,dX\,dt
    =: -J_1-J_2+J_3.
\end{multline*}
Note that $J_3$ can be estimated using Cauchy-Schwarz and the Carleson condition on $\nabla A$ as before and controlled by allowable quantities. For $J_1$, integrating by parts in $t$ gives that 
\[
J_1=-\frac12\iint\dr_t\br{a_{nn}^{-1}}w_k^2\Psi^mx_n\,dX\,dt
-\frac{m}2\iint\frac{1}{a_{nn}}w_k^2(\dr_t\Psi) \Psi^{m-1}x_n=:J_{11}+J_{12}.
\]
Since $|\dr_t\Psi|x_ndX\,dt$ is a Carleson measure, we have 
\(
|J_{12}|\le C\norm{\nu}_{C}\int\abs{N(\Psi^{\frac{m-1}{2}}w_k)}^2dx\,dt\).
For $J_{11}$, we write $x_n=\frac12\dr_n(x_n^2)$ and then integrate by parts in $x_n$. This process gives three terms. When $\dr_n$ hits $w_k^2$ and $\Psi^m$, the two corresponding terms can both be bounded by allowable quantities using that $|\dr_t a_{nn}|^2x_n^3dX\,dt$ is a Carleson measure, that $|\dr_t a_{nn}|x_n^2\lesssim 1$ and that $|\dr_n\Psi|dX\,dt$ is a Carleson measure, respectively. When $\dr_n$ hits $\dr_t\br{a_{nn}^{-1}}$, we integrate by parts again in $\dr_t$ to get 
\[
-\frac12\underbrace{\iint (\dr_tw_k)w_k\dr_n(a_{nn}^{-1})\Psi^mx_n^2\,dX\,dt}_{=:J_{11*}}
-\frac{m}4\iint w_k^2\dr_n(a_{nn}^{-1})(\dr_t\Psi)\Psi^{m-1}x_n^2\,dX\,dt.
\]
Note that the second term can be controlled by an allowable quantity using the bound $|\dr_n a_{nn}|x_n\lesssim 1$ and the Carleson condition on $\dr_t\Psi$. We estimate $J_{11*}$ using Cauchy-Schwarz:
\begin{equation}\label{eq.J11*}
    |J_{11*}|\le \br{\iint |\dr_tw_k|^2\Psi^{m+2}x_n^3dX\,dt}^{1/2}
    \br{\iint|\dr_n a_{nn}|^2x_n\,w_k^2\Psi^{m-2}dX\,dt}^{1/2}.
\end{equation}
The square of the second factor is bounded by $\norm{\mu}_{C}\int_{\pom}N(\Psi^{\frac{m}2-1}\nabla u)^2dx\,dt$ using the Carleson condition on $\dr_n a_{nn}$. To estimate the first term in the product, we do a Whitney decomposition of $\supp\Psi$ and apply Lemma~\ref{lem.CacciCO} to each Whitney cube $Q$. Since the cubes in $\set{2Q}$ have bounded overlap, we deduce that 
\begin{multline}\label{eq.dtwkCO}
   \iint |\dr_tw_k|^2\Psi^{m+2}x_n^3dX\,dt\lesssim \iint |\dr_tu|^2\Psi^mx_n\,dX\,dt+
    \iint |\dr_t A|^2|\nabla u|^2\Psi^{m+2}x_n^3dX\,dt\\
    \lesssim \iint \br{|\nabla A|^2x_n+|\dr_t A|^2x_n^3}|\nabla u|^2\Psi^m dX\,dt
    +\iint |\nabla^2u|^2\Psi^mx_n\,dX\,dt\\
    \lesssim \norm{\mu}_{C}\int_{\pom}N(\Psi^{\frac{m}2}\nabla u)^2dx\,dt +\iint |\nabla^2u|^2\Psi^mx_n\,dX\,dt
\end{multline}
where we have used $|\dr_t u|^2\lesssim |\nabla A|^2|\nabla u|^2+|\nabla^2u|^2$ by the PDE satisfied by $u$ in the second inequality. Plugging these estimates into \eqref{eq.J11*} gives 
\[
|J_{11*}|\le \varepsilon \iint |\nabla^2u|^2\Psi^mx_n\,dX\,dt
+ C_{\varepsilon}\norm{\mu}_{C}\int_{\pom}N(\Psi^{\frac{m}2-1}\nabla u)^2dx\,dt.
\]
We are left with the term $J_2$ from $J$. By the divergence theorem, 
\begin{multline*}
    J_2=\iint (\dr_k A)\nabla u\cdot\nabla(a_{nn}^{-1}) w_k\Psi^mx_n\,dX\,dt
    +\iint \frac{\dr_kA}{a_{nn}}\nabla u\cdot\nabla w_k\br{\Psi^m x_n}dX\,dt\\
    +m\iint \frac{\dr_kA}{a_{nn}}\nabla u\cdot\nabla \Psi\br{w_k\Psi^{m-1} x_n}dX\,dt
    +\iint \frac{\dr_ka_{nj}}{a_{nn}}\dr_j u \br{w_k\Psi^m}dX\,dt=: \sum_{i=1}^4J_{2i}.
\end{multline*}
We have 
\begin{equation*}
    |J_{21}|\lesssim\iint |\nabla A|^2|\nabla u|^2\Psi^mx_n\,dX\,dt
    \le C\norm{\mu}_{C}\int_{\pom}N\br{\Psi^{m/2}\nabla u}^2dx\,dt.
\end{equation*}
The terms $J_{22}$ and $J_{23}$ can be estimated in the same way as the first two terms coming from $I_1$. 

For the term $J_{24}$, we introduce $1=\dr_n(x_n)$ and then integrate by parts in $x_n$. This process gives 4 terms, when $\dr_n$ hits $ \frac{\dr_ka_{nj}}{a_{nn}}$,  $\dr_j u$, $w_k$, and $\Psi^m$, respectively. We only explain the situation when  $\dr_n$ hits $ \frac{\dr_ka_{nj}}{a_{nn}}$, as the other 3 terms are similar to the terms treated before. In this case, we have
\begin{multline*}
    \iint \dr_n\br{\frac{\dr_ka_{nj}}{a_{nn}}}\dr_j u \br{w_k\Psi^mx_n}dX\,dt
    = \iint \dr_n(a_{nn}^{-1})\dr_ka_{nj}\dr_j u \br{w_k\Psi^mx_n}dX\,dt\\
    + \iint a_{nn}^{-1}(\dr_n\dr_ka_{nj})\dr_j u \br{w_k\Psi^mx_n}dX\,dt.
\end{multline*}
The first term on the right-hand side can be estimated as $J_{21}$, while for the second term, we integrate by parts again in $x_k$. Since $k<n$, this process gives 4 terms, each of which can be handled and controlled by the allowable quantities as in the previous step. This completes the proof of the lemma.
\end{proof}
\medskip

To prove Lemma~\ref{lem.SH}, we use the PDE \eqref{eq.Hpde2} satisfied by $H$. Recall that matrix $B$ is defined as $B=A^T/a_{nn}$.
\begin{proof}[Proof of Lemma \ref{lem.SH}]
    We start by writing 
    \begin{multline*}
        \iint |\nabla H|^2\Psi^mx_n\,dX\,dt\lesssim \iint B\nabla H\cdot\nabla H\br{\Psi^m x_n}dX\,dt\\  
        =-\iint H B\nabla H\cdot \nabla\br{\Psi^m x_n}dX\,dt
        \underbrace{-\iint \divg(B\nabla H)H\Psi^mx_n\,dX\,dt}_{=:S}.
    \end{multline*}
    The first term on the right-hand side can be estimated as in the proof of Lemma~\ref{lem.Sdk}; in fact, the argument is slightly simpler here since 
 $b_{nn}=1$. More precisely, we write by the product rule
 \begin{multline*}
      -\iint H B\nabla H\cdot \nabla\br{\Psi^m x_n}dX\,dt\\
      = -m\iint H B\nabla H\cdot \nabla\Psi\br{\Psi^{m-1} x_n}dX\,dt 
       \underbrace{-\frac12\iint b_{nj}\dr_j\br{H^2}\Psi^mdX\,dt}_{=:I} 
      .
 \end{multline*}
Observe that the first term on the right-hand side is bounded by 
\(\varepsilon\iint |\nabla H|^2\Psi^mx_n\,dX\,dt+C_{\varepsilon}\|\nu\|_C\int_{\pom}N(\Psi^{m/2-1}H)^2\), which is allowable because $|H|\lesssim|\nabla u|$. The term $I$ can be estimated in the same way as the term $I$ in the proof of Lemma~\ref{lem.Sdk}. Namely, we integrate by parts in $x_j$, and when $\dr_j$ hits $b_{nj}$, we introduce $1=\dr_nx_n$ and then integrate by parts in $x_n$. For the term that involves $\dr_n\dr_j b_{nj}$, we use the fact that $j<n$ (because $b_{nn}=1$) to integrate by parts again in $x_j$. When $j=n$ we again get a 
boundary term $\frac12\int_{\pom}H^2\Psi^m$. For smooth boundary values,
it is bounded by $C\int_{\pom}N(\Psi^{m/2-1}\nabla u)^2$, with coefficient
independent of $\mu$ and $\nu$, because $|H|\lesssim|\nabla u|$.
It vanishes for the compact interior cutoffs used in the present calculation.

We now estimate the term $S$ using the PDE \eqref{eq.Hpde2}.
\[
S=-\iint \frac{1}{a_{nn}}(\dr_t H) H\Psi^mx_n\,dX\,dt-\iint TH\Psi^mx_n\,dX\,dt=:-S_1-S_2,
\]
where $T$ is defined as in \eqref{eq.Hpde2}.
The term $S_1$ is similar to $J_1$ in the proof of Lemma~\ref{lem.Sdk}. The difference is in the term 
\[
S_{11*}:=\iint(\dr_t H)H\dr_n(a_{nn}^{-1})\Psi^mx_n^2dX\,dt,
\]
which corresponds to $J_{11*}$ in the proof of Lemma~\ref{lem.Sdk}. As in \eqref{eq.J11*}, we need to estimate $\iint|\dr_t H|^2\Psi^{m+2}x_n^3dX\,dt$. Since $H=\sum_{1\le j\le n}a_{nj}\dr_j u$, the product rule gives that 
$|\dr_t H|\lesssim |\dr_t A||\nabla u|+|\dr_t\nabla u|$. So
\[
\iint|\dr_t H|^2\Psi^{m+2}x_n^3dX\,dt\lesssim\iint|\dr_t A|^2|\nabla u|^2\Psi^{m+2}x_n^3dX\,dt+ \iint  |\dr_t\nabla u|^2\Psi^{m+2}x_n^3dX\,dt,
\]
where the first term is further bounded by $C\|\mu\|_C\int_{\pom}N\br{\Psi^{m/2}\nabla u}^2dx\,dt$, and the second term can be estimated exactly as \eqref{eq.dtwkCO}. Hence, we conclude that $S_1$ can be controlled by allowable quantities. 

It remains to estimate $S_2$, which we further split into 3 terms: we write $S_2=S_{21}+S_{22}+S_{23}$, where 
 \[
S_{21}:=\iint (\partial_i b_{in})\partial_tu\,H\Psi^mx_ndX\,dt,
\]
\[
S_{22}:=\iint \left[(\partial_ta_{nn}^{-1})H-(\partial_tb_{in})\partial_iu\right]H\Psi^mx_ndX\,dt,
\]
and 
\[
S_{23}:=\iint \sum_{j<n} \left[\partial_i(b_{ij}\partial_j(a_{nk} \partial_k u))-
\partial_k(b_{kn} \partial_j(a_{ji} \partial_i u))\right] H\Psi^mx_ndX\,dt.
\]
For $S_{21}$, we use the PDE to get that  $|\dr_t u|\lesssim |\nabla A||\nabla u|+|\nabla^2u|$, and hence 
\begin{multline*}
    |S_{21}|\lesssim \iint |\nabla A|^2|\nabla u||H|\Psi^mx_ndX\,dt+\iint |\nabla A||\nabla^2 u||H|\Psi^mx_n\\
    \lesssim (1+C_\varepsilon) \norm{\mu}_{C}\int_{\pom}N\br{\Psi^{m/2}|\nabla u|}^2dx\,dt + \varepsilon \iint |\nabla^2 u|^2\Psi^m x_n.
\end{multline*}
To estimate $S_{22}$, we write $x_n=\frac12\dr_n\br{x_n^2}$ and integrate by parts in $x_n$. This yields 3 terms. When $\dr_n$ hits $H$ or $\Psi^m$, we can control the corresponding terms by allowable quantities using the Carleson condition on $|\dr_t A|$ (and also on $\dr_n\Psi$ when $\dr_n$ hits  $\Psi^m$) and H\"older inequality. When $\dr_n$ hits the quantity in the bracket $[\quad ]$, we essentially need to treat differently (when $\dr_n$ hits on the $H$ or $\dr_i u$ in the bracket, it yields terms similar as before)
\[
\iint \dr_n\dr_t(a_{nn}^{-1})\, H^2\Psi^m x_n^2dX\,dt -\iint (\dr_n\dr_t b_{in}) \dr_i u\, H\Psi^m x_n^2dX\,dt. 
\]
We integrate by parts again in $t$ for each integral, which yields 4 terms. All of these terms are similar to the terms that we have treated already, namely,  $S_{11*}$, and $J_{11*}$ as well as $J_{23}$ in the proof of Lemma~\ref{lem.Sdk}. So $|S_{22}|$ can be bounded by allowable quantities. 
\smallskip

It remains to estimate $S_{23}$. Put $q_j=a_{ji}\partial_i u$ and define,
with repeated indices $i,k$ summed from $1$ to $n$,
\[
 F_j=\partial_i(b_{ij}a_{nk})\partial_k u
       -\partial_k(b_{kn}a_{ji})\partial_i u,\qquad j<n,
\]
\[
 G_i=\sum_{j<n}\big[-(\partial_jb_{ij})H
                         +(\partial_jb_{in})q_j\big],\qquad 1\le i\le n.
\]
Since $b_{ij}a_{nk}=b_{kn}a_{ji}$, the spatial part of $T$ in
\eqref{eq.Hpde2} has the exact divergence form
\[
 T_{\mathrm{sp}}=\sum_{j<n}\partial_jF_j+\partial_iG_i,
 \qquad |F|+|G|\lesssim|\nabla A||\nabla u|.
\]
Indeed, commute $\partial_j$ past the coefficients in the two terms of
$T_{\mathrm{sp}}$. Their difference inside the resulting tangential
derivative is $F_j$, because the second derivatives of $u$ cancel by
the displayed coefficient identity. The two commutators give $G$.
Integrating this formula against $H\Psi^mx_n$ gives terms bounded by
\begin{multline*}
 \varepsilon\iint|\nabla H|^2\Psi^mx_n
 +C_\varepsilon\|\mu\|_C
       \int_{\pom}N(\Psi^{m/2}\nabla u)^2\\
 +C\|\nu\|_C\int_{\pom}N(\Psi^{(m-1)/2}\nabla u)^2,
\end{multline*}
except for $-\iint G_nH\Psi^m$, where the derivative hits $x_n$.
Here $b_{nn}=1$ gives
\[
 -G_nH=\sum_{j<n}(\partial_jb_{nj})H^2.
\]
This is of the same form as $I_1$ in Lemma \ref{lem.Sdk}: insert
$1=\partial_nx_n$, integrate in $x_n$, and transfer any resulting mixed
coefficient derivative to $H^2\Psi^m$ in the tangential variable $x_j$.
The same bounds apply. Absorb the terms with
$\varepsilon\iint|\nabla H|^2\Psi^mx_n$ into the left-hand side.

This completes the proof of Lemma~\ref{lem.SH}.
\end{proof}

\begin{proof}[Completion of the cutoff estimates]
For a cutoff compactly supported in $\om$, sum Lemma \ref{lem.Sdk} over
$k<n$, add Lemma \ref{lem.SH}, and use \eqref{eq.grad2u_split}.
The last term in that estimate is bounded by
$C\|\mu\|_C\|N(\Psi^{m/2-1}\nabla u)\|_2^2$, since $0\le\Psi\le1$.
Choose the small parameters in the two lemmas so the total coefficient
of $\iint|\nabla^2u|^2\Psi^mx_n$ on the right is at most $1/2$.
This proves Lemma \ref{lem.S2cutoff} for these cutoffs.

For a general $\Psi$, choose smooth functions
$0\le\eta_\delta(x_n),\chi_R(X,t)\le1$
such that $\eta_\delta=0$ on $x_n<\delta$, $\eta_\delta=1$ on
$x_n>2\delta$, and $|\eta_\delta'|\lesssim\delta^{-1}$.
Take $\chi_R=1$ on $|X|<R$, $|t|<R^2$, supported where
$|X|<2R$, $|t|<4R^2$, with
$|\nabla\chi_R|\lesssim R^{-1}$ and $|\partial_t\chi_R|\lesssim R^{-2}$.
Set $\Psi_{\delta,R}=\Psi\eta_\delta\chi_R$.
For a cutoff $\zeta$, denote its measure in Lemma \ref{lem.Sdk} by
$\nu_\zeta$. The product rule gives
\begin{equation}\label{eq:cutoffExhaustion}
 \|\nu_{\Psi_{\delta,R}}\|_C\le C(\|\nu_\Psi\|_C+1),\qquad
 N(\Psi_{\delta,R}^{m/2-1}\nabla u)
 \le N(\Psi^{m/2-1}\nabla u).
\end{equation}
To verify the uniform Carleson bound, the density for $\eta_\delta$ is
at most $C\delta^{-1}\boldsymbol1_{\{\delta<x_n<2\delta\}}$.
Its integral on any boundary box of side $s$ is at most $Cs^{n+1}$.
The density for $\chi_R$ is at most $C/R$ on a region of volume
$CR^{n+2}$, so its mass on such a box is at most
$C\min\{s^{n+2}/R,R^{n+1}\}\le Cs^{n+1}$.
The implicit constant in \eqref{eq.Psi_bd} is uniform as well.

Apply the compact-cutoff estimate to $\Psi_{\delta,R}$ and let
$\delta\downarrow0$, $R\to\infty$. Fatou's lemma and
\eqref{eq:cutoffExhaustion} prove Lemma \ref{lem.S2cutoff} as stated.
If its right-hand side is finite, the Hessian integral is now finite,
so the same limit in the two component estimates proves Lemmas
\ref{lem.Sdk} and \ref{lem.SH}, by dominated convergence on their Hessian
terms and Fatou on their left-hand sides. If the right-hand side is
infinite the assertions are immediate. This also explains the unit term
in all three estimates without requiring pointwise boundary traces of
$\nabla u$ or $H$.
\end{proof}

\section{Localization on bounded Lipschitz cylinders}\label{S.bddLip}

In this section we introduce a localization technique that will allow us to establish solvability of the Neumann problem on Lipschitz cylinders of the form $\Omega=\mathcal O\times\R$ for a bounded Lipschitz domain
$\mathcal O\subset\R^n$ (with sufficiently small Lipschitz norm) using previously  established solvability on $\R^n_+\times\R$. This proves Theorem \ref{MainT} for bounded Lipschitz cylinders. 
For unbounded Lipschitz graph domains, the result already follows by the change of variables in Lemmas~\ref{lem:graphflat} and~\ref{lem:graphtransfer}.

In what follows, we follow the localization argument in Section 11 of \cite{DLP1} where a similar decomposition was considered as well as \cite{DL}. Recalling the definition of the Lipschitz domain
$\mathcal O\subset\R^n$ from Definition \ref{DefLipDomain} we restrict ourselves to the case
 \( \diam(\mathcal{O}) < \infty \).     \vglue1mm

Fix $p\in(1,\infty)$. We choose small bounds $\ell(p)\le1$ and
$\mu(p)>0$, and a parameter $L\ge1$, as specified below. We assume
$\ell\le\ell(p)$ and $\|\mu\|_C\le\mu(p)$. Write
$\mathbb B_s(P)$ for a spatial Euclidean ball. Fix a geometric
constant $c_*$ such that a working cylinder of diameter $d$, centred
at $P_j\in\partial\mathcal O$, satisfies
$4\Z_j\subset\mathbb B_{c_*d}(P_j)$ when $\ell\le1$.
Before rescaling, refine the finite boundary cover so that
$(1/2)\Z_j$ cover $\partial\mathcal O$ and each working cylinder has
an original coordinate chart available at the larger scale
$16c_*e^L d$. This is possible by compactness, by decreasing the common
working diameter $d$ and increasing the number of cylinders.
The graph functions are defined on tangential balls of that larger
radius, and their epigraphs agree with $\mathcal O$ in the corresponding
spatial balls. Their Lipschitz constants remain at most $\ell$.
 
 By rescaling the PDE, it may also be assumed that the $\ell$-cylinders $\Z_j$ in the definition above have diameter $d=1$. Let $N,\, C_0$ be as above and hence 
 there are $N$ such $\ell$-cylinders $(1/2)\Z_j$ needed to cover the boundary $\partial\mathcal O$.
 It follows that there exists a partition of unity $(\varphi_j)_{j=1}^N$  such that each  $\varphi_j\in 
 C_0^\infty(\R^n)$, is supported in
$(1/2)\Z_j$, and altogether, 
$$\sum_{j=1}^N \varphi_j(x)=1,\qquad \mbox{ for all }\qquad x\in 
\left\{
x\in\mathcal O:\,  \mathrm{dist}\left( x,{\partial\mathcal O}
\right) \leq \frac{1}{2}\right\}.$$
As each $\varphi_j$ is smooth, and there are $N$ of them, it follows that $\|\nabla \varphi_j\|_{L^\infty}\le M$ for some $M>0$ for all $j$. This completes 
the description of the  partition of the spatial component of our domain $\Omega$. \medskip

Fix a scale $r_0\in (0,1/2]$ to be specified later, let \( \Delta_i  = \partial\Base \times (ir_0^2,(i+1)r_0^2] \) for $i\in\Z$. There exists a nonnegative smooth function $\psi\in C_c^\infty(\R)$, supp $\psi\subset [0,3r_0^2]$ such that for 
$$\psi_i(\cdot)=\psi(\cdot-ir_0^2), \qquad |\psi'|\lesssim r_0^{-2},\qquad \sum_{i\in \Z} \psi_i \equiv 1.$$\medskip

We first take $g$ smooth and compactly supported in time. We do not
require its integral to vanish. Put $g_{ij}=g\psi_i\varphi_j$; only
finitely many of these functions are nonzero. We construct the
corresponding solutions before using the decomposition.

Let $h$ be a smooth datum supported in a bounded time interval
$(a,b)$, and put
\[
 S=\sigma(\partial\mathcal O),\qquad
 I_h=\int_{\partial\Omega}h,\qquad
 F_h(t)=\int_{\partial\mathcal O}h(P,t)\,d\sigma(P).
\]
Fix $\theta\in C_0^\infty(0,1)$ with $\theta\ge0$ and
$\int\theta=1$. For $R\ge1$, set $T_R=b+R$ and
\begin{equation}\label{eq:compensatingPulse}
 h_R(P,t)=h(P,t)-\frac{I_h}{SR}
                       \theta\!\left(\frac{t-T_R}{R}\right).
\end{equation}
Then $\int_{\partial\Omega}h_R=0$, and
\[
 \|h_R-h\|_{L^p(\partial\Omega)}
 =|I_h|S^{1/p-1}R^{1/p-1}\|\theta\|_{L^p(\R)}\longrightarrow0.
\]
Each $h_R$ belongs to the energy-dual trace space. Indeed, writing
$F_R(t)=\int_{\partial\mathcal O}h_R(P,t)\,d\sigma(P)$ and
$m_\zeta(t)=\fint_{\mathcal O}\zeta(X,t)\,dX$, spatial trace and
Poincar\'e inequalities give
\[
 |\langle h_R,\operatorname{Tr}\zeta\rangle|
 \le C\|h_R\|_2\|\nabla\zeta\|_2
       +|\mathcal O|^{-1/2}\|D_t^{-1/2}F_R\|_2
                              \|D_t^{1/2}\zeta\|_2.
\]
The last norm of $F_R$ is finite: $F_R$ is smooth and compactly
supported with integral zero, so $\widehat F_R(\xi)=O(|\xi|)$
near zero. Section~\ref{Def} therefore supplies an energy solution
$V_R$ with datum $h_R$.

We normalize $V_R$ to vanish for $t<a$. To justify this normalization,
its spatial mean is constant on $(-\infty,a)$, by the weak equation.
Subtract that constant. Spatial Poincar\'e and the global gradient
energy give a sequence $s_k\to-\infty$ with
$\|V_R(\cdot,s_k)\|_2\to0$. The homogeneous Neumann energy
inequality from $s_k$ to any $t<a$ implies $V_R(\cdot,t)=0$.
For $R,R'\ge1$, the difference $V_R-V_{R'}$ has zero initial
value and zero Neumann data before $\min(T_R,T_{R'})$.
The same energy inequality yields
\begin{equation}\label{eq:pulseCausality}
 V_R=V_{R'}\quad\hbox{on }
 \mathcal O\times(-\infty,\min(T_R,T_{R'})).
\end{equation}
Thus these solutions define a single function $v[h]$ on every
bounded time interval, by taking $R$ with $T_R$ beyond that interval.
It belongs to $L^2_{\loc}(\R;W^{1,2}(\mathcal O))$, solves the
Neumann problem with datum $h$, and vanishes for $t<a$. The energy
inequality also proves uniqueness in this causal local-energy class
and hence linearity of $h\mapsto v[h]$.
Its mean satisfies
\begin{equation}\label{eq:causalMean}
 \fint_{\mathcal O}v[h](X,t)\,dX
   =\frac1{|\mathcal O|}\int_{-\infty}^tF_h(s)\,ds.
\end{equation}
We require no uniform global energy bound for $V_R$. In particular,
$v[h]$ need not be a global energy solution when $I_h\ne0$.

For each $i,j$ define $v_{ij}=v[g_{ij}]$. Then
\begin{equation}\label{defvij}
\begin{cases}
    \LL v_{ij}=0&\text{in }\Omega,\\
    \partial_\nu^A v_{ij}=g_{ij}&\text{on }\partial\Omega,\\
    v_{ij}=0&\text{for }t<ir_0^2.
\end{cases}
\end{equation}
The finite sum $u=\sum_{i,j}v_{ij}$ is the causal local-energy
solution with datum $g$. If $\int_{\partial\Omega}g=0$, the
energy solution with datum $g$ exists by the preceding trace estimate;
after its zero-past normalization it equals this sum by uniqueness
on every finite time interval. We will pass to general data after
proving the uniform estimate.

Observe that the boundary data of $v_{ij}$ is supported on $(\partial\Base\cap (1/2)\Z_j) \times (ir_0^2,(i+3)r_0^2]$ and hence in $\Delta_i\cup \Delta_{i+1}\cup \Delta_{i+2}$. Furthermore we have that
$$\|g_{ij}\|_{L^{p}(\partial\Omega)}\le C\|g\|_{L^{p}(\Delta_i\cup \Delta_{i+1}\cup \Delta_{i+2})}.$$

\vglue1mm

Let us fix $i,j$ and consider $v=v_{ij}$, where for convenience we drop the indices $i,j$.
Without loss of generality (by shifting the time variable) and relabeling the $\ell$-cylinders $\Z_j$ we may assume that $\partial_\nu^A v$ is supported in $(\partial\Base\cap (1/2)\Z_1) \times (-6r_0^2,-3r_0^2)$ and let
us call $h$ the datum on this set.

We first state a key theorem from \cite{DLP2}:

\begin{theorem}\cite[Theorem 1.4]{DLP2}\label{thm.NtoLoc}
Let $\om=\mathcal O\times\R$ where $\mathcal O$ is either a bounded or unbounded Lipschitz domain in $\Rn$. Let $\mathcal L=-\dr_t +\divg (A\nabla \cdot)$, and let $\mathcal L^*=\dr_t+\divg(A^T\nabla\cdot)$ be the adjoint operator of $\mathcal L$.
    Let $p\in (1,\infty)$. Suppose that $(N_{\mathcal L})_p$ and $(D_{\mathcal L^*})_{p'}$ are solvable in $\om$. Then for any backward parabolic cube $J_r=J_r(x,t):=Q_r(x)\times (t-r^2,t)$ centered at some $(x,t)\in\pom$, and any weak local solution in $J_{2r}\cap \Omega$ to $\mathcal Lu=0$   with zero Neumann data on $J_{2r}\cap\pom$, we have
\begin{equation} \label{locEST}
\|\wt N(|\nabla u|\1_{J_{r}})\|_{L^p(\partial \Omega)} \leq C r^{(n+1)/p}\fiint_{J_{2r}\cap \om} |\nabla u| dX\,dt.
\end{equation}
\end{theorem}

Here we have stated the theorem in its $L^1$ version instead of the $L^2$ version in the paper, but one follows from the other using boundary reverse H\"older's inequality for the gradient which in turn follows from the boundary Caccioppoli inequality.
The theorem above is useful in that it applies to local solutions with zero Neumann data. This allows us to verify the assumptions of this result (solvability) on a different domain which is unbounded and then apply it to our case of a bounded domain. To explain this more fully, consider a parabolic ball that sits inside one of the coordinate patches of $\partial\mathcal O\times\R$.
In the normalized coordinates of a working cylinder $\Z$, put
$r_*=c_*$ and $R_*=e^Lr_*$. Its centre is a boundary point, which
we take to be the origin. The reserved chart gives a Lipschitz graph
$\phi$ on $\{|x|<16R_*\}$, with $\phi(0)=0$ and Lipschitz constant
at most $\ell$, whose epigraph agrees with $\mathcal O$ in
$\mathbb B_{16R_*}(0)$.
Choose a spatial cutoff of the graph function so that
\[
 \tilde\phi=\phi\quad\text{on }\{|x|\le8R_*\},\qquad
 \tilde\phi=0\quad\text{on }\{|x|\ge16R_*\},\qquad
 \|\nabla\tilde\phi\|_\infty\le C\ell.
\]
The last bound follows from $|\phi(x)|\le\ell|x|$ and is independent
of $R_*$. Thus the graph domain
\[
 \tilde\Omega=\tilde{\mathcal O}\times\R,
 \qquad \tilde{\mathcal O}=\{(x,x_n):x_n>\tilde\phi(x)\},
\]
agrees with $\Omega$ in $\mathbb B_{8R_*}(0)\times\R$.

We extend the coefficients within this common region. Choose
$\chi\in C^\infty(\R)$ with $0\le\chi\le1$, $\chi=1$ on
$(-\infty,0]$ and $\chi=0$ on $[1,\infty)$, and put
\[
 \eta(X)=\chi\left(\frac{\log(|X|/r_*)}{L}\right),\qquad
 L=\log(R_*/r_*),
\]
with $\eta=1$ near the origin. Define
\[
 \tilde A(X,t)=
 \begin{cases}
  I+\eta(X)(A(X,t)-I),&X\in\tilde{\mathcal O}\cap\mathbb B_{R_*}(0),\\
  I,&X\in\tilde{\mathcal O}\setminus\mathbb B_{R_*}(0).
 \end{cases}
\]
Only values of $A$ in the original chart are used. The extension is
uniformly elliptic, with fixed bounds also containing those of $I$.
Since $4\Z\subset\mathbb B_{r_*}(0)$, it satisfies $\tilde A=A$ on
$[4\Z\cap\mathcal O]\times\R$. In particular, this agreement holds
for all times.

Let $\tilde\delta$ denote distance to $\partial\tilde\Omega$, and let
$\mu_{\tilde A}$ be the measure in \eqref{E:1:carl} for this domain
and matrix. Write
$M_A=\|A-I\|_{L^\infty((\mathbb B_{R_*}(0)\cap\mathcal O)\times\R)}$.
We claim that
\begin{equation}\label{eq:localExtensionCarleson}
 \|\mu_{\tilde A}\|_C
 \le C_1\|\mu\|_C+\frac{C_2M_A^2}{L^2},
\end{equation}
where the constants are independent of $L$, the working scale and
the chosen patch.

Indeed, $\partial_t\eta=0$, $|\nabla\eta|\lesssim(L|X|)^{-1}$ and
$\tilde\delta(X)\le|X|$. The latter inequality also shows that $|X|$
is comparable throughout each Whitney ball. Hence
\[
 \sup_{B_{\tilde\delta(X)/2}(X,t)}
        \tilde\delta(Y)|\nabla\eta(Y)|^2
 \lesssim\frac{1}{L^2|X|}.
\]
For $n\ge2$, the measure $|X|^{-1}\,dX\,dt$ is Carleson on
$\tilde\Omega$. For a boundary parabolic ball of radius $s$ far from
the spatial origin, this follows from $|X|\gtrsim s$ there. Otherwise
its spatial projection lies in $\mathbb B_{Cs}(0)$, and
$\int_{\mathbb B_{Cs}(0)}|X|^{-1}\,dX\lesssim s^{n-1}$;
the time interval has length at most $Cs^2$.
This bounds the cutoff contribution by $C_2M_A^2/L^2$.

The inherited coefficient contribution is bounded by the original
Carleson measure. Any Whitney ball whose spatial projection meets
$\mathbb B_{R_*}(0)$ has its spatial centre in $\mathbb B_{2R_*}(0)$
and is contained in
$\mathbb B_{3R_*}(0)\times\R$. At these centres the original and
extended boundary distances agree, since a nearest boundary point
lies in their common region $\mathbb B_{4R_*}(0)$.
For boundary balls of radius $s\le R_*$ meeting this support, the
boundary centre belongs to the original patch. For $s>R_*$, the
supported density is contained in an original parabolic ball of
radius $Cs$ centred spatially at the origin, with the same time
centre. Thus the original Carleson estimate applies with a fixed
factor in both cases. The identities
\[
 \nabla\tilde A=\eta\nabla A+(A-I)\nabla\eta,
 \qquad \partial_t\tilde A=\eta\partial_tA
\]
and the product inequality prove \eqref{eq:localExtensionCarleson}.

We now specify the choice of $L$ used in refining the cover. Since
$M_A\le\Lambda+|I|$, choose $L\ge1$, uniformly for all patches, so that
\[
 \frac{C_2(\Lambda+|I|)^2}{L^2}\le C_1\mu(p).
\]
Then $\|\mu\|_C\le\mu(p)$ implies
$\|\mu_{\tilde A}\|_C\le2C_1\mu(p)$.
Choose $\mu(p)$ and $\ell(p)$ small enough to meet the half-space
solvability thresholds after the graph change of variables below.
These choices depend only on $n,p,\lambda,\Lambda$ and are made
before fixing the refined cover and its partition of unity. The
smaller temporal scale $r_0$ is chosen afterwards.\medskip

Choose $J_{2r}$ inside the common coordinate patch
$[4\Z\cap\mathcal O]\times\R$. A solution $v$ on
$J_{2r}\cap\Omega$ then solves the extended equation on the same
local set $J_{2r}\cap\tilde\Omega$; no extension of $v$ is needed.
To apply Theorem~\ref{thm.NtoLoc}, we require global solvability on
$\tilde\Omega$.
This is an unbounded domain above a graph of a Lipschitz function and hence the solvability of the Neumann problem $(N_{\mathcal L})_p$ and the Dirichlet problem of its adjoint $(D_{\mathcal L^*})_{p'}$ in $\tilde\Omega$ can be deduced from the solvability
of these two problems on the domain $\R^n_+\times\R$. Apply Lemmas \ref{lem:graphflat} and \ref{lem:graphtransfer} to
$\tilde\phi$ and $\tilde A$. They give a smooth interior bi-Lipschitz map
$\Phi:\R^n_+\to\tilde\OO$ with constant determinant $c_0$.
For a global solution $u_0$ on $\tilde\Omega$, put
$\rho(z,t)=(\Phi(z),t)$ and
\[
 U=u_0\circ\rho,\qquad
 -\partial_tU+\divg(B\nabla U)=0,\qquad
 B=(D\Phi)^{-1}(\tilde A\circ\rho)(D\Phi)^{-T}.
\]
The constant Jacobian permits division of the transformed weak equation
by $c_0$. By \eqref{eq:graphCM}, the Carleson norm of $B$ is small when
both $\ell(p)$ and $\mu(p)$ are small. Neumann data transform by
\eqref{eq:graphg}, and the adjoint Dirichlet problem transfers with the
same exponents.\medskip

It follows that the PDE for $U$ is of the type considered in Theorem \ref{MainT} on $\R^n_+\times\R$
for which the solvability of the Neumann problem has been established.
 Applying Theorem \ref{MainT} for $\Rn_+\times\R$ to the operator $\LL_1$, we have that 
 for any $1<p<\infty$, 
the $L^p$ Neumann problem for the operator $\LL_1$ is solvable on $\R^n_+\times\R$ if the Carleson norm of the matrix $B$ is sufficiently small (i.e., if $\ell(p)$ and $\mu(p)$ are sufficiently small). Specifically, for some $C=C(A,\lambda,\Lambda,\ell, n,p)>0$,
$$\|\tilde N(\nabla U)\|_{L^p(\partial(\R^n_+\times\R))}\le C\|\wt g\|_{{L}^p(\partial(\R^n_+\times\R))},$$
for Neumann datum $\wt g$ of $U$. 

As $U=u_0\circ\rho$ and $\rho$ is bi-Lipschitz, the same estimate holds for $u_0$ on $\tilde\Omega$. Thus $(N_{\tilde{\mathcal L}})_p$ is solvable on $\tilde{\Omega}$.
Notice that the assumption that $(D_{\tilde{\mathcal L}^*})_{p'}$ is solvable is satisfied because under the above conditions on the matrix $A$ the small Carleson case was fully resolved in \cite{DDH,DH18} and thus we may apply the result of Theorem \ref{thm.NtoLoc} for $v$ with zero Neumann data on $J_{2r}\cap \tilde{\Omega}$. But on  
$J_{2r}\cap \tilde{\Omega}=J_{2r}\cap {\Omega}$ we have $A=\tilde{A}$ and hence \eqref{locEST} does apply to
our $v$ on $J_{2r}\cap \Omega$. We also note that because of the interior regularity of coefficients $(|\nabla A|\lesssim \delta(X)^{-1})$ we can here and below replace the averaged version of the nontangential maximal function $\tilde{N}$ by its stronger sup-version $N$. We do that from now onwards. \medskip

By construction $v=0$ for $t\le-6r_0^2$. Let
\[
 J_j=(1/2)\Z_j\times(-3,0),\qquad
 2J_j=\Z_j\times(-6,0),\qquad j=1,\ldots,N.
\]
We apply the local theorem on these sets using fixed finite coverings
by backward parabolic cubes with their enlargements in $2J_j$.
The datum $h$ can meet several of these patches. For each $j$,
choose a spatial cutoff $\chi_j$ equal to one on $\Z_j$ and
supported in $2\Z_j$. Transfer $h_j=\chi_jh$ to the boundary
of the corresponding graph domain $\tilde\Omega_j$, extending it
by zero outside the common patch. Let $w_j$ be its graph-domain
energy solution. These smooth compact data are admissible for the
graph-domain energy construction, and the already established
solvability there gives
\begin{equation}\label{solvb-w}
 \|N(\nabla w_j)\|_{L^p(\partial\tilde\Omega_j)}
 \le C\|h_j\|_{L^p(\partial\tilde\Omega_j)}
 \le C\|h\|_{L^p(\partial\Omega)}.
\end{equation}
On $2J_j\cap\Omega$, the function $v-w_j$ solves the common
PDE with zero Neumann data. Theorem~\ref{thm.NtoLoc} therefore gives
\begin{equation}\label{locEST2}
 \|N(|\nabla(v-w_j)|\1_{J_j})\|_{L^p(\partial\Omega)}
 \le C_p\fiint_{2J_j\cap\Omega}|\nabla(v-w_j)|,
 \qquad j=1,\ldots,N.
\end{equation}
Here, as above, the maximal functions of the localized fields are
compared on the common boundary patches, with fixed aperture changes.
The local volume integral of $|\nabla w_j|$ is bounded by
$C\|h\|_p$ using \eqref{solvb-w}, integration along the graph
coordinates, and H\"older on the bounded boundary boxes. Hence
\begin{equation}\label{locEST3}
 \|N(|\nabla v|\1_{J_j})\|_{L^p(\partial\Omega)}
 \le C\|h\|_{L^p(\partial\Omega)}
       +C\iint_{2J_j\cap\Omega}|\nabla v|.
\end{equation}
The local theorem applies for the fixed exponent $p$: its graph-domain
solvability hypotheses were ensured when choosing $\ell(p)$ and
$\mu(p)$ above.

Summing over the finite cover and using interior estimates on the
remaining compact part of $\mathcal O$, with backward cylinders
ending at or before zero, gives
\begin{equation}\label{eq:localFiniteN}
 \|N((\nabla v)\1_{\mathcal O\times(-\infty,0)})\|_{L^p(\partial\Omega)}
 \le C\|h\|_{L^p(\partial\Omega)}
       +C\iint_{\mathcal O\times(-6r_0^2,0)}|\nabla v|.
\end{equation}
The interior estimates here are the first-gradient local estimates
used in passing from $\tilde N$ to $N$ above, followed by the
interior gradient $L^2$-to-$L^1$ estimate proved below.
All fixed-volume factors are included in $C$.
The right-hand side is finite because $v$ has finite local energy
on this bounded time interval. Thus the maximal norm is finite
before either absorption below is performed. The construction by
\eqref{eq:pulseCausality} lets us use one energy solution $V_R$ on
all the time intervals in this calculation. No estimate for its
maximal function after the compensating pulse starts is used.
\medskip

Let $$\mathcal O_\varepsilon=\mathcal O\cap \{X:\, d(X,\partial \mathcal O)>\varepsilon\}.$$
Since trivially,
$$\iint_{(\mathcal O\setminus \mathcal O_\varepsilon) \times (-3,0)}|\nabla v|\lesssim \varepsilon\|N((\nabla v)\1_{\mathcal O\times (-3,0)})\|_{L^1(\partial \Omega)}\lesssim \varepsilon \| N((\nabla v)\1_{\mathcal O\times (-\infty,0)})\|_{L^p(\partial \Omega)},$$
it follows that we can fix $\varepsilon>0$ small such that the estimate 
\begin{equation}\label{est-Brownlikef}
\| N((\nabla v)\1_{\mathcal O\times (-\infty,0)})\|_{L^p(\partial \Omega)}\le 2C\|h\|_{L^p(\partial\Omega)}+C\iint_{\mathcal O_\varepsilon\times (-6r_0^2,0)}|\nabla v|.
\end{equation}
holds for all $v$ we consider with constant only depending on the coefficients of our PDE and the Lipschitz constant of the graph. 

\subsection{An interior estimate on a short time interval.}

Put $\tau=6r_0^2\le3/2$ and $I_\tau=(-\tau,0)$. We use only
that $v$ is a local energy solution of $v_t=\divg(A\nabla v)$ and
vanishes for $t\le-\tau$. The parameter $\varepsilon>0$ has already
been fixed, and we may assume $\varepsilon<1$.
Choose spatial cutoffs $0\le\zeta,\varphi\le1$ such that
\[
 \begin{split}
 &\zeta=1\text{ on }\mathcal O_\varepsilon,\qquad
   \supp\zeta\subset\mathcal O_{\varepsilon/2},\\
 &\varphi=1\text{ on }\mathcal O_{\varepsilon/2},\qquad
   \supp\varphi\subset\mathcal O_{\varepsilon/4},\qquad
   |\nabla\zeta|+|\nabla\varphi|\le C\varepsilon^{-1}.
 \end{split}
\]
Testing the original equation with $v\zeta^2$, integrating from
$-\tau$ to $0$, and dropping the nonnegative final-time term gives
Caccioppoli's inequality. Together with Cauchy--Schwarz it yields
\begin{equation}\label{eq:shortTimeCaccioppoli}
 \iint_{\mathcal O_\varepsilon\times I_\tau}|\nabla v|
 \le C\varepsilon^{-1}
       \|v\|_{L^2(\mathcal O_{\varepsilon/2}\times I_\tau)}.
\end{equation}
Here we used $\tau\le3/2$ and the boundedness of $\mathcal O$.

Next, testing with $(-t)v\varphi^2$ gives the time-weighted identity
\begin{align*}
 \iint_{\mathcal O\times I_\tau}v^2\varphi^2
 &=-2\iint_{\mathcal O\times I_\tau}(-t)\varphi^2
                                      A\nabla v\cdot\nabla v\\
 &\quad-4\iint_{\mathcal O\times I_\tau}(-t)v\varphi
                                      A\nabla v\cdot\nabla\varphi.
\end{align*}
There is no time-boundary contribution: $v$ vanishes initially and
$-t=0$ at the final time. Both energy tests are justified by Steklov
averages and limiting time cutoffs, so no pointwise time derivative
is required. Drop the nonpositive first term and use $0\le-t\le\tau$.
Young's inequality then gives
\[
 \iint v^2\varphi^2
 \le \frac12\iint v^2\varphi^2
       +C\tau^2\varepsilon^{-2}
          \iint_{\mathcal O_{\varepsilon/4}\times I_\tau}|\nabla v|^2.
\]
Consequently,
\begin{equation}\label{eq:shortTimeEnergy}
 \|v\|_{L^2(\mathcal O_{\varepsilon/2}\times I_\tau)}
 \le C\tau\varepsilon^{-1}
       \|\nabla v\|_{L^2(\mathcal O_{\varepsilon/4}\times I_\tau)}.
\end{equation}

We use the following interior estimate on backward cylinders
$J_\rho(Y,s)=B_\rho(Y)\times(s-\rho^2,s)$:
\[
 \left(\fint_{J_\rho}|\nabla v|^2\right)^{1/2}
 \le C\fint_{J_{8\rho}}|\nabla v|,
 \qquad J_{8\rho}\subset\Omega.
\]
Indeed, Caccioppoli bounds the left-hand side by
$C\rho^{-1}(\fint_{J_{2\rho}}|v-c|^2)^{1/2}$.
Local boundedness for the scalar solution $v-c$, followed by the
parabolic $L^1$ Poincar\'e inequality, bounds this by
$C\rho^{-1}\fint_{J_{4\rho}}|v-c|
 \le C\fint_{J_{8\rho}}|\nabla v|$ for a suitable constant $c$.
The $L^1$ form of local boundedness follows from its $L^2$ form by
interpolation on nested cylinders. The Poincar\'e inequality here follows from spatial Poincar\'e and
$\frac{d}{dt}\int v\eta=-\int A\nabla v\cdot\nabla\eta$ for a
normalized spatial bump $\eta$; local boundedness is the scalar
estimate of \cite[Theorem B]{Aro68}.
All cylinders in this argument have the same final time as their
respective enlargements.

Cover $\mathcal O_{\varepsilon/4}\times(-2,0)$ by finitely many
such cylinders of radius comparable to $\varepsilon$, with enlarged
cylinders contained in $\mathcal O_{\varepsilon/8}\times(-\infty,0)$.
Their number and the constants depend on $\varepsilon$ and the fixed
domain, but not on $\tau$. Since $v=0$ for $t\le-\tau$, summing the
preceding estimate gives
\[
 \|\nabla v\|_{L^2(\mathcal O_{\varepsilon/4}\times I_\tau)}
 \le C_\varepsilon
       \iint_{\mathcal O_{\varepsilon/8}\times I_\tau}|\nabla v|.
\]
Combining this with \eqref{eq:shortTimeCaccioppoli} and
\eqref{eq:shortTimeEnergy}, we obtain
\[
 \iint_{\mathcal O_\varepsilon\times I_\tau}|\nabla v|
 \le C_\varepsilon\tau
       \iint_{\mathcal O_{\varepsilon/8}\times I_\tau}|\nabla v|.
\]
Finally, put $N^-=N((\nabla v)\1_{\mathcal O\times(-\infty,0)})$.
For $X\in\mathcal O_{\varepsilon/8}$, the point $(X,t)$ belongs to
the cones with vertices $(P,t)$ for $P$ in a surface ball of radius
$c\varepsilon$ about a nearest boundary point, after decreasing
$\varepsilon$ within the fixed chart scale if necessary.
Averaging $N^-(P,t)$ over that ball and then integrating in $X$ and
$t\in I_\tau$ bounds the last integral by
$C_\varepsilon\int_{\partial\mathcal O\times I_\tau}N^-$.
Thus
\begin{equation}\label{eqFF}
 \iint_{\mathcal O_\varepsilon\times(-6r_0^2,0)}|\nabla v|
 \le C_\varepsilon r_0^2
       \int_{\partial\mathcal O\times(-6r_0^2,0)}
          N((\nabla v)\1_{\mathcal O\times(-\infty,0)}).
\end{equation}
The constant is independent of $r_0$. All local estimates use only
past times and boundedness and ellipticity of $A$.

\subsection{Solvability for small $r_0$} We are now ready to combine \eqref{eqFF} and  \eqref{est-Brownlikef}. Since $\varepsilon>0$ was fixed earlier we get that
\begin{multline}\label{est-Brownlikef2}
\| N((\nabla v)\1_{\mathcal O\times (-\infty,0)})\|_{L^p(\partial \Omega)}\le 2C\|h\|_{L^p(\partial\Omega)}+
C_\varepsilon r_0^2\int_{\partial\mathcal O\times (-6r_0^2,0)} N((\nabla v) 1_{\mathcal O\times (-\infty,0)})\\
\le 2C\|h\|_{L^p(\partial\Omega)}+C_{\varepsilon,p}r_0^2 \| N((\nabla v)\1_{\mathcal O\times (-\infty,0)})\|_{L^p(\partial \Omega)}.
\end{multline}
Hence if $r_0\in (0,1/2]$ is chosen small enough such that $C_{\varepsilon,p}r_0^2\le 1/2$
the last term of \eqref{est-Brownlikef2} can be absorbed by the left-hand side and we obtain that
\begin{equation}\label{est-BrownlikefF}
\| N((\nabla v)\1_{\mathcal O\times (-\infty,0)})\|_{L^p(\partial \Omega)}\le 4C\|h\|_{L^p(\partial\Omega)}.
\end{equation}
This is the desired estimate for $v$ with nonvanishing data in $(1/2)\Z_j\times (-6r_0^2, -3r_0^2)$.
From now on, we consider $r_0$ fixed, chosen such that $r_0=\min\{1/2,(2C_{\varepsilon,p})^{-1/2}\}$.
\medskip

The next ingredient we introduce is the exponential decay of $\nabla v$ on parts of the parabolic cylinder
where the Neumann data vanish. The following is proven in \cite{DL}:

\begin{proposition}\label{prop:Smooth Exponential Decay}
    Let \( u \) be such that \( {L}u = 0 \) in \( \Base \times (0,\infty) \) and 
    moreover suppose that \( \partial_\nu^A u=0 \) on \( \partial \Base \times [0,\infty) \). 
    Then there exists a \( \beta >0 \) such that for all \( t>0 \) and $c=\fint_{\mathcal O} u(\cdot,0)$
    \begin{align}
        \| u(\cdot,t)-c \|_{L^2(\mathcal{O})} 
        \lesssim e^{-\beta t} \| u(\cdot,0) -c\|_{L^2(\mathcal{O})}
        \lesssim e^{-\beta t} \| \nabla u(\cdot,0)\|_{L^2(\mathcal{O})}
 .
    \end{align}
\end{proposition}

For purposes that will soon become clear, we introduce a modified nontangential maximal function which we denote by ${N}^{r_0}$. In any $L^p$ norm, ${N}$ and ${N}^{r_0}$ are equivalent, with a constant depending on $r_0$ (which was already fixed and hence it suffices to work with ${N}^{r_0}$). 

For a point $(P,\tau)\in\partial\Omega$ let $\Gamma(P,\tau)$ be the usual nontangential cone of some fixed aperture. Let 
$$\Gamma^{r_0}(P,\tau)=\Gamma(P,\tau)\cap \{(X,t)\in\Omega:\delta(X,t)\le r_0\},$$
$$I^{r_0}(P,\tau)=\{X\in\mathcal O:d(X,\partial \mathcal O)>r_0\}\times (\tau-r_0^2,\tau+r_0^2).$$
For a locally bounded function $v:\Omega\to\R$ let 
$${N}^{r_0}(v)(P,\tau)=\sup_{(X,t)\in\Gamma^{r_0}(P,\tau)}|v(X,t)|+\sup_{(X,t)\in I^{r_0}(P,\tau)}|v(X,t)|.$$
Observe that the main difference from the usual nontangential maximal function is that we limit the maximal 
width of nontangential cones (in time) to $r_0^2$ and hence for a point $(P,\tau)\in\partial\Omega$ the value of ${N}^{r_0}(v)(P,\tau)$ will depend on values of $v$
only for time in the set $\mathcal O\times (\tau-r_0^2,\tau+r_0^2)$, i.e., a set whose width in the time variable is proportional to $r_0^2$.\medskip

Let us now return to our solution $v$ with Neumann data supported in $(1/2)\Z_j\times (-6r_0^2, -3r_0^2)$.
Notice that thanks to the estimate \eqref{est-BrownlikefF} we have,
for $\Delta_i=\partial\mathcal O\times (ir_0^2,(i+1)r_0^2]$, $i\in\Z$, that
\begin{equation}
\|N^{r_0}(\nabla v)\|_{L^p(\Delta_{-2})}\le \| N((\nabla v)\1_{\mathcal O\times (-\infty,0)})\|_{L^p(\partial \Omega)}\le 4C\|h\|_{L^p(\partial\Omega)}.
\end{equation}

We claim the following lemma holds (see also {\cite[Lemma 5.10]{DiS}} for a version of this lemma that applies to the Dirichlet problem and \cite[Lemma 4.3]{DL} for an $L^1$ Neumann problem version under different assumptions):

\begin{lemma}\label{lemma:Exponential Decay} Let $p>1$. There exist $\ell(p)>0$ and $\mu(p)>0$ with the following property. Assume that 
\begin{itemize}
\item the Carleson measure $\mu$ of the coefficients of $\LL u=-\partial_t u+\divg(A\nabla u)=0$ on $\mathcal O\times\R$ satisfies $\|\mu\|_{C}\le \mu(p)$,
\item for each $j=1,2,\dots, N$ we have $\ell\le \ell(p)$ in the Definition \ref{DefLipDomain}, i.e., 
locally the Lipschitz norm of the graph describing $\partial\mathcal O$ is sufficiently small.
\end{itemize}
 Let $v$ be a solution to $\LL v=0$ and
    suppose that $\partial^A_\nu v=0$  on \( \partial\Base \times (ir_0^2,\infty) \) for some $i\in\Z$.
    There exists an \( \alpha=\alpha(r_0)>0 \) such that for all $k>i$ we have that
    \begin{align*}
        \| {N}^{r_0}(\nabla v) \|_{L^p(\Delta_k)}
        \lesssim e^{-\alpha|i-k|} \| { N}^{r_0}(\nabla v) \|_{L^p(\Delta_{i})}.     
    \end{align*}
\end{lemma}

\begin{proof} Assume that $\|{N}^{r_0}(\nabla v) \|_{L^p(\Delta_{i})}<\infty$ and that $\partial_\nu^Av\big|_{\partial\Omega}\equiv0$ for all times $\ge ir_0^2$. Let \( \Omega_i := \Base \times (ir_0^2,(i+1)r_0^2] \).
Clearly, then  
$$\int_{\Omega_{i}}|\nabla v|\lesssim \| {N}^{r_0}(\nabla v) \|_{L^1(\Delta_{i})}\lesssim_{r_0} \| {N}^{r_0}(\nabla v) \|_{L^p(\Delta_{i})}.$$
Hence using the condition $\partial_\nu^A v=0$ for $\tau>ir_0^2$ and $L^2$ integrability of the gradient (by the boundary Caccioppoli inequality), we get that for some $\tau\in (ir_0^2,(i+1)r_0^2)$
$$\left(\int_{\mathcal O\times\{\tau\}}|\nabla v|^2\right)^{1/2}\lesssim \|{N}^{r_0}(\nabla v) \|_{L^p(\Delta_{i})}<\infty.$$
Hence, by Sobolev embedding on $\mathcal O\times\{\tau\}$ (for $c$ as in Proposition \ref{prop:Smooth Exponential Decay})
$$\left(\int_{\mathcal O\times\{\tau\}}|v-c|^2\right)^{1/2}\lesssim \| {N}^{r_0}(\nabla v) \|_{L^p(\Delta_{i})}<\infty.$$

Fix now some $k>i+3$. We propagate the $L^2$ initial data given on $\mathcal O\times\{\tau\}$ using  Proposition \ref{prop:Smooth Exponential Decay} to $t\in ((k-3)r_0^2,(k+4)r_0^2)$. It follows that for all such $t$ we have 

$$\| v(\cdot,t)-c \|_{L^2(\mathcal{O})} \lesssim e^{-\beta r_0^2(k-i-4)} \| v(\cdot,\tau)-c \|_{L^2(\mathcal{O})}
 \lesssim e^{-\beta r_0^2 |i-k|} \|{N}^{r_0}(\nabla v) \|_{L^p(\Delta_{i})}<\infty.$$
By integrating over the interval $((k-3)r_0^2,(k+4)r_0^2)$ and then applying boundary H\"older regularity for solutions 
\begin{equation}
\sup_{\Omega_{k-2}\cup\Omega_{k-1}\cup\Omega_{k}\cup\Omega_{k+1}\cup\Omega_{k+2}}|v-c|\lesssim e^{-\beta r_0^2 |i-k|} \| {N}^{r_0}(\nabla v) \|_{L^p(\Delta_{i})}<\infty.
\end{equation}
Observe that, by the boundary Caccioppoli inequality, the  $\sup_{\Omega_{k-2}\cup\Omega_{k-1}\cup\Omega_{k}\cup\Omega_{k+1}\cup\Omega_{k+2}}|v-c|$
controls the $L^2$ norm of $\nabla v$ over the smaller set $\mathcal O\times ((k-1)r_0^2, (k+2)r_0^2)$. Applying
\eqref{locEST} over a covering of $\partial \mathcal O\times (kr_0^2, (k+1)r_0^2)$ by backwards parabolic cubes $J$ as in Theorem  \ref{thm.NtoLoc}, we have for all $k>i+3$:
$$
\int_{\Delta_k} { N}^{r_0}(\nabla v)^p \lesssim e^{-\alpha p|i-k|}\| { N}^{r_0}(\nabla v )\|_{L^p(\Delta_{i})}^p,
$$
which is our claim.  Here we are using the assumptions of Lemma \ref{lemma:Exponential Decay} (smallness of $\mu(p)$ and $\ell(p)$) which allow us to verify the assumptions of Theorem \ref{thm.NtoLoc} on an unbounded domain; i.e., we use the process described in detail below Theorem \ref{thm.NtoLoc}.\medskip

The fact that our claim  holds for $k>i$ (and not just $k>i+3$ as proven) can be seen by using the maximum principle in this region (where there is no decay).
\end{proof}
\medskip

The decay argument uses energy inequalities on finite time intervals
and subtracts the constant spatial mean after the datum stops. It
therefore applies to the causal local-energy solutions constructed
above, even when their global half-time energy is infinite.
In particular, we can apply Lemma \ref{lemma:Exponential Decay} to each $v_{ij}$ defined by \eqref{defvij}.
It follows that
\begin{equation}\label{eq60}
\int_{\Delta_k} {N}^{r_0}(\nabla v_{ij})^p \lesssim e^{-\alpha p(k-i)}\| { N}^{r_0}(\nabla v_{ij}) \|_{L^p(\Delta_{i+4})}^p,\qquad\mbox{for all }k=i+5,i+6,\dots.
\end{equation}
Meanwhile, \eqref{est-BrownlikefF} means that for each $v_{ij}$ we have
\begin{equation}\label{eq60a}
\sum_{k=i-1}^{i+4}\|N^{r_0}(\nabla v_{ij})\|_{L^p(\Delta_k)}^p
\lesssim\|g_{ij}\|_{L^p(\partial\Omega)}^p
\le C\|g\|_{L^p(\Delta_i\cup\Delta_{i+1}\cup\Delta_{i+2})}^p.
\end{equation}
And since $v_{ij}\equiv0$ for $t<ir_0^2$ it follows that $\| { N}^{r_0}(\nabla v_{ij}) \|_{L^p(\Delta_{k})}=0$ when $k<i-1$.\medskip
  
We sum the localized estimates using Young's inequality for sequences.
Put
\[
 a_k:=\|g\|_{L^p(\Delta_k\cup\Delta_{k+1}\cup\Delta_{k+2})},
 \qquad b_i:=\|N^{r_0}(\nabla u)\|_{L^p(\Delta_i)},
\]
and define
\[
 \kappa_j=\begin{cases}
 0,&j<-1,\\
 1,&-1\le j\le4,\\
 e^{-\alpha j},&j>4.
 \end{cases}
\]
Subadditivity, \eqref{eq60}, \eqref{eq60a} and the vanishing for
$k<i-1$ give
\[
 b_i\lesssim_N\sum_{k\in\Z}\kappa_{i-k}a_k.
\]
Since $\kappa\in\ell^1(\Z)$ and the sets
$\Delta_k\cup\Delta_{k+1}\cup\Delta_{k+2}$ have bounded overlap,
Young's inequality yields
\[
 \|N^{r_0}(\nabla u)\|_{L^p(\partial\Omega)}
 =\|b\|_{\ell^p}
 \lesssim_N\|\kappa\|_{\ell^1}\|a\|_{\ell^p}
 \lesssim\|g\|_{L^p(\partial\Omega)}.
\]
The equivalence of $N$ and $N^{r_0}$ therefore gives
\[
 \|N(\nabla u)\|_{L^p(\partial\Omega)}
 \lesssim_{r_0}\|g\|_{L^p(\partial\Omega)},
\]
for smooth data compactly supported in time. This estimate has been
proved for the causal local-energy solutions constructed above; its
constant does not depend on the time support or the total integral
of the datum. In particular it holds for differences of such solutions.

We finish the passage to general data, keeping track of the spatial
constant. Choose smooth compactly supported $g_m\to g$ in
$L^p(\partial\Omega)$ and let $u_m=v[g_m]$. Subtract a constant
from each $u_m$ so that $\fint_{\mathcal O}u_m(X,0)\,dX=0$.
Writing $m_m(t)=\fint_{\mathcal O}u_m(X,t)\,dX$, the weak equation
and \eqref{eq:causalMean} give, on each bounded interval $J$,
\[
 \sup_{t\in J}|m_m(t)-m_l(t)|
 \le C_{J,p,\mathcal O}\|g_m-g_l\|_p,
 \qquad
 \|N(\nabla(u_m-u_l))\|_p\le C\|g_m-g_l\|_p.
\]
Integration in boundary charts and over the interior part of the
domain, followed by spatial $L^1$ Poincar\'e, yields
\[
 \|\nabla(u_m-u_l)\|_{L^1(\mathcal O\times J)}
 +\|u_m-u_l\|_{L^1(\mathcal O\times J)}
 \le C_{J,p,\mathcal O}\|g_m-g_l\|_p.
\]
Thus $u_m$ and their gradients converge locally in $L^1$, including
up to the boundary. Interior scalar local boundedness and Caccioppoli
also give convergence in the interior local-energy norms. Passing
to the weak identities shows that the limit $u$ satisfies
\[
 \iint_\Omega[A\nabla u\cdot\nabla\zeta-u\partial_t\zeta]
   =\int_{\partial\Omega}g\zeta
\]
for every smooth test $\zeta$ compactly supported in
$\overline\Omega$. Fatou on the interior Whitney averages and then
on the boundary gives
\begin{equation}\label{eq:boundedCompletion}
 \|\tilde N(\nabla u)\|_p\le C\|g\|_p,
 \qquad
 \|\tilde N(\nabla(u_m-u))\|_p\le C\|g_m-g\|_p\longrightarrow0.
\end{equation}
The same difference estimates show independence of the approximation,
with the stated normalization. In particular we do not interpret the
infinite sum of the $v_{ij}$ without a normalization; the finite-sum
estimate is extended by this completion.

For completeness, when $g$ also belongs to the energy-dual trace
space, the approximation can be chosen to converge in that norm
as well, with $\int_{\partial\Omega}g_m=0$. Here is a density
argument that preserves the cancellation. First apply smooth time Fourier cutoffs compactly supported away from zero,
and then regularize in boundary charts. The time cutoffs converge
strongly in both $L^p$ and the energy-dual norm; for each fixed
time cutoff the boundary regularizations do so as well.
For a fixed regularization $f$, its time primitive $H$ is in
$H^1(\R;L^2(\partial\mathcal O))$, and $H,\partial_tH=f$
belong to $L^p(\partial\Omega)$. For $\chi\in C_0^\infty(\R)$ equal
to one near zero, put
\[
 f_M=\partial_t\big[\chi(t/M)H\big].
\]
Then $f_M$ is smooth and compactly supported in time, has zero time
integral at every boundary point, and converges to $f$ in $L^p$ and
$L^2$. Moreover $\chi(t/M)H\to H$ in
$H^{1/2}(\R;L^2(\partial\mathcal O))$. The trace--Poincar\'e
bound used above for $h_R$, applied to $f_M-f$, therefore gives
convergence in the energy-dual norm: its spatial flux term is the
time derivative of the boundary integral of $(\chi(t/M)-1)H$.
A diagonal choice gives the asserted approximation.

For these $g_m$, the finite-sum solution is the energy solution,
as proved after \eqref{defvij}. Coercivity of the energy form gives
convergence of $u_m$ to the energy solution for $g$ in $\dot\E$,
modulo constants. The local limit above consequently agrees with
that solution modulo a constant. Hence \eqref{eq:boundedCompletion}
proves $(N)_p^{\mathcal L}$ in the sense of Section~\ref{Def}, and
also constructs its extension to arbitrary $L^p$ data. This proves
Theorem~\ref{MainT} for bounded Lipschitz cylinders. \qed
\medskip

\appendix
\section{Remarks on the nontangential maximal function of $D^{1/2}_t u$}\label{APA}

Recall that in \cite{DLP1} we established nontangential estimates for
$\tilde N(\nabla u)$. For the parabolic Regularity problem,
\cite[Theorem 1.1]{Din23} then gives the corresponding bounds for
$\tilde N(D^{1/2}_t u)$ and $\tilde N(H_tD^{1/2}_t u)$, provided the
solution attains its boundary data nontangentially and those data belong
to the parabolic Sobolev space of \cite[Definition 1.1]{Din23}.
Here $H_t$ denotes the Hilbert transform in time.
An analogous bound was shown in \cite{AEN}
where the $L^2$ solvability was considered for the Regularity and Neumann problems for special matrices.\medskip

Hence, a natural question arises: Does the solution $u$, the existence of which is guaranteed by Theorem \ref{MainT}, together with the bound for $\tilde N(\nabla u)$, also enjoy an estimate for $\tilde N(D^{1/2}_t u)$?
\medskip

In this appendix we will show that for the Neumann problem the answer depends on whether the base set $\mathcal O\subset\R^n$ in our domain $\Omega=\mathcal O\times\R$ is bounded or not.

We prove that in the case where this set is bounded the answer is negative (i.e., such a bound does not exist) for all $1<p\le \infty$. In the unbounded case we show that such estimates do hold for all $1<p<\infty$. This answer applies to both $\tilde N(D^{1/2}_t u)$ and $\tilde N(H_tD^{1/2}_t u)$ by similar considerations. We focus here on
$\tilde N(D^{1/2}_t u)$ for simplicity, given the similarity of the two bounds.

\subsection {The bounded case}

Let $\mathcal O$ be a bounded Lipschitz domain. For every $1<p\le\infty$ we prove that there is no constant $C=C(\mathcal L,\mathcal O,p)$ such that
\begin{equation}\label{A1}
\|\tilde N(D^{1/2}_t u)\|_{L^p(\pom)}\le C\|g\|_{L^p(\pom)}
\end{equation}
for all energy solutions with Neumann datum $g$. We use smooth, compactly supported functions of time whose integral over $\pom$ is zero.\medskip

We first justify spatial integration of the half derivative. If $u\in\dot\E(\mathcal O\times\R)$ and $v(t):=\int_{\mathcal O}u(X,t)\,dX$, then
\begin{equation}\label{eq:massHalf}
D^{1/2}_tv=\int_{\mathcal O}D^{1/2}_tu(X,\cdot)\,dX
\quad\hbox{in }L^2(\R),\qquad
\|D^{1/2}_tv\|_2\le |\mathcal O|^{1/2}\|D^{1/2}_tu\|_{L^2(\Omega)}.
\end{equation}
Indeed, for a smooth function compactly supported in $\overline\Omega$ the identity follows by taking the Fourier transform in time, and the inequality follows by Cauchy--Schwarz in $X$. Let $u_j$ be such functions converging to $u$ in $\dot\E$. The same inequality for $u_j-u_k$ shows that $\int_{\mathcal O}u_j\,dX$ converges in $\dot H^{1/2}(\R)$, modulo constants. After subtracting constants so that the averages on $\mathcal O\times(-1,1)$ agree, spatial Poincar\'e and the one-dimensional fractional Poincar\'e inequality give $u_j\to u$ in $L^2(\mathcal O\times I)$ for every bounded time interval $I$. Thus the limit of the spatial integrals is $v$, modulo a constant. The right-hand sides in \eqref{eq:massHalf} converge in $L^2(\R)$ by Cauchy--Schwarz, which proves the identity.\medskip

We also verify that the data used below give energy solutions. Write $S:=\sigma(\partial\mathcal O)$ and let $g(x,t)=S^{-1}F(t)$, where $F\in C_0^\infty(\R)$ and $\int_\R F=0$. The function $D_t^{-1/2}F$, defined by the Fourier multiplier $|\tau|^{-1/2}$, belongs to $L^2(\R)$: near zero $\widehat F(\tau)=O(|\tau|)$, and at infinity $\widehat F$ decays rapidly. For a smooth test function $\psi$, put $m_\psi(t):=|\mathcal O|^{-1}\int_{\mathcal O}\psi(X,t)\,dX$. The spatial trace and Poincar\'e inequalities, followed by \eqref{eq:massHalf}, give
\begin{align*}
\left|\int_{\pom}g\psi\,d\sigma\right|
&\le \left|S^{-1}\int_\R F(t)\int_{\partial\mathcal O}(\psi-m_\psi)\,d\sigma\,dt\right|
   +\left|\int_\R Fm_\psi\,dt\right|\\
&\le C(\mathcal O)\|F\|_2\|\nabla\psi\|_{L^2(\Omega)}
  +|\mathcal O|^{-1/2}\|D_t^{-1/2}F\|_2\|D_t^{1/2}\psi\|_{L^2(\Omega)}.
\end{align*}
Hence $g$ defines a bounded functional on the energy space, and the coercive form $a_\delta$ from Section~\ref{Def} gives an energy solution with this Neumann datum. Testing its weak equation with a function of time alone gives, in distributions on $\R$,
\begin{equation}\label{flux-flow}
 v'(t)=\frac{d}{dt}\int_{\mathcal O}u(X,t)\,dX
       =\int_{\partial\mathcal O}g(x,t)\,d\sigma(x)=F(t).
\end{equation}
Here the time term in \eqref{eq.NeumannBdy} is $\langle v',\psi\rangle$ by \eqref{eq:massHalf} and $\partial_t=D_t^{1/2}H_tD_t^{1/2}$.\medskip

Next we estimate a fixed time average of the spatial integral. Put $d:=\diam(\mathcal O)$ and $\tau_0:=d^2$. Choose $\rho\in C_0^\infty(-1,1)$ with $\rho\ge0$ and $\int\rho=1$, and write $\rho_a(t):=a^{-1}\rho(t/a)$ for $a>0$. We claim that, for every energy solution,
\begin{equation}\label{eq:boundedAvg}
\big\|(D^{1/2}_tv)*\rho_{\tau_0}\big\|_{L^p(\R)}
\le C(\mathcal O,p)\|\tilde N(D^{1/2}_tu)\|_{L^p(\pom)},
\qquad 1<p\le\infty.
\end{equation}
Let $w=D_t^{1/2}u$ and cover $\Omega$ by Whitney boxes $W$ of parabolic size $\ell\lesssim d$. After a fixed subdivision, each box is contained in an admissible averaging ball of comparable measure. It has a boundary shadow $\Delta_W$ with $|W|\simeq\ell|\Delta_W|$, on which its $L^2$ average is bounded by $\tilde N(w)$ at a fixed larger aperture. Consequently,
$$\iint_W|w|\,dX\,ds
\le |W|\left(\fiint_W|w|^2\right)^{1/2}
\lesssim\ell\int_{\Delta_W}\tilde N(w)\,d\sigma.$$
The shadows have bounded overlap at each dyadic size. If $W$ meets $\mathcal O\times(t-\tau_0,t+\tau_0)$, its shadow lies in $\partial\mathcal O\times(t-C_0\tau_0,t+C_0\tau_0)$ for a fixed geometric constant $C_0$. Summing first at each size and then using $\sum_{\ell\lesssim d}\ell\lesssim d$, we obtain
\begin{equation}\label{eq:massWhitney}
\int_{\mathcal O}\int_{|s-t|<\tau_0}|w(X,s)|\,ds\,dX
\lesssim d\int_{|s-t|<C_0\tau_0}\int_{\partial\mathcal O}\tilde N(w)(x,s)\,d\sigma(x)\,ds.
\end{equation}
By \eqref{eq:massHalf}, Fubini on this finite time interval and \eqref{eq:massWhitney},
$$\big|(D^{1/2}_tv)*\rho_{\tau_0}(t)\big|
\le \frac{\|\rho\|_\infty}{\tau_0}
\int_{\mathcal O}\int_{|s-t|<\tau_0}|w(X,s)|\,ds\,dX
\le C(\mathcal O)\int_{|s-t|<C_0\tau_0}B(s)\,ds,$$
where $B(s):=\int_{\partial\mathcal O}\tilde N(w)(x,s)\,d\sigma(x)$. Young's inequality in time and H\"older's inequality on $\partial\mathcal O$ give
$$\left\|\int_{|s-\cdot|<C_0\tau_0}B(s)\,ds\right\|_{L^p(\R)}
\le 2C_0\tau_0\|B\|_{L^p(\R)},\qquad
\|B\|_{L^p(\R)}\le S^{1-1/p}\|\tilde N(w)\|_{L^p(\pom)}.$$
The global change-of-aperture estimate returns the last norm to the prescribed aperture. This also holds for $p=\infty$, since the essential supremum over the boundary is the supremum of the same interior averages. Thus \eqref{eq:boundedAvg} follows, with constants independent of the datum.\medskip

Fix a nonzero function $\phi\in C_0^\infty(\R)$ and, for $T\ge1$, set
$$F_T(t):=\phi'(t/T),\qquad g_T(x,t):=S^{-1}F_T(t).$$
Each $F_T$ has integral zero, so let $u_T$ be the energy solution constructed above. By \eqref{flux-flow}, after subtracting a constant from $u_T$, its spatial integral and half derivative are
$$v_T(t)=T\phi(t/T),\qquad
D_t^{1/2}v_T(t)=T^{1/2}(D_t^{1/2}\phi)(t/T).$$
In particular,
$$\big\|(D_t^{1/2}v_T)*\rho_{\tau_0}\big\|_{L^p(\R)}
=T^{1/2+1/p}\big\|(D_t^{1/2}\phi)*\rho_{\tau_0/T}\big\|_{L^p(\R)}.$$
We use $1/p=0$ when $p=\infty$. The function $D_t^{1/2}\phi$ is nonzero, continuous and tends to zero at infinity, with decay $O(|t|^{-3/2})$. The approximate identity therefore converges in $L^p(\R)$ for every $1<p<\infty$ and uniformly for $p=\infty$. The last norm is consequently bounded below by a positive constant for all sufficiently large $T$. On the other hand,
$$\|g_T\|_{L^p(\pom)}=S^{1/p-1}T^{1/p}\|\phi'\|_{L^p(\R)}.$$
If \eqref{A1} held, \eqref{eq:boundedAvg} would now give
$$cT^{1/2+1/p}\le C\|\tilde N(D_t^{1/2}u_T)\|_{L^p(\pom)}
\le C\|g_T\|_{L^p(\pom)}\le C'T^{1/p},$$
with constants independent of $T$. This is impossible as $T\to\infty$, and proves the failure of \eqref{A1} for every $1<p\le\infty$.\qed\medskip

The same proof applies to $H_tD_t^{1/2}u$: the Hilbert transform commutes with spatial integration in $L^2$, and $H_tD_t^{1/2}\phi$ has the same continuity and decay properties used above.

The obstruction comes from the spatial mean: by \eqref{flux-flow}, it is a time primitive of the total boundary flux. Its half derivative therefore acquires the factor $T^{1/2}$ under a dilation of time. In particular, cancellation of the integral of $g$ over $\pom$ does not restore \eqref{A1}; the data above have that cancellation and their spatial masses vanish at both ends of the time axis.

\subsection{The graph domain case $\mathcal O=\{x_n>\phi(x)\}$ (unbounded)}

We now consider unbounded Lipschitz graph domains. The spatial integration argument above uses the finite measure of $\mathcal O$ and does not apply to this case.\medskip

We first reduce to the half-space. Apply Lemmas \ref{lem:graphflat} and \ref{lem:graphtransfer} to the graph
defining $\OO$. The resulting map has constant spatial Jacobian, so the
pulled-back operator has the form $-\partial_t+\divg(B\nabla\cdot)$ and
satisfies the same coefficient conditions, with constants depending also
on $\|\nabla\phi\|_\infty$. The time half derivative and Hilbert transform
commute with this pullback, the relevant maximal-function norms are
equivalent, and Neumann solvability transfers at the same exponent. It therefore
suffices to prove the estimates on $\R^n_+\times\R$. No smallness of the
Lipschitz constant is required in this reduction. All constants in the
graph-domain conclusions below may depend on the Lipschitz character.

\begin{lemma}[interior bounds]\label{lem:int}
Let $W=B_{h/2}(Y)\times(s-h^2/4,s+h^2/4)$ with $Y=(y,y_n)$, $y_n\ge h$, and let
$W^*=B_h(Y)\times(s-h^2,s+h^2)$ be its parabolic double. If $u$ is a weak
solution of $\LL u=0$, then under \eqref{E:elliptic}, \eqref{E:1:carl} and
\eqref{E:1:bound},
\[
\sup_W|\nabla u|\lesssim\Big(\fint_{W^*}|\nabla u|^2\Big)^{1/2},
\qquad
\sup_W|\partial_tu|\lesssim h^{-1}\Big(\fint_{W^*}|\nabla u|^2\Big)^{1/2} .
\]
The constants depend only on $n$, the ellipticity constants and $K$ in
\eqref{E:1:bound}; the suprema are essential suprema. Consequently, if
$(x,s)$ is a boundary point for which $W^*\subset\Gamma(x,s)$, then
\[
\sup_W|\nabla u|+h\sup_W|\partial_tu|
\lesssim\tilde N(\nabla u)(x,s).
\]
\end{lemma}

\begin{proof}
Set $U(\xi,\tau)=u(Y+h\xi,s+h^2\tau)$,
$\mathsf A(\xi,\tau)=A(Y+h\xi,s+h^2\tau)$ and
$\mathcal C_\rho=B_\rho(0)\times(-\rho^2,\rho^2)$. Thus $W$ and $W^*$
correspond to $\mathcal C_{1/2}$ and $\mathcal C_1$. We work inside
$\mathcal C_{3/4}$, where the original height is at least $h/4$, even when
$y_n=h$. Hence
\[
\|\nabla_\xi\mathsf A\|_{L^\infty(\mathcal C_{3/4})}
+\|\partial_\tau\mathsf A\|_{L^\infty(\mathcal C_{3/4})}\le C K.
\]
Subtracting the constant skew-symmetric part of $\mathsf A(0,0)$ changes
neither the equation nor these derivatives. We may therefore also take
$\|\mathsf A\|_{L^\infty(\mathcal C_{3/4})}\le C(\Lambda+K)$.
All derivatives and divergences below, except $\partial_\tau$, are spatial.

First, spatial difference quotients and the energy test with a squared
cutoff in $\mathcal C_{3/4}$ give
\begin{equation}\label{eq:intL2}
\int_{\mathcal C_{2/3}}\big(|\nabla^2U|^2+|\partial_\tau U|^2\big)
\le C\int_{\mathcal C_{3/4}}|\nabla U|^2.
\end{equation}
Indeed, the equation for $U_k=\partial_{\xi_k}U$ is
\[
\partial_\tau U_k-\divg(\mathsf A\nabla U_k)
=\divg\big((\partial_{\xi_k}\mathsf A)\nabla U\big).
\]
Testing with $\zeta^2U_k$, summing over $k$, and absorbing the terms with
$\zeta\nabla U_k$ bounds the Hessian integral by
\[
C\int_{\mathcal C_{3/4}}
\big(|\nabla\zeta|^2+|\zeta\partial_\tau\zeta|
+\zeta^2|\nabla\mathsf A|^2\big)|\nabla U|^2,
\]
where $\zeta=1$ on $\mathcal C_{2/3}$. The other term in
\eqref{eq:intL2} follows from
$\partial_\tau U=\mathsf A:\nabla^2U+(\divg\mathsf A)\cdot\nabla U$.
Here the calculations for $U_k$ are first made with spatial difference
quotients, whose forcing has the same uniform bound.

Time difference quotients now show that $\partial_\tau U$ also belongs
locally to the parabolic energy class in $\mathcal C_{2/3}$: their $L^2$
norms are bounded by that of $\partial_\tau U$ in \eqref{eq:intL2}, and
their divergence forcing is a time difference quotient of $\mathsf A$
times $\nabla U$, uniformly bounded in $L^2$. The same cutoff energy test
therefore bounds their spatial gradients. Passing to the limit justifies
the differentiated equation for $\partial_\tau U$ without any second
derivatives of $\mathsf A$.

For completeness, the coupled Moser estimate needed here follows directly
by testing the differentiated equations together. Put
$V=(U_1,\ldots,U_n,\partial_\tau U)$. Its components satisfy
\[
\partial_\tau V_\alpha-\divg(\mathsf A\nabla V_\alpha)
=\divg F_\alpha,\qquad
F_k=(\partial_{\xi_k}\mathsf A)\nabla U,\quad
F_{n+1}=(\partial_\tau\mathsf A)\nabla U,
\qquad |F|\le CK|V|.
\]
For $p\ge2$, test with $\zeta^2|V|^{p-2}V_\alpha$ and sum over $\alpha$.
The common scalar principal matrix is essential: its two principal terms
are bounded below by
\[
\lambda |V|^{p-2}|\nabla V|^2
+\lambda(p-2)|V|^{p-2}|\nabla|V||^2,
\]
also when $\mathsf A$ is nonsymmetric. Young's inequality for the forcing
and cutoff terms yields
\begin{equation}\label{eq:intMoser}
\begin{split}
&\sup_\tau\int \zeta^2|V|^p
+\iint\big|\nabla(\zeta|V|^{p/2})\big|^2\\
&\qquad\le Cp^4\iint
\big(|\nabla\zeta|^2+|\zeta\partial_\tau\zeta|
+K^2\zeta^2\big)|V|^p.
\end{split}
\end{equation}
Time averaging and truncated, regularized powers justify these tests in
the energy class. For $1/2\le r<R<2/3$, choose $\zeta=1$ on
$\mathcal C_r$, supported in $\mathcal C_R$, with
$|\nabla\zeta|\lesssim(R-r)^{-1}$ and
$|\partial_\tau\zeta|\lesssim(R-r)^{-2}$. The parabolic Sobolev inequality,
with $\chi=1+2/n$, then gives
\[
\|V\|_{L^{p\chi}(\mathcal C_r)}
\le\big[Cp^4(R-r)^{-2}\big]^{1/p}
\|V\|_{L^p(\mathcal C_R)}.
\]
Iterating with $p=2\chi^j$ on cylinders decreasing from
$\mathcal C_{5/8}$ to $\mathcal C_{1/2}$, and using \eqref{eq:intL2}, gives
\[
\sup_{\mathcal C_{1/2}}\big(|\nabla U|+|\partial_\tau U|\big)
\le C\|V\|_{L^2(\mathcal C_{5/8})}
\le C\Big(\fint_{\mathcal C_1}|\nabla U|^2\Big)^{1/2}.
\]
Rescaling proves the two interior bounds.

Finally, at an interior point $P=(Z,\tau)$ of height $d=\delta(P)$,
apply these bounds with centre $P$ and scale $d/4$. The corresponding
double lies inside the averaging ball $B_{d/2}(P)$ and has comparable
volume, so, for almost every $P$,
\[
|\nabla u(P)|+d|\partial_tu(P)|
\le C\Big(\fint_{B_{d/2}(P)}|\nabla u|^2\Big)^{1/2}.
\]
If $P\in W\subset\Gamma(x,s)$, this average is bounded by
$\tilde N(\nabla u)(x,s)$ and $h\le2\delta(P)$. Taking suprema on $W$
proves the last assertion without assuming an upper bound on $y_n/h$.
\end{proof}

\subsubsection*{Interior and area estimates with drift}

We now allow the first-order term in Theorem~\ref{Thalf}. We first work
in the half-space and write $h=x_n$, $z=(x,t)$ and
\[
 \mathcal L_{\mathbf b}u=-\partial_tu+\divg(A\nabla u)
                         +\mathbf b\cdot\nabla u=0,
 \qquad \kappa=(\|\mu\|_C+\|\mu_{\mathbf b}\|_C)^{1/2}.
\]
The drift hypotheses are \eqref{eq:halfDriftHyp}. By an energy solution
we mean $u\in\dot\E(\Omega)$ satisfying this equation in distributions;
equivalently, the weak identity has the additional term
$-\iint(\mathbf b\cdot\nabla u)\zeta$ on its left-hand side.
Interior difference quotients give the reinforced weak regularity used
below. We make no existence or uniqueness assertion for this equation.

For a field $F(z,h)$ put
\[
 \mathcal V(F)(z)=\left(\int_0^\infty|F(z,h)|^2\frac{dh}{h}\right)^{1/2},
 \qquad
 \mathcal C(F)(z)=\left(\int_0^\infty\int_{\Delta_h(z)}
       |F(y,s,h)|^2\frac{dy\,ds\,dh}{h^{n+2}}\right)^{1/2},
\]
where $\Delta_h(z)$ is a parabolic boundary ball. Fixed changes of
aperture affect only constants, and
\[
 \mathcal C(h\nabla^2u)\simeq S(\nabla u),\qquad
 \mathcal C(h^2\partial_t\nabla u)\simeq A(\nabla u).
\]

\begin{lemma}[Carleson square-function estimate]\label{lem:driftCM}
Suppose $d\nu=|a(z,h)|^2\,dz\,dh/h$ is Carleson. For $1<p<\infty$,
\[
 \|\mathcal C(av)\|_p
 \lesssim_p\|\nu\|_C^{1/2}\|N(v)\|_p.
\]
\end{lemma}

\begin{proof}
For $1<p<2$, let $G=N(v)$ and $O_k=\{G>2^k\}$, with the aperture
fixed so that $|v(z,h)|>2^k$ implies $\Delta_{ch}(z)\subset O_k$.
Decompose $O_k$ into parabolic Whitney cubes $Q$ and assign each point
of $\{2^k<|v|\le2^{k+1}\}$ to the cube containing its projection.
Its height is at most $C\ell(Q)$. The assigned region lies in a
Carleson box over $CQ$, and its conical projection is contained in
another fixed dilate of $Q$.
The square function of $a$ restricted to this region has squared
$L^2$ norm at most $C\|\nu\|_C|Q|$ and support of measure at most
$C|Q|$. H\"older on this support bounds its $p$th power integral by
$C\|\nu\|_C^{p/2}|Q|$. Subadditivity of the power $p/2$ gives
\[
 \|\mathcal C(av)\|_p^p
 \lesssim\|\nu\|_C^{p/2}\sum_k2^{kp}|O_k|
 \lesssim\|\nu\|_C^{p/2}\|G\|_p^p.
\]
Truncation and monotone convergence justify the decomposition.
At $p=2$, Fubini and ordinary Carleson embedding suffice.

For $p>2$, dualize $\mathcal C(av)^2$ against a nonnegative
$w\in L^{(p/2)'}$. Fubini gives, up to a geometric constant,
\[
 \int\mathcal C(av)^2w
 =C\iint |v(z,h)|^2\fint_{\Delta_h(z)}w\,d\nu(z,h).
\]
The nontangential maximal function of the integrand is at most
$C N(v)^2Mw$, at fixed larger apertures. Carleson embedding, H\"older
and the maximal inequality on $L^{(p/2)'}$ prove the assertion.
\end{proof}

\begin{lemma}[interior and area estimates with drift]\label{lem:driftInputs}
Lemma~\ref{lem:int} holds for $\mathcal L_{\mathbf b}$, with constants
depending also on $K_{\mathbf b}$. Moreover, for $1<p<\infty$,
\begin{equation}\label{eq:driftInputs}
 \|A(\nabla u)\|_p+\|\mathcal C(h\partial_tu)\|_p
 \lesssim\|S(\nabla u)\|_p+\kappa\|\tilde N(\nabla u)\|_p.
\end{equation}
\end{lemma}

\begin{proof}
Use the rescaling of Lemma~\ref{lem:int}, and put
$\boldsymbol\beta(\xi,\tau)=h\mathbf b(Y+h\xi,s+h^2\tau)$.
On $\mathcal C_{3/4}$,
\[
 |\boldsymbol\beta|+|\nabla_\xi\boldsymbol\beta|
       +|\partial_\tau\boldsymbol\beta|\le C K_{\mathbf b}.
\]
The equations for $V=(\nabla_\xi U,\partial_\tau U)$ become
\[
 \partial_\tau V_\alpha-\divg(\mathsf A\nabla V_\alpha)
 =\divg F_\alpha+\boldsymbol\beta\cdot\nabla V_\alpha+G_\alpha,
\]
where $F_\alpha$ is as in that lemma and
\[
 G_k=(\partial_{\xi_k}\boldsymbol\beta)\cdot\nabla U,
 \qquad G_{n+1}=(\partial_\tau\boldsymbol\beta)\cdot\nabla U,
 \qquad |F|+|G|\le C(K+K_{\mathbf b})|V|.
\]
The spatial difference-quotient energy estimate first gives
\eqref{eq:intL2}, with the additional bounded coefficients
$|\boldsymbol\beta|^2+|\nabla\boldsymbol\beta|$ on the right.
Time difference quotients then give the equation and local energy
regularity of $\partial_\tau U$. In the coupled Moser test, the drift
term is bounded by Young's inequality against the principal gradient
term and $C|\boldsymbol\beta|^2|V|^p$; the term $G$ is bounded by
$C|V|^p$. Thus \eqref{eq:intMoser} and its iteration apply. In
particular, $N(\nabla u)\lesssim\tilde N(\nabla u)$.
Only the displayed weak first derivatives of $A$ and $\mathbf b$ are
used in these difference-quotient arguments.

Put $v=\partial_tu$. Its differentiated equation is
\[
 \partial_tv-\divg(A\nabla v)-\mathbf b\cdot\nabla v
 =\divg((\partial_tA)\nabla u)+(\partial_t\mathbf b)\cdot\nabla u.
\]
Testing with a squared cutoff times $v$ on nested interior Whitney
boxes of size $r$ gives
\begin{multline}\label{eq:driftTimeCacc}
 \iint_W|\nabla v|^2\lesssim
 r^{-2}\iint_{W^*}|v|^2
 +\iint_{W^*}|\partial_tA|^2|\nabla u|^2\\
 +r^2\iint_{W^*}|\partial_t\mathbf b|^2|\nabla u|^2.
\end{multline}
Here $r|\mathbf b|\lesssim K_{\mathbf b}$; Young's inequality bounds
the drift contribution by an absorbable gradient term and
$Cr^{-2}|v|^2$. No smallness is needed. The PDE also gives
\begin{equation}\label{eq:driftPDEbound}
 |\partial_tu|\lesssim|\nabla^2u|
                       +( |\nabla A|+|\mathbf b|)|\nabla u|.
\end{equation}
Sum \eqref{eq:driftTimeCacc} over Whitney boxes in a cone, using
\eqref{eq:driftPDEbound}. The resulting coefficient measures have
densities $h|\nabla A|^2$, $h|\mathbf b|^2$,
$h^3|\partial_tA|^2$ and $h^5|\partial_t\mathbf b|^2$.
Lemma~\ref{lem:driftCM} proves the area estimate in
\eqref{eq:driftInputs}. The other estimate follows directly from
\eqref{eq:driftPDEbound} and the same lemma.
\end{proof}

\begin{lemma}[boundary trace and the half-time estimate]\label{lem:boundaryBridge}
Let $1<p<\infty$ and let $u$ be a reinforced weak solution of
$\mathcal L_{\mathbf b}u=0$ in the half-space under the coefficient
and drift hypotheses of Theorem~\ref{Thalf}. Suppose that
$G:=\tilde N_a(\nabla u)\in L^p(\pom)$, with a fixed sufficiently large
aperture $a$. Then $u$ has a nontangential trace $f\in L^p_{\loc}(\pom)$,
and, for a.e.\ $t$ and a.e.\ $x,y\in\R^{n-1}$,
\begin{equation}\label{eq:spatialHajlasz}
 |f(x,t)-f(y,t)|\le C|x-y|\big(G(x,t)+G(y,t)\big).
\end{equation}
If also $D_t^{1/2}f\in L^p(\pom)$, then
\begin{equation}\label{A1a}
\|\tilde N(D^{1/2}_t u)\|_{L^p(\pom)}+
\|\tilde N(H_tD^{1/2}_t u)\|_{L^p(\pom)}
\lesssim \|\tilde N(\nabla u)\|_{L^p(\pom)}+
\|D^{1/2}_t f\|_{L^p(\pom)}.
\end{equation}
For an energy solution this $f$ agrees with its energy trace.
\end{lemma}

\begin{proof}
By Lemma \ref{lem:driftInputs}, integration on vertical segments gives
$$f(x,t):=\lim_{h\downarrow0}u(x,h,t),$$
with
\begin{equation}\label{eq:bridgeVertical}
 |u(x,h,t)-f(x,t)|\le ChG(x,t)
\end{equation}
for a.e.\ $(x,t)$. At a fixed positive height the solution is locally
bounded, so this also gives $f\in L^p_{\loc}(\pom)$.
To check nontangential attainment, fix a cone aperture $a_0$ and take
$(y,h,s)$ with $|y-x|+|s-t|^{1/2}<a_0h$.
Join $(x,h,t)$ to $(y,h,t)$ and then to $(y,h,s)$.
The spatial segment has length at most $a_0h$, and the time segment
has length at most $a_0^2h^2$. Both lie in a fixed larger cone at $(x,t)$.
Lemma \ref{lem:driftInputs} bounds the spatial gradient there by $CG(x,t)$
and the time derivative by $Ch^{-1}G(x,t)$, after fixing $a$ large enough.
Together with \eqref{eq:bridgeVertical}, this gives
$$|u(y,h,s)-f(x,t)|\le C_{a_0}hG(x,t)\longrightarrow0.$$
The interior locally Lipschitz representative supplied by the same lemma
justifies these segment integrations. For energy solutions, convergence
in \eqref{eq:bridgeVertical} in $L^p_{\loc}$ and the energy trace theorem
identify $f$ with the energy trace.

For \eqref{eq:spatialHajlasz}, put $r=|x-y|>0$ and transport both
boundary points, at the same time $t$, to $(x,r,t)$.
The vertical segments and the horizontal segment at height $r$ lie in
fixed cones at the two boundary points. Their total length is at most
$3r$, so Lemma \ref{lem:driftInputs} gives the asserted bound.
In the notation of \cite[Definition 1.1]{Din23}, this means
$$\|f\|_{\dot M_x^{1,p}(\pom)}\lesssim\|G\|_{L^p(\pom)},$$
since the spatial Haj\l asz seminorm is the infimum of the $L^p$ norms
of functions satisfying \eqref{eq:spatialHajlasz} with constant one.
Consequently $D_t^{1/2}f\in L^p$ implies
$f\in\dot L^p_{1,1/2}(\pom)$, whose seminorm is the sum of this spatial
seminorm and $\|D_t^{1/2}f\|_p$.
The half-space is a uniform domain with Ahlfors regular boundary,
and the drift has size $|\mathbf b|\le K_{\mathbf b}/\delta$.
With $B=-\mathbf b$ in the sign convention of that paper, all the hypotheses of
\cite[Theorem 1.1]{Din23} now hold, and its estimate gives \eqref{A1a}.
Changes of aperture affect only the constant. The graph-domain
transfer is given in the proof of Proposition~\ref{prop:halfDrift}.
\end{proof}

\subsubsection*{The case $p=2$, including drift}

Let $u$ be an energy solution of $\mathcal L_{\mathbf b}u=0$ under
the hypotheses of Theorem~\ref{Thalf}, with
$\tilde N(\nabla u),S(\nabla u)\in L^2(\pom)$. We shall prove the estimate
\begin{equation}\label{A2}
\|D^{1/2}_t u\big|_{\pom}\|_{L^2(\pom)}\lesssim \|\tilde N(\nabla u)\|_{L^2(\pom)}+\|S(\nabla u)\|_{L^2(\pom)},
\end{equation}
Lemma~\ref{lem:boundaryBridge} then gives the two half-time maximal
estimates at $p=2$. When $\mathbf b=0$, Theorem~\ref{thm.S<Np}
removes the square-function assumption and its term on the right.

Put $w=D_t^{1/2}u$ and write
$$\mathcal N:=\|\tilde N(\nabla u)\|_2,\qquad
\mathcal S:=\|S(\nabla u)\|_2,\qquad
\mathcal A:=\|A(\nabla u)\|_2,$$
where all three norms are on $\pom$. The square-function norm is
finite by hypothesis, and Lemma~\ref{lem:driftInputs} gives
$\mathcal A\lesssim\mathcal S+\kappa\mathcal N<\infty$. Fubini in the definitions of the square
and area functions, the PDE, and the Carleson condition give
\begin{equation}\label{eq:p2WeightedInputs}
\begin{split}
 \|x_n^{1/2}\nabla^2u\|_{L^2(\Omega)}&\lesssim\mathcal S,\qquad
 \|x_n^{3/2}\partial_t\nabla u\|_{L^2(\Omega)}\lesssim\mathcal A,\\
 \|x_n^{1/2}\partial_tu\|_{L^2(\Omega)}
 &\lesssim\mathcal S+\kappa\mathcal N.
\end{split}
\end{equation}
For the last inequality use
$\partial_tu=\divg(A\nabla u)+\mathbf b\cdot\nabla u$ and
\[
 |\partial_tu|\lesssim|\nabla^2u|
                   +( |\nabla A|+|\mathbf b|)|\nabla u|.
\]
The coefficient term is bounded by Carleson embedding and
$N(\nabla u)\lesssim\tilde N(\nabla u)$ at a fixed larger aperture.
These estimates do not use the boundary value of $w$.\medskip

\noindent\emph{Time regularization.} Choose real even functions
$p_j\in C_0^\infty(\R\setminus\{0\})$, $0\le p_j\le1$, such that
$p_j(\tau)\to1$ for every $\tau\ne0$. Let $P_j$ be their Fourier
multipliers in time and put $u_j=P_ju$, $w_j=D_t^{1/2}u_j=P_jw$.
These functions are well-defined modulo the energy normalization;
in fact $u_j\in L^2(\Omega)$ because $p_j$ is supported away from zero
and $w\in L^2(\Omega)$. Also
$$w_j\in L^2(\Omega),\qquad
\partial_nw_j=D_t^{1/2}P_j\partial_nu\in L^2(\Omega),$$
since $|\tau|^{1/2}p_j(\tau)$ is bounded for each $j$.
Thus, with $\mathcal H:=L^2(\R^{n-1}\times\R)$,
$w_j\in H^1((0,\infty);\mathcal H)$ and $w_j(0)\in\mathcal H$ exists.
The bounds at this stage may depend on $j$.
The multipliers commute with derivatives and are contractions on each
weighted space in \eqref{eq:p2WeightedInputs}, since the weights are
independent of time. In particular those three estimates hold uniformly
for $u_j$. Here
$\partial_tu_j=P_j[\divg(A\nabla u)+\mathbf b\cdot\nabla u]$;
we do not commute $P_j$ with either coefficient.\medskip

\noindent\emph{The height identities.} Write $h=x_n$ and, with inner
products taken in $\mathcal H$, set
$$F_j(h):=\|w_j(h)\|_{\mathcal H}^2,\qquad
B_j(h):=\|\partial_nw_j(h)\|_{\mathcal H}^2.$$
The functions $F_j$ and $B_j$ are locally absolutely continuous on
$(0,\infty)$, and $F_j,B_j,F_j'\in L^1(0,\infty)$.
For $B_j$ this local regularity follows from the Hessian bound on every
strip $a<h<b$ with $a>0$ and the bounded time-frequency support.
Self-adjointness of $D_t^{1/2}$ gives, for a.e.\ $h>0$,
\begin{equation}\label{eq:p2Bprime}
 B_j'(h)=2\langle\partial_{nn}u_j,
                    -H_t\partial_t\partial_nu_j\rangle_{\mathcal H},
 \qquad
 \int_0^\infty h^2|B_j'(h)|\,dh\lesssim\mathcal S\mathcal A.
\end{equation}
The last estimate is Cauchy--Schwarz with $h^2=h^{1/2}h^{3/2}$ and
\eqref{eq:p2WeightedInputs}.
Since $B_j'\in L^1(1,\infty)$ and $B_j\in L^1(0,\infty)$, its limit
at infinity is zero. Therefore
$B_j(h)=-\int_h^\infty B_j'(s)\,ds$, and Tonelli's theorem gives
\begin{equation}\label{eq:p2NormalEnergy}
 \int_0^\infty hB_j(h)\,dh
 \le\frac12\int_0^\infty s^2|B_j'(s)|\,ds
 \lesssim\mathcal S\mathcal A.
\end{equation}
In particular the weighted normal energy is a conclusion of the calculation.

Differentiating $F_j'=2\langle w_j,\partial_nw_j\rangle_{\mathcal H}$
once more on interior strips, and moving the half derivative in time,
we obtain
\begin{equation}\label{eq:p2Fsecond}
 F_j''(h)=2B_j(h)+2C_j(h),\qquad
 C_j(h):=\langle-H_t\partial_tu_j,\partial_{nn}u_j\rangle_{\mathcal H}.
\end{equation}
The $L^2$ isometry of $H_t$ and \eqref{eq:p2WeightedInputs} imply
$$\int_0^\infty h|C_j(h)|\,dh
\lesssim\mathcal S\big(\mathcal S+\kappa\mathcal N\big).$$
Together with \eqref{eq:p2NormalEnergy}, this proves
$\int_0^\infty h|F_j''(h)|\,dh<\infty$, uniformly in $j$.
It follows that $F_j'$ has limit zero at infinity: its derivative is
integrable on $(1,\infty)$, and $F_j'\in L^1(0,\infty)$.
Likewise $F_j(\infty)=0$ because $F_j,F_j'\in L^1(0,\infty)$.
Consequently, Fubini with the preceding absolute bound yields
\begin{equation}\label{eq:p2TraceRegularized}
\begin{split}
 \|w_j(0)\|_{\mathcal H}^2
 &=-\int_0^\infty F_j'(h)\,dh
   =\int_0^\infty hF_j''(h)\,dh\\
 &\lesssim\mathcal S\mathcal A+
       \mathcal S\big(\mathcal S+\kappa\mathcal N\big)
 \lesssim\mathcal S^2+\kappa^2\mathcal N^2,
\end{split}
\end{equation}
where the last step uses the drifted area bound
\eqref{eq:driftInputs}. These identities account for
both height endpoints, with no boundary integrability assumed for $w$.
\medskip

\noindent\emph{Identification of the trace.} Let $f=\operatorname{Tr}u$
be the energy trace. The energy trace theorem and approximation by smooth
functions give
$$w_j(0)=D_t^{1/2}P_jf.$$
Indeed, time multipliers commute with the energy trace, and $D_t^{1/2}P_j$
is bounded on the energy space for each fixed $j$.
Moreover $P_jf\to f$ in
$\dot H^{1/4}_{\partial_t-\Delta_x}(\pom)$.
The multiplier inequality
$$|\tau|\le\big(|\xi|^4+\tau^2\big)^{1/2}$$
shows that $D_t^{1/2}$ maps this space continuously into
$\dot H^{-1/4}_{\partial_t-\Delta_x}(\pom)$.
Thus $w_j(0)\to D_t^{1/2}f$ in distributions.
By \eqref{eq:p2TraceRegularized} a subsequence also converges weakly in
$L^2(\pom)$, so its limit is $D_t^{1/2}f$ and
$$\|D_t^{1/2}f\|_2^2\lesssim
\mathcal S^2+\kappa^2\mathcal N^2.$$
Finally, vertical translations of $u$ converge to $u$ in the energy norm
as their heights tend to zero. Continuity of the energy trace and the same
multiplier bound identify $D_t^{1/2}f$ with the distributional boundary
trace of $D_t^{1/2}u$. This proves \eqref{A2}.\medskip

\subsection{$p>2$ with drift}\label{P>2}
Throughout this subsection let $u$ be an energy solution of
$\mathcal L_{\mathbf b}u=0$ under the hypotheses of Theorem~\ref{Thalf},
with $\tilde N(\nabla u),S(\nabla u)\in L^p(\pom)$.
Write $w=D_t^{1/2}u$ and $b=w|_{\pom}$; the scalar $b$ denotes
the boundary half derivative, while $\mathbf b$ denotes the drift.
The constants may depend on $K_{\mathbf b}$ and both Carleson norms;
no smallness is assumed. We shall prove
\begin{equation}\label{eq:target}
\|D^{1/2}_t u\big|_{\pom}\|_{L^p(\pom)}\lesssim \|\tilde N(\nabla u)\|_{L^p(\pom)}+\|S(\nabla u)\|_{L^p(\pom)},
\end{equation}
Equation \eqref{A1a} gives the two maximal bounds in Theorem~\ref{Thalf}.
Only when $\mathbf b=0$ do we use Theorem~\ref{thm.S<Np} to
bound the right-hand side by $C\|\tilde N(\nabla u)\|_{L^p}$. The $L^2$ proof \eqref{A2} uses
self-adjointness of $D_t^{1/2}$; pairing against $|b|^{p-2}b$ does not
permit the same argument. Instead we prove a Fefferman-Stein estimate
on finite cubes and let their sides tend to infinity, using the
normalization supplied by Lemma \ref{lem:traceNormalization}.
We work on $\om=\R^n_+\times\R$; the drifted graph-domain
transfer is verified in Proposition~\ref{prop:halfDrift}.

\subsubsection*{Notation}
$\pom=\R^{n-1}\times\R$ carries the parabolic metric and the measure
$d\sigma=dx\,dt$; a boundary parabolic cube $Q=Q_x\times I_Q$ of side $r$ as in
\eqref{eqdef.bdypcube} has $|I_Q|=2r^2$ and $|Q|\sim r^{n+1}$, with centre
$(x_Q,t_Q)$. We write $T(Q):=Q_x\times(0,r)\times\R$ for the (time-unrestricted)
box over $Q$ and $T^*(Q):=Q_x\times(0,r)\times CI_Q$ for its restriction to a
fixed dilate of the time interval of $Q$, and recall $\delta(X,t)=x_n$ on $\R^n_+\times\R$. The functionals
$\tilde N$, $N$, $S$ and $A$ are those of Section \ref{Def}, with an aperture
large enough for Lemma \ref{lem:vert}; by the remark following
\eqref{A<S+cN.Lp} the enlargements of aperture we make below change the $L^p$
norms of all of them only by a constant. Lemma \ref{lem:driftInputs} gives
$N(\nabla u)\lesssim\tilde N(\nabla u)$ pointwise for solutions, after such an
enlargement. The same lemma supplies the area bound
\eqref{eq:driftInputs}, without a smallness assumption. $M$ is the parabolic Hardy-Littlewood
maximal operator on $\pom$, $M_x$ and $M_t$ the maximal operators in the spatial
and the time variables separately, and
\[
M^{\sharp}:=M_x\circ M_t ,
\]
\emph{in that order}: the inner average is in time over a long interval, the
outer one in space over a ball of radius comparable to the side of the cube.
$M^{\sharp}$ is bounded on $L^\theta(\pom)$ for every $\theta>1$, being a
composition of two maximal operators acting in complementary variables. 

Recall that $\partial_t=D^{1/2}_tH_tD^{1/2}_t$ and \eqref{eq:algebra} holds.

We fix a boundary parabolic cube $Q\subset\pom$ of side $r$. Let $\chi=\chi(t)$ be
smooth with
\begin{equation}\label{eq:chi}
\chi\equiv1\ \mbox{on }|t-t_Q|\le(2r)^2,\quad
\supp\chi\subset\{|t-t_Q|<(4r)^2\},\quad
|\chi'|\lesssim r^{-2},
\end{equation}
so that the plateau of $\chi$ contains a fixed dilate of the time interval $I_Q$
of $Q$. Since $D^{1/2}_t$ annihilates constants, we can write
\begin{equation}\label{eq:split}
b=\Big[\ \underbrace{D^{1/2}_t\big[(u-c_Q)\chi\big]}_{=:\ w_{\mathrm{loc}}}
\ +\ \underbrace{D^{1/2}_t\big[(u-c_Q)(1-\chi)\big]}_{=:\ w_{\mathrm{far}}}\ \Big]\Big|_{\pom}
\end{equation}
for any constant $c_Q$, fixed in \eqref{eq:cQ} below. Here $D^{1/2}_t$ acts in
$t$ alone at each fixed $X$, so $w_{\mathrm{loc}}$ and $w_{\mathrm{far}}$ are
defined at every interior point, and on $\pom$ they are given by the same two
formulas applied to the boundary trace of $u$, as Step 0 of the proof of
Proposition \ref{prop:L} shows for the local part. The far part will be
evaluated at interior points, at height $\sim r$. We shall prove the following estimates.
\begin{itemize}
\item \textbf{Local}: $\fint_Q|w_{\mathrm{loc}}|^2\,d\sigma$
is shown to be controlled by the $L^2$-averaged maximal functions of $\tilde N(\nabla u)$,
$S(\nabla u)$, $A(\nabla u)$ at any point of $Q$.
\item \textbf{Far}: $w_{\mathrm{far}}$ behaves like a constant on $Q$
up to an $M^{\sharp}[\tilde N(\nabla u)]$ error, on average over $Q$.
\end{itemize}

\noindent {\bf Interior and local estimates.} The estimates of this part are local estimates on a single weak solution.

\begin{lemma}[vertical and horizontal transport]\label{lem:vert}
Let $Y=(y,y_n)$ with $0\le y_n\le r$ and let $\hat Y=(y,r)$. Let $Z=(z,r)$ with
$|z-y|\le Cr$. Then, for a.e.\ $s$,
\begin{equation}\label{eq:vert}
|u(Y,s)-u(Z,s)|\ \lesssim_C\ r\,\tilde N(\nabla u)(y,s),
\end{equation}
where $\tilde N$ is taken with an aperture $a\gtrsim C$.
\end{lemma}

\begin{proof}
Join $Y\to\hat Y\to Z$. Every point $(y,\rho)$, $\rho\le r$, of the vertical line
segment lies in $\Gamma(y,s)$, and every point of the horizontal segment lies at
height $r$ within distance $Cr$ of $y$, hence in $\Gamma(y,s)$ once the aperture $a$ of the cone is
$a\gtrsim C$. We integrate $|\nabla u|$ along the two segments and apply
Lemma \ref{lem:driftInputs} on Whitney balls centred at the segment points, whose
doubles lie in $\Gamma(y,s)$ (after if needed a further enlargement of $a$). Each line segment contributes at most
$Cr\sup|\nabla u|$.
\end{proof}

We now fix the constant in \eqref{eq:split}. Let $\varphi_Q\in C_0^\infty(\R^n_+)$
be a bump function in spatial variables only with
\begin{equation}\label{eq:bump}
\varphi_Q\ge0,\quad\int\varphi_Q=1,\quad
\supp\varphi_Q\subset B_{r/2}(\hat X_Q),\quad
\|\varphi_Q\|_\infty\lesssim r^{-n},\quad
\|\nabla\varphi_Q\|_\infty\lesssim r^{-n-1},
\end{equation}
where $\hat X_Q:=(x_Q,r)$, and set
\begin{equation}\label{eq:cQ}
\mathcal U_Q(s):=\int_{\R^n_+}u(X,s)\varphi_Q(X)\,dX,
\qquad c_Q:=\mathcal U_Q(t_Q) .
\end{equation}
Note $\supp\varphi_Q$ sits at heights between $r/2$ and $3r/2$, and is therefore
compactly contained in $\R^n_+$. Write
$D_Q:=4Q_x\times(0,4r)$, $I^*_Q:=\{|s-t_Q|<(4r)^2\}$ and
$W_Q:=B_{r/2}(\hat X_Q)\times I^*_Q$.

\begin{lemma}[oscillation of $u$]\label{lem:osc}
Set $\Theta_Q^2:=\fint_{CQ}\tilde N(\nabla u)^2$. Then
$\mathcal U_Q\in W^{1,1}_{\loc}(\R)$ with
\begin{equation}\label{eq:Uprime}
\mathcal U_Q'(s)=-\int_{\R^n_+}A\nabla u\cdot\nabla\varphi_Q
                   +\int_{\R^n_+}(\mathbf b\cdot\nabla u)\varphi_Q ,
\end{equation}
and, with $c_Q$ as in \eqref{eq:cQ},
\begin{align}
\int_{D_Q}|u(X,s)-\mathcal U_Q(s)|^2dX
&\ \lesssim\ r^2\int_{D_Q}|\nabla u(X,s)|^2dX
&&\mbox{for a.e.\ }s,\label{eq:poin}\\
|\mathcal U_Q(s)-\mathcal U_Q(s')|
&\ \lesssim\ r\Big(\fint_{W_Q}|\nabla u|^2\Big)^{1/2}\ \lesssim\ r\,\Theta_Q
&&\mbox{for }s,s'\in I^*_Q,\label{eq:timeosc}\\
\iint_{D_Q\times I^*_Q}|u-c_Q|^2\,dX\,dt&\ \lesssim\ r^{n+4}\,\Theta_Q^2 .
&&\label{eq:osc}
\end{align}
\end{lemma}

\begin{proof}
$D_Q$ is a convex box of diameter $\lesssim r$ and $u(\cdot,s)\in W^{1,2}(D_Q)$
for a.e.\ $s$. Poincar\'e in space gives \eqref{eq:poin} with
$\langle u\rangle_{D_Q}$ in place of $\mathcal U_Q(s)$; since $\varphi_Q$ is a
probability density supported in $D_Q$ with
$\|\varphi_Q\|_\infty\lesssim|D_Q|^{-1}$,
$|\langle u\rangle_{D_Q}-\mathcal U_Q(s)|
\lesssim(\fint_{D_Q}|u-\langle u\rangle_{D_Q}|^2)^{1/2}$. Note that no boundary condition on $\partial D_Q$ is used,
and $D_Q$ may touch $\pom$.

Identity \eqref{eq:Uprime} is the weak formulation
$\iint_\om[A\nabla u\cdot\nabla\Phi-u\,\partial_t\Phi
                   -(\mathbf b\cdot\nabla u)\Phi]=0$ tested against
$\Phi(X,t)=\varphi_Q(X)\eta(t)$, $\eta\in C_0^\infty(\R)$, admissible precisely
because $\supp\varphi_Q\Subset\R^n_+$. On this support,
$|\mathbf b|\le C K_{\mathbf b}/r$, so
$|\mathbf b|\,|\varphi_Q|\lesssim K_{\mathbf b}r^{-n-1}$.
Integrating the two terms in $s$ over $W_Q$, of measure
$\sim r^{n+2}$, gives the first inequality in \eqref{eq:timeosc}. Since $W_Q$
sits at height $\sim r$, this is the interior oscillation in time of the spatial
average of a solution over a Whitney region, \cite[Lemma 3.1]{KL}, in the form in
which \cite{Din23} uses it; $W_Q$ is covered by boundedly many Whitney balls,
each in the cone over a fixed fraction of $CQ$, which gives the second
inequality.

For the spatial term in \eqref{eq:osc}, cover $D_Q\times I^*_Q$ by Whitney boxes at dyadic heights $h\le4r$. At each height their boundary shadows have bounded overlap inside a fixed dilate $CQ$, and the ratio of box volume to shadow measure is comparable to $h$. Hence
$$\iint_{D_Q\times I^*_Q}|\nabla u|^2
\lesssim\sum_{h\le4r}h\int_{CQ}\tilde N(\nabla u)^2
\lesssim r\,|Q|\Theta_Q^2.$$
Here the sum is over dyadic heights, and $C$ is chosen large enough to contain all the shadows. Split
$u-c_Q=(u-\mathcal U_Q(s))+(\mathcal U_Q(s)-\mathcal U_Q(t_Q))$.
Integrating \eqref{eq:poin} over $I^*_Q$ and using the preceding estimate bounds the first term by $Cr^{n+4}\Theta_Q^2$. Equation \eqref{eq:timeosc} and $|D_Q||I^*_Q|\sim r^{n+2}$ give the same bound for the second term. This proves \eqref{eq:osc}.
\end{proof}

\begin{remark} A bump function of radius $\sim r$
compactly contained in the domain exists at every scale on $\R^n_+\times\R$, but
on a bounded cylinder only for $r\lesssim\diam\OO$. This is the place we use that our domain is unbounded.
\end{remark}

\begin{lemma}[normalization of the half derivative]\label{lem:traceNormalization}
Let $2\le p<\infty$, let $u$ solve $\mathcal L_{\mathbf b}u=0$ in $\R^n_+\times\R$ under the
coefficient hypotheses of this subsection, and put
$\Theta=\tilde N(\nabla u)$ and $E_p=\|\Theta\|_{L^p(\pom)}<\infty$.
Then the boundary trace $f=u|_{\pom}$ has a distributional half derivative
$b=D_t^{1/2}f$. This is also the distributional boundary trace of
$D_t^{1/2}u$.

Write $D=n+1$ for the homogeneous dimension of $\pom$. Fix a smooth test
function $\psi$ supported in the unit boundary cube, and set
\[
 \psi_R(x,t)=R^{-D}\psi\big((x-x_0)/R,(t-t_0)/R^2\big),\qquad R>0.
\]
Then, uniformly in $(x_0,t_0)$ and $R$,
\begin{equation}\label{eq:traceNormalization}
 |\langle b,\psi_R\rangle|\le C_{p,\psi}E_p R^{-D/p}.
\end{equation}
In particular these pairings tend to zero as $R\to\infty$.
\end{lemma}

\begin{proof}
Lemma \ref{lem:driftInputs} and vertical integration give the trace $f$ for a.e.\
$(x,t)$ and
\begin{equation}\label{eq:verticalTraceP}
 |u(x,h,t)-f(x,t)|\lesssim h\Theta(x,t),\qquad h>0.
\end{equation}
In particular $f\in L^p_{\loc}(\pom)$. We must also control its time tails.
Let $Q_j$ be the cube centred at $(x_0,t_0)$ of side $r_j=2^jR$, and let
$a_j=c_{Q_j}$ be the constants in \eqref{eq:cQ}. By Lemmas
\ref{lem:vert} and \ref{lem:osc}, followed by H\"older's inequality,
\begin{equation}\label{eq:traceGrowthP}
 |f(x,t)-a_j|\lesssim r_j\Theta(x,t)+E_p r_j^{1-D/p}
 \quad\text{if }x\in(Q_j)_x,\ |t-t_0|<r_j^2.
\end{equation}
Indeed, transport to the support of $\varphi_{Q_j}$ gives the first term,
and \eqref{eq:timeosc} gives the second since
$\Theta_{Q_j}\le C r_j^{-D/p}E_p$. Comparing the two constants on a common
Whitney region by \eqref{eq:osc} also gives
\[
 |a_{j+1}-a_j|\lesssim E_p r_j^{1-D/p}.
\]
Consequently, on the region in \eqref{eq:traceGrowthP},
\begin{equation}\label{eq:traceGrowthSum}
 |f(x,t)-a_0|\lesssim r_j\Theta(x,t)
             +E_p\sum_{i=0}^j r_i^{1-D/p}.
\end{equation}

The kernel formula and scaling give
\[
 |D_t^{1/2}\psi_R(x,t)|\le C_\psi R^{-D-1}
 \boldsymbol1_{(Q_0)_x}(x)
 \big(1+|t-t_0|/R^2\big)^{-3/2}.
\]
Split time into $|t-t_0|<R^2$ and the shells
$r_{j-1}^2\le|t-t_0|<r_j^2$, $j\ge1$. On the $j$th region the last
bound is at most $C_\psi R^{2-D}r_j^{-3}$. Since
$|(Q_0)_x|\simeq R^{D-2}$, H\"older's inequality and
\eqref{eq:traceGrowthSum} yield
\begin{multline}\label{eq:traceTailSum}
 \iint_{\pom}|f-a_0|\,|D_t^{1/2}\psi_R|\,dx\,dt
 \lesssim E_p R^{-D/p}
 \left[\sum_{j\ge0}2^{-2j/p}
 +\sum_{j\ge0}2^{-j}\sum_{i=0}^j2^{i(1-D/p)}\right]
 \lesssim E_p R^{-D/p}.
\end{multline}
For the second series we interchange the sums:
\[
 \sum_{j\ge0}2^{-j}\sum_{i=0}^j2^{i(1-D/p)}
 =\sum_{i\ge0}2^{i(1-D/p)}\sum_{j\ge i}2^{-j}
 =2\sum_{i\ge0}2^{-iD/p}<\infty.
\]
This covers both $p=2$ and $p=D$.
The same estimates give absolute integrability against the time weight
$(1+|t|)^{-3/2}$ on every bounded spatial set.

Define $\langle b,\psi_R\rangle=\iint(f-a_0)D_t^{1/2}\psi_R$.
This is independent of $a_0$, because the time integral of
$D_t^{1/2}\psi_R$ is zero for each $x$. The preceding bounds, applied to
arbitrary compactly supported smooth tests, define a distribution and
prove \eqref{eq:traceNormalization}. They also justify the time tails
in \eqref{eq:split}, at the boundary and at every fixed positive height.
Finally \eqref{eq:verticalTraceP} gives, for such a test $\zeta$,
\[
 |\langle D_t^{1/2}u(\cdot,h,\cdot)-b,\zeta\rangle|
 \lesssim hE_p\|D_t^{1/2}\zeta\|_{L^{p'}}\longrightarrow0.
\]
This proves the distributional trace assertion without a global energy
or boundary solvability assumption.
\end{proof}

\noindent{\bf The local estimate.}
Throughout this subsection $Q$ is a fixed boundary parabolic cube of side $r$ and
$(x_0,t_0)\in Q$ is a fixed point. Write
\begin{equation}\label{eq:good}
\Xi_Q:=M\big[\tilde N(\nabla u)^2\big]^{1/2}+M\big[S(\nabla u)^2\big]^{1/2}
+M\big[A(\nabla u)^2\big]^{1/2}\Big|_{(x_0,t_0)} .
\end{equation}
The strategy of our proof is to repeat the global calculation done for $w$ in the $L^2$ case but locally
with cutoffs of size $Q$. The bound we obtain is again $L^2$ based, but local.
Since $(x_0,t_0)\in CQ$ and $\int_{CQ}F^2\le|CQ|M[F^2](x_0,t_0)$, the
following four inequalities hold:
For $0<a\le4r$,
\begin{align}
\iint_{4Q_x\times(0,a)\times CI_Q}|\nabla u|^2
&\ \lesssim\ a\,\Xi_Q^2|Q|,
\label{eq:G0}\\
\iint_{T^*(2Q)}|\nabla^2u|^2x_n&\ \lesssim\ \Xi_Q^2|Q|,
\label{eq:G2}\\
\iint_{T^*(2Q)}|\partial_t\nabla u|^2x_n^3&\ \lesssim\ \Xi_Q^2|Q|,
\label{eq:G3}\\
\iint_{T^*(2Q)}\big(|\nabla A|^2+|\mathbf b|^2\big)|\nabla u|^2x_n
&\ \lesssim\ \kappa^2\Xi_Q^2|Q|.
\label{eq:G4}
\end{align}
Estimate \eqref{eq:G0} is proved one dyadic layer at a time: cover
$\{x_n\sim h\}$ by Whitney boxes $W$ of size $\sim h$ with bounded overlap and
let $\Delta_W$ be the boundary cube of side $\sim h$ below $W$; then
$\iint_W|\nabla u|^2\le|W|\fint_{\Delta_W}\tilde N(\nabla u)^2$ with
$|W|/|\Delta_W|\sim h$, and then we sum over $h\le a$. Multiplying by
$x_n^{2j-1}$ and summing again gives, for $j=\tfrac12,1,2$,
\begin{equation}\label{eq:G1}
\iint_{T^*(2Q)}|\nabla u|^2x_n^{2j-1}\ \lesssim\ r^{2j}\Xi_Q^2|Q|.
\end{equation}
For \eqref{eq:G4}, apply Carleson embedding to $\nabla u$
restricted to $T^*(2Q)$. Its nontangential supremum has boundary
support in $CQ$ and is bounded there by $C\tilde N(\nabla u)$,
by Lemma~\ref{lem:driftInputs}. Thus the added drift contribution is
at most $C\|\mu_{\mathbf b}\|_C\int_{CQ}\tilde N(\nabla u)^2$.
Note $\Theta_Q\le\Xi_Q$, so we may use these with $\Xi_Q$.
The bounds \eqref{eq:G0}--\eqref{eq:G1} apply to time-restricted integrals.
The half derivative and the Hilbert transform of a compactly supported
function need not be compactly supported. We therefore keep their
integrals over all of $\R$. Only after applying the global $L^2$ identities
or inequalities below do we use the supports of the inputs, which contain
a factor $\chi$ or $\chi'$, to apply \eqref{eq:G0}--\eqref{eq:G1}.

Set
\[
U:=(u-c_Q)\chi,\qquad V:=D^{1/2}_tU ,
\]
so that $V=w_{\mathrm{loc}}$; we write $V$ for it throughout this subsection,
the estimate being carried out in the interior and only then traced. We see that
\begin{equation}\label{eq:vderivs}
\begin{split}
 \nabla V&=D^{1/2}_t(\chi\nabla u),\qquad D^{1/2}_tV=-H_t\partial_tU,\\
 D^{1/2}_tV&=-H_t\big[\chi\big(\divg(A\nabla u)+\mathbf b\cdot\nabla u\big)\big]
             -H_t\big[(u-c_Q)\chi'\big].
\end{split}
\end{equation}
Let $\Psi(X)=\psi_1(x')\psi_2(x_n)$ be independent of $t$, with
\begin{equation}
\Psi\equiv1\ \mbox{on }Q_x\times(0,r),\qquad
\supp\Psi\subset2Q_x\times(0,2r),\qquad
|\nabla\Psi|\lesssim r^{-1},
\end{equation}
so in particular $\psi_2\equiv1$ on $[0,r)$ and
\begin{equation}\label{eq:dnPsi}
\partial_n\Psi\equiv0\ \mbox{on }\{x_n<r\},\qquad
\supp\partial_n\Psi\subset2Q_x\times[r,2r).
\end{equation}
Fix an integer $m\ge2$ and put $\mathcal R:=\supp\partial_n\Psi\times\R$, a
region contained in $2Q_x\times[r,2r)\times\R$ by \eqref{eq:dnPsi}. Both
integrations by parts below are in $x_n$, so $\nabla\Psi$ enters only through
$\partial_n\Psi$ and every error term they produce is supported on $\mathcal R$, at
height $\sim r$; this is what keeps the two auxiliary lemmas below off the
boundary.

\begin{proposition}[local estimate]\label{prop:L}
Let $\om=\R^n_+\times\R$ and let $u$ be an energy solution of
$\mathcal L_{\mathbf b}u=0$ under the coefficient and drift hypotheses
of Theorem~\ref{Thalf}, with
$Q$, $\Psi$, $\Xi_Q$ and $w_{\mathrm{loc}}$ as above. Then
\[
\int_{\pom}w_{\mathrm{loc}}^2\,\Psi^m\,d\sigma\ \lesssim\ \Xi_Q^2|Q|,
\qquad\mbox{hence}\qquad
\fint_Q|w_{\mathrm{loc}}|^2\,d\sigma\ \lesssim\ \Xi_Q^2 .
\]
\end{proposition}

Before we prove this claim we state two further auxiliary estimates.

\begin{lemma}[$L^2$ size of $V$ at height $\sim r$]\label{lem:A}
$\displaystyle\iint_{\mathcal R}V^2\,dX\,dt\ \lesssim\ r\,\Xi_Q^2|Q|$.
\end{lemma}

\begin{proof}
Points $Y=(y,h)$ of $\mathcal R$ have $h\sim r$, so $\mathcal R$ is at a
positive distance from $\pom$ and a spatial cutoff around it may be taken
to be compactly supported in $\R^n_+$. Let $\Phi=\Phi(X)$ be such a cutoff,
independent of $t$, with $\Phi\equiv1$ on $\supp\partial_n\Psi$,
$\supp\Phi\subset 3Q_x\times(r/2,3r)\Subset\R^n_+$ and
$|\nabla\Phi|\lesssim r^{-1}$. Localizing the energy function $u$ in
space and time gives $U\Phi\in L^2$ with $D_t^{1/2}(U\Phi)\in L^2$;
the latter also follows by time mollification and passage to the limit.
Thus $H_tU\Phi^2$ is an admissible energy test function. Since $\Phi$
is independent of $t$, \eqref{eq:algebra} gives
\[
\iint_\om|D^{1/2}_tU|^2\Phi^2
=\iint_\om\big(\partial_tU\big)\big(H_tU\big)\Phi^2 ,
\]
by the properties of $D^{1/2}_t$ and $H_t$. This is analogous to the way \cite{Din23} estimates the local part of $D^{1/2}_tu$ and by
which \cite{DLP1} localises $D^{1/2}_t$ in space. Using
$\partial_tU=\chi[\divg(A\nabla u)+\mathbf b\cdot\nabla u]+(u-c_Q)\chi'$
and integrating
the first term by parts in $X$ this gives
\begin{multline*}
\iint_\om|D^{1/2}_tU|^2\Phi^2
=-\iint\chi A\nabla u\cdot H_t(\chi\nabla u)\Phi^2\\
-2\iint\chi A\nabla u\cdot\nabla\Phi\,(H_tU)\Phi
+\iint(u-c_Q)\chi'(H_tU)\Phi^2
+\iint\chi(\mathbf b\cdot\nabla u)(H_tU)\Phi^2 .
\end{multline*}
As $H_t$ is an isometry of $L^2(dt)$, after Cauchy-Schwarz it disappears.
Since $\supp\Phi\subset D_Q$ and $\chi$ is supported in $I^*_Q$, \eqref{eq:G0}
with $a=3r$ gives $\iint\chi^2|\nabla u|^2\Phi^2\lesssim r\Xi_Q^2|Q|$, and by
\eqref{eq:osc}, $\iint_{\supp\Phi}|U|^2\le\iint_{D_Q\times I^*_Q}|u-c_Q|^2
\lesssim r^{n+4}\Xi_Q^2=r^2\cdot r\Xi_Q^2|Q|$. The first three terms are
therefore bounded respectively by
\[
r\Xi_Q^2|Q|,\qquad
r^{-1}\big(r\Xi_Q^2|Q|\big)^{1/2}\big(r^{n+4}\Xi_Q^2\big)^{1/2}
= r\Xi_Q^2|Q|,\qquad
r^{-2}\,r^{n+4}\Xi_Q^2= r\Xi_Q^2|Q|,
\]
using $|\chi'|\lesssim r^{-2}$ in the third term. On the support
of $\Phi$, $|\mathbf b|\lesssim K_{\mathbf b}/r$, so the fourth term
is bounded by
\[
 \frac{C K_{\mathbf b}}r
 \left(\iint\chi^2|\nabla u|^2\Phi^2\right)^{1/2}
 \left(\iint|U|^2\Phi^2\right)^{1/2}
 \lesssim K_{\mathbf b}r\Xi_Q^2|Q|.
\]
Here the Hilbert transform is removed by its $L^2(dt)$ isometry
before restricting time to the support of $\chi$.
The left-hand side dominates
$\iint_{\mathcal R}V^2$, thus proving Lemma \ref{lem:A}.
\end{proof}

\begin{lemma}[weighted square function of $V$]\label{lem:sqv}
$\displaystyle\iint_{T(2Q)}|\nabla V|^2x_n^2\,dX\,dt\lesssim r\Xi_Q^2|Q|$.
\end{lemma}

\begin{proof}
Since $u$ has finite energy, $\nabla u(X,\cdot)\in L^2(dt)$ for a.e.\ $X$.
By \eqref{eq:vderivs}, $\nabla V=D^{1/2}_t(\chi\nabla u)$, so Plancherel
in $t$, first for time mollifications if necessary, gives
$\|D^{1/2}_tf\|^2_{L^2_t}=\int|\tau||\hat f|^2
\le\|f\|_{L^2_t}\|\partial_tf\|_{L^2_t}$. We apply Cauchy-Schwarz in $X$ with
$x_n^2=x_n^{1/2}\cdot x_n^{3/2}$ which gives
\[
\iint|\nabla V|^2x_n^2\le
\Big(\iint|\chi\nabla u|^2x_n\Big)^{1/2}
\Big(\iint\big(|\chi\partial_t\nabla u|+|\chi'\nabla u|\big)^2x_n^3\Big)^{1/2}.
\]
After applying \eqref{eq:G1} with $j=1$ to the first factor and \eqref{eq:G3}
together with \eqref{eq:G1} for $j=2$ (using $|\chi'|\lesssim r^{-2}$) to
the second we obtain the bound.
\end{proof}

\begin{proof}[Proof of Proposition \ref{prop:L}]
The identity \eqref{eq:ident} below and the two integrations by parts that
produce it are analogous to the proof of \eqref{A2} given earlier. We have however replaced
$D^{1/2}_tu$ by $V=D^{1/2}_tU=D^{1/2}_t((u-c_Q)\chi)$. Whatever weights we used previously in the form of $x_n^k$
are replaced by  $\Psi^mx_n^k$ and the global norms
$\|S(\nabla u)\|_{L^2}$, $\|A(\nabla u)\|_{L^2}$ are replaced by the local bounds
\eqref{eq:G2}--\eqref{eq:G1}. We must also justify the truncation at $x_n=0$ and estimate the error terms produced by the two cutoffs.
We give these terms in full.

\smallskip\emph{Step 0: truncation and the boundary value.} Both integrations by parts below are in
$x_n$ and cross $\{x_n=0\}$, where $\Psi\equiv1$. Each integrand carries a
positive power of $x_n$ there, but the factor $\partial_nV$ it multiplies is
not known to be bounded as $x_n\to0$, so the boundary terms must be produced and discarded
rather than ignored. For this reason we work on $\om_\epsilon:=\{x_n>\epsilon\}$, $0<\epsilon<r$,
and set $F(\epsilon):=\int_{\{x_n=\epsilon\}}V^2\Psi^m\,dx'dt\ge0$. We may assume
$\Xi_Q<\infty$, so that $\tilde N(\nabla u)\in L^2(CQ)$ for every fixed dilate of $Q$.
On each closed height interval inside $(0,r)$, local energy regularity
of $U$ and Lemma \ref{lem:sqv} give $V\Psi^{m/2}$ and its height
derivative in $L^2$.
Its squared $L^2$ norm $F$ is therefore locally absolutely continuous.
We claim the following:
\begin{enumerate}[(a)]
\item\label{it:liminf} There is a sequence $\epsilon_j\downarrow0$ with
$\epsilon_jF'(\epsilon_j)\le o(1)$: $F$ is locally absolutely continuous, and if
$\epsilon F'(\epsilon)\ge\delta_0>0$ on $(0,\epsilon_0)$ then
$F(\epsilon)\le F(\epsilon_0)-\delta_0\log(\epsilon_0/\epsilon)\to-\infty$, which would
contradict $F\ge0$.
\item\label{it:trace} If $\liminf_jF(\epsilon_j)<\infty$ then
$\int_{\pom}w_{\mathrm{loc}}^2\Psi^m\le\liminf_jF(\epsilon_j)$.
\end{enumerate}

We prove \ref{it:trace} in the weak topology of $L^2(\pom)$, rather than by
comparing $V(\cdot,\epsilon,\cdot)$ with its boundary value at almost every
point: the energy class gives the trace of $u$, and not a vertical limit of its
half time derivative. Since $\psi_2\equiv1$ on $[0,r)$, the cutoff $\Psi$ does not
depend on $x_n$ on $\{x_n<r\}$, and we write $\Psi_0(x)$ for it there, so that
$\Psi=\Psi_0$ on $\pom$. Pass first to a subsequence along which $F(\epsilon_j)$
tends to its lower limit, and then to a further one along which the functions
$V(\cdot,\epsilon_j,\cdot)\Psi_0^{m/2}$, whose $L^2(\pom)$ norms are
$F(\epsilon_j)^{1/2}$, converge weakly in $L^2(\pom)$. Their limit $G$ satisfies
$\|G\|^2_{L^2(\pom)}\le\liminf_jF(\epsilon_j)$, and it remains to identify $G$
with $w_{\mathrm{loc}}\Psi_0^{m/2}$.

Lemma \ref{lem:driftInputs} makes $h\mapsto u(x,h,t)$ Lipschitz on $(0,r)$ with constant
$\lesssim\tilde N(\nabla u)(x,t)$, which is finite at a.e.\ point of $CQ$. Hence
the trace $f_u:=\lim_{h\downarrow0}u(\cdot,h,\cdot)$ exists a.e.\ there and
satisfies $|u(x,h,t)-f_u(x,t)|\lesssim h\tilde N(\nabla u)(x,t)$, which is
\eqref{eq:verticalTraceP}. Let $\zeta\in C_0^\infty(\pom)$. The factor
$\Psi_0^{m/2}$ is independent of $t$, so $D^{1/2}_t$ falls on $\zeta$ alone, and
its self-adjointness at the fixed height $\epsilon$ gives
\[
\int_{\pom}V(x,\epsilon,t)\Psi_0(x)^{m/2}\zeta(x,t)\,dx\,dt
=\int_{\pom}U(x,\epsilon,t)\Psi_0(x)^{m/2}\big(D^{1/2}_t\zeta\big)(x,t)\,dx\,dt .
\]
The right-hand integrand is supported in $\supp\Psi_0\times I^*_Q$, a compact
subset of $CQ$ on which $D^{1/2}_t\zeta$ is bounded. There
$U(\cdot,\epsilon,\cdot)\to(f_u-c_Q)\chi$ in $L^1$ as $\epsilon\downarrow0$, by
the vertical estimate above and $\tilde N(\nabla u)\in L^2(CQ)$. Therefore
$G=D^{1/2}_t\big[(f_u-c_Q)\chi\big]\Psi_0^{m/2}$ and
$\|G\|^2_{L^2(\pom)}=\int_{\pom}w_{\mathrm{loc}}^2\Psi^m$, which proves
\ref{it:trace}. This is also the sense in which the local part of
\eqref{eq:split} is restricted to $\pom$: its boundary value is the half time
derivative of the trace of $u$.

By \eqref{eq:dnPsi}, $\partial_n\Psi\equiv0$ on $\{x_n<r\}$, so for
$\epsilon<r$
\begin{equation}\label{eq:Fprime}
F'(\epsilon)=2\int_{\{x_n=\epsilon\}}V\,(\partial_nV)\,\Psi^m .
\end{equation}
All estimates below are uniform in $\epsilon$, each being an integral of an
absolute value over a subregion of $\om$.

\smallskip\emph{Steps 1--2: the identity.} Since $\Psi$ has compact support in
the $X$ variable,
$$F(\epsilon)=-\iint_{\om_\epsilon}\partial_n(V^2\Psi^m)\,dX\,dt.$$
Inserting $\partial_n(x_n)$ into this solid integral as in the proof of
\eqref{A2}, and using \eqref{eq:Fprime}, we obtain
\begin{equation}\label{eq:ident}
F(\epsilon)=\epsilon F'(\epsilon)
+\underbrace{2\iint_{\om_\epsilon}|\partial_nV|^2\Psi^mx_n}_{=:A_1}
+\underbrace{2\iint_{\om_\epsilon}V\,D^{1/2}_t(\chi\partial^2_{nn}u)\Psi^mx_n}_{=:A_2}
+E_0+E_1,
\end{equation}
where $E_0:=-m\iint_{\om_\epsilon}V^2\Psi^{m-1}\partial_n\Psi$ and
$E_1:=2m\iint_{\om_\epsilon}V(\partial_nV)\Psi^{m-1}(\partial_n\Psi)x_n$ are
supported in $\mathcal R$. The boundary term at $x_n=\epsilon$ is exactly
$\epsilon F'(\epsilon)$; it has no sign, and is disposed of along the sequence
of Step 0\ref{it:liminf}. By Lemma \ref{lem:A},
$|E_0|\lesssim r^{-1}\iint_{\mathcal R}V^2\lesssim\Xi_Q^2|Q|$, and
by Cauchy-Schwarz with Lemmas \ref{lem:A} and \ref{lem:sqv},
$|E_1|\lesssim r^{-1}(r\Xi_Q^2|Q|)^{1/2}(r\Xi_Q^2|Q|)^{1/2}
=\Xi_Q^2|Q|$.

\smallskip\emph{Step 3: the term $A_2$.} $\Psi^mx_n$ is independent of $t$ and
$\om_\epsilon$ is a product region in $t$, so $D^{1/2}_t$ may be moved onto $V$
by self-adjointness. Then \eqref{eq:vderivs} replaces $D^{1/2}_tV$ and
$A_2=A_2'+A_{2,\mathbf b}+E_2$ with
\[
\begin{split}
 A_2'&=-2\iint H_t\big[\chi\divg(A\nabla u)\big]
                         (\chi\partial^2_{nn}u)\Psi^mx_n,\\
 A_{2,\mathbf b}&=-2\iint H_t[\chi\mathbf b\cdot\nabla u]
                         (\chi\partial^2_{nn}u)\Psi^mx_n,\\
 E_2&=-2\iint H_t\big[(u-c_Q)\chi'\big]
                         (\chi\partial^2_{nn}u)\Psi^mx_n.
\end{split}
\]
$A_2'$ is the term estimated in \eqref{A2} since $H_t$ is an isometry of $L^2(dt)$ and
the weight is $t$-independent, so it may be discarded inside each
Cauchy-Schwarz factor, and
$|\divg(A\nabla u)|\lesssim|\nabla^2u|+|\nabla A||\nabla u|$ with
\eqref{eq:G2}, \eqref{eq:G4} gives
$|A_2'|\lesssim(1+\kappa)\Xi_Q^2|Q|$.
For the drift term, the weight $\Psi^mx_n$ is independent of time.
Apply Cauchy--Schwarz over all time and then the isometry of $H_t$:
\[
 |A_{2,\mathbf b}|
 \le2\left(\iint\chi^2|\mathbf b|^2|\nabla u|^2\Psi^mx_n\right)^{1/2}
       \left(\iint\chi^2|\partial_{nn}u|^2\Psi^mx_n\right)^{1/2}
 \lesssim\kappa\Xi_Q^2|Q|,
\]
by \eqref{eq:G4} and \eqref{eq:G2}. In particular no derivative
is moved onto $\mathbf b$, and no smallness is needed.
For $E_2$, note
that $\supp\chi'$ is exactly where $\chi$ may vanish, so the correct input is
the $\chi$-free Lemma \ref{lem:osc}:
$\iint|u-c_Q|^2|\chi'|^2x_n\le4r\|\chi'\|^2_\infty
\iint_{D_Q\times I^*_Q}|u-c_Q|^2\lesssim r^{-4}\cdot r^4\Xi_Q^2|Q|$,
whence $|E_2|\lesssim\Xi_Q^2|Q|$. 

\smallskip\emph{Step 4: the term $A_1$.} As in \eqref{A2} we insert $x_n^2$;
with $f:=|\partial_nV|^2\Psi^m\ge0$ and $\partial_n(x_n^2)=2x_n$,
\[
A_1=\iint_{\om_\epsilon}f\,\partial_n(x_n^2)
=-\epsilon^2\!\!\int_{\{x_n=\epsilon\}}\!\!f-\iint_{\om_\epsilon}(\partial_nf)x_n^2
\ \le\ -\iint_{\om_\epsilon}(\partial_nf)x_n^2 .
\]
Here the boundary term has a sign and is dropped; monotone convergence as
$\epsilon\to0$ then makes the finiteness of $A_1$ a conclusion. Since $\partial_nV=D^{1/2}_t(\chi\partial_nu)$ and
$\partial^2_{nn}V=D^{1/2}_t(\chi\partial^2_{nn}u)$, the two terms of $\partial_nf$
in which $\partial_n$ falls on a $D^{1/2}_t$-factor are equal, so
$A_1\le2A_{11}+E_3$ with
$A_{11}=-\iint D^{1/2}_t(\chi\partial^2_{nn}u)D^{1/2}_t(\chi\partial_nu)\Psi^mx_n^2$
and $E_3=-m\iint|\partial_nV|^2\Psi^{m-1}(\partial_n\Psi)x_n^2$.
Lemma \ref{lem:sqv} gives $|E_3|\lesssim r^{-1}\iint_{\mathcal R}|\nabla V|^2x_n^2
\lesssim\Xi_Q^2|Q|$. Moving one $D^{1/2}_t$ across and using
$D^{1/2}_tD^{1/2}_t=-H_t\partial_t$,
\[
A_{11}=\iint(\chi\partial^2_{nn}u)H_t\big[\chi\partial_t\partial_nu\big]\Psi^mx_n^2
+\iint(\chi\partial^2_{nn}u)H_t\big[\chi'\partial_nu\big]\Psi^mx_n^2.
\]
Finally a Cauchy-Schwarz with $x_n^2=x_n^{1/2}x_n^{3/2}$ and using
\eqref{eq:G2}, \eqref{eq:G3}, \eqref{eq:G1} for $j=2$, gives
$|A_{11}|\lesssim\Xi_Q^2|Q|$.

\smallskip\emph{Step 5: combining the estimates.} Combining estimates given above (with constants independent of
$\epsilon$), we have
$F(\epsilon)\le\epsilon F'(\epsilon)+C(1+\kappa)\Xi_Q^2|Q|$
for $0<\epsilon<r$; evaluating along $\epsilon_j$ and applying Step
0\ref{it:trace} finishes the proof. The second assertion follows since
$\Psi^m\equiv1$ on $Q$.
\end{proof}

\noindent{\bf The far-field estimate.}
The far part will be differenced against its values at the interior points
$Z=(z,r)$, $z\in B_r(x_Q)$, taken at the centre time $t_Q$; we call these the
anchors. Figure \ref{fig:far} shows the configuration. For $s$ in the set
$\{s:\,\chi(s)=1\}$ --- which by \eqref{eq:chi} contains the time interval of $Q$, and in
particular $t_Q$ --- the argument of $D^{1/2}_t$ in
\eqref{eq:split} for $w_{\mathrm{far}}$ vanishes at $\sigma=s$, so the principal value disappears and
\[
w_{\mathrm{far}}(Y,s)=-c_0\int_\R\big(u(Y,\sigma)-c_Q\big)\big(1-\chi(\sigma)\big)
|s-\sigma|^{-3/2}\,d\sigma .
\]
We keep in the integral above the smooth factor $1-\chi$, since replacing it by the indicator of
$\{|\sigma-t_Q|>(2r)^2\}$ would make the estimate of $\mathrm{II}$ below
depend on a point value of $\tilde N(\nabla u)$ at the cutoff point.

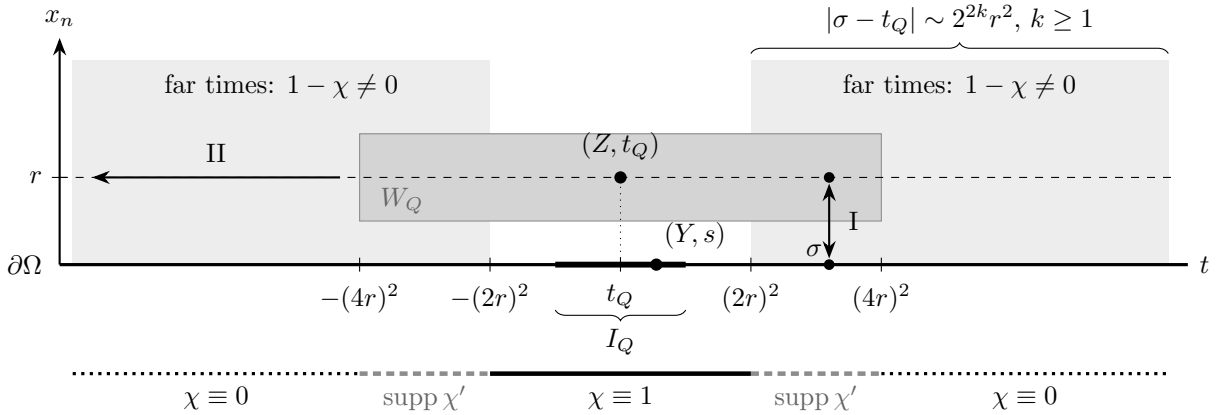
\begin{figure}[htpb]
\centering
\resizebox{0.98\textwidth}{!}{%
\begin{tikzpicture}[x=0.86cm,y=1.15cm,>=Stealth,font=\footnotesize]

\fill[black!7] (-8.4,0) rectangle (-2,2.35);
\fill[black!7] (2,0) rectangle (8.4,2.35);
\node at (-5.2,2.05) {far times: $1-\chi\neq0$};
\node at (5.2,2.05) {far times: $1-\chi\neq0$};

\fill[black!17] (-4,0.5) rectangle (4,1.5);
\draw[black!45] (-4,0.5) rectangle (4,1.5);
\node[black!60] at (-3.35,0.73) {$W_Q$};

\draw[very thick] (-8.6,0) -- (8.7,0) node[right] {$t$};
\draw[thick,->] (-8.6,0) -- (-8.6,2.6) node[above] {$x_n$};
\node[left] at (-8.7,0) {$\pom$};

\draw[dashed] (-8.6,1) -- (8.5,1);
\draw (-8.72,1) -- (-8.48,1);
\node[left] at (-8.7,1) {$r$};

\draw[line width=2.2pt] (-1,0) -- (1,0);
\draw[decorate,decoration={brace,amplitude=4pt,mirror}] (-1,-0.52) -- (1,-0.52)
   node[midway,below=3pt] {$I_Q$};

\foreach \x in {-4,-2,0,2,4} \draw (\x,0.08) -- (\x,-0.08);
\node[below] at (0,-0.1) {$t_Q$};
\node[below] at (2,-0.1) {$(2r)^2$};
\node[below] at (4,-0.1) {$(4r)^2$};
\node[below] at (-2,-0.1) {$-(2r)^2$};
\node[below] at (-4,-0.1) {$-(4r)^2$};

\draw[dotted] (0,0) -- (0,0.94);
\fill (0,1) circle (2.3pt);
\node[above=2pt] at (0,1.02) {$(Z,t_Q)$};
\fill (0.55,0) circle (2.3pt);
\node[above right=0pt] at (0.5,0.06) {$(Y,s)$};

\draw[line width=1.7pt] (-2,-1.25) -- (2,-1.25);
\node at (0,-1.53) {$\chi\equiv1$};
\draw[line width=1.7pt,black!45,dash pattern=on 3pt off 2pt] (2,-1.25) -- (4,-1.25);
\draw[line width=1.7pt,black!45,dash pattern=on 3pt off 2pt] (-4,-1.25) -- (-2,-1.25);
\node[black!60] at (3,-1.53) {$\mathrm{supp}\,\chi'$};
\node[black!60] at (-3,-1.53) {$\mathrm{supp}\,\chi'$};
\draw[dotted,line width=1.2pt] (4,-1.25) -- (8.4,-1.25);
\draw[dotted,line width=1.2pt] (-8.4,-1.25) -- (-4,-1.25);
\node at (6.2,-1.53) {$\chi\equiv0$};
\node at (-6.2,-1.53) {$\chi\equiv0$};

\fill (3.2,0) circle (2.0pt);
\node[above left=-2pt] at (3.18,0.04) {$\sigma$};
\draw[<->,thick] (3.2,0.07) -- (3.2,0.93);
\node[right=2pt] at (3.25,0.5) {I};
\fill (3.2,1) circle (2.0pt);

\draw[<-,thick] (-8.1,1) -- (-4.3,1);
\node[above=1pt] at (-6.2,1.02) {II};

\draw[decorate,decoration={brace,amplitude=4pt}] (2,2.42) -- (8.4,2.42)
   node[midway,above=3pt] {$|\sigma-t_Q|\sim2^{2k}r^2$, $k\ge1$};

\end{tikzpicture}%
}
\caption{The far-field estimate in the $(t,x_n)$ variables. The spatial variable
$x$ is suppressed.}
\label{fig:far}
\end{figure}

The integration by parts in $\sigma$ below generates endpoint terms at
$\sigma\to\pm\infty$. The following lemma gives sequences along which these terms vanish.

\begin{lemma}[the endpoint terms vanish]\label{lem:endpoint}
Let $1<p<\infty$, let $\|\tilde N(\nabla u)\|_{L^p(\pom)}<\infty$, let $Z=(z,r)$
and let $\mathcal K(\sigma)$ be any function with
$|\mathcal K(\sigma)|\lesssim r^2|\sigma-t_Q|^{-3/2}$. Then for a.e.\
$z\in B_r(x_Q)$ there are sequences $\sigma_j\to+\infty$ and
$\sigma_j'\to-\infty$ along which
\[
\big(u(Z,\sigma_j)-c_Q\big)\mathcal K(\sigma_j)\longrightarrow0,
\qquad
\big(u(Z,\sigma_j')-c_Q\big)\mathcal K(\sigma_j')\longrightarrow0 .
\]
\end{lemma}

\begin{proof}
Write $\Theta(z,\sigma):=\tilde N(\nabla u)(z,\sigma)$. By hypothesis
$\Theta\in L^p(\pom)=L^p(\R^{n-1}\times\R)$, so by Fubini
$\Theta(z,\cdot)\in L^p(\R)$ for a.e.\ $z$. Fix one such $z$ and split
\[
u(Z,\sigma)-c_Q=\underbrace{\big[u(Z,\sigma)-\mathcal U_Q(\sigma)\big]}_{\mathrm{(i)}}
+\underbrace{\big[\mathcal U_Q(\sigma)-\mathcal U_Q(t_Q)\big]}_{\mathrm{(ii)}} .
\]

\emph{The contribution of (ii).} By \eqref{eq:Uprime},
\eqref{eq:bump} and Lemma \ref{lem:driftInputs} at scale $r$, available at every time
since $\om=\R^n_+\times\R$ is invariant under translation in $t$,
\begin{align*}
|\mathcal U_Q'(s)|&\le
\left(\Lambda\|\nabla\varphi_Q\|_\infty
       +C K_{\mathbf b}r^{-1}\|\varphi_Q\|_\infty\right)
\int_{\supp\varphi_Q}|\nabla u(\cdot,s)|\\
&\lesssim r^{-1}\sup_{\supp\varphi_Q}|\nabla u(\cdot,s)|
\lesssim r^{-1}\Theta(z,s).
\end{align*}
Hence by H\"older
$|\mathrm{(ii)}|\le\int_{t_Q}^{\sigma}|\mathcal U_Q'|
\lesssim r^{-1}|\sigma-t_Q|^{1-1/p}\|\Theta(z,\cdot)\|_{L^p(\R)}$, and
therefore
\begin{equation}\label{eq:endpt}
|\mathrm{(ii)}\cdot\mathcal K(\sigma)|
\ \lesssim\ r\,|\sigma-t_Q|^{-1/2-1/p}\|\Theta(z,\cdot)\|_{L^p(\R)}
\ \longrightarrow\ 0\qquad(\sigma\to\pm\infty),
\end{equation}
the exponent being negative for every $p\in(1,\infty)$. \medskip

\emph{(i) tends to zero along a sequence.} Every point of a path from $Z$ to a
point of $\supp\varphi_Q$ of length $\lesssim r$ may be taken at height
$\ge r/2$ and within distance $Cr$ of $z$, hence lies in $\Gamma(z,\sigma)$ once
the aperture is $\gtrsim C$; so the argument of Lemma \ref{lem:vert} gives
$|\mathrm{(i)}|\lesssim r\,\Theta(z,\sigma)$, and
$|\mathrm{(i)}\cdot\mathcal K(\sigma)|\lesssim r^3|\sigma-t_Q|^{-3/2}\Theta(z,\sigma)$.
Since $\Theta(z,\cdot)\in L^p(\R)$ we have
$\liminf_{\sigma\to+\infty}\Theta(z,\sigma)=0$, so there is
$\sigma_j\to+\infty$ with $\Theta(z,\sigma_j)\to0$; along that sequence both
$\mathrm{(i)}\cdot\mathcal K$ and, by \eqref{eq:endpt},
$\mathrm{(ii)}\cdot\mathcal K$ tend to $0$. The same argument at $-\infty$ gives
$\sigma_j'$.
\end{proof}

\begin{proposition}[far-field estimate]\label{prop:far}
Let $c^*:=\fint_{B_r(x_Q)}w_{\mathrm{far}}\big((z,r),t_Q\big)\,dz$. Then
\begin{equation}\label{eq:far}
\fint_Q|w_{\mathrm{far}}-c^*|\,d\sigma\ \lesssim\
M^{\sharp}\big[\tilde N(\nabla u)\big](x_Q,t_Q).
\end{equation}
\end{proposition}

\begin{proof}
Fix $(y,s)\in Q$ and $z\in B_r(x_Q)$, write $Y:=(y,0)$ and $Z:=(z,r)$, and
split
\[
w_{\mathrm{far}}(Y,s)-w_{\mathrm{far}}(Z,t_Q)=\mathrm I(y,s,z)+\mathrm{II}(s,z),
\]
\begin{align*}
\mathrm I&:=-c_0\int_\R\big(u(Y,\sigma)-u(Z,\sigma)\big)(1-\chi(\sigma))
|s-\sigma|^{-3/2}d\sigma,\\
\mathrm{II}&:=-c_0\int_\R\big(u(Z,\sigma)-c_Q\big)(1-\chi(\sigma))
\Big[|s-\sigma|^{-3/2}-|t_Q-\sigma|^{-3/2}\Big]d\sigma .
\end{align*}
Averaging over $z$ and then over $(y,s)\in Q$ bounds the left-hand side of
\eqref{eq:far}, since $c^*$ is the $z$-average of $w_{\mathrm{far}}(Z,t_Q)$.

\smallskip\noindent\emph{Term $\mathrm I$: the spatial difference.} Both $Y$ and
$Z$ project into a fixed dilate of $Q_x$ and $|z-y|\lesssim r$, so
Lemma \ref{lem:vert} gives
$|u(Y,\sigma)-u(Z,\sigma)|\lesssim r\,\tilde N(\nabla u)(y,\sigma)$. Since
$1-\chi$ is supported in $\{|\sigma-t_Q|\ge(2r)^2\}$ and
$|s-\sigma|\sim|\sigma-t_Q|$ there, summing dyadically over
$|\sigma-t_Q|\sim2^{2k}r^2$, $k\ge1$,
\begin{multline}\label{eq:termI}
|\mathrm I|\ \lesssim\ r\sum_{k\ge1}(2^{2k}r^2)^{-3/2}(2^{2k}r^2)
\fint_{I_k}\tilde N(\nabla u)(y,\cdot)\\
=\sum_{k\ge1}2^{-k}\fint_{I_k}\tilde N(\nabla u)(y,\cdot)
\ \lesssim\ M_t\big[\tilde N(\nabla u)\big](y,t_Q),
\end{multline}
where $I_k:=\{|\sigma-t_Q|<2^{2k+4}r^2\}$ contains the
corresponding shell and has comparable length; the same dyadic sum bounds
the interior analogue in \cite{Din23}. The gain
$2^{-k}$ is the spatial scale $r$ against the kernel scale $2^kr$.
Averaging over $y\in Q_x$,
\begin{equation}\label{eq:termIav}
\fint_Q|\mathrm I|\,d\sigma\ \lesssim\
\fint_{Q_x}M_t\big[\tilde N(\nabla u)\big](y,t_Q)\,dy
\ \le\ M_x\circ M_t\big[\tilde N(\nabla u)\big](x_Q,t_Q).
\end{equation}

\smallskip\noindent\emph{Term $\mathrm{II}$: the kernel difference.} Since
$|s-t_Q|\lesssim r^2$ while $|\sigma-t_Q|\gtrsim r^2$ on the support of
$1-\chi$, the mean value theorem gives
$|\Delta K(\sigma)|\lesssim r^2|\sigma-t_Q|^{-5/2}$ for
$\Delta K(\sigma):=|s-\sigma|^{-3/2}-|t_Q-\sigma|^{-3/2}$. Hence
\[
\mathcal K(\sigma):=
\begin{cases}
-\displaystyle\int_\sigma^\infty\Delta K(\tau)\,d\tau,
   &\sigma>t_Q+(2r)^2,\\[2mm]
\displaystyle\int_{-\infty}^{\sigma}\Delta K(\tau)\,d\tau,
   &\sigma<t_Q-(2r)^2.
\end{cases}
\]
Thus $\mathcal K'=\Delta K$ on each far interval and
$|\mathcal K(\sigma)|\lesssim r^2|\sigma-t_Q|^{-3/2}$. Integration by parts
gives three terms:
\[
\mathrm{II}=c_0\int(\partial_\sigma u)(Z,\sigma)(1-\chi)\mathcal K\,d\sigma
-c_0\int\big(u(Z,\sigma)-c_Q\big)\chi'(\sigma)\mathcal K\,d\sigma
+\ \mbox{(endpoints at }\sigma\to\pm\infty).
\]

For the first term, $Z$ lies at the fixed height $r$, so Lemma \ref{lem:driftInputs} applies
at scale $r$ at every time $\sigma$, $\om=\R^n_+\times\R$ being invariant under
translation in $t$, and yields
$|\partial_\sigma u(Z,\sigma)|\lesssim r^{-1}\tilde N(\nabla u)(z,\sigma)$.
Hence, dyadically as before,
\begin{equation}\label{eq:termII}
\Big|\int(\partial_\sigma u)(1-\chi)\mathcal K\Big|
\lesssim\sum_{k\ge1}r^{-1}\cdot r^2(2^{2k}r^2)^{-3/2}(2^{2k}r^2)
\fint_{I_k}\tilde N(\nabla u)(z,\cdot)
=\sum_{k\ge1}2^{-k}\fint_{I_k}\tilde N(\nabla u)(z,\cdot),
\end{equation}
which is $\lesssim M_t[\tilde N(\nabla u)](z,t_Q)$; averaging over
$z\in B_r(x_Q)$ gives $M^{\sharp}[\tilde N(\nabla u)](x_Q,t_Q)$ as in
\eqref{eq:termIav}.

For the second term, put $J:=I^*_Q$ and
$\Theta(z,\sigma):=\tilde N(\nabla u)(z,\sigma)$. On $\supp\chi'$ we have
$|\chi'\mathcal K|\lesssim r^{-3}$. The spatial transport estimate and
\eqref{eq:Uprime}, with the bump in \eqref{eq:bump}, give
\[
 |u(Z,\sigma)-c_Q|
 \le |u(Z,\sigma)-\mathcal U_Q(\sigma)|
       +|\mathcal U_Q(\sigma)-\mathcal U_Q(t_Q)|
 \lesssim r\Theta(z,\sigma)+r^{-1}\int_J\Theta(z,\tau)\,d\tau .
\]
Here the path between $Z$ and the support of $\varphi_Q$ stays at height
comparable to $r$, and the interval joining $t_Q$ to $\sigma$ is contained
in $J$. Keeping the time integral, we obtain
\begin{equation}\label{eq:transitionaverage}
 \left|\int(u(Z,\sigma)-c_Q)\chi'(\sigma)\mathcal K(\sigma)\,d\sigma\right|
 \lesssim r^{-3}\int_J|u(Z,\sigma)-c_Q|\,d\sigma
 \lesssim \fint_J\Theta(z,\sigma)\,d\sigma .
\end{equation}
Averaging over $z\in B_r(x_Q)$ gives the required bound by
$M^\sharp[\tilde N(\nabla u)](x_Q,t_Q)$.

At the inner endpoints, near $\sigma=t_Q$, the factor $1-\chi$ vanishes
identically and nothing is produced; at $\sigma\to\pm\infty$ the endpoint terms
vanish along a sequence by Lemma \ref{lem:endpoint}, the integration by parts
being performed on $[t_Q,\sigma_j]$ and followed by the limit $j\to\infty$; the surviving
integral converges absolutely by \eqref{eq:termII}.
\end{proof}

\begin{remark}[the estimate is an average over $Q$, not a supremum]
Proposition \ref{prop:far} cannot be upgraded to
$\sup_{Q}|w_{\mathrm{far}}-c^*|\lesssim M^{\sharp}[\tilde N(\nabla u)](x_Q,t_Q)$:
by \eqref{eq:vert} term $\mathrm I$ at a boundary point $(y,0)$ is controlled only
by $M_t[\tilde N(\nabla u)](y,\cdot)$, and $\sup_{y\in Q_x}M_t[\tilde N(\nabla u)]$
is bounded on no $L^\theta$. The averaged form \eqref{eq:far} is what the sharp
function needs, so nothing is lost.
\end{remark}

\subsubsection*{The sharp function estimate}

Recall the Fefferman-Stein sharp function
$$b^\sharp(x,t):=\sup_{Q\ni(x,t)}\ \inf_{c\in\R}\ \fint_Q|b-c|\,d\sigma.$$

\begin{proposition}\label{prop:sharp}
For every boundary parabolic cube $Q\subset\pom$,
\begin{multline*}
\inf_{c}\fint_Q|b-c|\,d\sigma\ \lesssim\
\inf_{(x_0,t_0)\in Q}\Big\{M\big[\tilde N(\nabla u)^2\big]^{1/2}
+M\big[S(\nabla u)^2\big]^{1/2}\\
+M\big[A(\nabla u)^2\big]^{1/2}
+M^{\sharp}\big[\tilde N(\nabla u)\big]\Big\}(x_0,t_0).
\end{multline*}
\end{proposition}

\begin{proof}
Take $c=c^*$ from Proposition \ref{prop:far}. By \eqref{eq:split} and
Cauchy-Schwarz,
\[
\fint_Q|b-c^*|
\le\Big(\fint_Q|w_{\mathrm{loc}}|^2\Big)^{1/2}+\fint_Q|w_{\mathrm{far}}-c^*|
\ \lesssim\ \Xi_Q+M^{\sharp}\big[\tilde N(\nabla u)\big](x_Q,t_Q)
\]
by Propositions \ref{prop:L} and \ref{prop:far}. To estimate the far
term at an arbitrary $(x_0,t_0)\in Q$, retain the averages in
\eqref{eq:termI}, \eqref{eq:termII} and \eqref{eq:transitionaverage}
before bounding them by maximal functions. After the spatial averaging,
each term is an average of $\tilde N(\nabla u)$ over $Q_x\times I_k$,
$B_r(x_Q)\times I_k$, or $B_r(x_Q)\times J$. The spatial radii are
comparable to $r$ and the time lengths are at least a fixed multiple of
$r^2$. Enlarge these sets by fixed factors to contain $(x_0,t_0)$.
Each average is then bounded by
$C M_xM_t[\tilde N(\nabla u)](x_0,t_0)$, and the coefficients $2^{-k}$
in the two dyadic sums are summable. The local estimate already holds
with $\Xi_Q$ evaluated at this arbitrary point. We therefore obtain the
stated bound for every $(x_0,t_0)\in Q$ and may take the infimum.
\end{proof}

We can now prove \eqref{eq:target} in the following form.

\begin{proposition}\label{prop:star}
Let $2<p<\infty$ and let $u$ be an energy solution of
$\mathcal L_{\mathbf b}u=0$ in $\om=\R^n_+\times\R$ under
the hypotheses of Theorem~\ref{Thalf}, with
$\tilde N(\nabla u),S(\nabla u)\in L^p(\pom)$.
Then \eqref{eq:target} holds.
\end{proposition}

\begin{proof}
The assumed square-function bound and Lemma~\ref{lem:driftInputs}
give $A(\nabla u)\in L^p(\pom)$. Lemma
\ref{lem:traceNormalization} defines $b$ as a distribution and justifies
the tails of the far part. Propositions \ref{prop:L} and \ref{prop:far}
then give $b\in L^1_{\loc}(\pom)$, and Proposition \ref{prop:sharp} gives
\[
b^\sharp\ \lesssim\ M\big[\tilde N(\nabla u)^2\big]^{1/2}
+M\big[S(\nabla u)^2\big]^{1/2}
+M\big[A(\nabla u)^2\big]^{1/2}+M^{\sharp}\big[\tilde N(\nabla u)\big].
\]
Each term is bounded in $L^p$: for the first three,
$\|M[F^2]^{1/2}\|_{L^p}=\|M[F^2]\|_{L^{p/2}}^{1/2}\lesssim\|F\|_{L^p}$
\emph{if and only if $p>2$}; for the fourth, $M^{\sharp}$ is bounded on $L^p$ for
$p>1$. Hence
\begin{equation}\label{eq:sharpLp}
\|b^\sharp\|_{L^p}\ \lesssim\ \|\tilde N(\nabla u)\|_{L^p}+\|S(\nabla u)\|_{L^p}
+\|A(\nabla u)\|_{L^p} ,
\end{equation}
with no a priori $L^q$ assumption on $b$.

We give the finite-cube form of the sharp function inequality needed here.
For a boundary parabolic cube $Q$, write $b_Q=\fint_Qb$. Then
\begin{equation}\label{eq:localFS}
 \|b-b_Q\|_{L^p(Q)}\le C_p\|b^\sharp\|_{L^p(Q)}.
\end{equation}
To see this for merely locally integrable $b$, subdivide each spatial side
of $Q$ into two and its time interval into four, and repeat. For each cube
$P$ in this grid put $\alpha_P=\fint_P|b-b_P|$. Stop at the maximal proper
subcubes $P'$ for which $\fint_{P'}|b-b_P|>2\alpha_P$.
Their total measure is at most $|P|/2$. On the complement
$|b-b_P|\le2\alpha_P$ a.e., and on each stopped cube
$|b_{P'}-b_P|\le2^{n+2}\alpha_P$, by maximality and the fixed ratio of
child and parent measures. Recursion gives
\[
 |b-b_Q|\le C_n\sum_{P\in\mathcal S}\alpha_P\boldsymbol1_P
 \quad\text{a.e. on }Q,
\]
where $\mathcal S$ is the collection of stopped cubes, including $Q$.
The sets $E_P$ obtained by removing the next stopped cubes from $P$ are
pairwise disjoint and satisfy $|E_P|\ge|P|/2$. Cubes with $\alpha_P=0$
need not be subdivided. The exceptional set of infinite stopping chains
has measure zero, since its measure decreases by a factor of at least
two at each generation.

For a nonnegative $v\in L^{p'}(Q)$, let $M_Qv$ be its dyadic maximal
function in this grid. Since $\alpha_P\le2\inf_P b^\sharp$, we have
\[
 \int_Q\sum_{P\in\mathcal S}\alpha_P\boldsymbol1_P v
 \le4\sum_{P\in\mathcal S}\int_{E_P}b^\sharp M_Qv
 \le4\|b^\sharp\|_{L^p(Q)}\|M_Qv\|_{L^{p'}(Q)}.
\]
The dyadic maximal inequality and duality prove \eqref{eq:localFS}.
One can first take finite sums and then use monotone convergence, so no
finiteness of $\|b\|_{L^p(Q)}$ has been assumed.

Finally take cubes $Q_R$ centred at the origin with side $R\to\infty$.
Choose $\psi\in C_0^\infty(Q_1)$ with $\int\psi=1$ and use $\psi_R$ from
Lemma \ref{lem:traceNormalization}. By that lemma and \eqref{eq:localFS},
\begin{align*}
 |b_{Q_R}|
 &\le |\langle b,\psi_R\rangle|
       +\left|\int_{Q_R}(b-b_{Q_R})\psi_R\right|\\
 &\lesssim R^{-(n+1)/p}
 \big(\|\tilde N(\nabla u)\|_{L^p}+\|b^\sharp\|_{L^p}\big)
 \longrightarrow0.
\end{align*}
Fatou's lemma applied to \eqref{eq:localFS} now gives
$\|b\|_{L^p}\lesssim\|b^\sharp\|_{L^p}$. Combining this with
\eqref{eq:sharpLp} and the drifted area bound
\eqref{eq:driftInputs} proves \eqref{eq:target}.
\end{proof}

\subsection{The estimate for $1<p<2$}\label{ss:halfDrift}

\begin{proposition}[the boundary estimate below two]\label{prop:bilinearTrace}
Let $1<p<2$ and let $u$ be an energy solution of
$\mathcal L_{\mathbf b}u=0$ under the hypotheses of Theorem~\ref{Thalf}.
If $\tilde N(\nabla u),S(\nabla u)\in L^p$, then, for $f=\operatorname{Tr}u$,
\begin{equation}\label{eq:bilinearTrace}
 \|D_t^{1/2}f\|_p
 \lesssim \|\tilde N(\nabla u)\|_p+\|S(\nabla u)\|_p.
\end{equation}
\end{proposition}

\begin{proof}
For $0<p<2$, the vertical--conical comparison
\cite[Proposition 2.1(b), Remark 2.2]{AHM} gives
\[
 \|\mathcal V(F)\|_p\lesssim_p\|\mathcal C(F)\|_p.
\]
The extension in that remark applies to the parabolic boundary metric,
whose balls have measure comparable to $h^{n+1}$; this gives precisely
the measure in the definition of $\mathcal C$. Lemma~\ref{lem:driftInputs} therefore gives
\begin{equation}\label{eq:bilinearInputs}
 \|\mathcal V(h u_{hh})\|_p
 +\|\mathcal V(h^2\partial_tu_h)\|_p
 +\|\mathcal V(h\partial_tu)\|_p
 \lesssim \|S(\nabla u)\|_p+\kappa\|\tilde N(\nabla u)\|_p.
\end{equation}
Write $D=D_t^{1/2}$ and set
\[
 m(s)=(1+s)e^{-s},\qquad P_h=m(hD).
\]
Initially let $\varphi$ be smooth, compactly supported in $x$ and
Schwartz in $t$, with time Fourier support in a compact subset of
$\R\setminus\{0\}$. Put $F(h)=\langle u(h),DP_h\varphi\rangle$,
with pairings in $(x,t)$. Differentiating twice in height and using
$\langle a,D^2b\rangle=\langle\partial_ta,H_tb\rangle$ gives
\begin{equation}\label{eq:bilinearIdentity}
\begin{split}
 \int_0^\infty hF''(h)\,dh
 ={}&\int_0^\infty\langle h u_{hh},hD m(hD)\varphi\rangle\frac{dh}{h}\\
 &+2\int_0^\infty\langle h^2\partial_tu_h,H_tm'(hD)\varphi\rangle\frac{dh}{h}\\
 &+\int_0^\infty\langle h\partial_tu,hH_tD m''(hD)\varphi\rangle\frac{dh}{h}.
\end{split}
\end{equation}
The three scalar test multipliers are
\[
 \psi_1(s)=s(1+s)e^{-s},\qquad
 \psi_2(s)=-se^{-s},\qquad
 \psi_3(s)=s(s-1)e^{-s}.
\]
For each of them, time Littlewood--Paley theory gives
\[
 \|\mathcal V(\psi_i(hD)\varphi)\|_q
 +\|\mathcal V(H_t\psi_i(hD)\varphi)\|_q
 \lesssim_q\|\varphi\|_q,\qquad 1<q<\infty.
\]
Indeed, for $\mathcal H=L^2((0,\infty),dh/h)$, the symbol
$K_i(\tau)(h)=\psi_i(h\sqrt{|\tau|})$ satisfies
\[
 \|K_i(\tau)\|_{\mathcal H}^2=\int_0^\infty|\psi_i(s)|^2\frac{ds}{s},
 \qquad
 |\tau|^2\|\partial_\tau K_i(\tau)\|_{\mathcal H}^2
 =\tfrac14\int_0^\infty|s\psi_i'(s)|^2\frac{ds}{s}<\infty.
\]
The Hilbert-space-valued multiplier theorem in one dimension, followed
by integration in $x$, proves the assertion. The factors $H_t$ may
be applied to $\varphi$ first. The cancellation $m'(0)=0$ is needed
for the second family.

Cauchy--Schwarz in $dh/h$, H\"older in $(x,t)$ and
\eqref{eq:bilinearInputs} show that
\begin{equation}\label{eq:bilinearAbsolute}
 \int_0^\infty h|F''(h)|\,dh
 \lesssim\big(\|S(\nabla u)\|_p
                +\kappa\|\tilde N(\nabla u)\|_p\big)\|\varphi\|_{p'}.
\end{equation}
We next identify the height endpoints. Choose a real even smooth time
multiplier $R$, supported away from zero and equal to one on the time
Fourier support of $\varphi$. Since $Du,\partial_hu\in L^2(\Omega)$,
\[
 u_R=Ru\in H^1((0,\infty);L^2(\pom)).
\]
Thus $F$ has a continuous value at zero, and commutation of $R$ with
the energy trace gives $F(0)=\langle Df,\varphi\rangle$.
The test family $DP_h\varphi$ and its height derivative are bounded
at zero and decay exponentially at infinity in $L^2(\pom)$.
Using $u_R,\partial_hu_R\in L^2(\Omega)$ gives $F,F'\in L^1(0,\infty)$.
The formula for $F''$ holds distributionally on interior height
intervals, and \eqref{eq:bilinearAbsolute} makes its right-hand side
locally integrable. Hence $F'$ is locally absolutely continuous.
Its limit at infinity exists and is zero, since $F''\in L^1(1,\infty)$
and $F'\in L^1$. Similarly $F(\infty)=0$. Absolute Fubini now gives
\[
 F(0)=-\int_0^\infty F'(h)\,dh
     =\int_0^\infty hF''(h)\,dh.
\]
All estimates in \eqref{eq:bilinearAbsolute} use derivatives of the
original $u$ and multipliers on the tests. We do not commute $R$ with
$A$ or $\mathbf b$ or treat $u_R$ as a solution.

For an arbitrary $\varphi\in C_0^\infty(\pom)$, apply this bound to
$P_j\varphi$, where $P_j$ has symbol
$\chi(\tau/j)-\chi(j\tau)$ and $\chi$ is real, even, smooth, compactly
supported and equal to one near zero. These multipliers are uniformly
bounded on $L^{p'}$ and converge strongly to the identity there.
Also $P_jf\to f$ in the energy trace space. The multiplier bound
used after \eqref{eq:p2TraceRegularized} gives
$D P_jf\to Df$ in distributions. We therefore obtain
\[
 |\langle Df,\varphi\rangle|
 \lesssim\big(\|S(\nabla u)\|_p
                +\kappa\|\tilde N(\nabla u)\|_p\big)\|\varphi\|_{p'}.
\]
Duality proves \eqref{eq:bilinearTrace}, without any prior boundary
integrability of $Df$.
\end{proof}

\begin{proposition}[the estimate with drift]\label{prop:halfDrift}
Under the hypotheses of Theorem~\ref{Thalf}, every energy solution
with $\tilde N(\nabla u),S(\nabla u)\in L^p$, $1<p<\infty$, satisfies
\eqref{eq:halfDriftConclusion}.
\end{proposition}

\begin{proof}
In the half-space, the boundary estimate follows from
Proposition~\ref{prop:bilinearTrace} for $1<p<2$, from the direct
height calculation \eqref{A2} for $p=2$, and from
Proposition~\ref{prop:star} for $p>2$. All three arguments now use
$\mathcal L_{\mathbf b}$ under \eqref{eq:halfDriftHyp}; their
constants require no smallness. Lemma~\ref{lem:boundaryBridge}
gives the two maximal estimates for every $1<p<\infty$.

Finally, for a graph domain use the map $\Phi$ of
Lemma~\ref{lem:graphflat}. The pulled-back drift is
\[
 \mathbf b_0=(D\Phi)^{-1}(\mathbf b\circ\Phi).
\]
The map preserves time and has constant Jacobian, so the weak
equation has this drift and no additional time coefficient.
The pointwise drift bounds follow from $h|D^2\Phi|\lesssim1$.
The measures involving $\mathbf b$ and $\partial_t\mathbf b$
transfer by change of variables. Differentiating $\mathbf b_0$
in space gives in addition a term bounded by
$C|D^2\Phi||\mathbf b\circ\Phi|$; its Carleson norm is controlled
by $C K_{\mathbf b}^2\|\nu_\Phi\|_C$.
Likewise the chain rule gives
\[
 |\nabla^2(u\circ\Phi)|
 \lesssim|\nabla^2u|\circ\Phi+|D^2\Phi|\,|\nabla u|\circ\Phi.
\]
Lemma~\ref{lem:driftCM} and \eqref{eq:graphGeomCM} therefore bound
the pulled-back square function by the original square function
plus the gradient maximal function. Energy, boundary measure and
the time operators transfer as in Lemma~\ref{lem:graphtransfer}.
This proves the graph-domain assertion.
\end{proof}

\subsection{Proof of Theorem \ref{Thalf}}

We first give the trace estimate needed to pass from energy data to
general $L^p$ data when $1<p<2$.

\begin{lemma}[weighted boundary trace below two]\label{lem:weightedTrace}
Let $1<q<2$ and let $v$ be a weak solution of $\mathcal Lv=0$ in
$\R^n_+\times\R$ under the coefficient hypotheses of
Theorem~\ref{Thalf}, with zero drift. Fix an aperture $a$ large enough
for Lemmas~\ref{lem:int} and~\ref{lem:vert}, and suppose
$G_v:=\tilde N_a(\nabla v)\in L^q(\pom)$. For each spatial cube
$K\subset\R^{n-1}$ there is a constant $c_v$ such that its boundary
trace $f_v$ satisfies, with $\omega(s)=(1+|s|)^{-3/2}$,
\begin{equation}\label{eq:weightedtrace}
\int_K\int_\R |f_v(x,s)-c_v|\omega(s)\,ds\,dx\le C_{K,q}\|G_v\|_{L^q(\pom)}.
\end{equation}
Moreover, $D_t^{1/2}f_v$ is well-defined in distributions and is the
distributional boundary trace of $D_t^{1/2}v$.
\end{lemma}

\begin{proof}
Let $\varrho$ be the side length of $K$ and $x_K$ its centre.
Choose a nonnegative spatial bump $\varphi\in C_0^\infty(\R^n_+)$ with
$$\int\varphi=1,\qquad\supp\varphi\subset B_{\varrho/2}((x_K,\varrho)),\qquad
\|\varphi\|_\infty\lesssim\varrho^{-n},\qquad\|\nabla\varphi\|_\infty\lesssim\varrho^{-n-1},$$
and set $\Phi_v(s):=\int_{\R^n_+}v(X,s)\varphi(X)\,dX$. The vertical segment argument of Lemma \ref{lem:vert} gives the boundary trace $f_v:=v\big|_{\pom}$. The paths in that proof join $(x,0)$, for $x\in K$, to every point of $\supp\varphi$ with length $\lesssim\varrho$. Their portions away from the vertical segment lie at heights comparable to $\varrho$. Lemma \ref{lem:int} therefore gives, for a.e.\ $(x,s)\in K\times\R$,
$$|f_v(x,s)-\Phi_v(s)|\lesssim\varrho G_v(x,s).$$
The weak equation, tested with $\varphi(X)$ times a smooth compactly supported function of time, gives
$$\Phi_v'(s)=-\int_{\R^n_+}A(X,s)\nabla v(X,s)\cdot\nabla\varphi(X)\,dX.$$
Every point of $\supp\varphi$ lies at height comparable to $\varrho$ within distance $C\varrho$ of each $x\in K$. The same interior bound thus gives $|\Phi_v'(s)|\lesssim\varrho^{-1}G_v(x,s)$ for a.e.\ $(x,s)\in K\times\R$. Integrating these two inequalities over the whole time axis yields
\begin{equation}\label{eq:traceglobal}
\|f_v-\Phi_v\|_{L^q(K\times\R)}\lesssim\varrho\|G_v\|_{L^q(\pom)},\qquad
\|\Phi_v'\|_{L^q(\R)}\lesssim\varrho^{-1}|K|^{-1/q}\|G_v\|_{L^q(\pom)}.
\end{equation}
In particular $\Phi_v\in W^{1,q}_{\loc}(\R)$. Using its continuous representative, put $c_v:=\Phi_v(0)$. H\"older's inequality gives
$$|\Phi_v(s)-c_v|\le |s|^{1-1/q}\|\Phi_v'\|_{L^q(\R)}.$$
Let $\omega(s):=(1+|s|)^{-3/2}$ and $q'=q/(q-1)$. Since $\omega\in L^{q'}(\R)$ and
$$\int_\R |s|^{1-1/q}\omega(s)\,ds<\infty\qquad\mbox{for }1<q<2,$$
splitting $f_v-c_v=(f_v-\Phi_v)+(\Phi_v-c_v)$ and using \eqref{eq:traceglobal} proves \eqref{eq:weightedtrace}.

Let $\psi\in C_0^\infty(\pom)$ have spatial support in $K$. The kernel formula for $D^{1/2}_t$ gives
$$|D^{1/2}_t\psi(x,s)|\le C_\psi\omega(s),\qquad
\int_\R D^{1/2}_t\psi(x,s)\,ds=0.$$
The half derivative preserves the spatial support of $\psi$. Hence \eqref{eq:weightedtrace} makes the distributional pairing
$$\langle D^{1/2}_tf_v,\psi\rangle:=\int_K\int_\R(f_v(x,s)-c_v)D^{1/2}_t\psi(x,s)\,ds\,dx$$
absolutely convergent and independent of the subtracted constant. This also identifies $D^{1/2}_tf_v$ with the distributional boundary trace of $D^{1/2}_tv$. Indeed, the vertical segment in Lemma \ref{lem:vert}, now of length $h$, gives
$$\|v(\cdot,h,\cdot)-f_v\|_{L^q(K\times\R)}\lesssim h\|G_v\|_{L^q(\pom)},\qquad 0<h<\varrho.$$
Pairing this difference with $D^{1/2}_t\psi\in L^{q'}(K\times\R)$ shows that $D^{1/2}_tv(\cdot,h,\cdot)\to D^{1/2}_tf_v$ in distributions as $h\downarrow0$.

\end{proof}

\begin{proof}[Proof of Theorem \ref{Thalf}]
Proposition~\ref{prop:halfDrift} proves the estimate with drift for
every $1<p<\infty$. If $\mathbf b=0$, Lemma~\ref{lem:int} and
Theorem~\ref{thm.S<Np} give
$\|S(\nabla u)\|_p\lesssim\|\tilde N(\nabla u)\|_p$.
Thus the square-function hypothesis follows from the gradient maximal
bound and \eqref{eq:halfDriftConclusion} gives the stronger assertion.

For the Neumann-data assertion assume $(N)_p^{\mathcal L}$, with
$\mathcal L=\mathcal L_0$ and $1<p<\infty$. On energy data, the
solutionwise estimate just proved gives immediately
\[
 \|D_t^{1/2}\operatorname{Tr}u\|_p
 +\|\tilde N(D_t^{1/2}u)\|_p
 +\|\tilde N(H_tD_t^{1/2}u)\|_p
 \lesssim\|\tilde N(\nabla u)\|_p\lesssim\|g\|_p.
\]
We give the extension to arbitrary $g\in L^p(\pom)$.

Choose $g_j\in C_0^\infty(\pom)$ with zero integral and $g_j\to g$
in $L^p$. Such an approximation is obtained by subtracting from each
compact smooth approximation a smooth dilated bump of the same integral,
whose $L^p$ norm tends to zero with its scale. These data belong to the
energy-dual trace space, since their Fourier transforms are bounded
near zero and decay rapidly at infinity, and the squared dual weight
$(|\xi|^4+\tau^2)^{-1/4}$ is locally integrable for $n\ge2$.
Let $u_j$ be their energy solutions. The assumed Neumann estimate gives
\begin{equation}\label{eq:generalDataCauchy}
 \|\tilde N(\nabla(u_j-u_k))\|_p
 \lesssim\|g_j-g_k\|_p.
\end{equation}

Normalize the locally Lipschitz representatives by $u_j(0,1,0)=0$.
Lemma~\ref{lem:int} and integration over boundary shadows imply,
for each compact $K\Subset\Omega$,
\[
 \|\nabla(u_j-u_k)\|_{L^\infty(K)}
 +\|\partial_t(u_j-u_k)\|_{L^\infty(K)}
 \le C_K\|\tilde N(\nabla(u_j-u_k))\|_p.
\]
Indeed, the interior bound at a point of height $h$ is controlled by
the maximal function at every point of a boundary ball of radius
comparable to $h$. Integrating its $p$th power over that ball gives
the asserted bound away from height zero. Joining points to $(0,1,0)$
by spatial and time segments in a fixed interior compact set gives
the same bound for $u_j-u_k$. Thus $u_j$ converges locally uniformly,
and its spatial and time derivatives converge locally in $L^\infty$,
to a reinforced weak solution $u$.

For passage to the boundary weak identity we use $L^1$ convergence.
Put $G_{jk}=\tilde N(\nabla(u_j-u_k))$. At points $(x,h,t)$ above
a fixed bounded boundary box, Lemma~\ref{lem:int} gives
$|\nabla(u_j-u_k)(x,h,t)|\le C G_{jk}(x,t)$.
Integration in $h$, followed by H\"older on the boundary box, controls
the local $L^1$ norm of the gradient by $C\|G_{jk}\|_p$.
Vertical integration from a fixed positive height, using the interior
normalization above, gives the same bound for $u_j-u_k$.
Consequently $u_j$ and $\nabla u_j$ converge in $L^1$ on bounded
boxes meeting the boundary. This argument applies also when $p<2$.
Passing to the weak identities gives, for every smooth test $\zeta$
compactly supported in $\overline\Omega$,
\[
 \iint_\Omega[A\nabla u\cdot\nabla\zeta-u\partial_t\zeta]
 =\int_{\pom}g\zeta.
\]
Thus $u$ has Neumann datum $g$. Fatou on Whitney averages, first
locally and then in the boundary norm, gives
\begin{equation}\label{eq:generalDataGradient}
 \|\tilde N(\nabla(u_j-u))\|_p
 \lesssim\|g_j-g\|_p\longrightarrow0.
\end{equation}
The same estimates show that the normalized limit is independent of
the approximating sequence. For energy data it agrees with the energy
solution, by applying $(N)_p^{\mathcal L}$ to the differences.
This is the general-data solution furnished by completion of the
energy solution operator, and $\|\tilde N(\nabla u)\|_p\lesssim\|g\|_p$.

Put $f_j=\operatorname{Tr}u_j$ and $b_j=D_t^{1/2}f_j$.
The solutionwise estimate, applied to $u_j-u_k$, gives
\[
 \|b_j-b_k\|_p\lesssim\|g_j-g_k\|_p.
\]
Hence $b_j\to\beta$ in $L^p$, with $\|\beta\|_p\lesssim\|g\|_p$.
Lemma~\ref{lem:boundaryBridge} supplies the nontangential trace $f$
of $u$. For $p\ge2$, Lemma~\ref{lem:traceNormalization} and
\eqref{eq:generalDataGradient} identify $\beta=D_t^{1/2}f$.
For $1<p<2$, apply Lemma~\ref{lem:weightedTrace} to $v=u_j-u$.
For each $\psi\in C_0^\infty(\pom)$, it gives
\[
 |\langle b_j-D_t^{1/2}f,\psi\rangle|
 \le C_{p,\psi}\|\tilde N(\nabla(u_j-u))\|_p\longrightarrow0.
\]
The same trace estimates identify $D_t^{1/2}f$ with the distributional
boundary trace of $D_t^{1/2}u$. Finally \eqref{A1a} gives
\[
 \|\tilde N(D_t^{1/2}u)\|_p+
 \|\tilde N(H_tD_t^{1/2}u)\|_p\lesssim\|g\|_p.
\]
This uses only $(N)_p^{\mathcal L}$ at the specified exponent.
\end{proof}

\bibliographystyle{alpha}

\end{document}